\documentclass[12pt]{article}

\usepackage{amssymb}
\usepackage{latexsym}
\usepackage{amsmath}
\usepackage{mathrsfs}
\usepackage{upgreek}
\usepackage{mathabx}
\usepackage{esint}
\usepackage{xcolor}
\usepackage{bbm}
\usepackage{hyperref}
\usepackage{empheq}
\usepackage{cleveref}
\usepackage{mathtools}
\usepackage{amsthm}
\usepackage[title]{appendix}
\usepackage{amsfonts}
\usepackage{relsize}
\usepackage[margin=1in]{geometry} 
\def\scriptO{{{\it O}\kern -.42em {\it `}\kern + .20em}}
\usepackage{esvect}
\usepackage{enumerate}
\usepackage{cases}
\makeatletter

\newcommand{\Rmn}[1]{\expandafter\@slowromancap\romannumeral #1@}
\makeatother

\usepackage{graphicx}
\usepackage[left,modulo]{lineno}
\newtheorem{thm}{Theorem}[section]

  \newenvironment{hproof2}{%
  \proof}{\endproof}
  
  \newenvironment{hproof4}{%
  \proof}{\endproof}
  
 \newtheorem{Th}[thm]{Theorem}
  \newtheorem{op}[thm]{Open Question}
  
 \newtheorem{Prop}[thm]{Proposition}
 \newtheorem{Co}[thm]{Corollary}
\newtheorem{Lm}[thm]{Lemma}

\newtheorem{Dfi}[thm]{Definition}
\newtheorem{Rm}[thm]{Remark}
\newtheorem{Con}[thm]{Conjecture}
\numberwithin{equation}{section}
\newcommand{\be}{\begin{equation}}
\newcommand{\ee}{\end{equation}}
\newcommand{\bg}{\begin{gather}}
\newcommand{\eg}{\end{gather}}
\newcommand{\ba}{\begin{align}}
\newcommand{\ea}{\end{align}}
\newcommand{\bad}{\begin{aligned}}
\newcommand{\ead}{\end{aligned}}
\newcommand{\R}{\mathbb{R}}
\newcommand{\N}{\mathbb{N}}

\newcommand{\C}{\mathbb{C}}
\newcommand{\s}{\mathbb{S}}

\newcommand{\mca}[1]{\mathcal{#1}}

\newcommand{\Hb}{\mathbb{H}}

\newcommand{\nf}{\infty}
\newcommand{\Qr}{\mathcal{Q}}

\def\avint{\mathop{\mathchoice{\,\rlap{-}\!\!\int}
                              {\rlap{\raise.15em{\scriptstyle -}}\kern-.2em\int}
                              {\rlap{\raise.09em{\scriptscriptstyle -}}\!\int}
                              {\rlap{-}\!\int}}\nolimits}

\def\avint{\mathop{\,\rlap{-}\!\!\int}\nolimits}

\def\ovwe{\text{\larger[1.5]$\we$}}
\def\smi{\setminus}

\def\Xint#1{\mathchoice
{\XXint\displaystyle\textstyle{#1}}%
{\XXint\textstyle\scriptstyle{#1}}%
{\XXint\scriptstyle\scriptscriptstyle{#1}}%
{\XXint\scriptscriptstyle\scriptscriptstyle{#1}}%
\!\int}
\def\XXint#1#2#3{{\setbox0=\hbox{$#1{#2#3}{\int}$ }
\vcenter{\hbox{$#2#3$ }}\kern-.6\wd0}}

\def\dashint{\Xint-}

\usepackage{dsfont}
\def\mkt{\mkern2mu}
\def\mko{\mkern1mu}

\def\La{\Lambda}

\def\lf{\left}
\def\ve{\vec e}
\def\ot{\otimes}
\def\de{\delta}
\def\rg{\right}

\def\al{\alpha}

\def\la{\lambda}
\def\na{\nabla}

\def\nap{\nabla^{\perp}}
\newcommand{\we}{\wedge}
\def\imm{\textup{imm}}
\def\std{\textup{std}}

\def\bwe{\bigwedge}
\def\loc{\textup{loc}}

\def\vae{\varepsilon}
\def\e{\epsilon}
\def\vp{\varphi}

\def\ds{\displaystyle}
\def\Om{\Omega}
\def\II{\mathrm{I\!I}}
\def\bII{\vec{\II}}

\def\p{\partial}
\def\bn{\vec{n}}

\def\bH{\vec{H}}

\def\bv{\vec{v}}
\def\bg{\vec{g}}

\def\g{\nabla}

\def\lan{\langle}
\def\ran{\rangle}
\def\bL{\vec{L}}
\def\bR{\vec{R}}

\def\bw{\vec{w}}

\def\bPS{\vec\Psi}

\def\vp{\varphi}
\def\bP{\vec{\phi}}
\def\bP{\vec{\Phi}}
\def\dvol{d\textup{vol}}

\def\Si{\Sigma}

\newcommand{\res}{\mathbin{\hbox{\vrule height 5pt width .4pt depth 0pt\vrule height .6pt width 4pt depth 0pt}}} 
\newcommand{\metricsub}[3]{_{
    \mkern-#3mu
    \smash{\lower #1\hbox{$\scriptscriptstyle #2$}}
}}
\newcommand{\resg}{\res_g}
\newcommand{\dwe}{\mathbin{\dot{\wedge}}}
\newcommand{\dres}{\mathbin{\dot{\res}_g}}
\newcommand{\wres}{\mathbin{\ovs{\ovwe}{\res}_{\mkern -4mu g}}}

\newcommand{\D}{\mathcal D}

\newcommand{\ov}[1]{\overline{#1}}

\newcommand{\ti}[1]{\tilde{#1}}
\newcommand{\ovs}[2]{\overset{#1}{#2}}
\newcommand{\Sp}{\mathbb{S}}

\DeclareMathOperator{\curl}{curl}
\DeclareMathOperator{\dist}{dist}

\DeclareMathOperator{\supp}{supp}
\DeclareMathOperator{\card}{Card}

\newcommand{\lap}{\Delta}
\newcommand{\Imm}{\mathrm{Imm}}

\begin{document}

\title{The Analysis of Willmore Surfaces and its Generalizations in Higher Dimensions}
\author{Tian Lan \and Dorian Martino \and Tristan Rivière}
\maketitle
\begin{abstract}
We review recent progress concerning the analysis of Lagrangians on immersions into $\mathbb{R}^d$ depending on the first and second fundamental forms and their covariant derivatives. 
\end{abstract}

\centerline{ \it In honor of Professor Gang Tian on the occasion of his 65th birthday}


\section{Introduction}\label{sec:introduction}
The study of topological and geometrical properties of manifolds through the search of special metrics ``equipping'' these manifolds and solving some special Partial Differential Equations has a long and rich history which takes its origin maybe in the construction of special curves (like Brachistochrone curves in the XVII\textsuperscript{th} century, or Euler's elastica in the XVIII\textsuperscript{th} century) and constant Gauss curvature metrics on closed surfaces in relation with the uniformization theorem for Riemann Surfaces in the XX\textsuperscript{th} century. The development of what could be called ``intrinsic geometric analysis'' has been the source of spectacular results in differential topology, differential geometry, and complex geometry with the search of constant scalar curvature metrics, Einstein metrics, K\"ahler--Einstein metrics, solutions to the Ricci flow...  
Another branch of geometric analysis is dealing with the study of ``special submanifolds'' within a given Riemannian manifold and its interaction with the geometry of the manifold itself. This branch of geometric analysis is maybe rooted originally both in the calculus of variations with the variational constructions of closed geodesics, the resolution of the Lagrange-Plateau problem in Euclidean space as well in the ``explicit differential geometry'' of special submanifolds such as plane algebraic curves, algebraic surfaces, etc. The central objects of what can be described as ``extrinsic geometric analysis'' are minimal surfaces and their generalizations (constant and prescribed mean curvature surfaces). 
The underlying PDEs are related to the area variations under various constraints.
In the present work, we consider a branch of extrinsic variational geometric analysis that has seen substantial development in recent decades, and concerns immersions that arise as critical points of Lagrangians depending on the first and second fundamental form, and their covariant derivatives.
The simplest and maybe most studied model for such a Lagrangian is the Willmore functional of immersed surfaces into ${\mathbb R}^m$ with $m\ge 3$. Let $\Sigma$ be a closed oriented two-dimensional manifold and let $\vec{\Phi}$ be a $C^2$ immersion of this surface into ${\mathbb R}^m$. Denoting by $g_{\vec{\Phi}}$ the first fundamental form of this immersion and by $\bII_{\vec{\Phi}}$ its second fundamental form, the mean curvature of the immersion is given by
\[
\vec{H}_{\vec{\Phi}} \coloneqq \frac{1}{2}\mbox{tr}_{g_{\vec{\Phi}}}\bII_{\vec{\Phi}}.
\]
The Willmore energy is
\[
W(\vec{\Phi})\coloneqq \int_\Sigma|\vec{H}_{\vec{\Phi}}|^2\ d\textup{vol}_{g_{\vec{\Phi}}}.
\]
This energy has been initially introduced by Sophie Germain and Siméon Denis Poisson in an attempt to generalize to two-dimensional elastic membranes the famous Euler Elastica modeling the free energy of a beam \cite{Germain, poisson}. The very first derivation of the Euler--Lagrange equation of $W$ was made by Poisson himself around 1814 in the case of a graph (see page 60 of \cite{poisson} and chapter 10 of \cite{michelat-these}) in $\R^3$. Using concepts which were not completely clarified at the time (such as the Gauss curvature and the Laplace--Beltrami operator), we can rewrite Poisson's Euler--Lagrange equation for the Lagrangian $W$ in the following form
\begin{equation}
\label{EL-Will}
\Delta_gH_{\vec{\Phi}}+2H_{\vec{\Phi}}\,(H^2_{\vec{\Phi}}-K_{\vec{\Phi}})=0,
\end{equation}
where $H_{\vec{\Phi}}=\vec{n}_{\vec{\Phi}}\cdot\vec{H}_{\vec{\Phi}}$ is the mean curvature and $\vec{n}_{\vec{\Phi}}$ the unit Gauss map to the immersion, $\Delta_g$ is the negative Laplace--Beltrami operator which reads in local coordinates $(x^1,x^2)$ $\Delta_{g}\cdot \coloneqq (\det(g))^{-1/2}\partial_{x_i}(g^{ij}\,(\det(g))^{1/2} \partial_{x_j}\cdot)$ and $K_{\vec{\Phi}}\coloneqq \mbox{det}_{g_{\vec{\Phi}}}( \vec{n}_{\vec{\Phi}}\cdot\bII_{\vec{\Phi}})$ is the Gauss curvature.\footnote{The Euler-Lagrange equation in higher codimension has the same
cubic structure in the second fundamental form, see \cite[Theorem~2.1]{Weiner78}.} 
Throughout this paper, we shall often write $g=g_{\bP}$ when there is no ambiguity. With these notations we have also
\begin{equation}
\label{mean}
\vec{H}_{\vec{\Phi}}= H_{\vec{\Phi}}\ \vec{n}_{\vec{\Phi}}= \frac{1}{2}\,\left( \vec{n}_{\vec{\Phi}}\cdot\Delta_{g}\vec{\Phi}\right)\ \vec{n}_{\vec{\Phi}}= \frac{1}{2}\,\Delta_{g}\vec{\Phi}.
\end{equation}
Comparing \eqref{mean} with \eqref{EL-Will}, it appears that Poisson's Euler--Lagrange equation is a degenerate nonlinear fourth-order equation combining analytical difficulties which were way too advanced for the early XIX\textsuperscript{th} century. Hence in a natural way, in the following century, the attention has been exclusively devoted to special solutions solving a simpler second order version of (\ref{EL-Will}):
\[
H_{\vec{\Phi}}=0.
\]
This is the minimal surface equation corresponding to immersions that are critical points of the area.\\

Poisson's equation (\ref{EL-Will}) was rediscovered by W. Schadow\footnote{See the comment by Wilhelm Blaschke in~\cite[Exercise~7, $\S$ 83]{Bla}.} and also appeared in the PhD work of Gerhard Thomsen~\cite{thom} a century later (in 1923) in the framework of conformal geometry. Hence, the study of the Lagrangian $W$ was motivated by the merging of two theories: conformal geometry and minimal surfaces theory. 
On the one hand they proved that $W$ is invariant under the action of generic\footnote{Precisely, for any conformal transformation $\Psi$ from ${\mathbb R}^3\cup\{\infty\}$ into itself and for any immersion $\vec{\Phi}$ of an oriented closed two dimensional manifold $\Sigma$ into ${\mathbb R}^3$ such that $\Psi^{-1}(\{\infty\})\cap \vec{\Phi}(\Sigma)=\emptyset$ there holds
\[
W(\Psi\circ\vec{\Phi})=W(\vec{\Phi})\ .
\]} conformal transformations on the other hand, as it has been observed above, minimal surfaces are stable  critical points of $W$.
Probably because of the absence of known examples that were not just composition of minimal immersions with conformal transformations, the interest in functional $W$ disappeared again from the mathematical literature for a few decades until the famous paper \cite{Wil} by Thomas Willmore in 1965. In this paper Willmore proved that for any immersion $\vec{\Phi}$ of a closed surface $\Sigma$ the following lower bound holds
\[
W(\vec{\Phi})\ge 4\pi.
\]
Moreover, there is equality if and only if $\vec{\Phi}(\Sigma)$ is a round sphere. This lower bound could be interpreted as a Fenchel type theorem for surfaces or as a variant of the famous Chern--Lashof inequality (see \cite{ChLa}) stating that
\[
\int_\Sigma|K_{\vec{\Phi}}|\ d\textup{vol}_{g_{\vec{\Phi}}}\ge 4\pi.
\]
There is equality if and only if $\vec{\Phi}(\Sigma)$ is a convex surface. Willmore then conjectured that if $\Sigma$ is not diffeomorphic to $\s^2$ then the lower bound is increased and should be equal to $2\pi^2$, with equality if and only if $\vec{\Phi}(\Sigma)$ is a compact conformal transformation of the 2-torus obtained by rotating the vertical circle included in the $x-z$ plane of radius 1 and centered at the point $(\sqrt{2},0,0)$. Since this paper the Sophie Germain Functional was called {\it Willmore Functional}.\\

The first analytical work on the Willmore functional $W$ is the paper by Leon Simon \cite{Sim} in which he shows that the infimum of $W$ is achieved among all possible $C^2$ immersions of the torus by a smooth immersion.
Relying for a large part on the analysis of \cite{Sim}, Matthias Bauer and Ernst Kuwert in \cite{BaKu}, by proving some clever strict inequalities on connected sums of immersions of surfaces, have been able to extend Simon's result to any genus, that is, for any genus $g$ the infimum of $W$ among all immersions of the surface of genus $g$ is achieved. 
What should be noticed regarding the analysis in \cite{Sim} is that first the minimization is considered in a measure-theoretical sense using the theory of varifolds, then the regularity relies a lot on comparison arguments with local competitors which are graphs satisfying the biharmonic equation (i.e. the linearised version of (\ref{EL-Will})). In this sense, the analysis in \cite{Sim} is very much restricted to minimization operations.\\

The third author of the present work has been considering in the early 2000s the possibility to use more classical functional analysis to perform the minimization of $W$ with the perspective of considering more general variational arguments such as min-max operations, or the study of the associated gradient flow. One of the first difficulties was to obtain an Euler--Lagrange equation compatible with the functional space in which the variations of $W$ are studied. In order to illustrate this difficulty, it could be interesting to look at a simpler framework and to go one dimension lower. We consider the variations of the Euler Elastica for immersions $\vec{\Phi}$ of the segment $[0,1]$ into the plane $\R^2$ given by
\[
E(\vec{\Phi})=\int_{[0,1]}\kappa^2_{\vec{\Phi}}\ dl_{\vec{\Phi}},
\]
where $\kappa_{\vec{\Phi}}$ is the curvature of the immersion and $dl_{\vec{\Phi}}$ is the length element of the immersion. The critical points to $E$ are known to satisfy the Euler--Lagrange equation
\begin{equation}
\label{EEEL}
2\,\ddot{\kappa}_{\vec{\Phi}}+\kappa^3_{\vec{\Phi}}=0.
\end{equation}
This can be seen as the one-dimensional counterpart of the Willmore equation \eqref{EL-Will}. Obviously, this equation does not present any difficulty by itself and can be solved almost explicitly by multiplying by $\dot{\kappa}_{\vec{\Phi}}$, integrating and solving the following first-order ODE using elliptic integrals
\[
\dot{\kappa}^2_{\vec{\Phi}}+\frac{\kappa^4_{\vec{\Phi}}}{4}=c^2_0\qquad\Longleftrightarrow\qquad \frac{\dot{\kappa}_{\vec{\Phi}}}{\sqrt{{c^2_0}-\frac{\kappa^4_{\vec{\Phi}}}{4}}}=\pm 1.
\]
However this is not the solvability of (\ref{EEEL}) which is addressed at this stage, but rather the compatibility between the Euler--Lagrange equation and the Lagrangian from a purely variational perspective. In unit speed parameterization $\vec{\Phi}(t)$ for $s\in[0,L]$ and $L=\int dl_{\vec{\Phi}}$ we have $\kappa_{\vec{\Phi}}\, \vec{n}_{\vec{\Phi}}=\frac{d^2\vec{\Phi}}{ds^2}$ and thus it holds
\[
E(\vec{\Phi})=\int_0^L\left|\frac{d^2\vec{\Phi}}{ds^2}  \right|^2 ds.
\]
Hence, obviously, the natural space in which one should consider the variations of $E$ is the Sobolev space $W^{2,2}([0,L],{\R}^2)$. The problem is that the nonlinearity in the Euler--Lagrange equation
involves the cube of the second derivative of $\vec{\Phi}$ and a priori
\[
\kappa^3_{\vec{\Phi}}=\left(\vec{n}_{\vec{\Phi}}\cdot\frac{d^2\vec{\Phi}}{dt^2}\right)^3\qquad \notin L^1_{\textup{loc}}([0,L]).
\]
This does not define a distribution for an arbitrary $\vec{\Phi}$ in $W^{2,2}([0,L],{\R}^2)$. As paradoxical as it appears, the Euler--Lagrange equation \eqref{EEEL} is not compatible with the Lagrangian itself!\\

The same paradox is present one dimension higher for the Willmore functional. Using (\ref{mean}) we have
\[
W(\vec{\Phi})=\frac{1}{4}\int_\Sigma |\Delta_g\vec{\Phi}|^2\ d\textup{vol}_g.
\]
Naturally the function space to be considered is the space $W^{2,2}(\Sigma,{\mathbb R}^m)$ and, even assuming that the metric $g_{\vec{\Phi}}$ defined on $T\Sigma$ and its inverse $g^{-1}_{\vec{\Phi}}$ would be bounded in $L^\infty(\Sigma)$, the non-linearity 
in the Euler--Lagrange (\ref{EL-Will}) is containing terms such as
\[
2\,H_{\vec{\Phi}}\,|H_{\vec{\Phi}}|^2\qquad \notin L^1_{\textup{loc}}(\Sigma).
\]
As well, this does not define a distribution a priori. Hence, the question of constructing an ad-hoc framework to study the variations of $W$ was open. To that aim, the third author introduced in 2010 (paper published in 2014 \cite{Riv14}) the notion of {\it $W^{2,2}$ weak immersions}:

A map $\vec{\Phi}\in W^{1,\infty}\cap W^{2,2}(\Sigma,{\mathbb R}^m)$ is said to be a 
{\it  $W^{2,2}$ weak immersion}\footnote{In the original work as well as in subsequent works such as \cite{Ri16} the author introduced the space of {\it weak immersions} ${\mathcal E}_\Sigma$ which only requires $\vec{\Phi}\in W^{1,\infty}$ as well as (\ref{weakimm}) and the fact that the Gauss map is in $W^{1,2}(\Sigma,G_2({\R}^m))$. This later hypothesis is implied by the assumption $\vec{\Phi}\in W^{1,\infty}\cap W^{2,2}(\Sigma,{\mathbb R}^m)$  together with (\ref{weakimm}). The reverse is not true as observed in \cite{Lan}, that is weak immersions are not necessarily $W^{2,2}$ in a new bi-Lipschitz chart and hence the space of  {\it  $W^{2,2}$ weak immersions} we are defining here is smaller than the original space defined in \cite{Riv14}. It enjoys nevertheless the desired almost weak closure property (see Theorem~\ref{th-wclo}).} if for a given smooth reference metric
$g_0$ on the closed surface $\Sigma$ there exists $C_{\vec{\Phi}}>1$ such that
\begin{equation}
  \label{weakimm}  
  C_{\vec{\Phi}}^{-1}\, |X|_{g_0}^2\le |X|_{g_{\vec{\Phi}}}^2=|\vec{\Phi}_\ast X|^2_{{\mathbb R}^m}\le C_{\vec{\Phi}}\, |X|_{g_0}^2,\qquad \text{for a.e. } p\in \Sigma\,\text{ and all }X\in T_p\Sigma. 
\end{equation}
The interest of considering $W^{2,2}$ weak immersions is due to an ``almost weak closure'' property which is going to be a consequence
of the existence of an underlying smooth conformal structure for any $W^{2,2}$ weak immersion. Precisely we have the following result.
\begin{Th}
\label{th-conf}
Let $\vec{\Phi}$ be a $W^{2,2}$ weak immersion defined on a smooth closed orientable surface $\Sigma$ then there exists a constant Gauss curvature metric $h$ on $\Sigma$ and a function $\alpha\in C^0\cap W^{1,2}(\Sigma,{\mathbb R})$ such that
\begin{equation}
    \label{consgauss}
    g_{\vec{\Phi}}=e^{2\alpha}\, h.
\end{equation}
Moreover $\alpha$ satisfies the Liouville equation
\begin{equation}
\label{liouv}
-\Delta_{h}\alpha=e^{2\alpha}\, K_{g_{\vec{\Phi}}} -K_h.
\end{equation}
The Gauss--Bonnet theorem holds, denoting by $\gamma(\Sigma)$ the genus of $\Sigma$,
\be\label{gau-bon}
\int_\Sigma K_{g_{\vec{\Phi}}}\, d\textup{vol}_{g_{\vec{\Phi}}}= 4\pi\, (1-\gamma(\Sigma)).
\ee
\end{Th}
A proof of the first part of this result was proposed in \cite{Ri16} but a very last ingredient which requires $\vec{\Phi}\in W^{2,2}(\Sigma,{\mathbb R}^m)$
and not only $\vec{n}_{\vec{\Phi}}\in W^{1,2}(\Sigma,G_2({\mathbb R}^m))$ was missing to complete the argument and to ensure that $\phi$ given by \cite[Equation~(1.72)]{Ri16} in the middle of page 323 was a bi-Lipschitz homeomorphism. The complete argument for proving Theorem~\ref{th-conf} is given in Section \ref{sec-weakimm} of the present paper. Observe that
the existence of bi-Lipschitz isothermal charts does not hold if one only assumes that the Gauss map is in $W^{1,2}$ without assuming that the weak immersion is in $W^{2,2}$: in the local sense, a counterexample $z\rightarrow (z^2/|z|,0)\in \R^3$ is provided by the first author in \cite{Lan}, and when $\Sigma=\mathbb S^2$, Plotnikov \cite{Plot} construct a global counterexample by composing the map $z\rightarrow z^2/|z|$ with the stereographic projection.\\

A consequence of Theorem \ref{th-conf}, that is, the existence of an underlying smooth conformal structure, is the following almost sequential weak closure property for $W^{2,2}$ weak immersions proved first in \cite{Riv14} when the underlying conformal class is precompact and in the general case in \cite{LaRi2}.
It relies on Deligne--Mumford compactification of the moduli space of conformal structures (see for instance \cite{Hum} for the definition of nodal surfaces).
\begin{Th}
\label{th-wclo}
Let $\Sigma$ be a closed surface and let $\vec{\Phi}_k$ be a sequence of $W^{2,2}$ weak immersion into ${\mathbb R}^m$ with uniformly bounded $L^2$ norm of the second
fundamental form. Then, up to a subsequence, for any connected component of the limiting punctured nodal surface $\tilde{\Sigma}$ there exist a M\"obius transformation $\Xi_k$ of $\R^m$, a sequence of $W^{2,2}$ bi-Lipschitz homeomorphisms $\psi_k$ and at most finitely many points ${a_1, \dots, a_L}$ of $\tilde{\Sigma}$ such that
\[
\vec{\xi}_k \coloneqq \Xi_k\circ\vec{\Phi}_k\circ\psi_k\ \rightharpoonup\ \vec{\xi}_\infty \qquad\mbox{ weakly in }\quad W^{2,2}_{\textup{loc}}(\tilde{\Sigma}\setminus \{a_1,\dots, a_L\}),
\]
moreover $\vec{\xi}_k$ extends as a possibly branched $W^{2,2}$ weak immersion into ${\mathbb R}^m$.
\end{Th}
Some weeks after the proof of Theorem~\ref{th-wclo} was posted on arXiv, a proof of this result appeared on arXiv as well but assuming the sequence $\vec{\Phi}_k$ to be conformal with respect to some sequence of Riemann surface (see \cite{Kuwertli12}).
It is important to insist at this stage that the restriction to conformal weak immersions is missing the goal posed in \cite{Riv14} of developing  a variational theory for the Willmore energy. It does not fulfill the need of exploring the neighborhood of a weak immersion in order to deduce an Euler--Lagrange equation for instance and Theorem~\ref{th-conf} is an essential tool in that respect.\\

The notion of branched $W^{2,2}$ weak immersion into ${\mathbb R}^m$
is an extension of $W^{2,2}$ weak immersion allowing for isolated branched points $\{a_1,\dots, a_L\}$ at which the gradient of the weak immersion is degenerating. Branched $W^{2,2}$ weak immersions have an underlying conformal structure, that is to say, there exists a constant Gauss curvature metric $h$ on $\Sigma$ such that
\[
 g_{\vec{\Phi}}=e^{2\alpha}\, h.
\]
However the conformal factor $\alpha$ is not bounded in $L^\infty$ anymore at each branched point but has a relatively well understood asymptotic expansion: in a local conformal chart $x=(x^1,x^2)\in D^2$ for $h$ containing exactly one of the branched point, say $a_i$, and for which $x(a_i)=0$, there exists $d_i\in {\N}$ such that (see Lemma A.5 of \cite{Riv14})
\[
\|\alpha(x)-d_i\log|x|\|_{L^\infty(D^2)}<+\infty.
\]
The maximal number of branched points for a given closed two-dimensional manifold $\Sigma$ is bounded by the $L^2$ norm of the second fundamental form. Indeed, it satisfies the generalized Liouville equation
\begin{equation}
\label{general-liouv}
-\Delta_{h}\alpha=e^{2\alpha}\, K_{g_{\vec{\Phi}}} -K_h -\sum_{i=1}^L\ d_i\,\delta_{a_i}.
\end{equation}

Once a weak closure result is known, only part of the task of developing a variational theory for Willmore is fulfilled. Indeed, the paradox of having an Euler--Lagrange equation not compatible with the Lagrangian and the assumption remains. This paradox has been solved
 by the third author some years before Theorem~\ref{th-wclo} was known. He proved in \cite{Riv08} that the Willmore equation  can be rewritten in conservative form: In 3 dimension for instance (\ref{EL-Will}) is equivalent to (omitting the subscript $\vec{\Phi}$)
 \begin{equation}
     \label{will-cons}
     d\ast_{g}\lf[ d\vec{H}-\,2\, H\, d\vec{n}-H^2\,d\vec{\Phi}\rg]=0.
 \end{equation}
In local conformal coordinates this equation is equivalent to
\begin{equation}
  \label{will-cons-local}  
\mbox{div}\left[ \nabla\vec{H}-\,2\, H\, \nabla\vec{n}-H^2\,\nabla\vec{\Phi}\right]=0.
\end{equation}
One can observe at this stage that for a weak immersion there holds
\[
\vec{H}\in L^2,\quad \nabla\vec{n}\in L^2,\quad\mbox{ and }\quad \nabla\vec{\Phi}\in L^\infty.
\]
Therefore, the following quantity is a well-defined distribution
\[
\nabla\vec{H}-\,2\, H\, \nabla\vec{n}-H^2\,\nabla\vec{\Phi}\in W^{-1,2}+L^1.
\]
Using Poincar\'e lemma together with a result by Jean Bourgain and Ha\"\i m Brezis \cite{Bourgain03} we deduce that a weak immersion is a critical point to $W$ if and only if there exists $\vec{L}\in L^2$ such that
\[
\nabla^\perp\vec{L}=2\, H\, \nabla\vec{n}+H^2\,\nabla\vec{\Phi}-\nabla\vec{H},
\]
where $\nabla^\perp\coloneqq(-\partial_{2}\,,\partial_{1})$. At this stage the equation, though meaningful for weak immersions, was presenting the difficulty of having a non linearity $(2\, H\, \nabla\vec{n}+ H^2\,\nabla\vec{\Phi})$ a-priori only in $L^1$, whose Hodge decomposition (omitting to precise boundary condition) is given by 
\[
2\, H\, \nabla\vec{n}+H^2\,\nabla\vec{\Phi}=\nabla\vec{H}+\nabla^\perp\vec{L}.
\]
This difficulty has been overcome by identifying further conserved quantities which have been later on interpreted as direct consequences of Noether theorem by Yann Bernard \cite{Ber}. The Euler--Lagrange equation in divergence form (\ref{will-cons-local}) corresponds to the invariance of the Lagrangian by translation. The invariance by dilation following Noether theorem reads
\begin{equation}
\label{dil}
\mbox{div}\left[\vec{L}\cdot\nabla^\perp\vec{\Phi} \right]=0.
\end{equation}
The invariance by rotation is giving
\begin{equation}
\label{rot}
   \mbox{div}\left[ \vec{L}\times\nabla^\perp\vec{\Phi}+H\ \nabla^\perp\vec{\Phi}\right]=0.
\end{equation}
The next step in \cite{Riv08} consists in introducing the ``primitives'' of the two last conservation laws: there exist $S\in W^{1,2}$ and $\vec{R}\in W^{1,2}$ such that
\[
\nabla^\perp S\coloneqq\vec{L}\cdot\nabla^\perp\vec{\Phi}\qquad\mbox{ and }\qquad \nabla^\perp\vec{R}\coloneqq\vec{L}\times\nabla^\perp\vec{\Phi}+H\ \nabla^\perp\vec{\Phi}.
\]
Finally another result in \cite{Riv08} is the discovery of an equation which relates $S$ and $\vec{R}$ called ``the $\vec{R}-S$ system''
\begin{equation}
    \label{RS-sys}
    \nabla\vec{R}=\vec{n}\times \nabla^\perp\vec{R}+\nabla^\perp S\ \vec{n}.
\end{equation}
It is proved in \cite{Riv08} that the solutions to the $\vec{R}-S$ system combined with \eqref{struct} are smooth by applying integrability by compensation to this first order elliptic system combined with the structural equation 
\begin{equation}
\label{struct}
-2\vec{H}=\nabla S\cdot\nabla^\perp\vec{\Phi}+\nabla \vec{R}\times\nabla^\perp\vec{\Phi}.
\end{equation}

While the use of local conformal coordinates was essential in \cite{Riv08}, Section \ref{sec-willsur} of the present paper is presenting a proof  of the regularity without using conformal coordinates. The motivation is to ``prepare the ground'' for considering regularity questions for solution to the Euler--Lagrange equation to higher dimensional counterpart to the Willmore Lagrangian, which is the subject of the second part of the paper. Before coming to this second part we prove the following approximation theorem (in Section \ref{sec:ApproximationResults}), which is one of the main new result of the present paper.
\begin{Th}
\label{th-approx}
Let $(\Sigma,h)$ be a closed Riemann surface, where $h$ is a Riemannian metric on $\Sigma$ of constant Gaussian curvature $1$, $-1$, or $0$. Let $\vec{\Phi}\in W^{2,2}(\Si,\R^m)$ be a conformal weak immersion and $\alpha\in W^{1,2}(\Sigma)$ be such that
\begin{align*}
    g_{\bP}=e^{2\al}\mko h.
\end{align*}
Then there exist sequences of smooth immersions $(\vec{\Phi}_k)_{k\in \N}\subset C^\nf(\Si,\R^m)$, functions $(\alpha_k)_{k\in\N}\subset C^\nf(\Sigma)$, and constant Gaussian curvature metrics $(h_k)_{k\in\N}$ on $\Si$ such that
\begin{align*}
    \begin{dcases}
    (i)\ g_{\vec{\Phi}_k}=e^{2\alpha_k}\mko h_k,\\[1ex]
        (ii)\ \vec{\Phi}_k\rightarrow\vec{\Phi}  &\text{in }\,W^{2,2}(\Sigma,\R^m),\\[1ex]
        (iii)\:  \alpha_k\rightarrow\alpha &\text{in }\,C^0(\Sigma),\\[1ex]
(iv)\ h_k\rightarrow h &\text{in }\,C^\nf (\Si,T^*\Si\ot T^*\Si).
    \end{dcases}
\end{align*}

\end{Th}
This approximability property shows for instance that in order to prove an analytical property for weak immersions, it is sufficient to prove this property for $C^{\infty}$ immersions and to show that one can pass to the limit. This allows for instance, to define the Willmore flow for arbitrary initial data in the space of $W^{2,2}$ weak immersions. We prove that the density of $C^{\infty}$ immersions into weak immersions with VMO derivatives remains valid in higher dimensions, see Theorem \ref{th:Approx}. This result is fundamental in the recent work \cite{MarRiv2}, where the second and third authors study the closure of the set of weak immersions under a certain Sobolev-type bound on the second fundamental form whose motivation arises from the following context.\\

In 2018, Andrea Mondino and Huy The Nguyen \cite{mondino2018} proved that the Willmore functional is the only (up to a topological constant) conformally invariant functional among all the functionals defined on hypersurfaces of 3-dimensional manifolds and depending only on the first and second fundamental form. However, this is not the case in higher dimensions and codimensions. Easy examples can be found in dimension 4, with for instance, functionals defined on immersed submanifolds $\Sigma^4\subset \R^5$ such as
\begin{align*}
    \int_{\Sigma} \big|\mathring{\II} \big|^4_g\, d\textup{vol}_g, \quad \int_{\Sigma} \mathrm{tr}_g\left(\mathring{\II}^4\right)\, d\textup{vol}_g, \quad \int_{\Sigma} {\det}_g \big(\mathring{\II} \big)\, d\textup{vol}_g.
\end{align*}
Such conformally invariant functionals on submanifolds naturally arise in the context of the AdS/CFT correspondence in physics introduced by Juan Maldacena \cite{maldacena1998}. Broadly speaking, this correspondence states the existence of a duality between some gravitational theories on an asymptotically Anti De Sitter space-time $(X^{d+2},g)$ and some conformal field theories on the conformal boundary $\p_{\infty}X$, see for instance \cite{zwiebach2009} for an introduction. In this context, many questions require a priori to compute the volume of a noncompact minimal hypersurface such as the computation of the entanglement entropy or the expectation value of Wilson lines, see for instance \cite{graham2014} for various applications. In order to give meaning to this volume, one proceeds to a volume renormalization introduced by Henningson--Skenderis \cite{henningson1998} to study the Weyl anomaly. This method can be roughly summarized as follows: first compute the volume in a ball of size $R$, then compute an asymptotic expansion as $R\to +\infty$. Graham--Witten \cite{graham1999} proved in 1999 that one of the term in this expansion provides a conformally invariant functional on submanifolds of $\partial_{\infty} X$. For minimal hypersurfaces of $X^{d+2} = \mathbb{H}^{4}$, they obtained the Willmore energy for surfaces in $\R^3 = \partial_{\infty} \mathbb{H}^4$. In light of \cite{mondino2018}, this result is not surprising, since Graham--Witten proved that one has to find a conformally invariant functional on 2-dimensional surfaces by this procedure, and the Willmore energy is the only possibility. This renormalization procedure has been fully generalized by Gover--Waldron \cite{gover2017} in 2017. \\

The volume renormalization procedure for minimal hypersurfaces has the advantage of producing conformally invariant functionals having minimal hypersurfaces as critical points. However, as discussed in \Cref{sec:GR}, some of these functionals are not bounded from below or from above. Returning to the original idea of Germain and Poisson \cite{Germain,poisson}, one would like to study functionals that measure in some sense the umbilicity of a hypersurface. Such a functional should be able to provide ``best representation" of submanifolds and in particular, this requires to have a lower bound. In \cite{gover2021}, Gover--Waldron proved that conformally invariant functionals can also be stated in terms of singular Yamabe-type problem, a problem raised by Loewner--Nirenberg \cite{loewner1974} in 1974. 
In \cite{glaros2019,gover2021}, the authors prove that given any separating submanifold $\Sigma^{d-1}$ of some Riemannian manifold $(M^d,g)$, if one looks for the ``best" asymptotic expansion of a defining function\footnote{A defining function $s\in C^{\infty}(M)$ is defining for $\Sigma$ if $s=0$ and $|ds|_g=1$ on $\Sigma$. 
The question raised in \cite{gover2021} is to find the highest number $\ell\in\N$ and a defining function $s$ such that we have an expansion of the form $|ds|_g = 1+ O(s^{\ell})$ near $\Sigma$.} $s$ for $\Sigma$, then there is an obstruction, called the \textit{obstruction density}. This obstruction density for $d-1=2$ (i.e. when $\Sigma$ is a surface) and $M^d=\R^3$ turns out to be exactly the left-hand side of \eqref{EL-Will}. As proved by \cite{graham2017,gover2017}, this obstruction can be understood in general as the Euler--Lagrange equation of a coefficient in some renormalized volume expansion reminiscent to \cite{graham1999,graham2020}. These alternatives will be discussed in \Cref{sec:alternatives}. One common point of these functionals is that if $\Sigma^n\subset \R^d$ has even dimension $n\geq 4$, then the Euler--Lagrange equation of any of this generalized Willmore energy has a leading-order term of the form $\Delta_{g_{\Sigma}}^{\frac{n}{2}} \vec{H}_{\Sigma}$. This term is also the leading-order term of the functional 
\begin{align*}
    \bP\in \Imm(\Sigma^n,\R^d) \mapsto \int_{\Sigma^n} \big| \g^{\frac{n-2}{2}}\vec{\II}_{\bP} \big|^2_{g_{\bP}}\ d\textup{vol}_{g_{\bP}}.
\end{align*}
The study of compactness and regularity questions for such functionals has been the subject of recent works that we will describe in \Cref{sec:OP4d}.

\paragraph{Structure of the article.}

In \Cref{sec:OQ}, we review some recent progress surrounding the analysis of Willmore surfaces and their generalizations in higher dimensions. In \Cref{sec-preliminaries}, we set some notations and record some estimates that will be used later. In \Cref{sec-weakimm}, we introduce the setting of weak immersions. We prove that $W^{2,2}$ weak immersions induce a controlled complex structure and obtain Theorem \ref{th-approx}. We also prove that weak immersions have integer densities. In \Cref{sec-willsur}, we provide a new proof of the regularity of Willmore surfaces without the use of conformal coordinates. In \Cref{sec:4d}, we consider the case of scale-invariant Lagrangians on immersed 4-dimensional submanifolds in $\R^5$ with leading order term of the form $\int |dH|^2$. We apply the Noether theorem and compute the associated conservation laws.

\paragraph{Acknowledgments.}

This project is financed by Swiss National Science Foundation, project SNF 200020\textunderscore 219429.

\section{Open questions}\label{sec:OQ}
\subsection{In dimension 2}

In this section, we review some recent advances in the theory of Willmore surfaces (concerning mainly closed Willmore surfaces in codimension 1, but one can ask similar questions in higher codimension) and state some open questions. The field is subject to a fast development, and thus we might not record all the recent literature.

\subsubsection{Construction of Willmore surfaces}

Since the Willmore energy is conformally invariant, minimal surfaces of $\R^3$, $\s^3$ and $\Hb^3$ and their conformal transformations are Willmore.
One of the major problems is the understanding of closed Willmore surfaces, which are not conformal transformations of minimal surfaces. In 1985, Pinkall \cite{pinkall1985} proved that such surfaces exist. Despite the work of Hélein \cite{helein1998}, where he constructed the equivalent of Weierstrass--Enneper parametrization in the theory of minimal surfaces, only few methods for constructing explicit examples of Willmore surfaces are known in codimension 1, and most of them are not conformally invariant. For instance, Babich--Bobenko \cite{babich1993} constructed in 1993 examples of Willmore surfaces with a line of umbilic points by gluing minimal tori of $\Hb^3$ along their common boundary in $\partial_{\infty}\Hb^3 = \R^2$. In 2024, Dall'Acqua--Schätzle \cite{dellacqua2024} proved that any rotationally symmetric Willmore surface with a line of umbilic points is actually a Willmore surface obtained by such construction. To do so, they prove that in this case, the umbilic line should be contained in a plane intersecting the surface orthogonally. By an application of Cauchy--Kovalevskaya theorem, they prove that the mean curvature of the parts of the surface above and below this plane, seen as submanifolds of $\mathbb{H}^3$, completely vanishes. The $\s^1$-equivariant Willmore tori have been classified by Ferus--Pedit \cite{ferus1990} (and actually contains the construction of \cite{pinkall1985}) by analyzing the geodesic flow on these surfaces. The understanding of Willmore graphs has also been subject of numerous works. For example, it is known that entire Willmore graphs with finite energy and entire Willmore graphs of radial functions are flat planes \cite{chen2013,chen2017,luo2014}. Simply connected, orientable, complete Willmore surfaces with vanishing
Gaussian curvature have been fully classified in \cite{wu2023}.  The loop group methods introduced by Hélein \cite{helein1998}, has later been developed by Ma \cite{ma2006} and Xia--Shen \cite{xia2004} (see also \cite{dorfmeister2019} for a recent survey) and provide a way of constructing Willmore surfaces. However, this is a very abstract method and it would be interesting to know whether one can find more explicit constructions.

\begin{op}
    Is it possible to glue a given closed Willmore surface into another given closed Willmore surface?
\end{op}

A first work in this direction has been done by Marque \cite{marque20212}, where he exhibited the first example of bubbling for Willmore surfaces. Using a parametrization of Willmore spheres introduced by Bryant \cite{bryant1984} (discussed in Section \ref{sec:CGM} below), Marque constructed a sequence of Willmore spheres converging to the gluing of the inversions of a L\'opez surface and an Enneper surface. Li \cite{li2016} has also proved that embedded Willmore surfaces verify some rigidity properties. For instance, an embedded Willmore sphere must be a round sphere. Hence, the setting of immersions (opposed to embeddings) is natural and cannot be avoided.\\ 

An alternative way for such constructions is provided by variational methods, such as minimization or min-max procedures in a given class of surfaces. As an example of application, the Willmore conjecture for genus 1 in codimension 1 has been solved by Marques--Neves \cite{Marques14} in 2012 using a min-max scheme for minimal surfaces combined with Urbano theorem, but remains open for higher genus and general codimension. This conjecture states that up to conformal transformations and in all codimensions, the only minimizer of the Willmore functional in genus 1 is the Clifford torus, whereas in genus larger than 1, the likely candidates are the Lawson surfaces \cite[Conjecture 8.4]{kusner1989}. Even if there is some numerical evidence \cite{hsu1992}, its complete resolution remains largely open. The case of surfaces enjoying a priori the same symmetries as the Lawson surfaces has been recently studied in \cite{KLW2024}.

\begin{Con}[Willmore conjecture in codimension 1 \cite{kusner1989}]\label{con:Willmore_min}
    Given an integer $g>1$, the only minimizers of the Willmore energy among all the smooth surfaces of $\R^3$ are the conformal transformations of the Lawson surfaces of genus $g$.
\end{Con}

On the other hand, min-max procedures have recently become an active topic. One motivation for the study of the Willmore energy, is to apply Morse theory to understand the topology of the space $\Imm(\Sigma,\R^m)$ of immersions of $\Sigma$ into $\R^m$. For instance, a basic problem would be to understand min-max procedures for $W$: if $\gamma \in \pi_k(\Imm(\Sigma,\R^m))$ is a generator of regular homotopy of immersions, we consider the quantity
\begin{align*}
	\beta_\gamma \coloneqq \inf\left\{  \sup_{t\in\s^k} W(\bP_t)  :  (\bP_t)_{t\in\s^k} \in \gamma \right\}.
\end{align*}
One can ask a few natural questions: can we bound these numbers? Does there exist any immersion realizing these optimization problems? If so, how many are they? Can we classify them? A first issue is that the map $W\colon \Imm(\Sigma,\R^m) \to \R$ cannot be a Morse function, partly due to the conformal invariance. To overcome this, the third author adopted a viscosity approach in \cite{riviere2021}, previously introduced for the construction of geodesics in \cite{michelat2016}. Instead of considering $W$ alone, he adds a "smoother" times a small parameter $\sigma>0$ and then let $\sigma\to 0$. He applied successfully this method for a well chosen smoother, in the case $\Sigma = \s^2$, and obtained that the values of min-max procedures for $W$ on $\s^2$ is the energy of a bubble tree. We refer to \cite{riviere2017,riviere2021,michelat2020} for more information about the viscosity method for geodesics or minimal surfaces. A particular case of \cite{riviere2021} is the case $m=3$, i.e. the codimension 1 case, which can be applied to the sphere eversion and constitutes a new step towards the $16\pi$-conjecture, originally formulated by Kusner in the 1980s (see for instance \cite{kusner}).
\begin{Con}[16$\pi$-conjecture]\label{con:16pi}
    Let $\Omega$ be the set of path in $\Imm(\s^2,\R^3)$ joining the standard $\s^2$ to the sphere $\s^2$ with opposite orientation. Consider the following number
    \begin{align*}
        \beta_0 \coloneqq \inf_{\omega\in \Omega} \sup_{\bP\in \omega} W(\bP).
    \end{align*}
    Then $\beta_0$ is reached by a path $\omega\in C^0([0,1], \Imm(\s^2,\R^3))$ such that $\omega_{|[0,1/2]}$ and $\omega_{|[1/2,1]}$ are Willmore flows, and $\omega(1/2)$ is the inversion of the Morin surface, a minimal surface with 4 planar ends. 
\end{Con}

Max--Banchoff \cite{max1981} proved that any path $\omega\in \Omega$ must contain an immersion $\Phi\in \omega$ having a quadruple point (another proof is given by Hughes \cite{hughes1985}). Combining this with Li--Yau inequality shows that $\beta_0\geq 16\pi$.
As a consequence of \cite{riviere2021}, we know that the number $\beta_0$ is equal to the sum of the energies of finitely many branched Willmore spheres. In \cite{MiRi2023}, Michelat and the third author proved an energy quantization of the Gauss map in the $W^{1,(2,1)}$ topology for sequences of immersions arising in the viscosity method. Hence, the need for a classification of branched Willmore spheres became strongly motivated. As we will explain in the following \Cref{sec:CGM}, Bryant \cite{bryant1984} proved that any smooth Willmore sphere is a conformal transformation of a minimal surface in $\R^3$. A first generalization to the branched Willmore spheres has been obtained by Lamm--Nguyen \cite{lamm2015} when the total number of branched points counted with multiplicities is at most 3. Michelat and the third author \cite{michelat2022} successfully classified all ''variational branched Willmore spheres'' which include in particular weak limits of immersed Willmore spheres. The second author in \cite{martino2024} proved that every branched Willmore sphere in ${\mathbb R}^3$ is the inversion of a minimal surface. This implies in particular that $\beta_0$ is a multiple of $4\pi$. \\

In order to pursue the understanding of such min-max schemes, one can study the Morse index of closed Willmore surfaces. The Morse index of a Willmore immersion $\bP$ is, by definition, the number of negative eigenvalues of the quadratic form $\delta^2 W(\bP)$, this is a measure of how far is $\bP$ from being a minimizer. The study of the Morse index for Willmore spheres has been initiated by Alexis Michelat in his PhD thesis \cite{michelat-these} by proving that the Morse index of a Willmore immersion $\vec{\Phi}\coloneqq \frac{ \vec{\Psi} }{ |\vec{\Psi}|^2} $, for some minimal immersion $\vec{\Psi}\colon \Sigma\setminus \{p_1,\ldots,p_n\}\to \R^3$ with $n$ ends, is bounded from above by $n$ in full generality, and by $n-1$ when $\vec{\Phi}$ has some branch points of order at most 1 and some branch points of order 2 without any flux. However, Michelat conjectured in \cite{michelat2020} that whether the Willmore surface $\vec{\Phi}$ is branched or not, the Morse index $\mathrm{Ind}_W(\vec\Phi)$ of $\vec{\Phi}$ should be linear in $n$. He also conjectured that if $\vec \Phi$ has no branch point, then $\mathrm{Ind}_W(\vec\Phi)\leq n-3$. To do so, he proved that the computation of the index of $\vec\Phi$ can be reduced to the computation of the index of a matrix of size $n\times n$ whose entries are defined by fluxes of $\vec\Psi$ at each of its ends, together with some coefficients in the asymptotic expansion of some Jacobi fields. This result can be summarized as follows. We denote
\begin{align*}
    V\coloneqq \left\{ u\in C^{\infty}_{loc}\left(\Sigma\setminus \{p_1,\ldots,p_n\}\right) : \frac{u}{|\vec\Phi|^2} \in W^{2,2}(\Sigma)
    \right\}.
\end{align*}
Since $W^{2,2}(\Sigma)\hookrightarrow C^0(\Sigma)$, we can define the following quantity for any $u\in V$, 
\begin{align*}
    \mathrm{Eval}(u,p_i) \coloneqq \left( \frac{u}{|\vec\Phi|^2} \right)(p_i)\in \R.
\end{align*}
We denote $\mathcal{L}\coloneqq \Delta_{g_{\vec\Psi}} - 2 K_{g_{\vec\Psi}}$ the Jacobi operator of the minimal surface $\vec\Psi$. If $u\in V$ verifies $\mathcal{L}^2u=0$, then $u$ verifies the following expansion in a conformal charts near $p_i$ where $|\vec\Psi|^{-2} = |z|^{2m}(1+O(|z|))$, for some numbers $\mathrm{Log}(u,p_i)\in \R$: 
\begin{align*}
    u(z) = \frac{\mathrm{Eval}(u,p_i)}{|z|^{2m}} + \sum_{1\geq k+l\geq 1-2m} \Re\left( \alpha_{k,l}\, z^k\, \bar{z}^l \right) + \mathrm{Log}(u,p_i)\, \log|z| + O(1).
\end{align*}
The main result of \cite{michelat2020} can be stated as (see for instance \cite[Remark 6.5]{michelat2020})
\begin{align*}
    \mathrm{Ind}_W(\vec\Phi) = \dim\left\{ u\in V : \mathcal{L}^2 u =0,\ \sum_{i=1}^n \mathrm{Eval}(u,p_i)\, \mathrm{Log}(u,p_i) <0
    \right\}.
\end{align*}
In the case where $\Sigma=\s^2$ and the ends of $\vec\Psi$ are planar (so that $\vec\Phi$ as no branch point),  Hirsch, Kusner and Mäder-Baumdicker \cite{hirsch2023,hirsch2024}, revisiting Michelat's computations, proved that the Willmore index of $\vec\Phi$ is exactly $n-3$, and thus verified Michelat's conjecture in the case of unbranched Willmore spheres. To do so, they restricted themselves to the study of the following admissible variations 
\begin{align*}
    V'\coloneqq \left\{
        u\in C^{\infty}_{loc}(\s^2\setminus \{p_1,\ldots,p_n\}) : 
        \begin{array}{l}
             \text{ near each $p_i$, we have the expansion }\\
             u(z) = \Re(\alpha_i\, z^{-1}) + \gamma_i\, \log|z| +v_i(z) \\
             \text{ for some }v_i\in C^{2,\beta}(B(p_i,\varepsilon))
        \end{array}
    \right\}.
\end{align*}  
They show (still when $\vec\Phi$ is not branched) in \cite{hirsch2023} that
\begin{align*}
    \mathrm{Ind}_W(\vec\Phi) = \dim\left\{ u\in V' : \exists i\in\{1,\ldots,n\},\ \gamma_i \neq 0\right\} \in \{ n-3,n-2\}.
\end{align*}
In \cite{hirsch2024}, they obtain $\mathrm{Ind}_W(\vec\Phi) =n-3$ by combining a perturbative argument together with the lower semi-continuity of the Morse index for inversions of minimal surfaces.\\

Concerning branched Willmore spheres, the situation is more complex and the bound $n-3$ is no longer valid. In fact, one can show that the Morse index of the inversion of the L\'opez surface (one planar end and one end of multiplicity 3) is exactly 1. It would be interesting to know whether the number $n-1$ is sharp or not for the Morse index of branched Willmore spheres. 

\begin{op}
    Consider $\bP:\s^2\setminus \{p_1,\ldots,p_n\}\to \R^3$ a complete minimal surface with $n$ arbitrary ends such that $0\notin \bP(\s^2\setminus \{p_1,\ldots,p_n\})$ and $\frac{\bP}{|\bP|^2}$ is a branched Willmore surface. Is it true that the Willmore Morse index of $\frac{\bP}{|\bP|^2}$ is linear in $n$? What if we replace $\s^2$ by any closed Riemann surface of higher genus?
\end{op}

In the context of min-max procedures, Michelat \cite{michelat2018} proved that the Morse index of Willmore spheres obtained by realizing some width of Willmore sweepout with the viscosity method is bounded by the number of parameters of this sweepout. The "smoother" used in for this result is the Onofri energy whose second variation is given by a non-local operator. Michelat--Rivière \cite{michelat2023} then proved the upper semicontinuity of the nullity plus the Morse index. Michelat extended this result recently to the case of branched Willmore spheres in \cite{michelat2025} by developing a sharp analysis of weighted Rellich-type inequalities for fourth-order elliptic operators on degenerating annuli. In order to develop min-max procedures when the underlying surface is not a sphere, it would be interesting to have a generic way of estimating the index of a given Willmore surface which is not necessarily a conformal transformation of a minimal surface in $\R^3$. Thanks to the works of Chodosh--Maximo \cite{chodosh2016,chodosh2023}, it is known that the Morse index (for the area functional) is bounded from above and below by the topology of the surface. Since this also holds for smooth Willmore spheres (by the results of Michelat and Hirsch--Kusner--Mäder-Baumdicker), one can wonder whether such a property also holds or not for Willmore surfaces.

\begin{op}
    Can we estimate the Morse index of a given smooth Willmore surface of genus $g\geq 1$ by its topology and its Willmore energy?
\end{op}

\begin{op}
    Can we relate the Willmore index of conformal transformations of minimal surfaces in $\R^3$, $\Hb^3$ or $\s^3$ with their area index?
\end{op}

Now that the Willmore spheres have been classified, one can start working on a possible use of the Willmore functional as a quasi-Morse functional to describe the set of immersions of $\s^2$ into $\R^3$, in the same manner as geodesics and minimal surfaces are used to describe the fundamental groups of manifolds, see for instance \cite{klingenberg1978,moore2017}.

\begin{op}
    Can we describe the topology of the set $\Imm(\s^2,\R^3)$ using the Willmore functional as a pseudo-Morse function?
\end{op}

\subsubsection{Conformal Gauss map approach}\label{sec:CGM}

In 1984, Bryant \cite{bryant1984} (see also \cite{bryant1987}) proved that to each smooth immersion $\bP\colon \Sigma\to \R^3$, one can associate a map $Y\colon \Sigma\to \Sp^{3,1}$, where $\Sp^{3,1}$ is the De Sitter space 
\begin{align*}
    \s^{3,1}\coloneq \left\{ p\in \R^5 : |p|^2_{\eta} =1\right\}, \qquad \eta \coloneq (dx^1)^2 + \cdots+(dx^4)^2 - (dx^5)^2.
\end{align*}
Moreover, the map $Y\colon (\Sigma,g_{\bP})\to (\s^{3,1},\eta)$ is harmonic if and only if $\bP$ is Willmore, and is a space-like immersion away from the umbilic points of $\bP$. We refer to \cite{marque19} for an introduction and to \cite{bernard2023} for $\varepsilon$-regularity properties in this context. Thanks to this duality, Bryant exhibited the quartic $\Qr\coloneq \eta(\partial^2_{zz} Y, \partial^2_{zz} Y)\, (dz)^4$ which is holomorphic when $\bP$ is Willmore. This quartic can be seen as a measure of the distance of $\bP$ to the set of conformal transformations of minimal surfaces in $\R^3$. Indeed, if $\omega \coloneq \mathring{\II}(\partial_z,\partial_z) (dz)^2$ is the complex tracefree curvature of $\bP$, $K_Y$ is the Gaussian curvature of the metric $g_Y\coloneq Y^*\eta$ on $\Sigma$ and $K_Y^{\perp}$ is the curvature of the normal bundle of $Y$ in $\s^{3,1}$ (this is a 2-dimensional vector bundle), then Palmer \cite{palmer1991} proved in 1994 that 
\begin{align*}
    \frac{\Qr}{\omega^2} = (1-K_Y + i\, K_Y^{\perp})\, \frac{|\mathring{\II}|^2_{g_{\bP}} }{2}.
\end{align*}
The above formula is well-defined outside of the umbilic set of $\bP$, that is to say when $\omega\neq 0$. By the work of Bryant \cite{bryant1984}, we have that if $\Qr=0$ (i.e. $K_Y^{\perp}=0$ and $K_Y=1$), then $\bP$ is a conformal transformation of a minimal surface in $\R^3$. If we relax this information, then Marque \cite{marque2021} proved that $K_Y^{\perp}=0$ if and only if $\bP$ is the conformal transformation of a minimal surface in $\R^3$, $\Sp^3$ or $\Hb^3$. This approach seems very promising, but first requires a good understanding of the umbilic set and only a few advances have been made in understanding the conformal Gauss map of Willmore surfaces in codimension 1.\\

One of the major achievement in \cite{bryant1984} is the classification of smooth Willmore spheres. Indeed, since there are no non-zero holomorphic quartic differentials on $\Sp^2$, every smooth Willmore sphere must satisfy $\Qr=0$ and thus, be a conformal transformation of a minimal surface in $\R^3$. If $\bP\colon \s^2\to \R^3$ is a branched Willmore immersion (for instance obtained in a bubble tree), then the equation $\bar{\partial}\Qr=0$ is a-priori not verified at the branch points of $\bP$, hence $\Qr$ is a-priori merely meromorphic and it is not clear that Bryant's classification extends to branched Willmore immersions. Lamm--Nguyen \cite{lamm2015} computed an asymptotic expansion of $\Qr$ across branch points and proved that $\Qr=0$ if the number of branch points of $\bP$ is at most 3, when counted with multiplicities. This idea can also be seen as an application of Liouville theorem for holomorphic functions, as explained in \cite[Theorem 2.6.3]{marque19}. Then Michelat--Rivière \cite{michelat2022} proved that a branched Willmore sphere obtained as a weak limit of smooth Willmore surfaces or obtained in a bubbling process of a sequence of smooth Willmore surfaces (called a ''variational branched Willmore sphere'') also verifies $\Qr=0$, by showing that the limit of Bryant's quartics must remains holomorphic. The full classification has recently been obtained by the second author in \cite{martino2024}, by adapting the analysis of the conformal Gauss map developed in \cite{bernard2023} in order to obtain $\varepsilon$-regularity on the conformal Gauss map and \cite{martino2025} in order to prove an energy quantization result for Willmore surfaces with bounded Morse index. \\

The next step, which would be an interesting starting point for the Willmore conjecture, is to understand the geometry of smooth Willmore surfaces $\bP\colon \Sigma\to \R^3$ when $\Sigma$ has a positive genus. On the one hand, Palmer \cite{palmer1991} proved that Bryant's quartic can be interpreted as the curvature of the conformal Gauss map. Moreover, he proved that if $\Qr$ verifies the pointwise condition $|\Qr|_{g_Y}>1$, then $\Sigma$ is a torus and $\bP$ is unstable. On the other hand, Fisher-Colbrie \cite{fischer1985} proved that stable minimal surfaces in a positively curved manifold have finite total Gauss curvature. Now that some regularity properties have been understood in \cite{bernard2023,marque2025,martino2025}, it would be interesting to study stability questions from the viewpoint of the conformal Gauss map.

\begin{op}
    Can we prove that the Gauss curvature of the conformal Gauss map of a stable Willmore surface in $\R^3$ is integrable?
\end{op}

In higher codimension, the conformal Gauss map has already been studied \cite{montiel2000,ejiri1988,ma,ma2017} and there are Willmore spheres that are not conformal transformations of minimal surfaces. For instance, Montiel \cite{montiel2000} classified all the Willmore spheres in $\R^4$ and proved that in addition to conformal transformations of minimal surfaces, Willmore spheres can also be obtained as a rational curve in $\mathbb{CP}^3$ via the Penrose twistor fibration $\mathbb{CP}^3\to \s^4$. This classification has also been extended in \cite{michelat2022} to branched Willmore spheres arising as limits of smooth Willmore surfaces. The classification of smooth Willmore spheres in arbitrary codimension is much more involved and has recently been studied in \cite{MPW2025}.

\begin{op}
    Is it true that any branched Willmore spheres in $\R^4$ belongs to one of these two classes: either a conformal transformations of a branched minimal surface of $\R^4$, or \textit{via} the Penrose twistor fibration $\mathbb{CP}^3\to \s^4$ a branched rational curve in $\mathbb{CP}^3$?
\end{op}

One of the difficulties with such an approach comes from the size of the umbilic set of Willmore surfaces. Unlike minimal surfaces, Babich--Bobenko \cite{babich1993} proved that Willmore surfaces can have curves of umbilic points. At these points, the geometry of the conformal Gauss map degenerates. Examples of Willmore surfaces having such curves have been constructed by Babich--Bobenko \cite{babich1993} and by Dall'Acqua--Schätzle \cite{dellacqua2024}. However, the structure of such sets is still not fully understood. Schätzle \cite{schatzle2017} proved that the umbilic set of a smooth Willmore immersion is composed of isolated points and one dimensional real-analytic manifolds without boundary. 
In the recent work \cite{marque2025}, Marque and the second author proved that the structure of the umbilic set provides strong information on the geometry of Willmore immersions and on the value of its energy. For instance, if a Willmore surface has a geodesic umbilic line, then it is of Babich--Bobenko type.

\begin{op}
    What properties do the umbilic lines of closed Willmore surfaces satisfy? Do they satisfy any particular equation? Can we describe their location on the surface? Can we estimate their size in terms of the Willmore energy or any other geometric quantities?
\end{op}

One can also wonder about the geometric interpretation of the Bryant's quartic. A natural question would be to understand the set of Willmore surfaces having the same given Bryant's quartic, or even the set of quartics on a given surface that can be achieved as the Bryant's quartic of some Willmore immersion. For instance, we can consider two immersions $\bP_1,\bP_2\colon \Sigma\to \R^3$ having the same Bryant's quartic. Then one can roughly think of this restriction as a relation between the second fundamental forms of their conformal Gauss maps. If their induced metrics $g_{\vec\Phi_1}$ and $g_{\vec\Phi_2}$ are also conformal, then one can hope that these two conformal Gauss maps coincide, and thus $\vec\Phi_1$ and $\vec\Phi_2$ should be a conformal transformation of the other one.

\begin{op}
    Given two Willmore surfaces $\bP_1,\bP_2\colon \Sigma\to \R^3$ such that their Bryant's quartic satisfies $\Qr_1=\Qr_2\neq 0$. Is it possible to relate $\bP_1$ to a conformal transformation of $\bP_2$? It would be interesting to study the possible differences.
\end{op}

\begin{op}
    Given a complex structure on a given closed surface $\Sigma$, can we describe the set of holomorphic quartics which are the Bryant's quartic of some Willmore immersion?
\end{op}

The Poisson problem has been introduced in \cite{dalio2020} (analog of the Plateau problem for Willmore surfaces), where the authors solve the Plateau problem for data given for the immersion and the Gauss map, see also \cite{pozetta2021}. Due to the conformal invariance of the Willmore energy, it would be natural to solve the Plateau problem for conformally invariant data on the boundary.

\begin{op}
    Solve the Poisson problem for boundary data given by a boundary curve $\Gamma\subset \R^3$ and the image of its conformal Gauss map described as a closed curve in $\s^{3,1}$.
\end{op}

\begin{op}
    What is the optimal regularity assumption on the given boundary data in order to solve the Poisson problem? How many solutions are there?
\end{op}

Up to now, the conformal Gauss map approach has been mainly developed for surfaces of a Euclidean space. It would be interesting to study its generalization in manifolds. A possibility is provided by the normal tractor, see for instance \cite{curry2018}.

\subsubsection{Willmore flow}

Geometric flows are designed to produce explicit paths leading to critical points of a given functional. They are useful both from a theoretic and from a practical viewpoint. For instance, Conjecture \ref{con:16pi} is an example of application of the Willmore flow, which roughly says that there is no smooth Willmore sphere of energy strictly between the energy of $4\pi$ and $16\pi$. One could also think about the Willmore flow as a first step before studying the Canham--Helfrich flow  modeling the evolution of membranes of the cells, see for instance \cite{brazda2024,rupp2024} and the references therein. \\

The main analytical difficulty of the Willmore flow (despite the fact that it is a geometric flow, and thus has to be degenerate) is the fact that it is of order 4, meaning that there exists no comparison principle in the spirit of the mean curvature flow. For instance, if an initial data is included in the interior of a given sphere $S$, one can a-priori not say that the flow remains inside the interior of $S$. Geometric flows of order 4 have originally been introduced in the context of biharmonic surfaces \cite{escher1998}. The Willmore flow was first introduced in 2001 independently by Kuwert--Schätzle \cite{kuwert2001} and Simonett \cite{simonett2001}. They proved the existence of a solution in short time, and proved that if the initial data has small energy, then the flow smoothly converges to the round sphere. In 2003, Mayer--Simonett \cite{mayer2003} proved that the Willmore flow does not preserve the embeddedness of the initial data. Blatt \cite{blatt2009} constructed in 2009 an example of initial data of spherical type for which a singularity occurs in the Willmore flow. More recently, Dall'Acqua--Müller--Schätzle--Spener \cite{dall2024} produced examples of flows made of tori of revolution, and also produced degenerating examples. This flow is still widely unexplored. The parametric approach has been introduced by Palmurella--Rivière \cite{palmurella2022,palmurella2024}, where the authors show that there are only two possible kinds of singularity in finite time: either a concentration of energy with a possible branch point or a degeneration of the conformal class. 

\begin{op}
    Is it possible for the conformal class to degenerate in finite time along the Willmore flow?
\end{op}

In order to study the possible degeneracies, one can rely on the existing literature on the bubbling analysis of sequences $\bP_k\colon \Sigma\to \R^3$ of Willmore surfaces. When the underlying complex structure of $(\Sigma, g_{\bP_k})$ remains bounded in the moduli space, Bernard--Rivière \cite{bernard2014} proved in 2014 that some bubbles can appear, but the energy is always preserved. Then Laurain--Rivière \cite{laurain2018} proved in 2018 that when the underlying complex structure degenerates, then only one residue is responsible for a possible loss of energy. This has recently been extended in \cite{li2024}, where Li--Yin--Zhou proved that this residue precisely describes the loss of energy. In 2021, Marque \cite{marque20212} proved that the set of Willmore immersions with energy less than or equal to $12\pi$ is compact. In 2025, the second author \cite{martino2025} proved that the energy is preserved if $\bP_k$ has bounded Morse index, independently of the underlying complex structure. Recently, Michelat \cite{michelat2025} proved the lower semi-continuity of the Morse index under some smallness assumption of the residue (which corresponds to the assumption in \cite{laurain2018}). Since the analysis of Willmore immersions is now understood, it should be possible to study the case of the flow.

\begin{op}
    Can we prove an energy quantization property for the Willmore flow?
\end{op}

Once a singularity has occurred, one would like to continue the flow. However, the initial data is not smooth and the existence of a solution to the Willmore flow is not clear. In particular, the setting of smooth surfaces cannot be used and one has to find the proper notion of solution. 
\begin{op}
    Is it possible to define a solution of the Willmore flow starting from a branched surface? From a nodal surface?
\end{op}

As an application, one can think about the $16\pi$-conjecture and the classification of singularities along the Willmore flow. Lamm--Nguyen \cite{lamm2015} proved that if we start the flow with an initial data $\vec\Phi_0\colon \s^2 \to  \R^3$ verifying $W(\vec\Phi_0)\leq 12\pi$, then the flow starting from $\vec\Phi_0$ can only develop catenoid-type singularities. In a recent work \cite{mader2025}, Mäder-Baumdicker--Seidel identified 4 distinct homotopy classes in the set of immersed spheres with Willmore energy less than or equal to $12\pi$. Two of them contain round spheres and the two others lead to a notion of unavoidable singularities along the flow. A partial study of the set of immersions with energy less than $16\pi$ is also studied in \cite{mader2025}.

\begin{op}[Question 1.13 in \cite{mader2025}]
    Is it possible to classify the homotopy classes of the set of immersions with Willmore energy less than $16\pi$? In particular, given an immersion $\vec\Phi_0 \colon \s^2 \to \R^3$ satisfying $W(\vec\Phi_0)< 16\pi$, under which conditions the Willmore flow $(\vec\Phi_t)_t$ starting from $\vec\Phi_0$ lies in the homotopy class of a round sphere when $W(\vec\Phi_t)<12\pi$? 
\end{op}

\subsubsection{Critical points of the Willmore energy under various constraints}

Instead of looking at critical points of the Willmore functional among the set of all the immersions, one can restrict the set of admissible variations (for example, one needs to fix some geometric constraints for the study of the Canham--Helfrich functional in cell biology). Currently, three main possibilities are studied. \\

One possibility is to study the minimizers (and more generally the critical points) of the Willmore functional under the constraints of fixing both the area of the immersed surface and the volume it encloses. Since the Willmore functional is scale-invariant, this question is equivalent to prescribing the isoperimetric ratio of the immersion $\bP\colon \Sigma\to \R^3$ given by 
\begin{align*}
    & \mathcal{V}(\bP) \coloneqq -\frac{1}{3}\int_{\Sigma} \bP\cdot \vec{n}_{\bP}\, d\textup{vol}_{g_{\bP}}, \\[2mm]
    & \mathcal{I}(\bP) \coloneqq \frac{\textup{vol}_{g_{\bP}}(\Sigma) }{ \mathcal{V}(\bP)^{\frac{2}{3}} }.
\end{align*}
By the isoperimetric inequality, we have
\begin{align*}
    \mathcal{I}(\bP) \geq \mathcal{I}(\s^2) = (36\pi)^{\frac{1}{3}}.
\end{align*}
We can therefore study the following minimization problem for any $R\geq (36\pi)^{\frac{1}{3}}$:
\begin{align*}
    \beta^{\mathcal{I}}(R) \coloneqq \inf\left\{ W(\bP) : \bP\in\Imm(\Sigma,\R^3),\ \mathcal{I}(\bP) = R \right\}.
\end{align*}
Keller--Mondino--Rivière \cite{keller2014} proved in 2014 that there is a minimizer of $\beta^{\mathcal{I}}(R)$ provided that this number verifies some strict inequality. This inequality has been proved for all $R\in \left( (36\pi)^{\frac{1}{3}},+\infty \right)$ when $\Sigma=\s^2$ by Schygulla \cite{schygulla2012} in 2012, when $\Sigma=\mathbb{T}^2$ by Scharrer \cite{scharrer2022} in 2022 and by Ndiaye--Schätzle \cite{ndiaye2015} in all the remaining cases in 2015. 
\begin{op}
    Describe the solutions of $\beta^{\mathcal{I}}(R)$ for $R\in \left( (36\pi)^{\frac{1}{3}},+\infty \right)$. Are they unique? Do they have some symmetries?
\end{op}

Another possibility is to study the minimization of the Willmore energy under a constraint on the total mean curvature. Given an immersion $\bP\colon \Sigma\to \R^3$ (where $\Sigma$ is a Riemann surface, and thus carries an orientation), we denote 
\begin{align*}
    \mathcal{T}(\bP) \coloneqq \frac{1}{\textup{vol}_{g_{\bP}}(\Sigma)^{\frac{1}{2}}} \int_{\Sigma} H_{\bP}\, d\textup{vol}_{g_{\bP}}.
\end{align*}
By Hölder inequality, it holds that
\begin{align*}
    \mathcal{T}(\bP)^2 \leq W(\bP).
\end{align*}
We then study the following problem for any $R\in\R$
\begin{align*}
    \beta^{\mathcal{T}}(R) \coloneqq \inf\left\{ W(\bP)\colon \bP\in\Imm(\Sigma,\R^3),\ \mathcal{T}(\bP) = R \right\}.
\end{align*}
Scharrer--West \cite{scharrer2025} proved the existence of a minimizer for $R\in\left( 0, \sqrt{2}\, \mathcal{T}(\s^2) \right)\setminus \{\mathcal{T}(\s^2)\}$ and studied the properties of the map $(R\in \R\mapsto \beta^{\mathcal{T}}(R))$. For the case $R= \mathcal{T}(\s^2)$, they proved that 
\begin{align*}
    \beta^{\mathcal{T}}(\mathcal{T}(\s^2)) = \inf\left\{
        W(\vec\Phi) : \vec\Phi\in\Imm(\Sigma,\R^3)
    \right\}.
\end{align*}
Hence, the existence of a minimizer for $\beta^{\mathcal{T}}(\mathcal{T}(\s^2))$ should be very related to Conjecture \ref{con:Willmore_min}.

\begin{op}
    Is there a minimizer for $\beta^{\mathcal{T}}(R)$ with $R\notin\left( 0, \sqrt{2}\, \mathcal{T}(\s^2) \right)\setminus \{\mathcal{T}(\s^2)\}$?  
\end{op}

In the borderline case $R=0$, we know by continuity of the map $R\mapsto \beta^{\mathcal{T}}(R)$ that $\beta^{\mathcal{T}}(0)\leq 8\pi$. Knowing whether there exists a minimizer for $\beta^{\mathcal{T}}(0)$ thus amounts to knowing whether $\beta^{\mathcal{T}}(0)=8\pi$ is true (see \cite{scharrer2025}).

\begin{op}
    Can we study the evolution of minimizers of $\beta^{\mathcal{T}}(R)$ when $R$ varies?
\end{op}

The study of $\beta^{\mathcal{T}}$ is related to the study of the following functional, for a given parameter $\bar{H}\in \R$ called the spontaneous curvature
\begin{align*}
    \forall \bP\in \Imm(\Sigma,\R^3),\qquad W_{\bar{H}}(\bP) \coloneqq \int_{\Sigma} (H_{\bP} - \bar{H})^2\, d\textup{vol}_{g_{\bP}}.
\end{align*}
The Helfrich flow has been introduced in 2006 \cite{kohsaka2006}, where Kohsaka--Nagasawa proved the short-term existence of solutions to the gradient flow of $W_{\bar{H}}$ with the enclosed volume and the area prescribed by some given constants $V_0$ and $A_0$. The cases when the initial data has small energy or when it is radially symmetric have been studied by Rupp--Scharrer--Schlierf \cite{rupp2024}. This work extends earlier results of Rupp \cite{rupp2023,rupp20240} concerning the Willmore flow (i.e. the case $\bar{H}=0$) under a prescription of the isoperimetric ratio instead of a prescription of both the enclosed volume and the area. In order to study the flow for more general initial data, where singularities might appear, Scharrer--West \cite{scharrer20252} developed a bubbling analysis for constrained Willmore surfaces in the spirit of \cite{bernard2014,laurain2018}.

\begin{op}
    Can we prove an energy quantization property for the Helfrich flow?
\end{op}

Another important Willmore constrained variational problem consists in studying the variation of Willmore energy under constraints on the Riemann Surface the immersion is inducing.

Let $\vec{\Phi}$ be a weak immersion of a surface $\Sigma$, as   mentioned above it defines a smooth conformal class  --- i.e. a Riemann surface --- which is the one induced by the pull back metric $g_{\vec{\Phi}}=\vec{\Phi}^{\,\ast}g_{{\R}^3}$.  The conformal class defined by this immersion is said to be a degenerate point of the map $\mathcal{C}$ assigning the conformal class of $(\Sigma,g_{\vec{\Phi}})$ in the Teichmüller space if and only if there exists a non zero holomorphic quadratic form\footnote{A holomorphic quadratic form is a holomorphic section of the bundle given by $\wedge^{1,0}T^\ast\Sigma\otimes \wedge^{1,0}T^\ast\Sigma $ where $\Sigma$ is equipped with the conformal structure induced by $g_{\vec{\Phi}}=\vec{\Phi}^{\,\ast}g_{{\R}^m}$} $q$ such that
\begin{equation}
\label{isotherm}
<q,\vec{h}_0>_{WP}=0.
\end{equation}
In the above equation, we denoted $\vec{h}_0$ the Weingarten operator given by
\[
\vec{h}_0\coloneq-2\, \p\pi_{\vec{n}}\,{\otimes}\,\p\vec{\Phi},
\]
where $\pi_{\vec{n}}$ is the linear orthogonal projection map from ${\R}^m$ onto the normal space to the immersion $N_x\Sigma \coloneq (\vec{\Phi}_\ast T_x\Sigma)^\perp$ and $<\cdot,\cdot>_{WP}$ denotes the {\it Weil--Peterson} scalar product given in local conformal coordinates for which $$g_{\vec{\Phi}}=e^{2\lambda}\ [dx_1^2+dx_2^2],$$ by (see \cite{Riv15})
\[
\lf< \phi\ dz\otimes dz, \psi\ dz\otimes dz\rg> \coloneq e^{-4\la} \ \Re\lf[\ov{\phi}\,{\psi}\rg]\ .
\]
Immersions in ${\R}^3$ which are degenerate points of the map ${\mathcal C}$ have been first studied in the late XIXth century under the name of isothermic surfaces (see \cite{Riv13}). These are the surfaces, such that, away from umbilic points there exists locally conformal coordinates which give the principal directions. \\

Assuming now that $(\Sigma,g_{\vec{\Phi}})$ is a non degenerate point under some constraint on the conformal class of the induced metric then there exists a holomorphic quadratic form $q$ such that
\begin{equation}
 \label{conf-const}
 d\ast_{g}\lf[ d\vec{H}-\,3\, \pi_{\vec{n}}(d\vec{H})+\star(\ast_{g}\ d\vec{n}\wedge \vec{H})\rg]=<q,\vec{h}_0>_{WP}\ .
\end{equation}
where $\star$ is the standard Hodge operator on multivectors in ${\R}^m$ (see \cite{Riv15}). This equation is known as the {\it conformal Willmore equation}. In codimension one (i.e. $m=3$) this can be rewritten as follows
\begin{equation}
     \label{will-cons-conf}
     d\ast_{g}\lf[ d\vec{H}-\,2\, H\, d\vec{n}-H^2\,d\vec{\Phi}\rg]=<q,\vec{h}_0>_{WP}\ .
 \end{equation}
The holomorphic quadratic form is a Lagrange multiplier\footnote{The holomorphic quadratic form identifies to the tangent space to the moduli space (theorem 4.2.2 \cite{J06})} associated to the conformal class constraint.
Weak immersions solving (\ref{conf-const}) are proved to be smooth (\cite{bern-rivi-conf}).

\begin{Rm}
    \label{rm-conf-will}
    In \cite{bern-rivi-conf} it is proved that the Conformal Willmore equation (\ref{will-cons-conf}) is equivalent to the existence of $\vec{L}$ satisfying the set of the two conservation laws (\ref{dil}) and (\ref{rot}) 
\begin{equation}
\label{Noet-dil-rot}
    \left\{
    \begin{array}{l}
\displaystyle d\lf(\vec{L}\cdot d\vec{\Phi}\rg)=0,\\[5mm]
\displaystyle d\lf(\vec{L}\times d\vec{\Phi}+H\, d\vec{\Phi}\rg)=0,
    \end{array}
    \right .
\end{equation}
 coming respectively from Noether theorem for dilations and rotations which is itself equivalent to the so-called $(\vec{R},S)$ system (\ref{RS-sys}) combined with the structure equation (\ref{struct}).
For the higher dimensional counterparts to the Willmore functional and in particular for the Dirichlet energy of the mean curvature of a weak immersion of 4-dimensional manifolds into ${\mathbb R}^5$, we are asking in Open Question \ref{op-q-100} whether the conservation laws issued from Noether for dilations and rotations do correspond to the variation of this energy under some constraint.
\end{Rm}

An interesting question is to know whether or not a weak immersion which is both a critical point of the Willmore functional under some constraint on the conformal class  and isothermic (i.e. a degenerate point for the constraint) is solving or not equation (\ref{conf-const}) and is then smooth. The answer to this question is proved to be positive in \cite{Riv15} in full generality assuming that the surface has genus less than 3 or for arbitrary larger genus if the underlying conformal class is avoiding the ``tiny'' subset of the Teichm\"uller Space of $\Sigma$ made of hyper-elliptic points (a similar result has been obtained in \cite{KS13} using different methods but with restrictions both on the energy and the co-dimension). One application of this result is the existence of smooth minimizers of the Willmore energy inside a prescribed conformal classes (Theorem 1.5 \cite{Riv15}).
Once the existence of a smooth minimizer of the Willmore Energy in a fixed conformal class is known some open questions arise naturally.

\begin{op}
\label{op-q-1}
Identify  minimizers of the Willmore energy in any given conformal class. In particular identify the minimizers of the Willmore energy in the class of any flat torus.     
\end{op}

In a first approach one would aim at  giving a lower bound of the Willmore energy in a fixed conformal class. This approach has been initiated in the pioneer paper by P.Li and S.T. Yau \cite{LY}. The paper of Li and Yau has been then followed by several important contributions such as the one of S.Montiel and A.Ros \cite{MoRo} or more recently  by R.Bryant \cite{Bry} (see also recent results pointing in the direction of the open question \ref{op-q-1} in \cite{ndiaye2015} and \cite{HelNd1}).


\subsection{In higher dimensions}\label{sec:OP4d}

Even though geometric properties of generalized Willmore functionals have already been the subject to a lot of works, the study of their analytical properties is relatively new. We list below the few known results and propose different directions of research.

\subsubsection{Graham--Reichert functional}\label{sec:GR}

One advantage of the renormalized volume approach is to provide conformally invariant functionals having minimal surfaces as critical points, which seems to be an important property for physical applications \cite{anastasiou2025}. In dimension 4, this procedure has been applied independently by Zhang \cite{zhang} and Graham--Reichert \cite{graham2020}. They both end up with the following energy, already found by different methods by Guven \cite{guven} in 2005: for an immersion $\bP \colon \Sigma^4 \to \R^5$ the functional is given by
	\begin{align*}
		\mathcal{E}_{GR}(\bP) &\coloneqq \int_\Sigma |d H_{\bP}|^2_{g_{\bP}} - H_{\bP}^2\, |\II_{\bP}|^2_{g_{\bP}} + 7H^4_{\bP} \ d\textup{vol}_{g_{\bP}} .
	\end{align*}
	where $g_{\bP}\coloneqq \bP^*\xi$, $\xi$ is the flat metric on $\R^5$, $\II_{\bP}$ is the second fundamental form of $\bP(\Sigma)$ and $H_{\bP} = \frac{1}{4}\mathrm{tr}_{g_{\bP}} \II_{\bP}$ is its mean curvature. It turns out that the minus sign in the definition is sufficient to make the above functional unbounded from below, as proved in \cite{graham2020,martino2024duality}. The example constructed in \cite{martino2024duality} is based on the observation that the integrand of $\mathcal{E}_{GR}$ is a negative constant for $\Sp^2\times \R^2$, however this is not a closed submanifold. As mentioned by Graham in a private communication, a simple way to obtain a degenerating example of immersions from a closed manifold is to consider an ellipsoid with two semi-axes of length 1 and two others equal to some parameter $a\in(0,+\infty)$ that we can either send to $0$ or to $+\infty$ (since $\mathcal{E}_{GR}$ is scale invariant, both possibilities are equivalent and we recover $\s^2\times \R^2$ by sending $a\to +\infty$). At first glance, it seems that non-compact CMC immersions can provide an interesting source of examples for possible degeneracy of $\mathcal{E}_{GR}$ towards $-\infty$.
    \begin{op}
        Fix a closed submanifold $\Sigma^4\subset \R^5$. Is it possible to classify all the behaviour of immersions $\bP_k\colon\Sigma^4\to \R^5$ such that $\mathcal{E}_{GR}(\bP_k)\xrightarrow[k\to+\infty]{} -\infty$?
    \end{op}

An interesting starting point would be to find bounds in terms of the geometry of the immersion. For example, one can ask the following question.

\begin{op}
    Let $\bP\colon \Sigma^4\to \R^5$ be a smooth immersion and $g$ be its first fundamental form. Given bounds on intrinsic conformally invariant quantities of $g$, such as the total $Q$-curvature, the $L^2$-norm of its Weyl tensor, or the $L^1$-norm of its Bach tensor, can we deduce bounds on $\mathcal{E}_{GR}(\bP)$?
\end{op}

Another important question is the construction of critical points. We know that conformal transformations of minimal hypersurfaces are critical points of $\mathcal{E}_{GR}$, but to the best of the authors' knowledge, no other critical point is known.
\begin{op}
    Are there critical points of $\mathcal{E}_{GR}$ which are not conformal transformations of minimal hypersurfaces?
\end{op}

We now list a few natural questions regarding the two-dimensional case. First, it would be important to know if a generalization of Bryant's quartic exists or not.

\begin{op}
    Assume that $\bP\colon \Sigma^4\to \R^5$ is a smooth critical point of $\mathcal{E}_{GR}$ such that $(\Sigma,g_{\bP})$ carries a complex structure. Can we exhibit a holomorphic differential form? 
\end{op}

Another natural question is to know how many critical points $\mathcal{E}_{GR}$ has, or if there is some topological obstruction to the existence of critical points.

\begin{op}
    Does $\mathcal{E}_{GR}$ have critical points for every topological hypersurface $\Sigma^4\subset \R^5$? What about the critical points which are not conformal transformations of minimal surfaces?
\end{op}

In order to develop the analysis of $\mathcal{E}_{GR}$, it is important to first understand the regularity of its critical points. First computations in this direction have been performed by Bernard \cite{Bernard25}. In these recent works, Bernard rewrote the Euler--Lagrange equation of any conformally invariant functional of the form $\int_{\Sigma^4} |dH_{\bP}|^2_{g_{\bP}} + l.o.t.$, using the Noether theorem to obtain a system in divergence form. We describe these computations in a simpler setting in \Cref{sec:4d}. The full regularity is the subject of the recent work \cite{bernard2025}. It would also be interesting to have explicit formulas for higher-dimensional renormalized volumes.

\begin{op}
    In \cite{graham2020}, the authors provide an abstract formula for the generalized Willmore functional in any dimension. Compute it explicitly for hypersurfaces of $\R^7$. 
\end{op}

\subsubsection{Alternatives}\label{sec:alternatives}

In \cite{anastasiou2025}, the authors argue that the functional $\mathcal{E}_{GR}$ is relevant for physical purposes. However, it would be interesting to look for other possible functionals that could be of interest for geometric purposes. In \cite{blitz2023generalized,gover2020}, the authors define a generalized Willmore energy for hypersurfaces $\Sigma^4 \subset \R^5$ as a conformally invariant functional whose Euler--Lagrange equation has a leading term of the form $\lap^2 H$. They use tractor calculus to generate many different functionals arising from different geometric contexts.\\

The key idea of \cite{blitz2023generalized,gover2020}, which has recently been developed in \cite{allen2025}, is the link between $Q$-curvature and generalized Willmore energies. They show in \cite{gover2020} that one can define an extrinsic version of the GJMS (Graham--Jenne--Mason--Sparling) operators \cite{GJMS1992} on submanifolds. From these operators, one deduces a notion of extrinsic $Q$-curvature. In \cite{blitz2023generalized}, the authors prove that for a 4-dimensional submanifold $\Sigma^4\subset M^{n}$ with $n\geq 5$, the integral of the difference between the extrinsic $Q$-curvature and the intrinsic $Q$-curvature of $\Sigma$ is a generalized Willmore energy. Hence, one could hope to find similarities between the study of the standard $Q$-curvature and generalized Willmore energies. \\

In order to study the space of immersions of a given hypersurface $\Sigma^4\subset\R^5$ into $\R^5$ and to provide a notion of ``best representation", one would need a functional which is bounded from below and, as for the 2-dimensional case, would be a sort of conformally invariant distance of hypersurfaces to totally umbilic sets. One possibility has been introduced in \cite{martino2024duality} motivated by the generalization of the Dirichlet energy for the conformal Gauss map in dimension 2, see also \cite{allen2025} where Allen--Gover propose an alternative in any codimension using the normal tractor: if $\bP\colon \Sigma^4\to \R^5$ is a smooth immersion, then
\begin{align*}
    \mathcal{E}(\bP) = \int_{\Sigma} 4|dH_{\bP}|^2_{g_{\bP}} + \frac{1}{3}|\mathring{\II}_{\bP}|_{g_{\bP}}^4 + 2H^2_{\bP} |\II_{\bP}|_{g_{\bP}}^2 - 4H_{\bP}\, \text{tr}_{g_{\bP}}(\II_{\bP}^3) + 2\,  \text{tr}_{g_{\bP}}(\II_{\bP}^4)\ d \textup{vol}_{g_{\bP}}. 
\end{align*}
A priori, minimal hypersurfaces of $\R^5$ are not critical points of $\mathcal{E}$ anymore. However, it is proved in \cite{martino2024duality}, that $\mathcal{E}(\bP)\geq 0$ with equality if and only if $\bP(\Sigma)$ is a round sphere. Furthermore, $\mathcal{E}$ is coercive, in the sense that there exists a universal constant $C>0$ such that for any immersion $\bP\colon\Sigma\to \R^5$, it holds
\begin{align}\label{eq:coercivity}
    \int_{\Sigma} |d\II_{\bP}|^2_{g_{\bP}} + |\II_{\bP}|^4_{g_{\bP}}\, d\textup{vol}_{g_{\bP}} \leq C\, \mathcal{E}(\bP).
\end{align}
We refer to the appendix of \cite{martino2024duality} for various relations between the $L^p$-norms of $H$, $\II$, $\mathring{\II}$ and their derivatives. Under this coercivity property, one could hope to generalize the known results of Willmore surfaces to the higher dimensional case.

\begin{op}
    Given a 4-dimensional manifold $\Sigma^4$, can we explicitly compute or at least estimate the following number:
    \begin{align*}
        \inf\left\{ \mathcal{E}(\bP) \colon \bP\in \mathrm{Imm}(\Sigma^4,\R^5) \right\}.
    \end{align*}
\end{op}

\begin{op}\label{qu:minimize1}
    Can we minimize $\mathcal{E}$ among all the immersions of a given closed 4-dimensional manifold $\Sigma$ into $\R^5$? If yes, what are the minimizers?
\end{op}

\begin{op}
    Can we exhibit a generic class of critical points of $\mathcal{E}$? For example, one could consider conformal transformations of hypersurfaces satisfying an equation of the form $P(H)=0$ for some polynomial $P$.
\end{op}

It is possible that $\mathcal{E}$ has a specific structure for which there are no minimizers. In this case, it would be interesting to find another functional, coercive as well and which does possess minimizers. For instance, one can prove the following lower bound for $\mathcal{E}_{GR}$:
\begin{align*}
    \mathcal{E}_{GR}(\bP) \geq \int_{\Sigma} |dH_{\bP}|^2_{g_{\bP}} -\frac{1}{12}|\mathring{\II}_{\bP}|_{g_{\bP}}^4\, d\textup{vol}_{g_{\bP}}.
\end{align*}
Thus, there also exists a universal constant $C>0$ for which it holds
\begin{align*}
    \int_{\Sigma} |d\II_{\bP}|_{g_{\bP}}^2 + |\II_{\bP}|_{g_{\bP}}^4\, d\textup{vol}_{g_{\bP}} \leq C\left( \mathcal{E}_{GR}(\bP) + \int_{\Sigma} |\mathring{\II}_{\bP}|^4\, d\textup{vol}_{g_{\bP}}\right).
\end{align*}
The energy on the right-hand side is still a generalized Willmore energy. All these generalized Willmore energies are of the form
\begin{align}\label{eq:generic_Willmore}
    \bP\in\mathrm{Imm}(\Sigma^4, \R^5)\mapsto \int_{\Sigma} |d\II_{\bP}|_{g_{\bP}}^2\, d\textup{vol}_g+ O\left(\int_{\Sigma} |\II|^4_{g_{\bP}}\, d\textup{vol}_{g_{\bP}}\right).
\end{align}

\begin{op}\label{qu:minimize2}
    Can we prove that any generalized Willmore energy satisfying the coercivity condition \eqref{eq:coercivity} possesses minimizers?  
\end{op}

\begin{op}
    Given a conformally invariant functional of the form \eqref{eq:generic_Willmore}, can we produce a generic class of critical points?
\end{op}

In order to study Open Questions \ref{qu:minimize1} and \ref{qu:minimize2}, one could follow the strategy introduced by Simon \cite{Sim} and Bauer--Kuwert \cite{BaKu} in order to show the existence of a minimizer of the Willmore functional among the immersed surfaces of any given genus. In first step in \cite{BaKu} is to prove a key inequality on connected sums of non-umbilic immersed submanifolds. A generalization of this inequality has recently been proved in \cite{wu2025}.

\subsubsection{Analysis in high dimensions}

In order to answer the previous questions, one would need to develop a refined analysis for submanifolds in the spirit of the approach presented here and most likely combined with the existing theory of varifolds. An important starting point already raised in \cite[Remark 5.4.6]{helein2002}, would be to generalize the study of Coulomb frames in higher dimensions but also to generalize the work of Müller--Šverák  \cite{MS1995} on the construction of isothermal coordinates for immersed surfaces in a Euclidean space. The Coulomb frames are of major importance due to their relation to isothermal coordinates, see for instance Theorem \ref{thexicon} below or \cite[Section 5.4]{helein2002}. In higher dimensions, isothermal coordinates do not exist in general, the Cotton and Weyl tensors are the obstruction. However, harmonic coordinates always exist on any manifold of any dimension $n\geq 2$ (isothermal coordinates are harmonic). In recent work \cite{MarRiv1,MarRiv2}, the second and third authors show that for an immersion $\bP\colon \Sigma^4\to \R^5$, an estimate of the following form provides an atlas of harmonic coordinates on $\Sigma$ describing a $C^1$-differential structure with quantitative estimates on all the quantities involved
\begin{align}\label{eq:bound_E}
    \int_{\Sigma} |d\II_{\bP}|^2_{g_{\bP}} + |\II_{\bP}|^4_{g_{\bP}}\ d\textup{vol}_{g_{\bP}} \leq E.
\end{align}
In particular, the work \cite{MarRiv1} provides a possible generalization to the work of Müller--Šverák \cite{MS1995} together with an answer to Hélein's question. The construction of harmonic coordinates is done in two steps. First, we proceed by a continuity argument reminiscent of the one of Uhlenbeck \cite{Uh82} for the construction of a Coulomb gauge for Yang--Mills connections. In this first step, we construct both a controlled system of coordinates and a controlled Coulomb frame. We can then perturb these coordinates to obtain controlled harmonic coordinates. The crucial analytic ingredient that replaces the $\mathrm{div}$-$\curl$ structure in dimension 2 is the use of Lorentz spaces. An estimate of the form \eqref{eq:bound_E} combined with the Gauss--Codazzi equations implies that the Riemann tensor $\mathrm{Riem}^{g_{\bP}}$ lies in $L^{(2,1)}$. This Lorentz exponent 1 is just enough to be able to close the continuity argument and obtain the existence of harmonic coordinates if $\mathrm{Riem}^{g_{\bP}}$ is small in $L^{(2,1)}$. This construction extends to all dimensions $n\geq 3$ as long as $\mathrm{Riem}^{g_{\bP}}$ lies in $L^{\left(\frac{n}{2},1\right)}$.

\begin{op}
    Given a closed orientable $n$-dimensional manifold $\Sigma$ and an immersion $\bP\colon \Sigma\to \R^d$ with $\II_{\bP}\in L^n(\Sigma^n)$, is it possible to construct counterexamples to the existence of controlled harmonic coordinates? Or can we construct controlled harmonic coordinates on a ``large" subset of $\Sigma$?
\end{op}

The next question is to know whether this atlas of harmonic coordinates can be ``uniformized" to obtain a global metric on $\Sigma$ conformal to $g_{\bP}$ together with some control following from an estimate of the form \eqref{eq:bound_E}.
A first attempt would be to find a metric having constant scalar curvature. This question is at the origin of the Yamabe problem and is not solvable in full generality, the topology of $\Sigma$ plays a role when it has a boundary. An alternative has been proposed recently in \cite{laurain2025}, where the authors are able to solve locally $\Delta_{\bar{g}} \mathrm{Scal}_{\bar{g}}=0$ for some metric $\bar{g}$ conformal to $g$ in dimension 4. This equation is always solvable as long as the Sobolev inequalities are verified for the metric $g$ and the following regularity estimate holds
\begin{align}\label{eq:Laplacian}
    \forall u\in C^{\infty}_c(\Sigma),\qquad \| \nabla u\|_{L^4(\Sigma,g_{\bP})} \leq C\, \|\Delta_{g_{\bP}} u\|_{L^2(\Sigma,g_{\bP})}.
\end{align}
Sobolev inequalities are a consequence of the isoperimetric inequality, which always holds for submanifolds with a small $L^n$-norm of its mean curvature, thanks to the Michael--Simon inequality \cite{allard1975,michael1973}. A global version of the equation $\Delta_{\bar{g}} \mathrm{Scal}_{\bar{g}}=0$ would be that the $Q$-curvature of $\bar{g}$ is pointwise controlled by $|\mathrm{Riem}^{\bar{g}}|^2_{\bar{g}}$.

\begin{op}
    Given an immersion $\bP\colon \Sigma\to \R^m$ from a closed $4$-dimensional oriented manifold $\Sigma$ verifying \eqref{eq:bound_E}, can we find a metric $\bar{g}$ conformal to $g_{\bP}$ such that the $Q$-curvature of $\bar{g}$ is constant? Is it pointwise controlled by $|\vec{\II}_{\bP}|^4_{g_{\bP}}$?
\end{op}

In order to know whether an equation on a manifold is regularizing or not, one needs to develop a refined PDE theory. In particular, one needs to control the constants involved in inequalities such as Sobolev or Poincaré inequalities and elliptic regularity. Such inequalities are directly related to isoperimetric problems, see for instance \cite[Chapter IV]{chavel1984}. Gallot \cite{gallot1988} proved that the isoperimetric inequality holds on a complete manifold $(\Sigma^n,g)$ as soon as $\mathrm{Ric}_g\in L^p(\Sigma,g)$ for some $p>\frac{n}{2}$. Moreover, the Michael--Simon inequality (see Brendle \cite{brendle2021} for the optimal constants) implies that an isoperimetric inequality always holds for immersed $n$-dimensional submanifolds of arbitrary codimension in a Euclidean space whose mean curvature has small $L^n$-norm. 

\begin{op}
    Given a complete manifold $(\Sigma^n,g)$, is there a number $s>0$ such that if $\mathrm{Riem}_g\in W^{s,\frac{n}{s+2}}(\Sigma,g)$, then the isoperimetric inequality holds?
\end{op}

In order to develop a variational theory for submanifolds of the Euclidean space, one needs to understand the "closure" of this space under the norm we are interested in, in the same spirit that the notions of currents or varifolds arise as limits of smooth submanifolds having bounded area. In other words, it would be interesting to generalize Theorem \ref{th-wclo}.

\begin{op}
    Fix a number $s\in[0,n]$. Let $(\bP_k)_{k\in\N}$ be a sequence of smooth immersions from a given orientable $n$-dimensional submanifold $\Sigma^n\subset \R^{n+m}$ into $\R^{n+m}$. Assume that 
    \begin{align*}
        \sup_{k\in\N} \left\|\vec{\II}_{\bP_k}\right\|_{W^{s,\frac{n}{s+1}}(\Sigma,g_{\bP})} <+\infty.
    \end{align*}
    What is the limit $\displaystyle{\lim_{k\to+\infty}} \bP_k$? Can we develop a bubbling analysis adapted to the above energy?
\end{op}

The Sobolev space $W^{s,\frac{n}{s+1}}(\Sigma^n)$ always embeds in $L^n(\Sigma^n)$. Hence, one should not expect a strong limit in full generality, but rather a weak limit on $\Sigma$ minus a set having small measure. The case $s=\frac{n-2}{2}$ and $n\geq 4$ even has been treated in \cite{MarRiv2}, based on the approximation result obtained in Theorem \ref{th:Approx}.
\begin{Th}
    Let $(\bP_k)_{k\in\N}$ be a sequence of smooth immersions from a given orientable $n$-dimensional submanifold $\Sigma^n\subset \R^d$ into $\R^d$. Assume that 
    \begin{align*}
        \sup_{k\in\N} \int_{\Sigma} \big| \g^{\frac{n}{2}-1} \vec{\II}_{\bP_k} \big|^2_{g_{\bP_k}} + \cdots + \big| \vec{\II}_{\bP_k} \big|^n_{g_{\bP_k}}\, d\textup{vol}_{g_{\bP_k}} \leq C.
    \end{align*}
    Up to rescaling, we can assume that each $\bP_k(\Sigma)$ is contained in the unit ball $\mathbb{B}^d(0,1)$ and has diameter 1. Then, there exists a closed $C^1$ manifold $\Sigma_{\infty}$ (possibly non-connected) and a weak immersion $\bP_{\infty}\colon \Sigma_{\infty}\to \R^d$ such that the following holds. There exists a finite number of points $\vec{q}_1,\ldots,\vec{q}_K\in \R^d$ such that for each $r>0$, we have the following convergence in the pointed Gromov--Hausdorff topology
    \begin{align*}
        \left( \Sigma\setminus \bP_k^{-1}\left(\mathbb{B}^d(\vec{q}_1,r)\cup \cdots \cup \mathbb{B}^d(\vec{q}_1,r) \right),g_{\bP_k} \right) \to \left( \Sigma_{\infty}\setminus \bP_{\infty}^{-1}\left(\mathbb{B}^d(\vec{q}_1,r)\cup \cdots \cup \mathbb{B}^d(\vec{q}_1,r) \right),g_{\bP_{\infty}} \right).
    \end{align*}
    Moreover, $\bP_k$ converges to $\bP_{\infty}$ in $C^0\left( \Sigma_{\infty}\setminus \bP_{\infty}^{-1}\left(\mathbb{B}^d(\vec{q}_1,r)\cup \cdots \cup \mathbb{B}^d(\vec{q}_1,r) \right),g_{\bP_{\infty}} \right)$ for any $r>0$.
\end{Th}

It would be interesting to study the formation of singularities. It would also be useful to study the cases where $s$ is not an integer. Once we have a critical point, it is important to know whether the Euler--Lagrange equation is regularizing or not. In \cite{bernard2025}, we obtain the following result in dimension $4$.
\begin{Th}
    Let $\bP\colon \mathbb{B}^4(0,1)\to \R^5$ be a weak immersion critical point of the functional
    \begin{align*}
        E(\bP) = \int_{\mathbb{B}^4(0,1)} \big| d H_{\bP}\big|^2_{g_{\bP}}\ d\textup{vol}_{g_{\bP}}.
    \end{align*}
    Then $\bP$ is real-analytic in harmonic coordinates.
\end{Th}
The generalizations to the cases with higher codimension or with lower order terms as in the Graham--Reichert functional will be the subject of a forthcoming work.

 \section{Preliminaries}\label{sec-preliminaries}

\subsection{Notation}\label{geono} 
\begin{enumerate}[(i)]
\item In the Euclidean space $\R^n$, we define $B_r(x)\coloneqq\{y\in \R^n\colon |y-x|<r\}$, and write $B^n=B_1(0)\subset \R^n$ for the unit ball. In the case $n=2$, write $D^2=B^2$, $D_r(x)=B_r(x)\subset\R^2$. The phrase ``for every ball $B_r\subset U$'' refers to every ball of radius $r$ contained in $U$. Denote $\mathbb S^{n-1}=\partial B^n$.
\item We denote by $\R^{m\times n}$ the space of $m\times n$ real matrices, and by $\R^{n\times n}_{\text{sym}}$ the space of $n\times n$ symmetric real matrices.
\item For a real-valued function $f$, we denote $f^+\coloneqq\max(f,0)$ and $f^-\coloneqq\max(-f,0)$.
\item Let $\Sigma$ be a closed smooth manifold, $E$ be a vector bundle over $\Sigma$, we denote by $L^{p}(\Si,E)$ the space of $L^{p}$ sections of $E$. This convention also applies to the other function/distribution spaces.
\item For an open set $U\subset \R^n$, we denote by $\mathcal D'(U,\R^m)$ the space of $\R^m$-valued distributions on $U$ with the weak$^*$ topology, and we write $\mathcal D'(U)=\mathcal D'(U,\R)$.
\item We denote by $\mathcal L^n$ the Lebesgue measure on $\R^n$. For $U\subset \R^n$ with $\mathcal L^n(U)<\infty$ and $f\in L^1(U)$, we write $\fint_U f\coloneqq(\mathcal L^n(U))^{-1}\int _U f$.
\item For open sets $U$ and $V$, we write $V\Subset U$ if $\bar V$ is compact and $\bar V\subset U$.
\item Let $U$ be a smooth manifold, $f\colon U\rightarrow \R^m$. We denote by $df$ the differential of $f$. \\When $U$ is an open subset of $\R^n$, we denote $\p_{i}=\frac{\p}{\p x^i}$ when there is no ambiguity, and denote the Jacobian matrix of $f$ by $\g f$. If $n=2$, we write $\nabla^\perp f\coloneqq(-\p_{2}f,\p_{1}f)$. If in addition that $f=(\vec f_1,\vec f_2)$ for $\vec f_1,\vec f_2\colon U\to \R^m$, we write $\curl f\coloneqq\p_{1}\vec f_2-\p_{2}\vec f_1$.

\item \label{not-Hosta}
Let $(V, g)$ be an $n$-dimensional inner product space with a positively oriented orthonormal basis $(e_1,\dots,e_n)$. The \textit{Hodge star operator} $*_g$ is defined as the unique linear operator from $\bigwedge ^k V$ to $\bigwedge ^{n-k}V$ satisfying $$
    \alpha\we *_g\, \beta=\langle \alpha,\beta  \rangle_g\, e_1\we\cdots \we e_n, \qquad \text{for all }\, \al,\beta\in \bwe\nolimits^kV.$$
See for instance \cite[Section 3.3]{Jost17}. 
   When $V=\R^n$ and $g=g_{\text{std}}$ is the standard inner product on $\R^n$, we write $\star$ instead of $*_g$.
    \item Let $U\subset \R^n$ be an open set and let $g$ be a Riemannian metric on $U$. On the space of differential $\ell$-forms on $U$, we define the
\textit{codifferential} (see \cite[Definition~3.3.1]{Jost17}) by
\begin{align}
\label{defd*glap}
    &d^{*_g}\coloneqq(-1)^\ell*_g^{-1}d\,*_g=(-1)^{n(\ell+1)+1}*_gd\,*_g.
\end{align} 
 In particular, when $n=4$, we have $d^{*_g}=-*_g d\,*_g$. If $\al\in C^\nf(U,\bwe^\ell T^* U)$ and $\beta \in C^\nf(U,\bwe^{\ell+1} T^* U)$ with at least one of $\al,\beta$ compactly supported in $U$, then we have
\begin{align}\label{d^*adjoi}
    \int_U \lan d\al,\beta \ran_g \,\dvol_g = \int_U \lan \al,d^{*_g}\beta\ran_g\,\dvol_g.
\end{align}
    \item Levi-Civita symbol:\begin{align*}
  \vae^{a_1a_2\dots a_n}\coloneqq\begin{cases}
    1 & \text{if }(a_1,\dots,a_n)\text{ is an even permutation of }(1,\dots,n),\\
    -1 & \text{if }(a_1,\dots,a_n)\text{ is an odd permutation of }(1,\dots,n),\\
    0 & \text{otherwise}.
    \end{cases}
\end{align*}
\item Let $U\subset \R^n$ be open and bounded. We call $U$ a \textit{Lipschitz domain} if for each $y\in \p U$, there exist $r>0$ and a Lipschitz function $f\colon \R^{n-1}\rightarrow\R$ such that, upon rotating and relabeling the coordinate axes if necessary and writing $x=(x',x_n)$, we have
\begin{align*}
    U\cap B_r(y)=\{x\in B_r(y):x_n>f(x')\}.
\end{align*}
Equivalently, we say $\p U$ is Lipschitz or in $C^{0,1}$.
\item \label{not-overset} We define a bilinear map $\dwe\colon \big(\bwe^{k_1} T^*_p\Si\big )^m \times\big(\bwe^{k_2} T^*_p\Si\big )^m \rightarrow  \bwe^{k_1+k_2} T^*_p\Si$ by \begin{align}\label{defdwe}
              (\vec u_1 \mkt dx_I)\dwe(\vec u_2\mkt dx_J)=( \vec u_1 \cdot  \vec u_2) \mkt dx_I\wedge dx_J,
         \end{align}
         where $\vec u_1,\vec u_2\in\R^m$, $I,J $ are multi-indices. If one of $k_1,k_2$ is $0$, we use $\cdot$ instead of $\dwe$ for convenience. Similarly, we define bilinear maps $\times$, $\overset{\times}{\we}, \we, \overset{\ovwe}{\wedge}$ by \begin{align*}
            & \vec v_1 \times (\vec v_2\mkt dx_J)=( \vec v_1 \times  \vec v_2) \mkt dx_J,\quad  (\vec v_1 \mkt dx_I)\overset{\times}\wedge(\vec v_2\mkt dx_J)=( \vec v_1 \times  \vec v_2) \mkt dx_I\wedge dx_J,\\
             & \bw_1 \we (\bw_2\mkt dx_J)=( \bw_1 \we  \bw_2)\mkt  dx_J,\quad  (\bw_1 \mkt dx_I)\overset{\ovwe}\wedge(\bw_2\mkt dx_J)=( \bw_1 \we  \bw_2) \mkt dx_I\wedge dx_J,
         \end{align*}
         where $\vec v_1,\vec v_2\in \R^3$, $\bw_1,\bw_2\in\bigwedge\R^m$, $I,J$ are multi-indices.
\item \label{vec-ope}
For $\vec u=(\vec u_1,\vec u_2)$, $\vec v=(\vec v_1, \vec v_2)$, where $\vec u_i, \vec v_i\in \R^3$, we denote $$\vec u\cdot \vec v\coloneqq\sum_{i=1}^2 \vec u_i \cdot \vec v_i, \quad\vec u\times \vec v\coloneqq\sum_{i=1}^2 \vec u_i \times \vec v_i.$$
 \item We adopt Einstein summation convention for convenience of computation, and set $\delta_i^j=\de_{ij}=\mathbf 1_{i=j}$.
\item We use $C$ to denote positive constants, and $C(\alpha,\beta,\dots)$ to denote a positive constant depending only on $\alpha,\,\beta,\dots$. Similarly, when a term $A$ depends only on $\alpha,\,\beta,\dots$, we write $A=A(\alpha,\beta,\dots)$.
\item \label{not-met}Let $U\subset\R^{n}$ be open and $\bP\colon U\to\R^{m}$ be an immersion. Set $g=g_{\bP}=\bP^*g_{\text{std}}$, where $g_{\std}$ is the Euclidean metric on $\R^m$. We denote $$
   g_{ij}=\p_i\bP\cdot \p_j\bP,\quad (g^{ij})=(g_{ij})^{-1},\quad  \det g=\det(g_{ij}),\quad d\textup{vol}_g=(\det g)^{\frac 12}dx^1\we\cdots \we dx^n.$$
The pullback metric induces a pairing on $T^* U$ with $\lan dx^i,dx^j \ran_g=g^{ij}$. We denote by $|\cdot|_{g}$
the associated pointwise norm. We also write $|\cdot|_{\R^{m}}$ for the
Euclidean norm in $\R^{m}$ (abbreviated to $|\cdot|$ when unambiguous).

For $f\colon U\to \R$, define the \textit{Laplacian} \begin{align*}
    \lap_g f\coloneqq *_g d *_g df=(\det g)^{-\frac 12} \p_i \big(g^{ij}(\det g)^{\frac 12} \p_j f\big).
\end{align*}
Let $\pi_{\bn}$ denote the orthogonal projection onto the normal bundle of $\bP(U)\subset \R^m$. We define the \textit{Gauss map} and \textit{second fundamental form} \begin{gather*}
\bn\coloneqq \star \,\frac{\p_1 \bP\we \cdots \we \p_n \bP}{|\p_1 \bP\we \cdots \we \p_n \bP|},\\[2mm] 
\bII(X,Y)\coloneqq\pi_{\bn} X(d\bP(Y)),\qquad \forall X,Y\in T_p U,\ \forall p\in U.
\end{gather*}
The \textit{mean curvature vector} is defined by
\[
\vec{H}\coloneqq n^{-1}\,\mbox{tr}_{g}\,\bII=\frac 1n\, g^{ij}\,\bII_{ij}=\frac{1}{n} g^{ij}\, \bII\lf(\frac{\p}{\p x^i},\frac{\p}{\p x^j}\rg).
\]
When $m=n+1$, we write
\[
 \II_{ij}\coloneqq\bII_{ij}\cdot \bn,\quad H\coloneqq\bH \cdot \bn .
\] 
\item \label{not-res}    Let $(V,g)$ be a finite-dimensional inner product         space. For $\alpha\in \bigwedge^pV$ and $\beta\in \bigwedge^q   V$ with $q\le p$, we define the \textit{interior product} (see also \cite[Section~1.5]{Federer96} and \cite[Section~I]{Riv08}) $\alpha\,\resg \beta\in \bigwedge^{p-q}V$ satisfying 
    \begin{align}\label{defresg}
    \langle \alpha\,\resg \beta,  \gamma\rangle_g=\langle \alpha, \beta\wedge \gamma\rangle_g,\qquad \forall \gamma\in \bigwedge\nolimits^{p-q}V. 
    \end{align}
 In particular, if $\alpha,\beta \in \bigwedge^pV$, then we have $\alpha\,\resg\,\beta=\langle \alpha,\beta \rangle_g$. For $q>p$, set $\al\, \resg \beta\coloneqq0$. 
When $V=\R^5$ with the Euclidean metric, we write $\,\res$ instead of $\,\resg$. In addition, for $\al\in \bigwedge^p V$, $\beta\in\bigwedge V$, and $v\in V$, we have the following identities:
   \begin{numcases}{}   *_g\,(\alpha\wedge \beta )=(*_g\,\alpha)\,\resg \beta \label{*_gcommwed},\\[3mm]
   (\al\we \beta)\,\resg v=(\al\,\resg v) \we \beta+(-1)^p\, \alpha\we (\beta\,\resg v).\label{prodrule}
   \end{numcases}\smallskip
  A consequence of \eqref{*_gcommwed} is that, for $\al\in \bwe^p V$ and $\beta\in \bwe^q V$,
  \begin{align}\label{*_gcomres}
      *_g\, (\al\, \resg \beta)=\begin{dcases}
          (*_g\, \al) \we \beta,& \text{if }\dim(V) \text{ is odd},\\[1.3ex]
          (-1)^q (*_g\, \al) \we \beta, \quad& \text{if }\dim(V) \text{ is even}.
      \end{dcases}
  \end{align} 
\item \label{con-formsob}Throughout this paper, unless explicitly stated otherwise, we write $\nabla$, $|\cdot|$ for the flat Euclidean derivative and
pointwise tensor norms. We reserve $\nabla^g$, $\lan\cdot,\cdot \ran_g$, and $|\cdot|_g$ for the metric $g$. 
 For an open set $U\subset \R^n$, we write $\al\in W^{k,p}(U,\bwe^\ell \R^n)$  if $\alpha=\sum_{|I|=\ell}\alpha_I\mko dx^I$ with each $\al_I\in W^{k,p}(U)$, where $I$ ranges over all strictly increasing multi-indices of length $\ell$. We set
\begin{align}\label{convsonab}
    |\nabla\alpha|^2\coloneqq\sum_{i=1}^n\sum_I |\partial_i\mko \alpha_I|^2,\qquad 
    \|\al\|_{W^{k,p}(U)}\coloneqq\sum_{I} \|\al_I\|_{W^{k,p}(U)}.
\end{align}
Similarly, we write $X\in W^{k,p}(U, TU)$ if $X=X_j\mko \frac {\p}{\p x^j}$ with each $X_j\in W^{k,p}(U)$, and set 
\begin{align}\label{sobvecfie}
    |\nabla X|^2\coloneqq\sum_{i=1}^n\sum_{j=1}^n |\partial_i\mko X_j|^2,\qquad 
    \|X\|_{W^{k,p}(U)}\coloneqq\sum_{j=1}^n \|X_j\|_{W^{k,p}(U)}.
\end{align}
The same convention applies to any Banach function/distribution space on $U$
(e.g. $L^{p}$, $L^{p,q}$, $W^{k,(p,q)}$).
     \end{enumerate}




\subsection{A useful formula}\label{uselem}

 We provide the proof of an identity here without using conformal coordinates.
\begin{Lm} \label{lmuse}{\cite[Lemma 1.8]{Ri16}} Let $\Sigma$ be a $2$-dimensional oriented smooth manifold, and $\bP\colon\Sigma\rightarrow \R^3$ be a smooth immersion. Let $*_g$ be the Hodge star operator with respect to $g=g_{\bP}$. Then we have
    \begin{align}\label{dnHcon}
        -2H\,d\bP=d\bn +\bn \times *_g\, d\bn.
    \end{align}
\end{Lm}
\begin{proof}
   Let $p\in \Sigma$, $(x^1,x^2)$ be positive local coordinates around $p$, $\alpha\coloneqq dx^1\wedge dx^2$, $\II$ denotes the second fundamental form of $\bP$, and define $$
   g_{ij}\coloneqq\p_i\bP\cdot \p_j\bP,\quad (g^{ij})=(g_{ij})^{-1},\quad  \det g=\det(g_{ij}).$$ For any $k\in\{1,2\}$, we have 
   $$\p_{k}(\bn\cdot \bn)=2\,\p_{k}\bn \cdot \bn=0.
   $$ 
   Hence, it holds 
   \begin{align}\label{dn}
       \p_k\bn=g^{ij} (\p_k\bn\cdot \p_j\bP)\, \p_i\bP.
   \end{align}Fix $s\in \{1,2\}$, then we have \begin{align*}
       dx_s \wedge (\bn \times *_g \,d\bn)&=\bn \times \langle dx_s, d\bn\rangle_g\, \sqrt{\det g} \ dx^1\wedge dx^2\\[1mm]
       &=g^{sj}(\p_1\bP\times \p_2\bP)\times \p_j\bn \ \alpha \\[1mm]
       &=g^{sj}\,\vae^{ki}(\p_k\bP\cdot\p_j\bn )\, \p_i\bP \, \alpha\\[1mm]
       &=-g^{sj}\,\vae^{ki}\,\II_{kj}\,\p_i\bP \, \alpha, \end{align*}
       where the third equality follows from the identity \begin{align*}
           \forall, \vec u, \vec v,\vec w\in \R^3,\qquad (\vec{u}\times \vec{v})\times \vec{w}=(\vec u \cdot \vec w)\vec v -(\vec v \cdot \vec w)\vec u.
       \end{align*}  
       By \eqref{dn}, we also get\begin{align*}
            dx_s \wedge d\bn&= \vae^{sk}\,\p_k n\,\al=-g^{ij}\e^{sk}\,\II_{kj}\,\p_i\bP\, \al.
       \end{align*}
       Therefore, it holds
       \begin{align}\begin{aligned}\label{wedequ}
            dx_s \wedge (d\bn+\bn \times *_g \,d\bn)&=-(g^{sj}\vae^{ki}+g^{ij}\vae^{sk})\,\II_{kj}\,\p_i\bP\, \al\\[1mm]
            &=-\vae^{si}g^{kj}\,\II_{kj}\,\p_i\bP\,\al\\[1mm]
            &=-2H \,dx_s\wedge d\bP,
            \end{aligned}
       \end{align}
       where the second equality follows from the identity, valid for any $i,k,s\in\{1,2\}$ and real-valued function $f$ defined on $\{1,2\}$
       \begin{align*}
f(s)\vae^{ki}+f(k)\vae^{is}+f(i)\vae^{sk}=0.
       \end{align*}
       Since \eqref{wedequ} holds for any $s\in\{1,2\}$, the relation \eqref{dnHcon} follows immediately.
\end{proof}

\subsection{Some estimates for elliptic equations}
We first introduce a standard Morrey estimate for elliptic equations of divergence form with no lower-order terms.
\begin{Lm}[{\cite[Lemma 4.12]{hanlinpde}}]\label{derdecay}
Let $n\in \N^+$ and $\Lambda>0$. Suppose $\{a^{ij}\}_{i,j=1}^n\subset L^{\infty}(B^n)$ satisfies \begin{align}\label{elli}
    \La^{-1}|\xi|^2\le a^{ij}(x)\,\xi_i\, \xi_j \le \La |\xi|^2,\qquad \text{for a.e. } x\in B^n \text{ and all } \,\xi\in \R^n.
\end{align}
Suppose $u\in W^{1,2}(B^n)$ is a weak solution of the equation $ \p_j(a^{ij}\,\p_i u)=0$ in $B^n$, i.e. \begin{align*}
    \int_{B^n} a^{ij}\,\p_i u\,\p_j \vp=0,\qquad \forall \vp\in C^\infty_c(B^n).
\end{align*}
 Then there exist $\al=\alpha(n,\La)\in (0,1)$ and $C=C(n,\La)>0$ such that for any $0<r<1$ there holds \begin{align*}
    \int_{B_r(0)} |\nabla u|^2 \le C \,r^{n-2+2\al}\int_{B^n}|\nabla u|^2.
\end{align*}
\end{Lm}
Now we assume $n=2$. In 1992, S. Chanillo and Y.Y. Li obtained the following result, which generalized a famous estimate on integrability by compensation proved by H.C. Wente \cite{Wen}, in which it is assumed $a^{ij}=\delta_{ij}$.
\begin{Lm}[{\cite[Theorem 0.2]{CL63}}]\label{lm-chali}
Suppose $\{a^{ij}\}_{i,j=1}^2$ satisfies \eqref{elli} for some $\La\in (0,\infty)$, and $a^{ij}\equiv a^{ji}$. For $u,v\in W^{1,2}(D^2)$, there exists a unique solution $\phi\in W^{1,2}_0 (D^2)$ to the problem \begin{align*}
    \begin{cases}
   \p_i(a^{ij}\, \p_j\phi)=\p_1 u\,\p_2 v -\p_2 u\,\p_1 v&\mbox{in } D^2,\\[2mm]
  \phi=0 &\mbox{on } \p D^2.
\end{cases}
\end{align*}
We have $\phi \in C(\overline{D^2})$ with the estimate \begin{align*}
 \|\phi\|_{L^{\infty}(D^2)} +  \|\nabla \phi\|_{L^2(D^2)}\le C(\La)\, \|\nabla u\|_{L^2(D^2)}\,\|\nabla v\|_{L^2(D^2)}.
\end{align*}
\end{Lm}
We prove a stronger version of this result in Proposition \ref{pr:Wente} below.
\subsection{Bourgain--Brezis inequality}
 In 2002, J. Bourgain and H. Brezis proved a striking inequality \cite{Bourgain03}:
 \begin{align}\label{BBo}
   \lf\| u- \dashint_{\mathbb T^n} u\rg\|_{L^{\frac{n}{n-1}}(\mathbb T^n)}\le C\lf\|\g u\rg\|_{L^1+W^{-1,\frac{n}{n-1}}(\mathbb T^n)},
 \end{align}
     where $C=C(n)$, $\mathbb T^n$ is the $n$-dimensional torus. In the case $n=2$, F. Da Lio and the third author \cite{dalio2021} provided a proof for a stronger version of \eqref{BBo} in 2020. Here we denote $\dot W^{-1,2}(\R^2)$ as the dual space of the homogeneous Sobolev space $\dot W^{1,2}(\R^2)$.
     
     \begin{Lm}[{\cite[Lemma II.3]{dalio2021}}]\label{bourgainb}
Assume $g\in L^1(\R^2,\R^2)$, $ f\in \dot W^{-1,2}(\R^2,\R^2)$, $u$ be a tempered distribution on $\R^2$ such that $\g u= f+g$. Then there exist $c\in \R$ and a universal constant $C$ such that 
\begin{align*}
    \|u-c\|_{L^2(\R^2)}\le C( \| f\|_{\dot W^{-1,2}(\R^2)}+\|g\|_{L^1(\R^2)}).
\end{align*}
\end{Lm}

\begin{Co}\label{corofbb}
Assume $u\in \mathcal D'(D^2)$ such that $\g u= f+g$ with $g\in L^1(D^2,\R^2)$, $ f\in  W^{-1,2}(D^2,\R^2)$. Then for any $r\in (0,1)$, there holds
\begin{align*}
    \lf\|u-\dashint_{D_r(0)} u\rg\|_{L^2(D_r(0))}\le C(r) \big(\|f\|_{ W^{-1,2}(D^2)}+\|g\|_{L^1(D^2)}\big).
\end{align*}
\end{Co}

\subsection{Quasiconformal maps and Beltrami equations}\label{sec-quasi}
\begin{Dfi}\label{dfiquasicon}
    Let $U\subset \R^2$ be an open subset. A \textit{quasiconformal} map $f\colon U\to f(U)\subset \C$ is an orientation-preserving homeomorphism (i.e. $\det(\g f)>0$ a.e.) satisfying $f\in W^{1,2}_{\textup{loc}}(U)$ and there exists a constant $k\in [0,1)$ such that\begin{align*}
        |\p_{\bar z}f|\le k |\p_z f| \quad \mbox{a.e. on } U,
        \end{align*}
       where $\p_z\coloneqq(\p_x-i\p_y)/2$, $\p_{\bar z}\coloneqq(\p_x+i\p_y)/2$. We call the function $\mu_{f}\coloneqq\p_{\bar z }f/\p_z f$ the \textit{complex dilatation} of $f$.
\end{Dfi}
\begin{Lm}[{\cite[Theorem 4.30]{Imayoshi92}}]
    \label{exisolbel}
    Let $\mu\in L^\infty(\C,\C)$ be such that $\|\mu\|_{L^\infty(\C)}<1$.
    Then there exists a quasiconformal homeomorphism $f\colon \bar \C\rightarrow\bar{\C}$ verifying the \textit{Beltrami equation} $\p_{\bar z} f=\mu\,\p_{z}f$. Such $f$ is uniquely determined by the conditions $f(0)=0$, $f(1)=1$, $f(\infty)=f(\infty)$.
\end{Lm}
 \begin{Lm}[{\cite[Proposition 4]{CFMOZ}}]\label{lmregbel}
 Suppose $\mu\in W^{1,2}(\C)$ has compact support and $\|\mu\|_{\infty}<1$. Let $f\colon \C\rightarrow\C$ be a quasiconformal map solving the Beltrami equation $\p_{\bar z} f=\mu\,\p_{z}f$.
 Then it holds $f\in W_{\textup{loc}}^{2,q}(\C)$ for any $q<2$.
\end{Lm}
\subsection{Hardy space $\mathcal H^1$}
For the definition of the Hardy space $\mathcal H^1(\R^n)$, see \cite{grafakos2014-2,steinhar}.
\begin{Lm}[{\cite{CLMS}}]\label{CLMS}
        Let $n\ge 2$. Suppose $u\in W^{1,n}(\R^n,\R^n)$. Then $\det(\g u)\in \mathcal H^1 (\R^n)$ with \begin{align*}
            \|\det (\g u)\|_{\mathcal H^1(\R^n) }\le C(n) \|u\|_{W^{1,n}(\R^n)}.
        \end{align*}
    \end{Lm}
    Standard singular integral theory (see for instance \cite[Theorem III.4]{steinhar}) implies the following results.
    \begin{Lm}\label{thpoihar}
        Let $n\ge 2$, $f\in \mathcal H^1 (\R^n)$. Suppose $u\in L^1(B^n)$ satisfies $\Delta u=f$ in $\D'(B^n)$. Then it holds $u\in W^{2,1}_{\textup{loc}}(B^n)$ and for any $r\in (0,1)$, we have \begin{align}
            \|u\|_{W^{2,1}(B_r(0))}\le C(n,r) \big(\|f\|_{\mathcal H^1(\R^n) }+\|u\|_{L^1(B^n)}\big).
        \end{align}
    \end{Lm}
    
    \begin{Co}\label{coconfactor}
  Let $\bP\in W^{2,2}_{\textup{imm}}(D^2,\R^m)$ be weakly conformal (see Definitions \ref{defweakimm} and \ref{dficoncha}) with $g_{\bP} \coloneqq\bP^*g_{\text{std}}=e^{2\la}(dx_1^2+dx_2^2)$ for $\la\in L^\infty(D^2)$. Then we have $\la\in W^{2,1}_{\textup{loc}}(D^2)$.
    \end{Co}
    
    \begin{proof}
      Let $(\ve_1,\ve_2)\coloneqq e^{-\la}(\p_1 \bP, \p_2\bP)\in L^\infty\cap W^{1,2}(D^2)$, then  we have 
    \begin{align*}
        \begin{cases}
          \ve_1\cdot \p_1\ve_2=e^{-2\la}\p_1\bP\cdot \p_{1}\p_2\bP=\frac 12 e^{-2\la}\p_2|\p_1\bP|^2=\p_2\la,\\[3mm]
         \ve_1\cdot \p_2\ve_2=e^{-2\la}\p_1\bP\cdot \p_{2}\p_2\bP=-e^{-2\la}\p_2\p_1\bP\cdot \p_2\bP=-\p_1\la.
        \end{cases}
    \end{align*}
    In other words, it holds $-\nabla^\perp\la =\ve_1\cdot \g \ve_2$. Thus we have, denoting $\ve_j=(e_{j,1},\dots, e_{j,m})$ for $j\in\{1,2\}$,
    \be\label{della=K} 
        \Delta \la=\curl(\g^\perp \la)=-\g^\perp \ve_1\cdot \g\ve_2=-\sum_{i=1}^m \g^\perp e_{1,i}\,\g e_{2,i}. 
    \ee 
    For $1\le i\le m$, define $e^i\in W^{1,2}(D^2,\R^2)$ by $e^i=(e_{1,i},e_{2,i})$. Then we have 
    $$
    \det(\g e^i)=\g^\perp e_{1,i}\cdot \g e_{2,i}.
    $$
    For any fixed $r\in (0,1)$, we choose $\eta\in C_c^\infty(D^2)$ such that $\eta\equiv 1$ on $D_r(0)$. Then it holds
        \begin{align*}
            \Delta \la =-\sum_{i=1}^m \det(\g(\eta\mko e^i)) \qquad \text{in }\D'(D_r(0)).
        \end{align*} 
        Since $\eta\mko e^i\in W^{1,2}(\R^2)$ for each $i$, we have $\det(\g(\eta e^i))\in \mathcal H^1 (\R^2)$ by Lemma \ref{CLMS}. This finally implies $\la\in W^{2,1}_{\textup{loc}}(D^2)$ by Lemma \ref{thpoihar}.
        \end{proof}
\subsection{Functions of VMO}\label{sec-VMO}
      \begin{Dfi}\label{dfi-VMO}
\begin{enumerate}[(i)]
    \item Let $n\ge 1$, $U\subset \R^n$ be open, $f\in L^{1}_{\textup{loc}}(U)$. For $r>0$, we define 
\begin{align}\label{vmomodules}
    \beta_r(f)=\beta_{r,U}(f)\coloneqq\sup_{\substack{\rho\le r\\B_\rho\subset U}} \fint_{B_\rho} \left|f-\fint_{B_\rho} f \right|,
\end{align}
where $B_\rho$ ranges over balls of radius $\rho$ contained in $U$. We say $f$ is of \textit{vanishing mean oscillation} and write $f\in VMO(U)$ if $$\lim_{r\rightarrow 0}\beta_r(f)=0.$$ 

\item Let $N$ be an $n$-dimensional smooth manifold without boundary, $U\subset N$ be open and precompact. For $f\in L^1_{\textup{loc}}(U)$, we say $f\in VMO(U)$ if for any coordinate system $(V,\phi)$ (an open set $V\subset N$ with a smooth diffeomorphism $\phi\colon V\rightarrow \phi(V)\subset \R^n$), any $\xi\in C^\infty_c(V)$ there holds $(\xi f)\circ \phi^{-1}\in VMO(\phi(V\cap U))$. Similarly, for $f\in W^{1,1}_{\textup{loc}}(U)$ we say $df\in VMO(U)$ if for any coordinate system $(V,\phi)$, $\xi\in C_c^\infty(V)$ there holds $\g((\xi f)\circ \phi^{-1})\in VMO(\phi(V\cap U),\R^n)$.
\end{enumerate}
\end{Dfi}
\begin{Lm}[{\cite[Lemma 2.3]{Jo80}}]\label{bmo-estimate}
Let $B_R\subset \R^n$ be a ball of radius $R$. Suppose there exist constants $\beta, K>0$ such that for any $r>0$ and ball $B_r\subset B_R$ of radius $r$ satisfying $\dist(B_r,\p B_R)\ge Kr$, there holds \begin{align}\label{interiorbmocon}
    \int_{B_r} \left|f-\fint_{B_r} f \right|\le \beta \,r^n.
\end{align}
Then there exists a positive constant $C=C(n,K)$ such that 
\begin{align}\label{bmo-est-eq}
    \int_{B_R} \left|f-\fint_{B_R}f \right|\le C \beta\, R^n.
\end{align}
\end{Lm}

\begin{proof}
We can without loss of generality assume $B_R=B_1(0)$. Let $\sigma\coloneqq1-(6K+6)^{-1}>\frac 56$. We construct collections $\{\mathcal G'_i\}_{i=0}^\infty$, $\{\mathcal G_i\}_{i=0}^\infty$ of open balls in the following way.\\
Let $\tilde B\coloneqq B_{1-\sigma}(0)$, $\mathcal G'_0\coloneqq\{\tilde B\}$. For each $i\ge 1$, there exists a cover $\mathcal G_i'$ of $\{x\in\R^n:|x|=1-\sigma^i\}$ such that any element $B\in \mathcal G'_i$ is a ball of radius $\sigma^i(1-\sigma)$ with center lying on $\{x\in \R^n:|x|=1-\sigma^i\}$, and 
$\card(\mathcal G_i)\le C(n,K)\,\sigma^{-(n-1)i}$. Let $\mathcal G_i\coloneqq\{2B:B\in\mathcal G'_i\}$ for $i\ge 0$, where for any constant $c>0$ and any ball $B = B_r(x)$ we have denoted $cB=B_{cr}(x)$. 

For any $i\ge 0$ and $x\in B_1(0)$ with $1-\sigma^{i}\le |x|\le 1-\sigma^{i+1}$, there exists $|x'|=1-\sigma^i$ such that $|x'-x|\le \sigma^i(1-\sigma)$. Hence there exists a ball $B\in \mathcal G'_i$ centered at $x''\in \{y\in \R^n:|y|=1-\sigma^i\} $ such that $x'\in B$. Thus it holds $|x-x''|\le |x-x'|+|x'-x''|\le 2\sigma^i(1-\sigma)$. It follows that $B_1(0)\subset \bigcup_{i=0}^\infty \bigcup_{B\in \mathcal G_i}\bar B$. 

Now let $i\ge 1$ and $B\in \mathcal G_i$ centered at $x\in \{y\in \R^n:|y|=1-\sigma^i\}$. Then we have $B\cap \{y\in \R^n:|y|=1-\sigma^{i-1}\}\neq \emptyset$ since $\dist\big(x,\{y\in \R^n:|y|=1-\sigma^{i-1}\}\big)=\sigma^{i-1}(1-\sigma)<2\sigma^i(1-\sigma)$. Hence there exists $B'\in \mathcal G_{i-1}$ such that $B'\cap B\neq \emptyset$. If we define $B''\coloneqq3B'$, then it holds $B\cup B'\subset B''$ since the radius of $B$ is smaller than that of $B'$. Moreover, we have $\dist(B'',\p B_1(0))=\sigma^{i-1}-6\sigma^{i-1}(1-\sigma)=6K\sigma^{i-1}(1-\sigma)$. Therefore, it holds
\begin{align*}
\left|\fint_{B}f-\fint_{B''}f \right|\le \fint_B \left|f-\fint_{B''}f \right|\le \left( \frac 3\sigma \right)^n \fint_{B''} \lf|f-\fint_{B''}f\rg|\le C(n,K)\,\beta.
\end{align*}
As well, we have
\begin{align*}
    \left|\fint_{B'}f-\fint_{B''}f \right|\le \fint_{B'} \left|f-\fint_{B''}f \right|\le 3^n \fint_{B''} \left|f-\fint_{B''}f \right|\le C(n)\,\beta.
\end{align*}
Consequently, we obtain
\begin{align*}
    \left|\fint_{B}f- \fint_{B'}f\right|\le C(n,K)\, \beta.
\end{align*}
Then by induction we have  for all $i\ge 0$ and all $B\in \mathcal G_i$,
\begin{align*}
     \lf|\fint_{B}f- \fint_{2\tilde B}f\rg|\le C(n,K)\,i\beta.
\end{align*}
 Finally, since $\dist(B,\p B_1(0))=\sigma^i-2\sigma^i(1-\sigma)>2K\sigma^i(1-\sigma)$ for any $i\ge 0$ and $B\in \mathcal G_i$, we have 
\begin{align*}
    \int_{B_1(0)} \left|f-\fint_{2\tilde B}f \right| &\le \sum_{i=0}^\infty\sum_{B\in \mathcal G_i} \int_{B} \left|f-\fint_{2\tilde B}f \right|\\[3mm]
    & \le \sum_{i=0}^\infty\sum_{B\in \mathcal G_i} \left(\int_{B} \left|f-\fint_{B}f \right|+ |B|\,\left|\fint_{B}f- \fint_{2\tilde B}f \right| \right)\\[3mm]
 & \le C(n,K)\sum_{i=0}^\infty\sum_{B\in \mathcal G_i} |B|\,(i+1)\beta\\[3mm]
 & \le C(n,K) \sum_{i=0}^\infty \sigma^i (i+1)\beta\le C(n,K) \beta.
\end{align*} 
\end{proof}
\begin{Lm}[{\cite[Theorem 2]{Rei74}, \cite[Section 3]{Jo80}}]\label{lm-quasi-vmo}
    Let $U$ be an open subset of $\R^n$, $f\in VMO(U)$, $\vp\in W^{1,n}_{\textup{loc}}(U,\R^n)$ be a homeomorphism onto its image. Assume there exists a positive constant $K$ such that 
    $$
    \det(\g \vp)>K|\g \vp|^n \qquad \text{a.e. in } U. 
    $$
     Then we have $f\circ \vp^{-1}\in VMO(\vp(U))$ with 
     $$
    \forall r>0,\qquad \beta_{r,\vp(U)}(f\circ\vp^{-1})\le C(K) \,\beta _{r,U} (f).
    $$
\end{Lm}
\begin{Co}[{\cite{Jo80}}]\label{covmoext}
     Let $U\subset \R^n$ be a bounded Lipschitz domain and $f\in VMO(U)\cap L^\infty(U)$. Then there exists $\tilde f\in VMO(\R^n)\cap L^\infty(\R^n)$ such that $\tilde f= f$ on $U$ satisfying the inequality 
     $$
     \forall r>0,\qquad \beta_{r,\R^n}(\tilde f)\le C(U) \big(\beta_{r,U}(f)+r\|f\|_{L^\infty(U)}\big).
     $$
\end{Co}
\subsection{Sobolev--Lorentz spaces}\label{sec:soblor}
\begin{Dfi}\label{sumofspace}
Assume $X_1,X_2$ are Banach spaces that continuously embed in the same Hausdorff topological vector space $Z$. Let $$ X_1+X_2\coloneqq\{x_1+x_2:x_1\in X_1, x_2\in X_2\}.
$$ We define norms on $X_1+X_2$ and $X_1\cap X_2$ by \begin{gather*}
     \begin{dcases}\|x\|_{X_1\cap X_2}\coloneqq\|x\|_{X_1}+\|x\|_{X_2},\\[3mm]
     \|x\|_{X_1+X_2}\coloneqq\inf\{\|x_1\|_{X_1}+\|x_2\|_{X_2}:x=x_1+x_2,x_1\in X_1,x_2\in X_2\}.\end{dcases}\end{gather*}
     Equipped with these norms, the sets $X_1+X_2$ and $X_1\cap X_2$ are Banach spaces. For brevity we write $L^p+L^q(\Omega)\coloneqq L^p(\Omega)+L^q(\Om)$. The same convention applies to Sobolev--Lorentz spaces.
\end{Dfi}
\begin{Dfi}[Lorentz spaces]\label{def-Lor}
Let \(U\subset\R^{n}\) be a measurable set. Given a measurable function \(f\colon U\to\R\), we define the distribution function and the decreasing rearrangement of \(f\) as
\[
d_{f}(\lambda)\coloneqq\mathcal{L}^{n}\bigl\{x\in U:|f(x)|>\lambda\bigr\},
\qquad
f^{*}(t)\coloneqq\inf\bigl\{\lambda\ge 0 : d_{f}(\lambda)\le t\bigr\}.
\]
For \(1\le p<\infty\) and \(1\le q\le\infty\), we define the Lorentz quasinorm
\[
|f|_{L^{p,q}(U)}
\coloneqq\bigl\|t^{\frac1p}f^{*}(t)\bigr\|_{L^{q}(\R_{+},\,dt/t)}
    =p^{\frac1q}\bigl\|\lambda\,d_{f}(\lambda)^{\frac1p}\bigr\|_{L^{q}(\R_{+},\,d\lambda/\lambda)}.
\]
The Lorentz space \(L^{p,q}(U)\) consists of all measurable \(f\) with
\(|f|_{L^{p,q}(U)}<\infty\). We have (see for instance \cite[Chapter~4, Proposition~4.2]{Bennett88})
\begin{enumerate}[(i)]
    \item  $L^{p,p}(U)=L^p(U)$,
    \item  $L^{p,q}(U)\hookrightarrow L^{p,r}(U)$ if $q<r$,
    \item  $L^{p,q}(U)\hookrightarrow L^{s,r}(U)$ if $p>s$ and $\mca L^n(U)<\nf$.
\end{enumerate}
When \(p>1\), the Lorentz quasinorm is equivalent to a norm, which we denote by
\(\|\cdot\|_{L^{p,q}(U)}\).
\end{Dfi}
We record the following form of H\"older's inequality for Lorentz spaces.
\begin{Lm}[{\cite[Thm.~4.5]{Hunt}}]\label{lm:lorHold}
   Assume $f_1\in L^{p_1,q_1}(U)$ and $f_2\in L^{p_2,q_2}(U)$, with $p,p_1,p_2\in [1,\nf)$, $q,q_1,q_2\in [1,\nf]$, and $1/p=1/p_1+1/p_2$, $1/q\le 1/q_1+1/q_2$. Then $f_1f_2\in L^{p,q}(U)$, with
   \begin{align*}
       |f_1f_2|_{L^{p,q}(U)}\le C(p_1,p_2,q_1,q_2) |f_1|_{L^{p_1,q_1}(U)}|f_2|_{L^{p_2,q_2}(U)}.
   \end{align*}
\end{Lm}
\begin{Dfi}
    \label{dfi-So-Lor}
    Let $k\in \mathbb N^+$ and $U\subset \R^n$ be an open set. For $1< p< \infty$, $1\le q\le \infty$, we set \begin{align*}
    W^{k,(p,q)}(U)\coloneqq\Big\{f\in L^{p,q}(U)\colon \p^\al f\in L^{p,q}(U) \text{ for each } 0\le |\al|\le k\Big\}.
    \end{align*}
    We also define the negative-order Sobolev and Sobolev--Lorentz spaces by
    \begin{align*} W^{-k,p}(U)&\coloneqq\bigg\{f\in\mathcal D'(U)\colon f=\sum_{|\alpha|\le k} \p^\alpha f_{\alpha} \mbox{ for some }\{f_{\al}\}\subset L^p(U)\bigg\},\\[0.5mm]
   W^{-k,(p,q)}(U)&\coloneqq\bigg\{f\in\mathcal D'(U)\colon f=\sum_{|\alpha|\le k} \p^\alpha f_{\alpha} \mbox{ for some }\{f_{\al}\}\subset L^{p,q}(U)\bigg\}. \end{align*}
  The corresponding norms are \begin{align*}
       \|f\|_{W^{-k,p}(U)}&\coloneqq\inf \bigg\{\sum_{|\alpha|\le k}\|f_\alpha\|_{L^p(U)}\colon f=\sum_{|\alpha|\le k} \p^\alpha f_{\alpha}\bigg\},\\
       \|f\|_{W^{-k,(p,q)}(U)}&\coloneqq\inf \bigg\{\sum_{|\alpha|\le k}\|f_\alpha\|_{L^{p,q}(U)}\colon f=\sum_{|\alpha|\le k} \p^\alpha f_{\alpha}\bigg\}.
   \end{align*}
    When $U$ is bounded, we denote by $W_0^{k,p}(U)$ the closure of $C_c^\infty(U)$ in $W^{k,p}(U)$, equipped with the norm \begin{align*}
        \|f\|_{W^{k,p}_0(U)}\coloneqq\|\g^k f\|_{L^p(U)}.
    \end{align*}
     If $1<p<\infty$, $1< q\le \nf$, and \(1/p+1/p'=1\), \(1/q+1/q'=1\), then the same argument as in \cite[Section 1.1.15]{Mazya2011} implies
\begin{align}\label{SoLodua}
W^{-k,(p,q)}(U)=\bigl(W^{k,(p',q')}_{0}(U)\bigr)^{\!*}.
\end{align}
\end{Dfi}
Applying estimates for Riesz potentials, we obtain the following embedding results for Sobolev--Lorentz spaces, see for instance \cite[Eqs.~(1.3)--(1.5)]{Mingi11}, \cite[Ch.~4, Thm.~4.18]{Bennett88}, and \cite[Eq.~(1.2.4) and Thm.~3.1.4]{Adams96}.
\begin{Lm}\label{lm:LpembW-1}
    Let $U\subset \R^n$ be an open set, $1< p<n$, and $1\le q\le \nf$. Suppose $f\colon U\to \R$ is measurable.
\begin{enumerate}[(i)]
  \item If $f\in L^1(U)$, then $f\in W^{-1,(\frac n{n-1},\nf)}(U)$, with 
  \begin{align*}
      \|f\|_{W^{-1,\left( \frac{n}{n-1},\nf \right)}(U)}\le C(n) \|f\|_{L^1(U)}.
  \end{align*}
  \item If $f\in L^{p,q}(U)$, then $f\in W^{-1,(\frac{np}{n-p},q)}(U)$, with 
  \begin{align*}
      \|f\|_{W^{-1,\left( \frac{np}{n-p},q \right)}(U)}\le C(n,p) \|f\|_{L^{p,q}(U)}.
  \end{align*}
\end{enumerate}
\end{Lm}

\section{Weak immersions}\label{sec-weakimm}
We follow the definition as in \eqref{weakimm}.
\begin{Dfi}\label{defweakimm}
Let $\Si$ be an $n$-dimensional closed smooth manifold with a fixed smooth reference metric $g_0$. Fix $m,k\in \N$ and $1\le p \le \infty$. We
say that $\bP$ is a $W^{k,p}$ \textit{weak immersion} and write $\vec{\Phi}\in W^{k,p}_{\text{imm}}(\Sigma, \R^m) $ if $\vec{\Phi}\in W^{k,p}\cap W^{1,\infty}(\Sigma, \R^m)$ and there exists a constant $\La>0$ such that the following holds for a.e. $q\in \Si$ and any $X\in T_q\Sigma$:
\begin{align}\label{immcon}
\Lambda^{-1}\,g_0( X,X) \le | d\vec{\Phi}_q(X)|_{\R^m}^2  \le \Lambda \,g_0( X,X).
\end{align} 
We define $W^{k,p}_{\text{imm}}(B^n,\R^m)$ similarly. 
\end{Dfi}

In this section, we provide some additional material on the notion of weak immersions developed in the lecture notes \cite{Ri16}. First, we prove the existence of conformal coordinates for a weak immersion in \Cref{sec:ConformalCoordinates}. Then, we prove that weak immersions with VMO derivatives have integral densities and are locally injective in Section \ref{sec:Density}. Finally, we prove that a weak immersion can always be approximated by smooth immersions and that the underlying conformal structure (in the case of surfaces) converges as well in \Cref{sec:Approximation}. Throughout this paper, we denote by $D^2$ the unit disk of $\R^2$.

\subsection{Existence of conformal coordinates}\label{sec:ConformalCoordinates}

We start with the definition of conformal coordinates, or isothermal coordinates.

\begin{Dfi}\label{dficoncha}
    Let $\Si$ be a $2$-dimensional closed smooth manifold. Consider a weak immersion $\bP\in W^{1,\infty}_{\textup{imm}}(\Si,\R^m)$ and $U\subset \Si$ an open set. Let $\vp\colon U\subset \Si\rightarrow \vp(U)\subset \R^2$ be a $W^{1,2}_{\textup{loc}}$ bi-Sobolev homeomorphism (i.e. a homeomorphism in $W^{1,2}_{\textup{loc}}$ with a $W^{1,2}_{\textup{loc}}$ inverse). We say $\vp$ is a (weak) \textit{isothermal chart} of $\bP$ and $\bP\circ \vp^{-1}$ is \textit{(weakly) conformal} if there exists a (Lebesgue) measurable function $\la\colon\vp(U)\rightarrow \R$ such that for $i,j\in \{1,2\}$ there holds \begin{align*}
        \p_i(\bP\circ \vp^{-1})\cdot \p_j(\bP\circ \vp^{-1})=e^{2\la} \delta_{ij} \qquad \text{a.e. in }\vp(U). 
    \end{align*}
    We define isothermal charts for $\bP\in W^{1,\infty}_{\textup{imm}}(D^2,\R^m)$ similarly.
\end{Dfi}
Throughout this paper, we will frequently use the notion of non-smooth metrics and their conformal class.
\begin{Dfi}\label{dfi-weak-metric}
    Let $\Si$ be an $n$-dimensional smooth manifold. Let $g$ be a measurable section on the bundle $T^*\Sigma\ot T^*\Sigma$, we say that $g$ is a \textit{metric} if $g_p$ is symmetric and positive definite on $(T_p\Sigma)^2$ for a.e. $p\in \Sigma$. Let $g_1,g_2$ be two metrics, we say that $g_1$ is (weakly) conformal to $g_2$ if there exists a measurable function $\la\colon\Sigma\rightarrow\R$ such that $g_1=e^{\la}g_2$. 
\end{Dfi}

Our main goal in this subsection is to prove the following result:\begin{Th}\label{thexicon}
    Let $\bP\in W_{\textup{imm}}^{2,2}(D^2,\R^m)$, then there exists an open neighborhood $U$ of $0\in D^2$ with an orientation-preserving weak isothermal chart $\vp\in W^{2,2}\cap W^{1,\infty}(U,\R^2)$ of $\bP$. Moreover, we have $\vp^{-1}\in W^{2,2}\cap W^{1,\infty}(\vp(U))$.
\end{Th}

As a corollary, we obtain the following global result.
\begin{Co}\label{coconstr}
     Let $\Si$ be a $2$-dimensional closed manifold  with a smooth structure $\mathcal U$, and $\bP\in W^{2,2}_{\textup{imm}}(\Si,\R^m)$. Then we have:
     \begin{enumerate}[(i)]
         \item there exists a smooth atlas $\mathcal V$ on $\Si$ consisting of all the weak isothermal charts of $\bP$. If $\Si$ is oriented, $\mathcal V$ defines a complex structure on $\Si$.
         \item For any choice of two charts respectively in $\mathcal U$ and $\mathcal V$, the transition map between them is a $W^{1,\infty}_{\textup{loc}}\cap W^{2,2}_{\textup{loc}}$ bi-Sobolev homeomorphism.
         \item There exists a Riemannian metric $h$ conformal and smooth on $(\Si,\mathcal V)$ such that $h$ has constant Gaussian curvature $1$, $-1$ or $0$, and  $g_{\bP}\coloneqq\bP^*g_{\text{std}}=e^{2\al}h$ for some $\al\in W^{2,1}(\Si)\subset C^0(\Si)$.
     \end{enumerate}
\end{Co}

\begin{proof}
\begin{enumerate}[(i)] 
    \item We first assume $\bP\in W^{1,\infty}_{\textup{imm}}(D^2,\R^m)$ and $\vp_1,\vp_2\in W^{1,2}(D^2,\R^2)$ be isothermal charts as in Definition \ref{dficoncha}. For $k\in \{1,2\}$, since $\vp_k$ is a $W^{1,2}_{\textup{loc}}$ bi-Sobolev homeomorphism, by standard degree theory we have $\vp_k$ is differentiable a.e. and either $\det(\nabla \vp_k)>0$ a.e. or $\det(\nabla \vp_k)<0$ a.e., see \cite[Theorem 1.7 \& Lemma A.28]{Hencl14}. Without loss of generality we assume $\det(\g \vp_k)>0$ a.e.
Let $g_{ij}\coloneqq\p_i \bP\cdot \p_j \bP $, then following the computation in \cite[Section 1.5.1]{Imayoshi92}, since $\vp_k$ is an isothermal chart, we have 
\begin{equation}\label{Bel-metric}
\frac{\p_{\bar z} \vp_k}{\p_z \vp_k}=\frac{g_{11}-g_{22}+2ig_{12}}{g_{11}+g_{22}+2\sqrt{g_{11}g_{22}-g_{12}^2}}\qquad\text{a.e.}
\end{equation}
The condition $\bP\in W^{1,\infty}_{\textup{imm}}(D^2,\R^m)$ implies the right-hand side has an $L^\infty$ norm strictly less than $1$. In particular, $\vp_k$ is quasiconformal, and $\vp_1\circ(\vp_2)^{-1}$ is holomorphic on $\vp_2(D^2)$, see Definition \ref{dfiquasicon} and \cite[Corollary 1.2.8]{Nag88}. Now return to the case when $\bP\in W^{2,2}_{\textup{imm}}(\Sigma,\R^m)$. Let $\mathcal V$  be the set of all the weak isothermal charts. Then by the above argument each transition map between elements of $\mathcal V$ is either holomorphic or anti-holomorphic, and $\mathcal V$ is an atlas by Theorem \ref{thexicon}.
    \item This follows from $(i)$ and Theorem \ref{thexicon}.
    \item When $\Si$ is orientable, there exists a metric $h$ of constant Gaussian curvature $1$, $-1$ or $0$ and conformal under $\mathcal V$ by the uniformization 
    theorem for compact Riemann surfaces (see e.g. \cite[Section 4.4]{J06}). When $\Sigma$ is non-orientable, $h$ is constructed by using a $2$-cover from a closed orientable smooth surface to $\Sigma$. 
    
  The fact $\alpha\in W^{2,1}(\Sigma)$ follows from Corollary \ref{coconfactor}. Concerning the embedding $W^{2,1}(\Sigma)\hookrightarrow C^0(\Sigma)$ we refer to \cite[Theorem 3.3.4 \& 3.3.10]{helein2002}.
    \end{enumerate}
\end{proof}
\begin{Rm}\label{rm-gconstr}
If $\Sigma$ is oriented and equipped with a metric $g$ such that in any local coordinates the right-hand side in \eqref{Bel-metric} has an $L^\infty$ norm less than $1$, then by Lemma \ref{exisolbel} we can still define a $g$-associated conformal structure via the Beltrami equation \eqref{Bel-metric}. One major difference is about the regularity of the transition map as in Corollary \ref{coconstr} (ii): it may not be Lipschitz, but will be in $W^{1,p}$ for some $p>2$ instead.
\end{Rm}
Using the conformal structure associated to $\bP\in W^{2,2}_{\textup{imm}}(\Si,\R^m)$, the relation \eqref{liouv} follows from standard computations, see for instance \cite[Theorem 1.6]{Ri16}. By integrating \eqref{liouv} with respect to $d\textup{vol}_h$ (when $\Sigma$ is orientable), we have \eqref{gau-bon} by the Gauss--Bonnet theorem for the smooth metric $h$. Thus Theorem \ref{th-conf} is a consequence of Corollary \ref{coconstr}.\\

The proof of Theorem \ref{thexicon} is similar to the case when $\bP$ is a smooth immersion, while the major difficulty lies in the following analytical lemma, which is a weak inverse function theorem, and its higher-dimensional generalization is proved in \cite{Heinonen00}. 
It does not hold if we replace $W^{2,2}$ by some lower regularity from the perspective of Sobolev spaces. One counter example is the map $z\mapsto z^2/|z|$. 
Here we provide a different proof for the $2$-dimensional case. 
Instead of using properties of maps with bounded length distortion (BLD maps) \cite[Lemma 4.6 \& Theorem 4.7]{martio88}, we use regularity results for Beltrami equations \cite{CFMOZ}. Thanks to the embedding $W^{1,2}(D^2)\hookrightarrow VMO(D^2)$, we obtain another proof in Corollary~\ref{Co:locinj}.
\begin{Lm}[{\cite[Theorem 1.1]{Heinonen00}}]\label{thinv}
    Let $\phi\in W^{1,\infty}\cap W^{2,2}(D^2,\R^2)$. Assume there exists a constant $c>0$ such that $\det(\nabla \phi)>c$ a.e. in $D$. Then there exists an open neighborhood $U\subset D^2$ of $0$ such that $\phi$ is open and injective on $U$, and its inverse $\phi^{-1}\in W^{1,\infty}\cap W^{2,2}(\phi(U),\R^2)$.
\end{Lm}
\begin{proof}
    Without loss of generality assume $\phi(0)=0$ by translation. We denote $h_{ij}\coloneqq \langle \p_i \phi,\p_j \phi\rangle$. Choose $ \vp \in C_c^\infty(D^2)$ such that $\vp(\R^2)\subset [0,1]$ and $\vp\equiv 1$ on $D_{1/2}(0)$.  Then we regard $\phi$ as a function from $D^2\subset \C$ to $\C$ and define \begin{align*}
        \mu\coloneqq\vp \frac{\p_{\bar z} \phi}{\p_z \phi}.
    \end{align*} By direct computation we have \begin{align*}
        |\mu|^2\le \frac{|\p_{\bar z} \phi|^2}{|\p_z \phi|^2}=\frac{h_{11}+h_{22}-2\det(\nabla \phi)}{h_{11}+h_{22}+2\det (\nabla \phi)}\le \frac{\|\nabla \phi\|^2_{L^\infty(D^2)}-c}{\|\nabla \phi\|^2_{L^\infty(D^2)}+c}<1 \quad \text{ a.e. on $D^2$.}
    \end{align*}
    Hence, it holds $\|\mu\|_{L^\infty(\C)}<1$. Then we apply Lemma \ref{exisolbel} and obtain a quasiconformal map $f\colon\C\rightarrow \C$ such that $\p_{\bar z} f=\mu\,\p_{z}f$ almost everywhere.
    
    Let $h\coloneqq\phi\circ f^{-1}$. By the same computation as in \cite[Corollary 1.2.8]{Nag88}, the map $h$ is holomorphic on $f(D_{1/2}(0))$. Let $k\ge 1$ be the order of zero of $h$ at $0$. Then there exists a constant $ A\neq 0\in \C$, a bounded open set $U_1\subset f(D_{1/2}(0))$ containing $0$ and a map $g\in 
C^\infty(\overline {U_1},\C)$ conformal from $U_1$ to $g(U_1)\subset \C$ with 
$g(0)=0$ and $g'(0)=1$ such that $h=A\, g^k $ on $U_1$. Let $f_0\coloneqq g\circ f$, then since 
$h=\phi\circ f^{-1}$, we have $\phi=h\circ f=A(g\circ f)^k=A \,f_0^k$ on $f^{-1}(U_1)$. Since $f$ is quasiconformal and satisfies the equation $\p_{\bar z} f=\mu\,\p_{z}f$ with $\mu \in W^{1,2}(\C)$ of compact support, by Lemma \ref{lmregbel} we have $f\in W^{2,q}_{\textup{loc}}(\C)$ for any $q<2$. 
Fix $\al\in (0,1)$. By the Sobolev embedding we have $f\in C^{0,\al}_{\textup{loc}}(\C)$. In particular, this implies that $f_0=g\circ f\in C^{0,\al}(f^{-1}(\overline{U_1}))$, and there exists a constant $C_0>0$ such that for all $x\in f^{-1}\big( \overline{U_1}\big)$, it holds that
\begin{align}\label{holest}
     |f_0(x)|\le  C_0|x|^{\al}.
    \end{align}
Now we prove that $k=1$. By contradiction, suppose that $k\ge 2$. Then there exists a positive constant $C_1$ such that \begin{align}\label{sqrtc<fdf}
 \sqrt c\le \sqrt{ \det(D\phi)}  \le |\p_z \phi| =|kA f_0^{k-1}|\, |\p_z f_0|\le C_1|f_0|\, |\p_z f_0|\quad \text{ a.e. on }f^{-1}(\overline{U_1}).
\end{align}
By \eqref{holest} and \eqref{sqrtc<fdf}, there exists a positive constant $C_2$ such that for all $x\in f^{-1}(\overline{U_1})$, it holds that
\begin{align}\label{dz>|x|^-a}
    |\p_z f_0(x)|\ge C_2|x|^{-\al}.
\end{align}
Choose $p\in (1,\nf)$ such that $\al p\ge 2$. Then it follows from \eqref{dz>|x|^-a} that
\begin{align}\label{intpfpinf}
    \int_{f^{-1}(U_1)} |\p_z f_0|^p=\infty.
\end{align}
Since $f\in W^{2,q}_{\loc}(\C)$ for any $q<2$, we obtain by Sobolev embedding that $f\in W^{1,p}_{\textup{loc}}(\C)$. Hence we have $f_0=g\circ f \in W^{1,p}(f^{-1}(U_1))$, which contradicts~\eqref{intpfpinf}.

Consider an open set $U$ such that $0\in U\Subset f^{-1}(U_1)$. Since $k=1$, we have that $\phi=A(g\circ f)$. Hence, $\phi$ is quasiconformal on $U$, and $\phi^{-1}\in W^{1,\infty}\cap W^{2,2}(\phi(U),\R^2) $ since $
    D\phi^{-1}=(D\phi)^{-1}\circ \phi^{-1}$ a.e. on $\phi (U)$.
\end{proof}
We are now ready to prove Theorem \ref{thexicon}.
\begin{hproof2}
 We apply the Gram--Schmidt process to the frame $(\p_1\bP,\p_2\bP)$ and get 
 \begin{align}\label{gramsch}
    \vec f_1\coloneqq\frac{\p_1\bP}{|\p_1\bP|},\qquad \text{ and } \qquad \vec f_2\coloneqq\frac{\p_2\bP-(\p_2\bP \cdot \vec f_1)\vec f_1}{|\p_2\bP-(\p_2\bP \cdot \vec f_1)\vec f_1|}.
\end{align}
By the definition of weak immersions, there exists a constant $\Lambda>0$ such that for a.e. $p\in D^2$ and any $X\in T_p(D^2)$,
\begin{align}\label{immconD}
\Lambda^{-1} |X|^2 \le | d\vec{\Phi}_p(X)|^2  \le \Lambda |X|^2.
\end{align} 
In particular, we have $|\p_1 \bP|,|\p_2\bP|\in [\La^{-1},\La]$ and $\p_2\bP-(\p_2\bP \cdot \vec f_1)\vec f_1\in [\La^{-1}, \La(1+\La^2)]$.
Hence $(\vec{f}_1,\vec{f}_2)\in L^\infty \cap W^{1,2}(D^2,V_2({\R}^m))$, where $V_2({\R}^m)$ denotes the space of orthonormal 2-frames in ${\R}^m$. 
For $\theta\in W^{1,2}(D^2,{\R})$, we denote $(\vec{f}^{\,\theta}_1,\vec{f}^{\,\theta}_2)$ the rotation by the angle $\theta$ of the original frame  $(\vec{f}_1,\vec{f}_2)$:
\begin{align*}
\vec{f}^{\,\theta}_1+i\vec{f}^{\,\theta}_2=e^{i\theta}(\vec{f}_1+i\vec{f}_2).
\end{align*}
We have $(\vec{f}^{\,\theta}_1,\vec{f}^{\,\theta}_2)\in L^\infty\cap W^{1,2}(D^2,V_2(\R^m))$.
 We denote $g=g_{\bP}$. We look now for a rotation of this orthonormal frame realizing the following absolute minimum (for notation, see \ref{not-met})
\be
\label{I.8}
\begin{array}{l}
\ds\inf_{\theta\in W^{1,2}(D^2,{\R})}\int_{D^2} |\lan \vec{f}^{\,\theta}_1,d\vec{f}^{\,\theta}_2\ran|^2_{g}\ d\textup{vol}_{g}
\\[5mm]
\ds=\inf_{\theta\in W^{1,2}(D^2,{\R})}\int_{D^2} \sum_{i,j=1}^2g^{ij}\, \lan\vec{f}^{\,\theta}_1,\p_{i}\vec{f}^{\,\theta}_2\ran\,\lan\vec{f}^{\,\theta}_1,\p_{j}\vec{f}^{\,\theta}_2\ran\, \sqrt{\det g}\ dx^1\we dx^2.
\end{array}
\ee
For $k\in\{1,2\}$, we have 
\be
\label{I.10}
\lan\vec{f}^{\,\theta}_1,\p_{k}\vec{f}^{\,\theta}_2\ran=\p_{k}\theta+\lan\vec{f}_1,\p_{k}\vec{f}_2\ran.
\ee
Hence the following energy is strictly convex in $W^{1,2}(D^2,{\R})$:
\begin{align*}
E(\theta)\coloneqq\int_{D^2} |\lan \vec{f}^{\,\theta}_1,d\vec{f}^{\,\theta}_2\ran|^2_{g}\, d\textup{vol}_{g}=\int_{D^2} |d\theta+\lan \vec{f}_1,d\vec{f}_2\ran|^2_{g}\ d\textup{vol}_{g}.
\end{align*}
By a standard application of Mazur's lemma (see for instance \cite[Corollary 3.8]{Brezis11}), we obtain a unique $\theta\in W^{1,2}(D^2,\R)$ that achieves the minimum of $E$. It satisfies the Euler--Lagrange equation (see Notation \ref{not-Hosta} for the Hodge star operator $*_g$):
\begin{align}\label{eularequ}
 \forall \phi\in W^{1,2}(D^2,{\R}),\qquad \int_{D^2}d\phi\wedge *_g\lf( d\theta+\lan\vec{f}_1,d\vec{f}_2\ran  \rg) =0.
\end{align}
Denote $(\vec{e}_1,\vec{e}_2)\coloneqq(\vec{f}^{\,\theta}_1,\vec{f}^{\,\theta}_2)$ for this value of $\theta$. From \eqref{eularequ} we obtain that
\[
d\Big(  *_g \lan\vec{e}_1,d\vec{e}_2\ran  \Big)=0 \quad \text{in }\mathcal{D}'(D^2).
\] 
By the weak Poincar\'e lemma, there exists $\la\in W^{1,2}(D^2)$ such that 
\begin{align}\label{eq:def_lambda}
    d\la =  *_g \lan\vec{e}_1,d\vec{e}_2\ran.
\end{align}
Substituting this into \eqref{eularequ}, we have 
\[ 
\forall \phi\in W^{1,2}(D^2,{\R}),\qquad \int_{D^2}d\phi\wedge d\la =0.
\]
Hence for any $\phi\in C^\infty(\overline{D^2})$, we have by Stokes theorem
\[
\int_{\p D^2}\la\, d\phi =\int_{D^2} d(\la\, d\phi)=0.
\]
It follows that
$\la \big |_{\mathbb S^1}$ is constant by the fundamental lemma of the calculus of variations \cite[Lemma 1.1.1]{J98}. Up to adding a constant we can assume $\la \in W^{1,2}_0(D^2)$. By applying the operator $*_g\, d \,*_g$ to the equation \eqref{eq:def_lambda}, we obtain
\[\Delta_g \la=*_g\, d*_g d\la=-*_g d\lan\vec{e}_1,d\vec{e}_2\ran \quad \text{in }D^2. \]
This is written in coordinates as:
\[
\begin{cases}
\ds\,\p_{i}\lf(\sqrt{\det g}\ g^{ij}\,\p_{j}\la\rg)=\p_2 \vec e_1 \cdot \p_1 \vec e_2 -\p_1 \vec e_1 \cdot \p_2 \vec e_2\quad &\mbox{ in }D^2,\\[3mm]
\ds\la=0\quad &\mbox{ on }\p D^2.
\end{cases}
\]
We deduce from Lemma \ref{lm-chali} that $\la\in C^0(\overline{D^2})\cap W^{1,2}_0(D^2,{\R})$.  
For $i\in \{1,2\}$, we define the pullback of $\vec e_i$ by $\bP$ and the dual $1$-forms \begin{align}\begin{dcases}\label{pullvec}e_i\coloneqq g^{jk}\,\lan \vec e_i, \p_j\bP \ran \,\dfrac{\p}{\p x^k}\in L^\infty\cap W^{1,2}(D^2,TD^2),\\[2mm]
e_i^\ast\coloneqq\lan \vec{e}_i,\p_{j}\vec{\Phi}\ran \,dx^j=\lan \vec e_i, d\bP\ran\in L^\infty \cap W^{1,2}(D^2,T^*D^2).
\end{dcases}
\end{align}
It holds that $e_i^*(e_j)=\delta_{ij}$ and $d\bP(e_i)=\vec e_i$ for any $i,j\in\{1,2\}$.
Since both $e_i^*$ and $e_i$ are in $L^\infty\cap W^{1,2}(D^2)$, by a weak version of Cartan's magic formula, we have as in the smooth case \cite[Equation (1.70)]{Ri16}:
\be\label{dei*}
 de_i^*=-e_i^* ([e_1,e_2])\, e_1^*\wedge e_2^*=d\la\wedge e_i^\ast.
\ee
By direct computation we also have 
\[d\bP([e_1,e_2])=\pi_T (e_1(\vec e_2)-e_2(\vec e_1))\in L^2(D^2,\R^m),\] 
where $\pi_T$ is the orthogonal projection onto $d\bP(TD^2)$.
Since $\la\in L^\infty(D^2)\cap W^{1,2}(D^2,{\R})$, we deduce from the chain rule that $e^{-\la}\in W^{1,2}(D^2)$ and $d(e^{-\la})=-e^{-\la}\,d\la\in L^2(D^2,T^*D^2) $. Hence the product rule and \eqref{dei*} lead to
\[
d(e^{-\la}\, e_i^*)=0.
\]
By weak Poincar\'e lemma, we obtain the existence of a map $\vp^i\in W^{1,\infty}\cap W^{2,2}(D^2,{\R})$ satisfying
\be
\label{dp=}
d\varphi^i=e^{-\la}\, e_i^*.
\ee
Let $\vp\coloneqq(\vp^1,\vp^2)\in  W^{1,\infty}\cap W^{2,2}(D^2,{\R^2})$. By \eqref{pullvec} and \eqref{dp=}, we have
\[
\det (\g\varphi)=e^{-2\la}\det\big(\lan \vec{e}_i,\p_{j}\vec{\Phi}\ran\big)=e^{-2\la} \sqrt{\det(\p_i\bP\cdot \p_j\bP)}.
\]
Since $(e_1,e_2)$ is positive and using \eqref{immconD}, we obtain that all the eigenvalues of $(\p_i\bP \cdot \p_j\bP)_{1\le i,j\le 2}$ lie in $[\La^{-1},\La]$ and thus, it holds that
\[
    \det (\g\varphi) \geq \La^{-1}\,e^{-2\|\la\|_{L^\infty}}>0.
\]
It follows that there exists an open neighborhood $U\subset D^2$ of $0$ such that $\vp$ is open and injective on $U$, and its inverse $\vp^{-1}\in W^{1,\infty}\cap W^{2,2}(\vp(U),\R^2)$ by Lemma \ref{thinv}. We define 
\[\frac{\p}{\p \vp^i}\coloneqq \p_i (\vp^{-1})^j\, \frac{\p}{\p x^j}\in L^\infty\cap W^{1,2}(\vp(U),\vp_*(TD^2)),\]
where $(\vp^{-1})^j$ is the $j$-th component of $\vp^{-1}$. Since $e^\la( d\vp^1,d\vp^2 )$ is orthonormal with respect to $g$, 
its dual $e^{-\la}(\frac{\p}{\p \vp^1},\frac{\p}{\p \vp^2})$ is also 
orthonormal with respect to $g$. Therefore, by chain rule, we finally obtain \begin{align*}
     \p_i(\bP\circ \vp^{-1})\cdot \p_j(\bP\circ \vp^{-1})=d\bP\Big(\frac{\p}{\p \vp^i}\Big)\cdot d\bP\Big(\frac{\p}{\p \vp^j}\Big) =g \Big(\frac{\p}{\p \vp^i},\frac{\p}{\p \vp^j}\Big)=e^{2\la} \delta_{ij}.
\end{align*}
\end{hproof2}

\subsection{Constant of Wente's inequality for a metric and consequences}

In 1971, Wente \cite{Wen} discovered that solutions to elliptic equations of the following form are more regular than expected:
\begin{align*}
    \begin{cases}
        \Delta u = \nabla a \cdot \nabla^{\perp} b & \text{in }D^2,\\
        u = 0 & \text{on }\partial D^2.
    \end{cases}
\end{align*}
Indeed, it holds
\begin{align*}
    \|u\|_{L^{\infty}(D^2)} + \|\nabla u\|_{L^2(D^2)} \leq C\|\nabla a\|_{L^2(D^2)} \|\nabla b\|_{L^2(D^2)}.
\end{align*}
Since then, equations that possess this Jacobian structure have been intensively studied. Bethuel--Ghidaglia \cite{bethuel1993} proved that the constant $C$ is independent of the domain. Ge \cite{ge1998}, Braraket \cite{baraket1996} and Topping \cite{topping1997} studied such equations on surfaces and proved that the constant is independent of the underlying surface. A weighted Wente inequality has also been studied by Da Lio, Gianocca and Rivière in \cite{dalio2025,gianocca2025}. In this section, we provide a new proof of the independence of the constant with respect to the underlying metric and domain in Proposition \ref{pr:Wente}. As an application, we obtain the existence of a Coulomb frame under the smallness of the Gauss curvature in Proposition \ref{pr:Clb}.

\subsubsection{Wente inequality}

The strategy follows the one introduced in \cite[Theorem 1.3]{bethuel1993} and \cite[Lemma A.1]{dalio2025b}, relying on a direct study of the Green kernel.
	
	\begin{Prop}\label{pr:Wente}
    Let $\Omega\subset \R^2$ be a smooth simply connected open set. Let $g$ be a smooth metric on $\Omega$. Let $a,b\in W^{1,2}(\Omega)$ and $u\in W^{1,1}(\Omega)$ be the solution to 
		\begin{equation}\label{eq:Wente}
			\left\{
			\begin{aligned}
				& -\lap_g u = *_g (da\wedge db) \quad &&\text{in }\Omega,\\
				& u=0 \quad &&\text{on }\p \Omega.
			\end{aligned}
			\right.
		\end{equation}
		Then, we have the estimate 
		\begin{align*}
        \begin{cases}
 			\|u\|_{L^{\infty}(\Omega)}  \leq 18\,\|da\|_{L^2(\Omega,g)} \|db\|_{L^2(\Omega,g)}, \\[2mm]
            \|du\|_{L^2(\Omega,g)} \leq 3\sqrt{2}\,\|da\|_{L^2(\Omega,g)} \|db\|_{L^2(\Omega,g)}.
        \end{cases}
		\end{align*}
	\end{Prop}
	
	\begin{Rm}
		\begin{itemize}
			\item The first equation of \eqref{eq:Wente} can also be written as
			\begin{align*}
				-\p_i (g^{ij}\sqrt{\det g}\, \p_j u) = \na a\cdot \nap b.
			\end{align*}
			\item We can reduce the regularity on $g$, that is to say we can assume $g$ to be a metric of class $L^{\infty}(\Omega)$ such that there exists $C>0$ satisfying $C^{-1}g_{\text{std}}\leq g\leq Cg_{\text{std}}$. By using a mollifier argument, there exists a sequence $(g_k)_{k\in\N}$ of smooth metrics converging to $g$ in $L^{\infty}$ and almost everywhere such that $C^{-1} g_{\text{std}} \leq g_k\leq Cg_{\text{std}}$. Hence, we can pass to the limit in Proposition \ref{pr:Wente}. 
		\end{itemize}
	\end{Rm}
	
	\begin{proof}
    Up to dividing $a$ and $b$ by $\|da\|_{L^2(\Omega,g)}$ and $\|db\|_{L^2(\Omega,g)}$ respectively, we can assume that $\|da\|_{L^2(\Omega,g)}=\|db\|_{L^2(\Omega,g)}=1$.
		Given $p\in \Omega$, let $G_p$ be the solution to 
		\begin{equation*}
			\begin{cases}
				-\lap_g G_p = \delta_p & \text{in }\Omega,\\
				G_p = 0 & \text{on }\p \Omega.
			\end{cases}
		\end{equation*}
		By the maximum principle, it holds $G_p>0$ on $\Omega$, see for instance \cite[Chapter 4, Section 2, Theorem 4.17]{aubin1998}. Given $0\leq \alpha<\beta$, we define
		\begin{align*}
			\omega_p(\alpha,\beta) \coloneqq \{x\in \Omega : \alpha\leq G_p(x)\leq \beta \}, & & \Omega_p(\alpha) \coloneqq\{ x\in \Omega : \alpha \leq G_p(x)\}.
		\end{align*}
		By elliptic regularity, it holds $G_p\in C^{\infty}(\Omega\setminus \{p\})$. Thanks to Sard's theorem (see for instance \cite[Chapter 6, Theorem 6.10]{jlsmo}), for almost every $\alpha>0$, the set $G_p^{-1}(\{\alpha\})$ is a regular curve in $\Omega$. We orient this curve with the vector field
		\begin{align*}
			\nu^i = - \frac{g^{ij}\p_j G_p}{|dG_p|_g}.
		\end{align*}
		Hence we obtain by integration by parts
		\begin{align*}
			\int_{G_p^{-1}(\{\alpha\})} |dG_p|_g\, d\textup{vol}_g = -\int_{G_p^{-1}(\{\alpha\})} \p_{\nu} G_p\, d\textup{vol}_g = -\int_{\Omega_p(\alpha)} \lap_g G_p\, d\textup{vol}_g = 1.
		\end{align*}
		Using the coarea formula, see for instance \cite[Chapter IV, Theorem 1]{chavel1984}, we obtain for almost every $0\leq \alpha<\beta$:
		\begin{align}\label{eq:L2_Gp}
			\int_{\omega_p(\alpha,\beta)} |dG_p|^2_g\, d\textup{vol}_g = \int_{\alpha}^{\beta} \int_{G_p^{-1}(\{s\})} |dG_p|_g\, d\textup{vol}_g\, ds = \beta- \alpha.
		\end{align}
		Now, the solution $u$ to \eqref{eq:Wente} is given by 
		\begin{align}\label{eq:rep}
			u(p) = \int_{\Omega} G_p(x)\, da\wedge db.
		\end{align}
		For each $n\in \N$, there exists $\alpha_n\in [n,n+1]$ such that 
		\begin{align}
			\int_{G_p^{-1}(\{\alpha_n\})} |da|_g + |db|_g\, d\textup{vol}_g & \leq 2\int_n^{n+1} \int_{G_p^{-1}(s)} (|da|_g + |db|_g)\, |dG_p|_g\, d\textup{vol}_g \nonumber \\[2mm]
            &\leq  2\int_{\omega_p(n,n+1)} (|da|_g + |db|_g)\, |dG_p|_g\, d\textup{vol}_g.\label{eq:ch_a}
		\end{align}
		From \eqref{eq:rep}, it holds
		\begin{align}
			u(p) & = \sum_{n\in\N} \int_{\omega_p(\alpha_n,\alpha_{n+1})} G_p(x)\, *_g(da\wedge db)\, d\textup{vol}_g \nonumber\\[2mm]
			& = \sum_{n\in\N} \int_{\omega_p(\alpha_n,\alpha_{n+1})} [G_p(x) - n ]\, *_g(da\wedge db)\, d\textup{vol}_g \label{eq:rep1}\\[2mm]
            &\qquad + \sum_{n\in\N} n\int_{\omega_p(\alpha_n,\alpha_{n+1})} *_g(da\wedge db)\, d\textup{vol}_g. \label{eq:rep2}
		\end{align}
		We estimate the term \eqref{eq:rep1} by Cauchy--Schwarz:
		\begin{align}
			\left| \sum_{n\in\N} \int_{\omega_p(\alpha_n,\alpha_{n+1})} [G_p(x) - n ]\, *_g(da\wedge db)\, d\textup{vol}_g \right| & \leq 2 \sum_{n\in\N} \int_{\omega_p(\alpha_n,\alpha_{n+1})} |da|_g\, |db|_g\, d\textup{vol}_g \nonumber\\[2mm]
			& \leq 2\int_{\Omega} |da|_g\, |db|_g\, d\textup{vol}_g. \label{eq:rep11}
		\end{align}
		We estimate the term \eqref{eq:rep2} by integration by parts:
		\begin{align}
			& \sum_{n\in\N} n\int_{\omega_p(\alpha_n,\alpha_{n+1})} *_g(da\wedge db)\, d\textup{vol}_g \nonumber \\[2mm]
            = &\ \sum_{n\in\N} n\left(\int_{G_p^{-1}(\alpha_{n+1})} (\p_{\tau} b) a\, d\textup{vol}_g - \int_{G_p^{-1}(\alpha_{n})} (\p_{\tau} b) a\, d\textup{vol}_g \right) \nonumber \\[2mm]
			     =& - \sum_{n\in\N^*} \int_{G_p^{-1}(\alpha_{n})} (\p_{\tau} b) a\, d\textup{vol}_g . \label{eq:rep21}
		\end{align}
		For each $n\in \N^*$, we decompose $G_p^{-1}(\{\alpha_n\}) = \bigcup_{i\in I_n} \gamma^i_n$, where every $\gamma^i_n$ is a closed embedded curve in $\Omega$. Then, it holds
		\begin{align*}
			\left| \int_{G_p^{-1}(\alpha_{n})} (\p_{\tau} b) a\, d\textup{vol}_g \right| & = \left| \sum_{i\in I_n} \int_{\gamma^i_n} (\p_{\tau} b) a\, d\textup{vol}_g \right| \\[2mm]
			 & = \left| \sum_{i\in I_n} \int_{\gamma^i_n} (\p_{\tau} b) \left(a - \fint_{\gamma^i_n} a\, d\textup{vol}_g\right)\, d\textup{vol}_g \right| \\[2mm]
			 & \leq \sum_{i\in I_n} \left( \int_{\gamma^i_n} |\p_{\tau} b|\, d\textup{vol}_g\right)  \left( \int_{\gamma^i_n} |da|_g\, d\textup{vol}_g \right) \\[2mm]
			 & \leq \left(\int_{G_p^{-1}(\{\alpha_n\})} |db|_g\, d\textup{vol}_g \right)\left(\int_{G_p^{-1}(\{\alpha_n\})} |da|_g\, d\textup{vol}_g \right).
		\end{align*}
		Coming back to \eqref{eq:rep21}, we obtain
		\begin{align*}
			\left| \sum_{n\in\N} n\int_{\omega_p(\alpha_n,\alpha_{n+1})} *_g(da\wedge db)\, d\textup{vol}_g \right| \leq \sum_{n\in\N^*} \left(\int_{G_p^{-1}(\{\alpha_n\})} |db|_g\, d\textup{vol}_g \right)\left(\int_{G_p^{-1}(\{\alpha_n\})} |da|_g\, d\textup{vol}_g \right).
		\end{align*}
		Using the choice of the $\alpha_n$ in \eqref{eq:ch_a}, we obtain
		\begin{align*}
			& \left| \sum_{n\in\N} n\int_{\omega_p(\alpha_n,\alpha_{n+1})} *_g(da\wedge db)\, d\textup{vol}_g \right| \\[3mm]
            \leq &\ 4\sum_{n\in\N^*} \left(\int_{\omega_p(n,n+1)} (|da|_g+|db|_g)\, |dG_p|_g\, d\textup{vol}_g \right)^2 \\[3mm]
			\leq&\ 4\sum_{n\in\N^*} \|dG_p\|_{L^2(\omega_p(n,n+1),g)}^2\, \left( \|da\|_{L^2(\omega_p(n,n+1),g)} + \|db\|_{L^2(\omega_p(n,n+1),g)}\right)^2.
		\end{align*}
		Thanks to \eqref{eq:L2_Gp}, we obtain
		\begin{align}
			& \left| \sum_{n\in\N} n\int_{\omega_p(\alpha_n,\alpha_{n+1})} *_g(da\wedge db)\, d\textup{vol}_g \right| \nonumber \\[3mm]
			& \leq 4\sum_{n\in\N^*} \left( \|da\|_{L^2(\omega_p(n,n+1),g)} + \|db\|_{L^2(\omega_p(n,n+1),g)}\right)^2 \nonumber\\[3mm]
			& \leq 8\,\|da\|_{L^2(\Omega,g)}^2 + 8\|db\|_{L^2(\Omega,g)}^2 = 16. \label{eq:rep22}
		\end{align}
		We used Cauchy--Schwarz and the equalities $\|da\|_{L^2(\Omega,g)}=\|db\|_{L^2(\Omega,g)}=1$ for the last estimate. Coming back to \eqref{eq:rep}, together with \eqref{eq:rep11} and \eqref{eq:rep22}, we deduce that 
		\begin{align*}
			\|u\|_{L^{\infty}(\Omega)} \leq 18.
		\end{align*}
		We multiply \eqref{eq:Wente} by $u$ and integrate by parts:
		\begin{align*}
			\int_{\Omega} |du|^2_g\, d\textup{vol}_g & = -\int_{\Omega} u\, \lap_g u\, d\textup{vol}_g  \leq \|u\|_{L^{\infty}(\Omega)} \int_{\Omega} |da|_g\, |db|_g\, d\textup{vol}_g \leq 18.
		\end{align*}
	\end{proof}

	\subsubsection{Estimate of the moving frame}
	
	It has been proved in \cite{li2013} that a Wente-type inequality leads to an estimate on Coulomb frames. Using \ref{pr:Wente}, we obtain the following result.
	
	\begin{Prop}\label{pr:Clb}
        Let $\Omega\subset \R^2$ be a smooth simply connected open set. Consider a weak immersion $\bP :\Omega\to \R^m$ such that, with $g=g_{\bP}$, 
		\begin{align*}
			\int_{\Omega} |K_g|\, d\textup{vol}_g \leq \frac{1}{36}.
		\end{align*}
		Then, there exists an orthonormal frame $(\vec{e}_1,\vec{e}_2) \in W^{1,2}(\Omega,g)$ such that 
		\begin{align*}
			\int_{\Omega} |d\vec{e}_1|^2_g + |d\vec{e}_2|^2_g\, d\textup{vol}_g \leq \frac{3}{2}\int_{\Omega} |d\bn|^2_g\, d\textup{vol}_g.
		\end{align*} 
	\end{Prop}
	
	\begin{proof}
		Thanks to \cite[Lemma 4.1.3]{helein2002}, there exists a frame $(\vec{e}_1,\vec{e}_2) \in W^{1,2}(\Omega,g)$ such that 
		\begin{equation}\label{eq:Clb}
			\begin{cases}
				d^{*_g} \left\lan d\vec{e}_1, \vec{e}_2\right\ran_{\R^m} = 0 & \text{in }\Omega,\\
				\left\lan \p_{\nu} \vec{e}_1, \vec{e}_2 \right\ran_{\R^m} = 0 & \text{on }\p \Omega,
			\end{cases}
		\end{equation}
		where $\nu$ is the unit outward-pointing normal for $g$. Moreover, by decomposing each $d\vec{e}_i$ into the sum of its tangential part along $\bP$ and its normal part, we obtain the following estimate
		\begin{align}\label{eq:frame}
			\int_{\Omega} |d\vec{e}_1|^2_{g,\R^m} + |d\vec{e}_2|^2_{g,\R^m}\, d\textup{vol}_g \leq 2\int_{\Omega} |\left\lan d\vec{e}_1, \vec{e}_2 \right\ran_{\R^m}|^2_g\, d\textup{vol}_g + \int_{\Omega} |d\bn|^2_g\, d\textup{vol}_g.
		\end{align}
		By Poincaré Lemma and \eqref{eq:Clb}, there exists $\lambda\in W^{1,2}(\Omega,g)$ such that 
		\begin{equation}\label{eq:lbd}
			\begin{cases}
				d\lambda = *_g \left\lan d\vec{e}_1, \vec{e}_2 \right\ran_{\R^m} & \text{in }\Omega,\\
				\lambda = 0 & \text{on }\p \Omega.
			\end{cases}
		\end{equation}
		Indeed, the boundary condition in \eqref{eq:Clb} implies that $\lambda$ is constant on $\p \Omega$. Since the first relation of \eqref{eq:lbd} depends only on $d\lambda$ and not on $\lambda$ itself, we can add a constant in order to get the second relation of \eqref{eq:lbd}. We apply $d^{*_g}$ to \eqref{eq:lbd} to obtain
		\begin{equation}\label{eq:lbd2}
        \begin{cases}
            -\lap_g\lambda = -*_g d\left( \left\lan d\vec{e}_1, \vec{e}_2 \right\ran_{\R^m}\right) & \text{in }\Om,\\
            \lambda = 0 & \text{on }\p \Om.
        \end{cases}
		\end{equation}
		Thanks to Proposition \ref{pr:Wente}, it holds
		\begin{align*}
        \begin{cases}
 			\|\lambda\|_{L^{\infty}(\Omega)}  \leq 18\,\|d\vec{e}_1\|_{L^2(\Omega,g)}\, \|d\vec{e}_2\|_{L^2(\Omega,g)}, \\[2mm]
            \|d\lambda\|_{L^2(\Omega,g)} \leq 3\sqrt{2}\,\|d\vec{e}_1\|_{L^2(\Omega,g)}\, \|d\vec{e}_2\|_{L^2(\Omega,g)}.
        \end{cases}
		\end{align*}
		Denote $\vae_0\coloneqq\|K_g\|_{L^1(\Omega,g)}$. We multiply the first relation of \eqref{eq:lbd2} by $\lambda$ and integrate by parts:
		\begin{align*}
			\int_{\Omega} |d\lambda|^2_g\, d\textup{vol}_g & \leq \|\lambda\|_{L^{\infty}(\Omega)} \int_{\Omega} \left| d\left( \left\lan d\vec{e}_1, \vec{e}_2 \right\ran_{\R^m}\right) \right|_g\, d\textup{vol}_g \\[3mm]
			 & \leq \|\lambda\|_{L^{\infty}(\Omega)} \int_{\Omega} |K_g|\, d\textup{vol}_g \\[3mm]
			 & \leq 18\, \varepsilon_0\, \|d\vec{e}_1\|_{L^2(\Omega,g)}\, \|d\vec{e}_2\|_{L^2(\Omega,g)}.
		\end{align*}
		Thanks to \eqref{eq:lbd}, we obtain
		\begin{align*}
			\int_{\Omega} |\left\lan d\vec{e}_1, \vec{e}_2 \right\ran_{\R^m}|^2_g\, d\textup{vol}_g & \leq 18\, \varepsilon_0\, \|d\vec{e}_1\|_{L^2(\Omega,g)} \|d\vec{e}_2\|_{L^2(\Omega,g)} \leq 9\, \varepsilon_0 \left( \|d\vec{e}_1\|_{L^2(\Omega,g)}^2 + \|d\vec{e}_2\|_{L^2(\Omega,g)}^2 \right).
		\end{align*}
		Using \eqref{eq:frame}, we obtain
		\begin{align*}
			\int_{\Omega} |\left\lan d\vec{e}_1, \vec{e}_2 \right\ran_{\R^m}|^2_g\, d\textup{vol}_g \leq 9\, \varepsilon_0 \left(2\int_{\Omega} |\left\lan d\vec{e}_1, \vec{e}_2\right\ran_{\R^m}|^2_g\, d\textup{vol}_g + \int_{\Omega} |d\bn|^2_g\, d\textup{vol}_g\right).
		\end{align*}
		If $\vae_0\le\frac{1}{36}$, we obtain
		\begin{align*}
			\frac{1}{2}\int_{\Omega} |\left\lan d\vec{e}_1, \vec{e}_2 \right\ran_{\R^m}|^2_g\, d\textup{vol}_g \leq \frac{1}{4}\int_{\Omega} |d\bn|^2_g\, d\textup{vol}_g.
		\end{align*}
		We conclude thanks to \eqref{eq:frame}.
	\end{proof}

\subsection{Local injectivity and integral densities}\label{sec:Density}

\begin{Dfi}\label{defdensi}
   Let $\Si$ be an orientable $n$-dimensional closed smooth manifold and consider $\bP\in W^{1,\infty}_{\textup{imm}}(\Si,\R^m)$. For any measurable set $E\subset \R^m$, we define the volume of $E\cap \bP(\Sigma)$ as 
   \[
   \textup{Vol}(E)\coloneqq\int_{\Si} \mathbf 1_{\bP^{-1}(E)} \,d\textup{vol}_{g_{\bP}}.
   \]
   For $x\in \R^m$, we define the \textit{density} $\theta_{x}\in \N_0$ at $x$ by the following limit (whenever it exists)
   \begin{align*}
       \theta_x\coloneqq\lim_{r\rightarrow 0^+}\frac{\textup{Vol}(\bP(\Si)\cap B_r(x) )}{r^n|B^n|}.
   \end{align*}
\end{Dfi}
In this section, we prove that weak immersions with VMO derivatives are locally injective and have a well-defined density, which is an integer. For notation on VMO functions, see Section \ref{sec-VMO}.
\begin{thm}\label{loin}
    Let $f\in W^{1,\infty}(B^n,\R^m)$. Suppose that there exist a function $\omega\colon (0,\infty)\rightarrow [0,\infty]$ with $\lim_{r\rightarrow 0} \omega (r)=0$ and a constant $\La>0$ such that \begin{align*}\begin{dcases}
        \beta_{r,B^n}(\g f)\le \omega(r),\quad &\text{for all }\,r>0,\\[3mm]
        \La^{-1}|v|\le |d f_p(v)|\le \La |v|, \quad&\text{for a.e. }p\in B^n \text{ and all }\,v\in T_p(B^n).
        \end{dcases}
        \end{align*} In particular, $\g f\in VMO(B^n)$. Then there exists a constant $C>0$ and an open neighborhood $U$ of the origin, both depending only on $\La,\, \omega,\,n,\,m$ such that $$|f(x)-f(y)|\ge C|x-y|,\qquad \text{for all } \,x,y\in U.$$
\end{thm}
\begin{proof}
    Let $V$ be the set of full-rank $m\times n$ matrices, then $V$ is open in the space of $m\times n$ matrices. By the definition of weak immersion, there exists a compact subset $K$ of $V$ such that $\g f\in K$ a.e. in $B^n$. Hence there exists $\vae>0$ such that the set $K'\coloneqq\{M\in V\colon \dist(M,K)\le \vae\}$ is compact and contained in $V$. It follows that there exists a constant $C_0>0$ such that for any $v\in \R^n$, $M\in K'$, it holds\begin{align}\label{lowb}
        |Mv|\ge C_0 |v|.
    \end{align}
    Moreover, the constants $C_0,\,\vae$, and $K$ can be chosen to depend on $\La,\, \omega,\,n,\,m$ only. We choose $r_0, \vae_0\in (0,1/8)$ small enough (to be determined later) depending on $\La,\, \omega,\,n,\,m $ only such that for every ball $B$ of radius $\le 4r_0$ we have $\fint_B\lf|\g f-\fint_B \g f\rg|<\min(\vae,\vae_0)$. 
    Now fix $x_0,y_0\in B_{r_0}(0)$. Let $r\coloneqq|x_0-y_0|$ and consider the ball $B'\coloneqq B_{2r}(x_0)$. We denote
    $$
    M_0\coloneqq\fint_{B'} \g f.
    $$ 
    Since $2r<4r_0$ and $\g f\in K$ a.e., we have 
    $$
    \fint_{B'}|\g f-M_0|<\min(\vae,\vae_0).
    $$
    Hence, it holds $M_0\in K'$. Letting $f_0(x)\coloneqq f(x)-M_0 x$, we apply Morrey's inequality \cite[Theorem 5.6.4]{evans} in the case $p=2n$ and John--Nirenberg inequality \cite[Theorem 3.5]{hanlinpde}, for $\vae_0$ small enough we have \begin{align}\begin{aligned}
         \label{err-est}
    |f_0(x_0)-f_0(y_0)|&\le C_1\mkt  r^{1/2}\lf(\int_{B'}|\g f_0|^{2n}\rg)^{\frac{1}{2n}}\\
&= C_1 \mkt  r^{1/2}\lf(\int_{B'}|\g f-M_0|^{2n}\rg)^{\frac{1}{2n}}\\[2mm]
&\le C_2\mkt  r\mkt  \vae_0= 2^{-1}C_0\mkt  r=2^{-1}C_0 |x_0-y_0|,\end{aligned}
\end{align} where $C_1,\, C_2$ are positive constants depending only on $n$, and we choose $\vae_0= (2C_2)^{-1}C_0$. Therefore, by \eqref{lowb} and \eqref{err-est} we have $$|f(x_0)-f(y_0)|\ge |M_0(x_0-y_0)|-|f_0(x_0)-f_0(y_0)|\ge \frac{C_0}{2}|x_0-y_0|.$$
\end{proof}
By using the embedding $W^{1,(n,q)}(B^n)\hookrightarrow VMO(B^n)$ for $1\le q<\infty$ (see Definition \ref{dfi-So-Lor} and \cite[Theorem 6.1(iii)]{Cianchi98}), we obtain the following corollary.
\begin{Co}\label{Co:locinj}
    Let $m,n\in\N$ with $2\le n\le m$, and let $1\le q<\infty$. Suppose $f\in W^{1,\infty}_{\textup{imm}}(B^n,\R^m)$ and $\g^2 f\in L^{n,q}(B^n)$, then there exists a constant $C>0$ and an open neighborhood $U$ of the origin such that $$|f(x)-f(y)|\ge C|x-y|, \qquad \text{for all } \,x,y\in U.$$
\end{Co}
\begin{Rm}
     When $m>n$, Corollary~\ref{Co:locinj} is sharp in the sense that there exists $f\in W_{\imm}^{1,\infty}(B^n,\R^m)$ with $\g^2 f\in L^{n,\infty}(B^n)$ but $f$ is not injective in any neighborhood of the origin. We construct it as follows.
     Without loss of generality, we can assume $m=n+1$. We define $\ti f_0\in C^\nf( \mathbb S^{n-1},\R^n)$ by 
     \begin{align*}
         \ti f_0(x_1,\dots,x_n)\coloneqq \Big(\big(2+\cos(\pi x_n)\big)x_1,\dots, \big(2+\cos(\pi x_n)\big)x_{n-1},\mko\sin(\pi x_n)\Big).
     \end{align*}
     Then $\ti f_0$ is a non-injective smooth immersion from $\mathbb S^{n-1}$ to $\R^n$. Hence, there exists a non-injective smooth immersion $\ti f\colon \mathbb S^{n-1}\to \mathbb S^n$. 
     Now for $x\in B^n$, we define 
     \begin{align*}
         f(x)=\begin{dcases}
             |x|\mkt\ti f\Big(\frac{x}{|x|}\Big)&\text{if }x\neq 0,\\
             0&\text{otherwise}.
         \end{dcases}
     \end{align*}
    Then we have $f\in  W_{\imm}^{1,\infty}(B^n,\R^{n+1})$, and there exists a constant $C>0$ such that $|\g^2 f(x)|\le C|x|^{-1}$ for all $x\in B^n\smi\{0\}$. This implies that $\g^2 f\in L^{n,\infty}(B^n)$, but $f$ is not injective in any neighborhood of $0$. 
    
    In the special case $m=n$, a sharp theorem states that any $W^{2,2}$ weak immersion (also called ``BLD maps'' in this case) is a local homeomorphism everywhere, see \cite[Theorem~1.1]{Heinonen00}. However, when $n\ge 3$, it remains open whether one can replace $W^{2,2}$ by $W^{2,(2,q)}$ for any $2<q<\nf$.
\end{Rm}
In the proof of \cite[Corollary~3.43]{Ri16}, it is mentioned that in the case $n=2$, if we assume in addition that $\bP\in W^{2,2}(\Sigma)$, then for any $x\in\R^m$, $\theta_x$ exists and is in $\N_0$. Here we prove a stronger result.
\begin{Th}\label{thden}
    Let $\Si$ be an $n$-dimensional closed smooth manifold and $\bP\in W^{1,\infty}_{\textup{imm}}(\Si,\R^m)$ be such that $d\bP\in VMO(\Sigma)$. Then we have 
    \be\label{denequ}
  \theta_y=\text{Card}(\vec{\Phi}^{-1}(y)),\qquad  \text{for all }\,y\in \R^m. 
    \ee
\end{Th}
\begin{proof}
    By Theorem \ref{loin}, for any point $y\in {\R}^m$, $\bP^{-1}(\{y\})$ is at most finite. Thus it suffices to show that for a bi-Lipschitz map $f\in W^{1,\infty}_{\textup{imm}}(B^n,\R^m)$ with $\g f\in VMO(B^n)$, $f(0)=0$, there holds \be\label{loc-den}
        \lim_{r\rightarrow 0^+}\frac{\int_{f^{-1}(B_r(0))}1\,d\textup{vol}_{g_f}}{r^n|B^n|}=1,
    \ee
    where $g_f\coloneqq f^*g_{\text{std}}$.
      For $r\in (0,1)$, we define $f_r\colon B^n\rightarrow \R^m$ by 
\begin{align}\label{deff_r}
f_r(w)\coloneqq\frac{f(r\omega )}{r}.
\end{align}
    Then $\{f_r\}$ is uniformly bounded in $W^{1,\infty}(B^n)$.  Moreover, we have the following limit 
    \begin{align}\label{con-der}
        \fint_{B^n} \lf|\g f_r-\fint_{B^n}\g f_r\rg|=\fint_{B_r(0)} \lf|\g f-\fint_{B_r(0)}\g f\rg|\xrightarrow[r\to 0^+]{} 0.
    \end{align}
    Since $f$ is bi-Lipschitz on $B^n$, there exists a positive constant $C_0$ such that 
    \be\label{bi-lip}
    |x|\le C_0 |f(x)|, \qquad \text{ for all }x\in B^n.
    \ee
    Let $\{r_k\}_{k=1}^\infty\subset \R_+$ be an arbitrary sequence with $r_k\rightarrow 0^+$, then there exists a subsequence $\{r_{k_i}\}_{i=1}^\infty$ and an $m\times n$ matrix $M$ such that 
    \be\label{L1con-der}
    \fint_{B^n} \g f_{C_0r_{k_i}}\xrightarrow[i\to \infty]{} M.
    \ee 
    Now it suffices to prove $$
        \lim_{i\rightarrow \infty}\frac{\int_{f^{-1} \lf(B_{r_{k_i}}(0)\rg )}1\,d\textup{vol}_{g_f}}{r_{k_i}^n|B^n|}=1. 
    $$Denote $\tilde f_i\coloneqq f_{C_0r_{k_i}}$. By \eqref{con-der}, \eqref{L1con-der} we have \begin{align}\label{dfilim}
          \lim_{i\rightarrow \infty} \int_{B^n}\lf|\g \tilde f_i-M\rg|=0.\end{align}
   Let $\tilde f\colon B^n\rightarrow \R^m$ defined by $\tilde f(\omega)=M\omega$. Since $\{\tilde f_i\}$ is bounded in $W^{1,\infty}(B^n)$ and $\tilde f_i(0)=0$ for any $i$, by \eqref{dfilim} we obtain that
  $$ \tilde f_i\xrightarrow[i\rightarrow\infty]{} \tilde f \quad\text{in } W^{1,p}(B^n),\qquad \text{for all } 1\le p<\infty.$$ 
By \eqref{bi-lip} and definition of $f_r$, for any $\omega\in B^n$, $i\in \mathbb N$, we have 
        $$\lf|\omega\rg|\le C_0\big|\tilde f_i(\omega)\big|.$$
    Hence for any $\omega\in B^n$, we have 
    \be \label{est-M}
   |\omega|\le C_0 \big|\tilde f(\omega)\big|=C_0|M\,\omega|.
   \ee 
   In particular, this implies $M$ has full rank. 
   By \eqref{bi-lip} we have $f^{-1}\big(B_{r_{k_i}}(0)\big)\subset B_{C_0r_{k_i}}(0)$. We can then apply the change of variables formula to get 
   \be \label{cha-to-g}
   (C_0r_{k_i})^{-n}\int_{f^{-1}\lf(B_{r_{k_i}}(0)\rg )}1\,d\textup{vol}_{g_f}=\int_{\tilde f_i^{-1}\big(B_{C_0^{-1}}(0)\big)}\sqrt{\det \big(\g \tilde f_i^T \g \tilde f_i\big)}.
   \ee
   By the definition of weak immersions, there exists a constant $C>0$ independent of $i$ such that for all $i\in\N$,
\begin{align*}
    C^{-1}\le\det\big(\g \tilde f_i^T \g \tilde f_i\big)\le C \qquad \text{a.e. on }B^n.\end{align*}
Since $\g \tilde f_i\rightarrow M$ in $L^p(B^n)$ for any $p\in [1,\infty)$, we have 
\be \label{con-dvol}
    \lf\|\sqrt{\det \big(\g \tilde f_i^T\g \tilde f_i\big)}- \sqrt{\det\big(M^T M\big)}\rg\|_{L^1(B^n)}\xrightarrow[i\rightarrow \infty]{} 0.\ee
By Sobolev embedding we have $\tilde f_i\rightarrow \tilde f$ in $C^0(\overline{B^n})$
   \footnote{In general, $f_r$ defined in \eqref{deff_r} does not have a $C^0$ limit as $r\to 0$ under the same assumptions. For instance, we may take 
   $$\bar f(x_1,x_2)\coloneqq\begin{cases}
       0 & \text{if }(x_1,x_2)=0, \\[2mm]
       \Big(x_1,x_2,100^{-1}\,x_1\cos\big(\log( 1-\log (x_1^2+x_2^2))\big)\Big) & \text{otherwise}.
   \end{cases}$$ 
   Then $\bar f\in W^{2,2}_{\textup{imm}}(D^2,\R^3)$ hence $\g \bar f\in VMO(D^2)$, but for any $\omega\in  \{(x_1,x_2)\in D^2:x_1\neq 0\}$, $\bar f_r(\omega)\coloneqq r^{-1}\bar f(r\omega)$ diverges as $r\rightarrow 0$. Moreover, we compute the contingent cone \cite[Definition 4.1.1]{Aubinjp09} and Clarke tangent cone \cite[Definition 4.1.5]{Aubinjp09} of $\bar f(D^2)$ at $0$: $$T_{\bar f(D^2)}(0)=\big\{(x_1,x_2,x_3)\in \R^3: |x_3|\le |x_1|\big\},\qquad C_{\bar f(D^2)}(0)=\{0\}\times \R\times \{0\}.$$
   Neither of the two tangent cones above is a $2$-plane in $\R^3$. However, if we assume $f\colon D^2\rightarrow \R^3$ is a $C^1$ embedding, then it holds that $$T_{f(D^2)}(0)=C_{f(D^2)}(0)=
   \text{Im}(\g f(0)).$$ }.
Since $\mathcal L^n\lf(\tilde f^{-1}\big(\p B_{C_0^{-1}}(0)\big)\rg)=0$, by \cite[Theorem 3.2.11]{Durrett19} we obtain
   \be\label{con-vol} \mathcal L^n\lf(\tilde f_i^{-1}\big(B_{C_0^{-1}}(0)\big)\rg)\rightarrow \mathcal L^n\lf(\tilde f^{-1}\big(B_{C_0^{-1}}(0)\big)\rg).\ee
   Since $\tilde f(\omega)=M\omega$, we have \begin{align}\begin{aligned}\label{com-vol} 
        \mathcal L^n\lf(\tilde f^{-1}\big(B_{C_0^{-1}}(0)\big)\rg)&=C_0^{-n}\mathcal L^n \lf\{\omega\in \R^n: \langle M^TM\omega,\omega\rangle \le 1\rg\}.\end{aligned}\end{align}
Since $M^TM$ is positive definite, there exist $Q\in SO(n)$ and $D=\text{diag}(\la_1,\dots,\la_n)$ with $\la_1,\dots,\la_n>0$ such that $M^TM=QDQ^{-1}$. It follows that \begin{align}
    \begin{aligned}\label{comp-pre}\mathcal L^n \lf\{\omega\in \R^n: \langle M^TM\omega,\omega\rangle \le 1\rg\}&=\mathcal L^n \lf\{\omega\in \R^n: \langle DQ^{-1}\omega,Q^{-1}\omega\rangle \le 1\rg\}\\[2.5mm]
&=\mathcal L^n \lf\{v\in \R^n: \langle Dv,v\rangle \le 1\rg\}\\
&=\mathcal L^n\bigg\{(v_1,\dots,v_n)\in \R^n: \sum_{j=1}^n \la_j v_j^2\le 1\bigg\}.
\end{aligned}
\end{align}
By a linear change of variable, we have 
\begin{align}
    \begin{aligned}\label{linchavar}
        \mathcal L^n\bigg\{(v_1,\dots,v_n)\in \R^n: \sum_{j=1}^n \la_j v_j^2\le 1\bigg\}&=\prod_{j=1}^n \la_j^{-\frac 12} |B^n| =\det(M^TM)^{-\frac 12} |B^n|.
    \end{aligned}
\end{align}
   Combining \eqref{cha-to-g}--\eqref{linchavar}, we finally conclude $$\lim_{i\rightarrow \infty}\frac{\int_{f^{-1} \big(B_{r_{k_i}}(0)\big )}1\, d\textup{vol}_{g_f}}{r_{k_i}^n|B^n|}=\frac{C_0^n\mathcal L^n\lf(\tilde f^{-1}\big(B_{C_0^{-1}}(0)\big)\rg)\sqrt{\det(M^T M)}}{|B^n|}=1.$$
   Thus, the limit \eqref{loc-den} is proved.
\end{proof}

\subsection{Approximation by smooth immersions}\label{sec:Approximation}

In this section, we prove that every weak immersion $\bP$ can be approximated by smooth immersions $\bP_k$ and that the conformal structures induced by $\bP_k$ converge to that of $\bP$. To do so, we first prove in \Cref{sec:ConvergenceCoordinates} that the conformal charts obtained by Coulomb frames converge. Combining this with an argument from Uhlenbeck, we prove in \Cref{sec:ConvergenceStructure} that the conformal structures converge as well. We summarize the approximation results that we obtain in \Cref{sec:ApproximationResults}.

\subsubsection{Convergence of isothermal coordinates}\label{sec:ConvergenceCoordinates}

The goal of this section is to prove Theorem \ref{thmconchart}, namely that if a sequence of weak immersions $(\bP_k)_{k\in\N}$ converges to some $\bP_{\infty}$ in the strong $W^{2,2}_{\textup{loc}}$-topology, assuming a control on the metrics $g_{\bP_k}$, then the conformal charts induced by $\bP_k$ converge to conformal charts for $\bP$.\\

In order to prove the convergence of the Coulomb frames, we first need to prove some convergence result for equations having varying coefficients.
\begin{Lm}\label{strconsol}
 Let $n\ge 2$, $U$ be a bounded open subset of $\R^n$, $\{a^{ij}_k\}\subset L^\infty(U)$, $\{f_k\}\subset W^{-1,2}(U)$. Suppose that there exists a constant $\La>0$ such that for all $\xi\in \R^n$, $k\in\mathbb N$, and a.e. $x\in U$, it holds
    \begin{align*}
         \La^{-1} |\xi|^2\le \sum_{i,j}^n a^{ij}_k(x) \,\xi_i\xi_j \le \La |\xi|^2. \end{align*}
Assume in addition that $a^{ij}_k\rightarrow a^{ij}$ a.e. for some $a^{ij}\in L^\infty(U)$ and $f_{k}\rightarrow f$ in $W^{-1,2}(U)$ for some $f\in W^{-1,2}(U)$. Let $u_k\in W^{1,2}(U)$ be a solution to the equation  \begin{align}\label{ellequ}
    \p_i\big(a^{ij}_k\,\p_j u_k\big)=f_k  \quad \text{in }U.   
    \end{align}
Suppose that $u_k\rightharpoonup u$ weakly in $W^{1,2}(U)$. Then we have $u_k\rightarrow u$ in $W_{\textup{loc}}^{1,2}(U)$.
\end{Lm}
\begin{proof}
    Let $ V\Subset U$ be a bounded open subset with smooth boundary. Let $\xi\in C^\infty_c(U)$ be a cut-off function such that $\xi\ge 0$ on $U$ and $\xi=1$ on $V$. Denote by $\lf\langle ,\rg\rangle $ the duality pairing between $W^{-1,2}(U)$ and $W^{1,2}_0(U)$. Using $ (u_k-u)\,\xi$ as a test function in the equation \eqref{ellequ}, we obtain
    \begin{align}\begin{aligned}\label{est-diff-W12}
    & \Lambda^{-1}\int_V |\g(u_k-u)|^2 \le \int_U a^{ij}_k\, \xi \,\p_j(u_k-u)\, \p_i(u_k-u)\\
    & \leq- \lf\langle f_k, \xi (u_k-u) \rg\rangle -\int_U \big(a^{ij}_k \, \p_ju_k\,\p_i\xi\, (u_k-u) + a^{ij}_k \,\xi\,\p_j u\, \p_i(u_k-u)\big).
\end{aligned}
\end{align}
Since $\xi(u_k-u)\rightharpoonup 0$ weakly in $W_0^{1,2}(U)$ and $\|f_k-f\|_{W^{-1,2}(U)}\rightarrow 0$ as $k\rightarrow \infty$, we have
\begin{align}\begin{aligned}
    \label{1st-est-Duk-diff} \lf|\lf\langle f_k, \xi (u_k-u) \rg\rangle\rg|&\le \lf|\lf\langle f_k-f,\xi (u_k-u)\rg\rangle\rg|+\lf|\lf \langle f,\xi (u_k-u)\rg\rangle\rg|\\[3mm]
    &\le \lf\|f_k-f\rg\|_{W^{-1,2}(U)}\lf\|\xi(u_k-u)\rg\|_{W^{1,2}(U)}+\lf|\lf \langle f,\xi (u_k-u)\rg\rangle\rg|\\[2mm]
    &\xrightarrow[k\to \infty]{} 0.\end{aligned}\end{align}
By Rellich–Kondrachov compactness theorem, we also get \begin{align}\begin{aligned} \label{2nd-est-Duk-diff} 
\lf|\int_U a^{ij}_k \, \p_ju_k\,\p_i\xi\, (u_k-u)\rg|&\le \lf\|a^{ij}_k\rg\|_{L^\infty(U)}\lf\|\p_j u_k\rg\|_{L^2(U)}\lf\|\p_i\xi\,(u_k-u)\rg\|_{L^2(U)}\xrightarrow[k\to \infty]{} 0.  
\end{aligned}
\end{align} 
Concerning the last term, by weak convergence $u_k\rightharpoonup u$ in $W^{1,2}(U)$ and dominated convergence theorem, we obtain \begin{align}
    \begin{aligned}
        \label{3rd-est-Duk-diff} 
        \lf|\int_U a^{ij}_k \,\xi\,\p_j u\, \p_i(u_k-u)\rg|
        &\le  \lf|\int_U a^{ij} \,\xi\,\p_j u\, \p_i(u_k-u)\rg|+\int_U \lf|(a^{ij}_k-a^{ij}) \,\xi\,\p_j u\, \p_i(u_k-u)\rg|\\
        &\le \lf|\int_U a^{ij} \,\xi\,\p_j u\, \p_i(u_k-u)\rg| + \lf\|(a^{ij}_k-a^{ij}) \,\xi\,\p_j u\rg\|_{L^2(U)}\lf\|\p_i(u_k-u)\rg\|_{L^2(U)}\\
        &\xrightarrow[k\to \infty]{} 0.
    \end{aligned}
\end{align} 
Combining \eqref{est-diff-W12}--\eqref{3rd-est-Duk-diff} we finally conclude $u_k\rightarrow u$ in $W^{1,2}_{\textup{loc}}(U)$.
 \end{proof}

 Using the proof of Theorem \ref{thexicon}, we are now ready to prove strong convergence of conformal charts for a convergent sequence of weak immersions with bounded induced metric.
 
\begin{Th}\label{thmconchart}
    Let $\bP_k\in W_{\textup{imm}}^{2,2}(D^2,\R^m)$ for $k\in \mathbb N \cup \{\infty\}$. Assume that $\bP_k \rightarrow \bP_{\infty}$ in $W^{2,2}(D^2)$ as $k\rightarrow \infty$, and that there exists a constant $\La>0$ such that for any $k\in \N \cup \{\infty\}$, any $X\in \R^2$ and a.e. $p\in D^2$, there holds 
    \begin{align}\label{immconD-seq}
\Lambda^{-1} |X|_{\R^2} \le | d(\vec{\Phi}_k)_p(X)|_{\R^m}  \le \Lambda |X|_{\R^2}.
\end{align} 
    Then there exists a neighborhood $U\subset D^2$ of $0$ and $W^{1,\infty}\cap W^{2,2}$ diffeomorphisms $\vp_k\colon U\rightarrow \vp_k(U)\subset \R^2 $ such that $\bP_k \circ \vp_k^{-1}$ is (weakly) conformal on $\vp_k(U)$ for any $k\in \N \cup \{\infty\}$, $\vp_k\rightarrow \vp_{\infty}$ in $W^{2,2}(U)$ and $\|\vp_k\|_{W^{1,\infty}(U)}\le C$, $\|\vp_k^{-1}\|_{W^{1,\infty}(\vp_k(U))}\le C$ for a constant $C>0$.
\end{Th}
\begin{proof}
    We follow the proof of Theorem \ref{thexicon}. For $k\in \N\cup \{\infty\}$, denote 
    \begin{gather*}
g_k= g_{ij,k}\ dx^i\otimes dx^j,\qquad g_{ij,k}=\langle \p_i\bP_k,\p_j \bP_k\rangle, \\[2mm]
 (g^{ij}_k)=(g_{ij,k})^{-1},\qquad  \det g_k=\det(g_{ij,k}).
\end{gather*}
Then we have $g_k\rightarrow g_\infty$ in $W^{1,2}(D^2)$ as $k\rightarrow\infty$, and $\{g_k\}$, $\{g^{ij}_k\}$ are bounded in $L^\infty(D^2)$. \\ For $k\in \N\cup \{\infty\}$, we apply Gram--Schmidt process to $(\p_1\bP_k,\p_2\bP_k)$ and get \begin{align}\label{gramsch-seq}
    \vec f_{1,k}\coloneqq\frac{\p_1\bP_k}{|\p_1\bP_k|},\qquad \vec f_{2,k}\coloneqq\frac{\p_2\bP_k-(\p_2\bP_k \cdot \vec f_{1,k})\vec f_{1,k}}{|\p_2\bP_k-(\p_2\bP_k \cdot \vec f_{1,k})\vec f_{1,k}|}.
\end{align}
By $\eqref{immconD-seq}$ we have $|\p_1 \bP_k|,|\p_2\bP_k|\in [\La^{-1},\La]$, and 
$$
\left|\p_2\bP_k-(\p_2\bP_k \cdot \vec f_{1,k})\vec f_{1,k}\right| 
\in [\La^{-1}, \La(1+\La^2)].
$$
Hence it holds that $(\vec f_{1,k},\vec f_{2,k})\in L^\infty \cap W^{1,2}(D^2,V_2({\R}^m))$, where $V_2({\R}^m)$ denotes the space of orthonormal 2-frames in ${\R}^m$. Since $(\bP_k)_{k\in\N}$ converges in $W^{2,2}$, we deduce that $\vec f_{i,k}\rightarrow \vec f_{i,\infty}$ in $W^{1,2}(D^2)$ as $k\rightarrow \infty$ for $i\in \{1,2\}$.
For $\theta\in W^{1,2}(D^2,{\R})$, we denote $(\vec{f}^{\,\theta}_{1,k},\vec{f}^{\,\theta}_{2,k})$ the rotation by the angle $\theta$ of the original frame  $(\vec f_{1,k},\vec f_{2,k})$:
\begin{align*}
\vec{f}^{\,\theta}_{1,k}+i\vec{f}^{\,\theta}_{2,k}=e^{i\theta} \left(\vec f_{1,k}+i\vec f_{2,k} \right).
\end{align*}
Then we still have $\left(\vec{f}^{\,\theta}_{1,k},\vec{f}^{\,\theta}_{2,k} \right)\in L^\infty\cap W^{1,2}(D^2,V_2(\R^m))$.\\
  We look now for a rotation of this orthonormal frame realizing the following absolute minimum
\be
\label{I.8-seq}
\begin{array}{l}
\ds\inf_{\theta\in W_0^{1,2}(D^2,{\R})}\int_{D^2} |\lan \vec{f}^{\,\theta}_{1,k},d\vec{f}^{\,\theta}_{2,k}\ran|^2_{g_k}\ d\textup{vol}_{g_k}
\\[5mm]
\ds=\inf_{\theta\in W_0^{1,2}(D^2,{\R})}\int_{D^2} \sum_{i,j=1}^2g^{ij}_k\, \lan\vec{f}^{\,\theta}_{1,k},\p_{i}\vec{f}^{\,\theta}_{2,k}\ran\,\lan\vec{f}^{\,\theta}_{1,k},\p_{j}\vec{f}^{\,\theta}_{2,k}\ran\, \sqrt{\det g_k}\  dx^1\we dx^2.
\end{array}
\ee
  For $i\in\{1,2\}$, we have 
\be \label{rot-fdf}
\lan\vec{f}^{\,\theta}_{1,k},\p_{i}\vec{f}^{\,\theta}_{2,k}\ran=\p_{i}\theta+\lan\vec f_{1,k},\p_{i}\vec f_{2,k}\ran.
\ee
Hence the following energy is strictly convex in $W_0^{1,2}(D^2,{\R})$:
\begin{align*}
E(\theta)\coloneqq\int_{D^2} |\lan \vec{f}^{\,\theta}_{1,k},d\vec{f}^{\,\theta}_{2,k}\ran|^2_{g_k}\, d\textup{vol}_{g_k}=\int_{D^2} |d\theta+\lan \vec f_{1,k},d\vec f_{2,k}\ran|^2_{g_k}\,d\textup{vol}_{g_k}.
\end{align*}
By a standard application of Mazur's lemma, we have a unique $\theta_k\in W_0^{1,2}(D^2,\R)$ that achieves the minimum of $E$. It is the unique solution in $W_0^{1,2}(D^2,\R)$ for the Euler--Lagrange equation:
\be\label{eularequ-seq}
\forall \phi\in W_0^{1,2}(D^2,{\R}),\qquad \int_{D^2}d\phi\wedge *_{g_k}\lf( d\theta_k+\lan\vec f_{1,k},d\vec f_{2,k}\ran  \rg) =0.
\ee
In particular, taking $\phi=\theta_k$ in \eqref{eularequ-seq}, we get 
$$
\La^{-4}\int_{D^2} |\g \theta_k |^2 \,dx^2\le \int_{D^2} |d\theta_k|^2_{g_k}\, d\textup{vol}_{g_k}\le \int_{D^2} |d\vec f_{2,k}|_{g_k}^2 \,d\textup{vol}_{g_k}\le \La^4 \int_{D^2} |\g\vec f_{2,k}|^2.
$$ 
Thus the sequence $\{\theta_k\}$ is bounded in $W_0^{1,2}(D^2)$ and each subsequence contains a further subsequence $\{\theta_{k_l}\}$ weakly converging to some $\theta'\in  W^{1,2}_0(D^2)$. Then using \eqref{eularequ-seq}, H\"older's inequality and dominated convergence theorem, we have 
$$
\forall \phi\in W_0^{1,2}(D^2,{\R}),\qquad \int_{D^2}d\phi\wedge *_{g_\infty}\lf( d\theta'+\lan\vec f_{1,\infty},d\vec f_{2,\infty}\ran  \rg) =0.
$$ 
Hence it holds $\theta_{\infty}=\theta'$ and we have $\theta_k\rightharpoonup \theta_{\infty}$ weakly in $W^{1,2}(D^2)$. Also by dominated convergence, we have 
$$
d\lf(*_{g_k}\lan\vec f_{1,k},d\vec f_{2,k}\ran \rg)\rightarrow d\lf(*_{g_\infty}\lan\vec f_{1,\infty},d\vec f_{2,\infty}\ran \rg) \qquad \text{in $W^{-1,2}(D^2)$. }
$$
Then by replacing the cutoff function $\vp$ by $1$ on $D^2$, we use the same argument as in Lemma \ref{strconsol} to obtain $\theta_k\rightarrow \theta_{\infty}$ in $W^{1,2}(D^2)$. Denoting $(\vec{e}_{1,k},\vec{e}_{2,k})\coloneqq(\vec{f}^{\,\theta_k}_{1,k},\vec{f}^{\,\theta_k}_{2,k})$, we have $\vec e_{i,k}\rightarrow \vec e_{i,\infty}$ as $k\rightarrow \infty$ for $i\in \{1,2\}$. We deduce from \eqref{eularequ-seq} that the following equation holds in the sense of distribution:
\[
d\Big(  *_{g_k} \lan\vec{e}_{1,k},d\vec{e}_{2,k}\ran  \Big)=0.
\] 
Therefore by the weak Poincar\'e lemma, there exists $\la_k\in W^{1,2}(D^2)$ with $\int_{D^2 }\la_k=0$ such that \[d\la_k =  *_{g_k} \lan\vec{e}_{1,k},d\vec{e}_{2,k}\ran.\]
Hence we have the convergence $\g\la_k\rightarrow \g\la_{\infty}$ in $L^2(D^2)$, from which we deduce that $\la_k\rightarrow \la_{\infty}$ in $W^{1,2}(D^2)$ by Poincar\'e inequality.
In order to deduce boundedness in $C^0$, we use the following equation:
\[\Delta_{g_k} \la_k=*_{g_k} d*_{g_k} d\la_k=-*_{g_k} d\lan\vec{e}_{1,k},d\vec{e}_{2,k}\ran.
\]
Let $\la_k'\in \bigcap_{1<p<2} W^{1,p}_0(D^2)$ be the solution to the following equation:
\[
\begin{cases}
\ds\,\p_{i}\lf(\sqrt{\det g_k}\ g_k^{ij}\,\p_{j}\la'_k\rg)=\p_2 \vec e_{1,k} \cdot \p_1 \vec e_{2,k} -\p_1 \vec e_{1,k} \cdot \p_2 \vec e_{2,k}\quad &\mbox{ in }D^2,\\[3mm]
\ds\la_k'=0\quad &\mbox{ on }\p D^2.
\end{cases}
\]
By Proposition \ref{pr:Wente}, 
we obtain that $\{\la_k'\}$ is bounded in $C^0(\overline{D^2})\cap W^{1,2}_0(D^2,{\R})$. Concerning the difference $\lambda_k'-\lambda_k$, we have the following equation: 
$$
\p_{i}\lf(\sqrt{\det g_k}\ g_k^{ij}\, \p_{j}(\la'_k-\la_k)\rg)=0.
$$
Thanks to \cite {Me63}, there exists $p=p(\La)\in(2,\infty)$ and a positive constant $C$ independent of $k$, such that 
$$ 
\|\la'_k-\la_k\|_{C^0(\overline{D_{1/2}(0)})}\le C\|\la'_k-\la_k\|_{W^{1,p}(D_{1/2}(0))} \le C(\La)\,\|\la'_k-\la_k\|_{W^{1,2}(D^2)}\le C.
$$
It follows that $\{\la_k\}_{k\in\N}$ is bounded in $C^0(\overline{D_{1/2}(0)})$. For $i\in \{1,2\}$, we define pullback of $\vec e_{i,k}$ and the dual $1$-forms 
\begin{align}\label{dualbas-seq}\begin{dcases}
    e_{i,k}\coloneqq  g^{jl}_k\,\lan \vec e_{i,k}, \p_j\bP_k \ran \,\frac{\p}{\p x^l}\in L^\infty\cap W^{1,2}(D_{1/2}(0),TD_{1/2}(0)),\\[2mm]
    e_{i,k}^\ast\coloneqq\lan \vec{e}_{i,k},\p_{j}\vec{\Phi_k}\ran \,dx^j=\lan \vec e_{i,k}, d\bP_k\ran\in L^\infty \cap W^{1,2}(D_{1/2}(0),T^*D_{1/2}(0)).
\end{dcases}
\end{align}
Then by the same proof of Theorem \ref{thexicon}, we have
\[
d(e^{-\la_k}\, e_{i,k}^*)=0.
\]
Using the weak Poincar\'e lemma, we obtain the existence of $\vp^i_{k}\in W^{1,\infty}\cap W^{2,2}(D_{1/2}(0),{\R})$ with $\int_{D_{1/2}(0)} \vp^i_{k}=0$ for $i\in \{1,2\}$ such that
\be
\label{dp=-seq}
d\varphi^i_{k}=e^{-\la_k}\, e_{i,k}^*.
\ee
Let $\vp_k\coloneqq(\vp^1_{k},\vp^2_{k})\in  W^{1,\infty}\cap W^{2,2}(D_{1/2}(0),{\R^2})$. From \eqref{dualbas-seq} and \eqref{dp=-seq} we deduce that for a.e. $p\in D_{1/2}(0)$ and any $v\in T_p(D_{1/2}(0))$,
\begin{align}\label{vpimm}
    e^{-\|\la_k\|_{C(\overline{D_{1/2}(0)})}}\,\La^{-1}|v|\le  | d(\vp_{k})_p(v)|=  e^{-\la_k}|d\bP_k(v)|\le e^{\|\la_k\|_{C(\overline{D_{1/2}(0)})}}\,\La|v|.
    \end{align}
Since $\la_k \rightarrow \la_{\infty}$, $e_{i,k}^*\rightarrow e^*_{i,\infty }$ in $W^{1,2}(D^2)$ as $k\rightarrow\infty$, and $\{\la_k\}$ is bounded in $C(\overline{D_{1/2}(0)})$, by dominated convergence theorem and Poincar\'e inequality we have $\vp_{k}\rightarrow \vp_{\infty}$ in $W^{2,2}(D_{1/2}(0))$. 
Combining \eqref{vpimm} and Lemma \ref{loin}, it follows that there exists an open neighborhood $U\subset D_{1/2}(0)$ of $0$ independent of $k$ such that $\vp_k$ is open and injective on $U$, and its inverse $\vp_k^{-1}\in W^{1,\infty}\cap W^{2,2}(\vp_k(U),\R^2)$. We obtain that $\vp_k$ is an isothermal chart of $\bP_k$ by the same proof of Theorem \ref{thexicon}.
\end{proof}

\subsubsection{Convergence of conformal structures}\label{sec:ConvergenceStructure}

As a consequence of Theorem \ref{thmconchart}, we deduce the convergence of the underlying conformal structure for sequences of metrics induced by weak immersions which converge in some Sobolev space. To begin with, we consider the case where $g_k\rightarrow g_{\infty}$ smoothly. \begin{Lm}\label{lem-stru-con-smooth}
    Let $\Sigma$ be an oriented connected $2$-dimensional closed smooth manifold. Assume that $(g_k)_{k\in \N\cup \{\infty\}}$ is a sequence of Riemannian metrics on $\Sigma$, and $$g_k\xrightarrow[k\rightarrow\infty]{} g_{\infty} \quad \text{in }\, C^\infty(\Sigma, T^*\Sigma\ot T^*\Sigma).$$ Then there exists a sequence $(h_k)_{k\in \N\cup \{\infty\}}$ of metrics of constant Gaussian curvature $1,-1$ or $0$ such that $h_k$ is conformal to $g_k$ and $h_k\rightarrow h_\infty$ in $C^\infty(\Sigma, T^*\Sigma\ot T^*\Sigma)$ as $k\rightarrow \infty$.
\end{Lm}
\begin{proof}
    Let $\chi(\Sigma)$ denote the Euler characteristic of $\Sigma$ and $K_{h_k}$ be the Gaussian curvature of $h_k$.\\
    
    {\bf (Case I)} $\chi(\Sigma)\leq -2$ and $K_{h_k}=-1$.

    By a theorem of Poincar\'e and its corollary (see \cite[Theorem 1.6.2]{Tromba92}), in this case, for any $s\in \N$, $s\ge 3$ and any metric $g\in W^{s,2}(\Sigma,T^*\Sigma\ot T^*\Sigma)$, there exists a unique metric $h\in  W^{s,2}(\Sigma,T^*\Sigma\ot T^*\Sigma)$ such that $h$ is weakly conformal to $g$ and $h$ has constant Gaussian curvature $-1$. Moreover, the map taking $g$ to $h$ is continuous in $W^{s,2}$. It follows that $h_k$ is uniquely determined by $g_k$, and converges to $h_\infty$ smoothly since $g_k\rightarrow g_\infty$ in $C^\infty(\Sigma, T^*\Sigma\ot T^*\Sigma)$.\\
    
    {\bf(Case II)} $\chi(\Sigma)=0$ and $K_{h_k}=0$.
    
    Let $h_k=e^{2\la_k}g_k$ for some $\lambda_k\in C^{\infty}(\Sigma)$. We denote the Gaussian curvature and negative Laplace--Beltrami operator associated to $g_k$ by $K_{g_k}$ and $\Delta_{g_k}$ respectively, the condition that $h_k$ has constant Gaussian curvature $0$ is then equivalent to the equation (see e.g. \cite[Theorem 2.6]{Ri16}) \be \label{deltagk=K}\Delta_{g_k}\la_k=K_{g_k}.\ee
    By Gauss--Bonnet theorem, we have 
    $$
    \int_{\Sigma}K_{g_k}\,d\textup{vol}_{g_k}=2\pi \chi(\Sigma)=0.
    $$ 
    Hence for each $k\in\N\cup\{\infty\}$, there exists a smooth solution $\la_k$ to \eqref{deltagk=K} and it is unique up to additive constants (see \cite[Theorem 4.7]{aubin1998}). By adding suitable constants, we can without loss of generality assume that 
    \be\label{normalizelambda}
    \forall k\in \N\cup\{\infty\},\qquad \int_{\Sigma} \la_k \,d\textup{vol}_{g_{\infty}}=0.
    \ee
   Since $g_k\rightarrow g_{\infty}$ in $C^\infty$, we have $K_{g_k}\rightarrow K_{g_\infty}$ in $C^\infty(\Sigma)$ as $k\rightarrow\infty$, and in particular, $\{K_{g_k}\}_{k\in\N\cup\{\infty\}}$ is bounded in $C^0(\Sigma)$. By \eqref{deltagk=K} and the convergence $g_k\rightarrow g_\infty$, there exists a constant $C_1>0$ independent of $k$ such that 
   \begin{align*}
        C_1\int_{\Sigma} |d\la_k|_{g_\infty}^2\,d\textup{vol}_{g_\infty}&\le\int_{\Sigma} |d\la_k|_{g_k}^2\,d\textup{vol}_{g_k}=-\int_{\Sigma} \lambda_k\, \Delta_{g_k}\la_k\,d\textup{vol}_{g_k} =\int_{\Sigma }-\la_k K_{g_k}\,d\textup{vol}_{g_k}.
        \end{align*}
        Furthermore, by \eqref{normalizelambda}, Cauchy--Schwarz inequality, and Poincar\'e inequality, there exist constants $C_2,C_3>0$ independent of $k$ such that
       \begin{align*}
            \int_{\Sigma }-\la_k K_{g_k}\,d\textup{vol}_{g_k}&\le \|K_{g_k}\|_{L^{\infty}(\Sigma)}\, \textup{vol}_{g_k}(\Sigma)^{\frac{1}{2}} \lf(\int_{\Sigma} \la_k^2 \,d\textup{vol}_{g_k}\rg)^{\frac 12} \\
        &\le C_2\, \left( \sup_{l\in\N} \|K_{g_l}\|_{L^{\infty}(\Sigma)}\right)\, \left( \sup_{l\in\N} \textup{vol}_{g_l}(\Sigma)^{\frac{1}{2}}\right)\lf(\int_{\Sigma} \la_k^2 \,d\textup{vol}_{g_\infty}\rg)^{\frac 12} \\
        &\le C_3\, \lf(\int_{\Sigma} |d\la_k|_{g_{\infty}}^2 \,d\textup{vol}_{g_\infty}\rg)^{\frac 12}.
        \end{align*}
        Consequently, $\{\la_k\}_{k\in\N\cup\{\infty\}}$ is bounded in $W^{1,2}(\Sigma)$. Now we apply the interior estimates as in \cite[Section 6.3.1]{evans} to the equation \eqref{deltagk=K} and obtain that $\{\la_k\}_{k\in\N\cup\{\infty\}}$ is bounded in $W^{s,2}(\Sigma)$ for any $s\in \N$. A standard compactness argument then implies $\la_k\rightarrow \la_{\infty}$ in $C^\infty(\Sigma)$ as $k\rightarrow\infty$. Finally, by setting $h_k=e^{2\la_k}g_k$, we obtain the required convergent sequence of metrics of Gaussian curvature $0$.\\
        
    {\bf(Case III)} $\chi(\Sigma)= 2$ and  $K_{h_k}=1$.

    We identify $\Sigma=\C\cup\{\infty\}$ with the canonical smooth structure and orientation. Define $\psi(z)\coloneqq1/z$ on $\Sigma$, and $\phi\colon \C\rightarrow S^2\setminus\{(0,0,1)\}$ be the inverse of the stereographic projection, i.e. $$\phi(x,y)\coloneqq\lf(\frac{2x}{1+x^2+y^2},\frac{2y}{1+x^2+y^2},\frac{-1+x^2+y^2}{1+x^2+y^2}\rg).$$
    Let $h\coloneqq\phi^*g_{S^2}$, where $g_{S^2}$ is the standard round metric on $S^2$. Then $h$ can be extended to a Riemannian metric on $\Sigma$ with constant Gaussian curvature $1$, and is conformal to $g_{\text{std}}=dx^2+dy^2$ on $\C$. Given $k\in\N\cup\{\infty\}$, there exist two functions $\sigma_k\colon \C\rightarrow (0,\infty)$ and $\mu_k\colon \C\rightarrow D_1(0)\subset \C$ such that
    $$
    g_k=(g_k)_{11}\,dx^2+2(g_k)_{12}\,dxdy+(g_k)_{22}\, dy^2=\sigma_k|dz+\mu_k\, d\bar z|^2,
    $$
    where $dz=dx+idy$ and $d\bar z=dx-idy$. Moreover, such $\mu_k$ is uniquely determined by $g_k$ and can be computed as (see e.g. \cite[Section 1.2.1]{Nag88})
    $$
    \mu_k=\frac{(g_k)_{11}-(g_k)_{22}+2i\,(g_k)_{12}}{(g_k)_{11}+(g_k)_{22}+2\sqrt{(g_k)_{11}(g_k)_{22}-(g_k)_{12}^2}}.
    $$
    In particular, we have $\mu_k\rightarrow \mu_{\infty}$ in $C^\infty_{\textup{loc}}(\C)$. On the other hand, we obtain that for all $\omega\in\C$,
    $$ 
    (\psi^{-1})^*(g_k)_\omega=\frac{\sigma_k\circ \psi^{-1}(\omega)}{|\omega|^4} \big|d\omega+ \mu_k\circ\psi^{-1}(\omega)\, \omega^2{\bar\omega}^{-2}d\bar{\omega}\big|^2.
    $$
Since $(\psi^{-1})^*g_k$ is smooth on $\C$, it follows that $ \mu_k\circ \psi^{-1}(\omega)\, \omega^2{\bar \omega}^{-2}$ can be extended to a smooth function on $\C$ and converges to $ \mu_\infty\circ \psi^{-1}(\omega)\, \omega^2{\bar \omega}^{-2}$ in $C^\infty_{\textup{loc}}(\mathbb C)$. In particular, it holds that $\|\mu_k\|_{L^\infty(\C)}<1$ and $\|\mu_k-\mu_\infty\|_{L^\infty(\C)}\rightarrow 0$ as $k\rightarrow \infty$. 
    
    For $k\in \N\cup\{\infty\}$, let $f_k$ be the unique orientation-preserving homeomorphism of $\Sigma=\C\cup\{\infty\}$ onto itself satisfying the relations 
    \begin{align*}
    \begin{cases} 
        \p_{\bar z} f_k=\mu_k\, \p_{z} f_k, \\[2mm]
        f_k(0)=0,\ f_k(1)=1,\ f_k(\infty)=\infty.
    \end{cases}
    \end{align*}
    In particular, $f_k$ is quasiconformal with complex dilatation $\mu_k$ (see Definition \ref{dfiquasicon}). 
    By considering the quasiconformal maps $f_k\circ f_{\infty}^{-1}$ and applying \cite[Proposition 4.36]{Imayoshi92}, we have $f_k\rightarrow f_\infty$ locally uniformly on $\C$.  
    We denote $\sigma_k\in L^\infty(\C,\R^{2\times 2})$ the unique matrix-valued function on $\C$ such that $\sigma_k$ is positive definite a.e., $\text{det}(\sigma_k)=1$, and such that
    $$
\mu_k=\frac{(\sigma_k)_{22}-(\sigma_k)_{11}-2i\,(\sigma_k)_{12}}{(\sigma_k)_{11}+(\sigma_k)_{22}+2}.
$$ 
Such $\sigma_k$ depends uniquely and smoothly on $\mu_k$, hence we have $\sigma_k\rightarrow \sigma_{\infty}$ in $C^\infty_{\textup{loc}}(\C,\R^{2\times 2})$. By \cite[Theorem 5.1]{Reshetnyak89}, $f_k$ satisfies the elliptic equation 
    $$
    \text{div}(\sigma_k(z) \nabla f_k(z))=0 \quad\text{on }\C.
    $$
    Hence standard elliptic estimates as in Case II imply that $f_k\rightarrow f_{\infty}$ in $C^\infty_{\textup{loc}}(\C)$. 
    On the other hand, the map $\psi\circ f_k\circ \psi^{-1}$ is quasiconformal with complex dilatation $\mu_k\circ \psi^{-1}(\omega)\, \omega^2{\bar \omega}^{-2}$ (see \cite[Proposition 4.13]{Imayoshi92}). Moreover, it holds
    $$
    \mu_k\circ \psi^{-1}(\omega)\, \omega^2{\bar \omega}^{-2}\xrightarrow[k\rightarrow\infty]{} \mu_\infty\circ \psi^{-1}(\omega)\, \omega^2{\bar \omega}^{-2} \qquad \text{ in }C_{\textup{loc}}^\infty(\C).
    $$
    By the same argument as before, we obtain 
    $$
    \psi\circ f_k\circ \psi^{-1}\xrightarrow[k\rightarrow\infty]{}\psi\circ f_\infty\circ \psi^{-1} \qquad \text{ in }C_{\textup{loc}}^\infty(\C).
    $$
    Hence it holds $f_k\rightarrow f_\infty$ in $C^\infty(\Sigma)$ as $k\rightarrow \infty$. Besides, by the proof of \cite[Lemma 5.B.3]{Ahlfors06}, both $f_k$ and $\psi\circ f_k\circ \psi^{-1}$ have no critical points on $\C$, hence each $f_k$ is a diffeomorphism on $\Sigma$. 
    Finally, for all $k\in \N\cup\{\nf\}$, let $h_k\coloneqq f_k^* h$, then $h_k$ is a Riemannian metric with constant Gaussian curvature $1$ on $\Sigma$, and $h_k\rightarrow h_\infty$ smoothly on $\Sigma$. Since $h$ is conformal to $g_{\text{std}}$ and $f_k^* g_{\text{std}}$ is conformal to $g_k$ on $\C$ (see for instance \cite[Section 1.5.1]{Imayoshi92}), we have $h_k=f_k^* h$ is conformal to $g_k$ on $\C\subset \Sigma$. The conformality holds on $\Sigma$ since $h_k$, $g_k$ are both Riemannian metrics on $\Sigma$, and $\C$ is dense in $\Sigma$.
\end{proof}
Then we consider the general case in which $\{g_k\}$ are metrics induced by weak immersions. The strategy of proof is an adaptation of \cite[Proposition 3.2]{Uh82}. For a sequence of functions $f_k$ and a compact set $K\subset \R^n$, we write $f_k\rightarrow f$ in $C^\infty(K)$ if each $f_k$ is defined on a neighborhood of $K$ (possibly depending on $k$) and for any $l\in \N_0$, $\|f_k-f\|_{C^l(K)}\rightarrow 0$ as $k\rightarrow \infty$.

\begin{thm}\label{thm-con-stru}
    Let $\Sigma$ be an oriented connected $2$-dimensional closed smooth manifold with a reference Riemannian metric $\ti g$. Assume that $\bP_k \rightarrow \bP_{\infty}$ in $W^{2,2}(\Sigma,\R^m)$ as $k\rightarrow \infty$, and that there exists a constant $\La>0$ such that for any $k\in \N \cup \{\infty\}$, there holds\begin{align}
\Lambda^{-1} |X|_{\ti g} \le | d(\vec{\Phi}_k)_p(X)|_{\R^m}  \le \Lambda |X|_{\ti g},\qquad \text{for a.e. }p\in \Sigma \text{ and all } \mkt X\in T_p\Sigma.
\end{align}
Let $g_k\coloneqq\bP_k^*\,g_{\text{std}}$ be the metric induced by $\bP_k$. Assume that $\ti g$ and $g_{\infty}$ are weakly conformal (i.e. $\Sigma$ is equipped with the complex structure induced by $\bP_\infty$ as in Corollary \ref{coconstr}), 
then there exists a sequence $(h_k)_{k\in \N\cup \{\infty\}}$ of metrics smooth in the complex structure induced by $\bP_k$, and a sequence $(\Psi_k)_{k\in \N}$ of bi-Lipschitz homeomorphisms on $\Sigma$ such that\begin{align*}
 \begin{dcases} 
    (i)\ \text{$h_k$ has constant Gaussian curvature $1,-1$ or $0$,}\\[3mm]
    (ii)\ \text{$\Psi_k^*h_k\xrightarrow[k\rightarrow\infty]{} h_{\infty}$ in $C^\infty(\Sigma,T^*\Sigma\ot T^*\Sigma)$,}\\[2mm]
    (iii)\ \text{$\Psi_k\xrightarrow[k\rightarrow\infty]{} id$ in $W^{2,2}(\Sigma,\Sigma)$,}\\[2mm]
    (iv)\ \text{$h_k$ is weakly conformal to $g_k$,}\\[3mm]
    (v)\ \text{$\{\Psi_k\}_{k\in \N}$, $\{\Psi_k^{-1}\}_{k\in\N}$ are bounded in $W^{1,\infty}(\Sigma,\Sigma)$. In particular, the distortions}\\
    \text{\quad \quad\cite[Definition 1.11]{Hencl14} of $\{\Psi_k\}$ are uniformly bounded in $L^\infty(\Sigma)$.} 
    \end{dcases}
             \end{align*}
\end{thm}
\begin{proof}
    By Theorem \ref{thmconchart}, there exists a covering of $\Sigma$ by open subsets $\{U_{\alpha}\}_{\alpha=1}^l$ with orientation-preserving bi-Lipschitz diffeomorphisms $\vp_{k,\alpha}\colon U_\alpha\rightarrow \vp_{k,\alpha}(U_\alpha)\subset \R^2$ (where $k\in \N\cup\{\infty\}$) such that $\bP_k\circ \vp_{k,\alpha}^{-1}$ is weakly conformal,
    $\vp_{k,\alpha}\rightarrow \vp_{\infty,\alpha}$ in $W^{2,2}(U_\alpha)$ as $k\rightarrow\infty$ and $\|\vp_{k,\alpha}\|_{W^{1,\infty}(U_\alpha)} + \|\vp_{k,\alpha}^{-1}\|_{W^{1,\infty}(\vp_{k,\alpha}(U_\alpha))}$ is uniformly bounded. Consequently, for any compact subset $K$ of $\vp_{\infty,\alpha}(U_{\alpha})$, $K\subset \vp_{k,\alpha}(U_\alpha)$ for large enough $k$ and $\vp_{k,\alpha}^{-1}\rightarrow \vp_{\infty,\alpha}^{-1}$ in $W^{2,2}(K)$. Then for any $1\le \alpha,\beta\le l$ and compact subset $K$ of $\vp_{\infty,\beta}(U_\alpha\cap U_\beta)$, the transition maps converge as well:
    \be\label{tran-con}
    \vp_{k,\alpha,\beta}\coloneqq\vp_{k,\alpha}\circ \vp_{k,\beta}^{-1}\xrightarrow[k\rightarrow \infty]{} \vp_{\infty,\alpha}\circ \vp_{\infty,\beta}^{-1}=:\vp_{\infty,\alpha,\beta}\quad \text{in $W^{2,2}(K)$.}   \ee 
    Furthermore, since each $\vp_{k,\alpha,\beta}$ is holomorphic, the convergence holds in $C^\infty(K)$.\\
    
    We want to prove that for $k$ large enough, there exists an open cover $\{V_{\alpha}\}_{\alpha=1}^l$ of $\Sigma$ with $V_\alpha\Subset U_\alpha$ and $C^\infty$ injective immersions $\rho_{k,\alpha}\colon \vp_{\infty, \alpha}(\overline{V_\alpha})\subset\R^2\rightarrow \vp_{k,\alpha}(U_\alpha)\subset\R^2$ satisfying  
    \begin{align}\begin{cases}(i)\ \rho_{k,\alpha}\xrightarrow[k\rightarrow\infty]{} id &\text{ in }C^\infty(\vp_{\infty, \alpha}(\overline{V_\alpha}),\R^2),\\[3mm]
(ii)\ \vp_{k,\beta,\alpha}\circ\rho_{k,\alpha}\circ\vp_{\infty,\alpha,\beta}=\rho_{k,\beta} &\text{ on } \vp_{\infty,\beta}(\overline{V_\alpha}\cap \overline{V_\beta}).\label{cocycle}
\end{cases}
    \end{align}  
    Once these maps $\rho_{k,\alpha}$ are defined, we consider the map $\Psi_k\coloneqq\vp_{k,\alpha}^{-1}\circ \rho_{k,\alpha}\circ \vp_{\infty,\alpha}$ on $\overline{V_\alpha}$. This definition is independent of $\alpha$ by \eqref{cocycle}. 
    By construction, it holds $\Psi_k\rightarrow id$ in $C^0(\Sigma,\Sigma)$. For $k$ large enough, $\Psi_k$ is an immersion, hence surjective (its image is both open and closed). 
    As for global injectivity, assume that there exist sequences $\{k_i\}_{i=1}^\infty\subset \N$, $\{p_{i,1}\}_{i=1}^\infty$, $\{p_{i,2}\}_{i=1}^\infty\subset \Sigma$ such that $\lim_{i\rightarrow \infty}k_i=\infty$ and $\Psi_{k_i}(p_{i,1})=\Psi_{k_i}(p_{i,2})$. 
    We can without loss of generality assume that $p_{i,1}\rightarrow p_1$ and $p_{i,2}\rightarrow p_2$ as $i\rightarrow \infty$. Since we have $\Psi_k\rightarrow id$ in $C^0(\Sigma,\Sigma)$ as $k\rightarrow \infty$, we obtain $p_1=p_2\in V_\alpha$ for some $\alpha$. For $i$ large enough we get $p_{i,1}=p_{i,2}$ since $\Psi_k$ is injective on $V_\alpha$.\\
    
     We now proceed to the construction of $\{V_\al\}$ and $\{\rho_{k,\al}\}$. Let $\{U_\alpha'\}_{\alpha=1}^l$ be an open cover of $\Sigma$ such that $U_\alpha'$ has smooth boundary and $U_\alpha'\Subset U_\alpha$ (for existence of such sets, see the proof of \cite[Proposition 8.2.1]{Daners08}). 
     We claim that there exist open sets $\{V_{\al,j}\}_{1\le \alpha\le j\le l}$ (write $V_{\al,\al-1}=V_{\al,\al}$), an integer $k_0\in\N$, 
     and $C^\infty$ injective maps $\rho_{k,\al}\colon \vp_{\infty,\al}(\overline{V_{\al,\al}})\rightarrow\vp_{k,\alpha}(U_\al)$ for $k\ge k_0$ such that for any $1\le j\le l$ and $1\le \al,\beta\le j$, it holds 
         \begin{subnumcases}{}
             (i)\ V_{\al,j}\subset V_{\al,j-1}\subset U_{\al}',\label{con-rho-V1}\\[3mm]
             (ii)\ \Sigma=\Big(\bigcup\limits_{\gamma\le j} V_{\gamma,j}\Big)\cup \Big(\bigcup \limits_{\gamma>j} U'_{\gamma}\Big),\label{con-rho-V2}\\
             (iii)\ \rho_{k,\alpha}\xrightarrow[k\rightarrow\infty]{} id &in $C^\infty(\vp_{\infty, \alpha}(\overline{V_{\alpha,\alpha}}),\R^2)$,\label{con-rho-V3}\\[2mm]
             (iv)\ \vp_{k,\beta,\alpha}\circ\rho_{k,\alpha}\circ\vp_{\infty,\alpha,\beta}=\rho_{k,\beta} &on $\vp_{\infty,\beta}(\overline{V_{\alpha,j}}\cap \overline{V_{\beta,j}})$.\label{con-rho-V4}
         \end{subnumcases}
     We prove this by induction on $j$. 
For $j=1$, we define $V_{1,1}=U_1'$ and $\rho_{k,1}\colon \vp_{\infty,1}(\overline{U_1'})\rightarrow \R^2$ as the identity map. 
For $k$ large enough, the image of $\rho_{k,1}$ is contained in $\vp_{k,1}(U_1)$ since $\vp_{k,1}\rightarrow \vp_{\infty,1}$ in $C(\overline{U_1'},\R^2)$ and $U_1'\Subset U_1$. 
Let $1\le j_0\le l$, suppose that we have constructed $\{V_{\alpha, j}\}_{1\le \al\le j\le j_0}$ and $\{\rho_{k,\alpha}\}_{1\le \al\le j_0}$ such that \eqref{con-rho-V1}--\eqref{con-rho-V4} hold for any $1\le j\le j_0$ and $1\le \al,\beta\le j$. We now construct $\rho_{k,j_0+1}$ and $\{V_{\al,j_0+1}\}_{1\le \al\le j_0+1}$ using the induction hypothesis.\\

For $1\le \alpha\le j_0$ we define 
$$
\omega_{k,j_0}\coloneqq\vp_{k,j_0+1,\alpha}\circ\rho_{k,\alpha}\circ\vp_{\infty,\alpha,j_0+1} \quad \text{on $\vp_{\infty,j_0+1}\left( \overline{V_{\alpha,j_0}} \cap U_{j_0+1}\right)$.}
$$ 
 By our induction hypothesis \eqref{con-rho-V4} for $j\le j_0$, this map is well-defined on $\vp_{\infty,j_0+1}\Big(\big( \bigcup\limits_{\alpha\le j_0}\overline{V_{\alpha,j_0}}\big) \cap U_{j_0+1}\Big)$ and independent of $\alpha$. 
 By the assumption $\Sigma=\big(\bigcup\limits_{\alpha\le j_0} V_{\alpha,j_0}\big)\cup \big(\bigcup \limits_{\alpha>j_0} U'_{\alpha}\big)$, we have  
$$
\Big(\Sigma\setminus \bigcup\limits_{\alpha\le j_0 } V_{\alpha,j_0}\Big)\cap \Big(\Sigma\setminus \bigcup\limits_{\alpha> j_0}U'_{\alpha}\Big)=\emptyset.
$$
Hence we can choose a smooth cut-off $\xi_{j_0}\in C^\infty(\Sigma)$ such that $\xi_{j_0}=0$ on a neighborhood of $\Sigma\setminus\bigcup\limits_{\alpha\le j_0} V_{\alpha,j_0}$ and  $\xi_{j_0}=1$ on a neighborhood of $\Sigma\setminus \bigcup\limits_{\alpha> j_0}U'_{\alpha}$. Now we define $V_{j_0+1,j_0+1}\coloneqq U_{j_0+1}'\Subset U_{j_0+1}$, and for $\alpha\le j_0$ 
     $$V_{\alpha,j_0+1}\coloneqq V_{\alpha,j_0}\cap \,\text{int}\Big\{x\in\Sigma:\xi_{j_0}(x)=1\Big\}.$$ 
The condition \eqref{con-rho-V1} then holds for $j=j_0+1$. By the definition of $\{V_{\al,j_0+1}\}_{\al\le j_0+1}$ we have \begin{align*}
    \Big(\bigcup\limits_{\alpha\le j_0+1} V_{\alpha,j_0+1}\Big)\cup \Big(\bigcup \limits_{\alpha>j_0+1} U'_{\alpha}\Big)&=\Big(\bigcup\limits_{\alpha\le j_0} V_{\alpha,j_0+1}\Big)\cup \Big(\bigcup \limits_{\alpha>j_0} U'_{\alpha}\Big)\\
    &=\Big( \Big(\bigcup\limits_{\alpha\le j_0} V_{\alpha,j_0}\Big)\cap\,\text{int}\Big\{x\in\Sigma:\xi_{j_0}(x)=1\Big\}\Big) \cup \Big(\bigcup \limits_{\alpha>j_0} U'_{\alpha}\Big).\end{align*}
    The definition of $\xi_{j_0}$ also implies $$\Sigma\setminus\bigcup\limits_{\alpha> j_0}U'_{\alpha}\subset \text{int}\Big\{x\in\Sigma:\xi_{j_0}(x)=1\Big\}.
    $$We then obtain \begin{align*}
     \Big(\bigcup\limits_{\alpha\le j_0+1} V_{\alpha,j_0+1}\Big)\cup \Big(\bigcup \limits_{\alpha>j_0+1} U'_{\alpha}\Big)&\supset \Big( \Big(\bigcup\limits_{\alpha\le j_0} V_{\alpha,j_0}\Big)\setminus \bigcup\limits_{\alpha> j_0}U'_{\alpha}\Big) \cup \Big(\bigcup \limits_{\alpha>j_0} U'_{\alpha}\Big)\\
    &=\Big(\bigcup\limits_{\alpha\le j_0} V_{\alpha,j_0}\Big)\cup \Big(\bigcup \limits_{\alpha>j_0} U'_{\alpha}\Big)\\
    &=\Sigma.
\end{align*}
Hence the condition \eqref{con-rho-V2} holds for $j=j_0+1$. Since $\xi_{j_0}$ vanishes on a neighborhood of $\Sigma\setminus\bigcup\limits_{\alpha\le j_0} V_{\alpha,j_0}$, and $\omega_{k,j_0}\circ \vp_{\infty,j_0+1}\in C^\infty\big(\big( \bigcup\limits_{\alpha\le j_0}V_{\alpha,j_0}\big) \cap U_{j_0+1}\big)$ we have $(\xi_{j_0} \circ\vp^{-1}_{\infty,j_0+1})\,\omega_{k,j_0}\in C^\infty(\vp_{\infty,j_0+1}(U_{j_0+1}))$. 
Then we define $$\rho_{k,j_0+1}\coloneqq (\xi_{j_0}\circ \vp^{-1}_{\infty,j_0+1})\,\omega_{k,j_0}+(1-\xi_{j_0}\circ \vp^{-1}_{\infty,j_0+1}) \,id\in C^\infty(\vp_{\infty,j_0+1}(U_{j_0+1}),\R^2).$$ 
Let $\alpha\le j_0$. Since $\xi_{j_0}=1$ on $\overline{V_{\alpha,j_0+1}}$, the following holds on $\vp_{\infty,j_0+1}(U_{j_0+1}\cap \overline{V_{\alpha,j_0+1}})$:
$$\rho_{k,j_0+1}=\omega_{k,j_0}=\vp_{k,j_0+1,\alpha}\circ\rho_{k,\alpha}\circ\vp_{\infty,\alpha,j_0+1}.$$
Since we have $
\overline{V_{j_0+1,j_0+1}}=\overline{U'_{j_0+1}}\subset U_{j_0+1}$, the condition \eqref{con-rho-V4} is thus satisfied for $\beta=j=j_0+1$, $\alpha\le j_0$. By using \eqref{con-rho-V1} and the identity $\vp_{k,\al,\beta}=\vp_{k,\beta,\al}^{-1}$, we then obtain \eqref{con-rho-V4} for $j=j_0+1$ and any  $\al,\beta\le j_0+1$.\\

Next we prove $\rho_{k,j_0+1}\rightarrow id$ in $C^\infty(\vp_{\infty,j_0+1}(\overline{U'_{j_0+1}}),\R^2)$ as $k\rightarrow\infty$. 
Fixing $\alpha\le j_0$, let $W$ be an open set such that $\vp_{\infty,\alpha}(\overline{V_{\alpha,j_0}}\cap \overline{U_{j_0+1}'})\subset W\Subset \vp_{\infty,\alpha}(U_\alpha\cap U_{j_0+1})$. 
The convergence \eqref{tran-con} implies that $$ \vp_{k,j_0+1,\alpha}\xrightarrow[k\rightarrow\infty]{} \vp_{\infty,j_0+1,\alpha} \quad \text{in $C^\infty(\overline W)$. } $$
By our induction hypothesis \eqref{con-rho-V3} for $\al\le j_0$, we also have $$ \rho_{k,\alpha}\xrightarrow[k\rightarrow\infty]{} id \quad \text{in $C^\infty(\vp_{\infty,\alpha}(\overline{V_{\alpha,j_0}}))$.}$$
Then the definition of $\omega_{k,j_0}$ implies  
$$ \omega_{k,j_0}\xrightarrow[k\rightarrow\infty]{} id \qquad \text{in $C^\infty\left(\vp_{\infty,j_0+1}(\overline{V_{\alpha,j_0}}\cap \overline{U'_{j_0+1}})\right)$.}
$$ 
Hence we have
      $$ \rho_{k,j_0+1}\xrightarrow[k\rightarrow\infty]{} id\qquad \text{in $C^\infty\Big(\vp_{\infty,j_0+1}\Big(\Big(\bigcup_{\alpha\le j_0}\overline{V_{\alpha,j_0}}\Big)\cap \overline{U'_{j_0+1}}\Big)\Big)$.}$$
On the other hand, it holds that $\rho_{k,j_0+1}=id$ on $\vp_{\infty,j_0+1}\big(U_{j_0+1}\setminus \bigcup\limits_{\alpha\le j_0}V_{\alpha,j_0}\big)$. Therefore, we have
$$
\rho_{k,j_0+1}\xrightarrow[k\rightarrow\infty]{} id\quad \text{in $C^\infty\big(\vp_{\infty,j_0+1}( \overline{U'_{j_0+1}})\big)$.}
$$
Hence the condition \eqref{con-rho-V3} holds for $j=j_0+1$. In particular, $\rho_{k,j_0+1}$ is an immersion on $\vp_{\infty,j_0+1}(\overline{V_{j_0+1,j_0+1}})=\vp_{\infty,j_0+1}(\overline{U'_{j_0+1}})$ for $k$ large enough. 
Also, since $\p U'_{j_0+1}$ is smooth, we have the following embedding (see \cite[Theorem 5.8.4]{evans}): 
$$
W^{1,\infty}\left( \vp_{\infty,j_0+1}(U'_{j_0+1}) \right)\subset C^{0,1}\left(\vp_{\infty,j_0+1}(\overline{U'_{j_0+1}}) \right).
$$
Consequently, there exists $k_0\in \N$ such that for any $k\ge k_0$ and $x,y\in \vp_{\infty,j_0+1}(\overline{U'_{j_0+1}})$ it holds 
$$ 
|(\rho_{k,j_0+1}-id)(x)-(\rho_{k,j_0+1}-id)(y)|\le \frac 12 \,|x-y|.
$$
Hence we have 
$$
|\rho_{k,j_0+1}(x)-\rho_{k,j_0+1}(y)|\ge \frac 12 \,|x-y|.
$$
In particular, this implies $\rho_{k,j_0+1}$ is injective on $ \vp_{\infty,j_0+1}(\overline{U'_{j_0+1}})$ for $k\ge k_0$. 
Moreover, since $V_{j_0+1,j_0+1}=U'_{j_0+1}\Subset U_{j_0+1}$ and $\vp_{k,j_0+1}\rightarrow\vp_{\infty,j_0+1}$ in $C^0(\overline{U'_{j_0+1}})$ as $k\rightarrow\infty$, 
for $k$ large enough we have $$\rho_{k,j_0+1}\big(\vp_{\infty,j_0+1}(\overline{V_{j_0+1,j_0+1}})\big)\subset \vp_{k,j_0+1}(U_{j_0+1}).$$ This completes the induction. Finally, for $1\le \alpha\le l$, we take $V_\alpha\coloneqq V_{\alpha,l}$, the constructions of $\{\rho_{k,\alpha}\}$ and $\{V_\alpha\}$ are thus accomplished.

Since $\Sigma=\bigcup\limits_{1\le \alpha\le l} V_\alpha$, we can choose smooth non-negative cut-off functions $\phi_\alpha\in C^\infty_c(\vp_{\infty,\alpha}(V_\alpha))$ such that $\sum\limits_{\alpha=1}^l \phi_\alpha\circ \vp_{\infty,\alpha}>0$ on $\Sigma$. 
Since $\vp_{k,\alpha}\rightarrow \vp_{\infty,\alpha}$ in $C^0(\overline{V_\alpha})$ as $k\rightarrow\infty$, we have $\phi_{\alpha}\in C_c^\infty(\vp_{k,\alpha}(V_\alpha))$ for $k\in \N\cup\{\infty\}$ large enough. For such $k$ we define $$g'_k\coloneqq\sum_{\alpha=1}^l \vp_{k,\alpha}^*(\phi_\alpha\, g_{\text{std}}).$$
Then $g_k'$ is a smooth metric under the complex structure induced by $\bP_k$ as defined in Corollary \ref{coconstr}. 
Since $\vp^*_{k,\alpha}\,g_{\text {std}}$ is weakly conformal to $g_k$ on $V_\alpha$, we have $g_k'$ is weakly conformal to $g_k$. 
Since we assumed $\ti g$ and $g_\infty$ are weakly conformal, we obtain that $\vp_{\infty, \alpha}$ is weakly conformal from $(V_\alpha, g)$ to $(\vp_{\infty,\alpha}(V_\alpha),g_{\text{std}})$. 
Hence $\vp_{\infty,\alpha}$ is smooth by Corollary \ref{coconstr} (i). 
Let $1\le \alpha\le l$, $x\in V_\alpha$, then by \eqref{tran-con} and \eqref{cocycle} we have \begin{align*}
(\Psi_k^* g'_k)_x&=\Big((\vp_{k,\alpha}^{-1}\circ \rho_{k,\alpha}\circ \vp_{\infty,\alpha})^*\sum_{x\in V_\beta} \vp_{k,\beta}^*(\phi_\beta \,g_{\text{std}})\Big)_x\\
&=\Big(\sum_{x\in V_\beta} (\vp_{k,\beta,\alpha}\circ \rho_{k,\alpha} \circ \vp_{\infty,\alpha})^*(\phi_\beta \,g_{\text{std}})\Big)_x\\
&\xrightarrow[k\to \infty]{} \Big(\sum_{x\in V_\beta} (\vp_{\infty,\beta})^*(\phi_\beta\, g_{\text{std}})\Big)_x=(g'_{\infty})_x \;\text{smoothly in $x$}.
\end{align*}
Consequently, we can apply Lemma \ref{lem-stru-con-smooth} to the sequence $\{\Psi_k^*g_k'\}$ (denote $\Psi_\infty\coloneqq id$) to obtain a sequence of Riemannian metrics ${h_k'}$ of constant Gaussian curvature $1,-1$ or $0$ such that $h_k'$ is conformal to $\Psi_k^*g_k'$ and $h_k'$ converges to $h_\infty'$ smoothly. Finally, we take $h_k\coloneqq(\Psi_k^{-1})^* h_k'$, then $\Psi_k^*h_k$ converges to $h_{\infty}$ smoothly, $h_k$ is weakly conformal to $g_k'$ and hence $g_k$. By the definition of $\Psi_k$, we also have $h_k$ is smooth in the complex structure induced by $\bP_k$ (i.e. $(\vp_{k,\alpha}^{-1})^*h_k$ is smooth), and $h_k$ has the same constant Gaussian curvature as $h_k'$.
\end{proof}
\begin{Rm}
   Let $K>0$. Assume $\{g_k\}$ is replaced by a sequence of metrics $\{\tilde {g_k}\}$ such that for any $k$ and any coordinate chart $\vp\colon U\colon\Sigma\rightarrow\R^2$, it holds $$\det( (\tilde {g_k})_{ij})> K|((\tilde {g_k})_{ij})|^2\quad \text{a.e.,}
   $$
   where we denote 
   $$(\tilde {g_k})_{ij}\coloneqq\tilde {g_k}\Big((d\vp)^{-1}\Big(\frac{\p}{\p x^i}\Big),(d\vp)^{-1}\Big(\frac{\p}{\p x^j}\Big)\Big).$$
   Assume $\tilde {g_k}\rightarrow \tilde g_{\infty}$ a.e. Then we can construct the conformal charts $\{\vp_{k,\alpha}\}$ as in Remark \ref{rm-gconstr}, and by using the same argument as in Theorem \ref{thm-con-stru}, we find $\{h_k\}$, $\{\Psi_k\}$ satisfying Theorem \ref{thm-con-stru} (i), (ii), (iv), but $\Psi_k$ may not be Lipschitz. Indeed, by \cite[Proposition 4.36]{Imayoshi92} and \cite{Iwaniec79}, we have $\vp_{k,\alpha}\rightarrow \vp_{\infty,\alpha}$, $\vp^{-1}_{k,\alpha}\rightarrow \vp^{-1}_{\infty,\alpha}$ in $W^{1,p}_{\textup{loc}}$ for some $p>2$, hence $\Psi_k\rightarrow id$, $\Psi^{-1}_k\rightarrow id$ in $W^{1,p}(\Sigma,\Sigma)$, in particular, in $C^0(\Sigma,\Sigma)$.
\end{Rm}

\subsubsection{Approximation results}\label{sec:ApproximationResults}

The first goal of this section is to prove Theorem \ref{th:Approx}, namely that a $W^{1,\infty}$-immersion can be approximated by smooth immersions as soon as its first derivative belongs to $VMO$. Then we prove Theorem \ref{th-approx} at the end. 
We define a weak $W^{1,\infty}$-immersion taking values in manifolds.
\begin{Dfi}
     Let $N$ be an $n$-dimensional smooth manifold without boundary, $U\subset N$ be open and precompact, $M$ be an $\ell$-dimensional manifold smoothly embedded in $\R^m$. We say $f\in W^{1,\infty}_{\textup{imm}}(U,M)$ if $f\in W^{1,\infty}_{\textup{imm}}(U,\R^m)$ and $f(x)\in M$ for any $x\in U$.
\end{Dfi}

We estimate the mean oscillation of a composition and the distance to a smooth approximation.

\begin{Lm}
    \label{lmcovvmo}
  Let $U\subset \R^n$ be open and bounded. Let $f\in W^{1,\infty}(U)$ with $\g f\in  \text{VMO}(U)$. Let $\vp\colon\overline U\rightarrow \overline W\subset \R^n$ be a $C^1$ diffeomorphism, where $W=\vp(U)$. \\
  (i) For any $r>0$, we have $\g (f\circ \vp^{-1})\in VMO(W)$, and there exists a positive constant $C=C(n,\|\g\vp\|_{L^\infty(U)},\|\g\vp^{-1}\|_{L^\infty(W)})$ such that
  \begin{align}\label{nabfphi}
      \beta_r(\g(f\circ\vp^{-1}))\le C\left(\beta_r(\g f)+\beta_r(\g \vp)\|\g f\|_{L^\infty(U)}\right).
  \end{align}
 (ii) Let $\eta$ be the standard mollifier. For $\vae>0$ we let $\eta_{\vae}(x)=\vae^{-n}\eta(\vae^{-1}x)$. For $x\in U$ and $\vae<\dist(x, U^c)$, we define $f_\vae(x)\coloneqq \int_{U} f(y) \,\eta_{\vae}(x-y) \,dy$.
    \\Let $K>0$ and assume in addition that $\p U$ is Lipschitz. Then there exists a positive constant $ C=C(\|\g\vp\|_{L^\infty(U)},\|\g\vp^{-1}\|_{L^\infty(W)}, U,K)>0$ such that for any $x_1,x_0\in U$ and $\vae>0$ satisfying $|x_1-x_0|\le K\vae$ and $\vae<\min\big(\dist(\vp(x_1),\p W), \dist(x_0,\p U)\big)$, there holds 
 \begin{align}\begin{aligned}\label{vmocov}
 &\left| \g (f_\vae \circ\vp^{-1})(\vp(x_0))-\fint_{B_{\vae} (\vp(x_1))}\g (f\circ\vp^{-1}) \right| \\[2mm]
 &\le C\Big(\beta_\vae(\g f)+\big(\omega_\vae(\g \vp)+\vae\big)\|\g f\|_{L^\infty(U)} \Big),
 \end{aligned}
 \end{align}
 where $\omega_\vae(\g \vp)\coloneqq\sup\{|\g \vp(x)-\g \vp(y)|:x,y\in U, |x-y|\le \vae\}$.
\end{Lm}
\begin{proof}
   (i) Since $\vp$ is quasiconformal, by Lemmas \ref{bmo-estimate} and \ref{lm-quasi-vmo}, we have $\g f\circ \vp^{-1}$ is in $VMO(W)$. Moreover, there exists a constant $C=C\big(n,\|\g \vp\|_{L^\infty(U)},\|\g \vp^{-1}\|_{L^\infty(W)}\big)$ such that the following estimate holds
   $$
   \beta_r \big(\g f\circ \vp^{-1}\big)\le C\,\beta_r(\g f).
   $$
   Hence for every ball $B_r\subset W$, we have
   \begin{align*}
       & \fint_{B_r} \left|\g(f\circ \vp^{-1})-\fint_{B_r}\g f\circ\vp^{-1}\fint_{B_r} \g \vp^{-1}\right| \\[1.5mm]
       &\le  \ \fint_{B_r} \left|\left(\g f\circ \vp^{-1}-\fint_{B_r}\g f\circ\vp^{-1}\right)\g \vp^{-1}\right| + \fint_{B_r}\left| \g f\circ\vp^{-1}\left(\g \vp^{-1}-\fint_{B_r}\g \vp^{-1}\right) \right|\\[2mm]
       &\le \|\g \vp^{-1}\|_{L^
       \infty(W)} \ \beta_r \big(\g f\circ \vp^{-1})+\beta_r(\g \vp^{-1})\|\g f\|_{L^\infty(U)}\\[4mm]
       &\le C\big(\beta_r(\g f)+\beta_r(\g \vp)\|\g f\|_{L^\infty(U)}\big),
   \end{align*}
   The estimate \eqref{nabfphi} is thus proved. \\
   (ii) Let $y\in B_{\vae}(\vp(x_1))$, we have 
   $$
   |\vp^{-1}(y)-x_1|\le \|\g \vp^{-1}\|_{L^\infty(W)}\, |y-\vp(x_1)|\le  \vae\,\|\g \vp^{-1}\|_{L^\infty(W)}.
   $$
   Hence it holds that $|x_0-\vp^{-1}(y)|\le \vae\,(K+\|\g \vp^{-1}\|_{L^\infty(W)})$. Moreover it holds $(\g \vp)^{-1}\in C^0(\overline U)$ and for all $r>0$, we have
   $$
   \omega_r\big ((\g \vp)^{-1}\big)\le C\,\omega_r(\g \vp).
   $$
   Since $\p U$ is Lipschitz, we can extend $(\g \vp)^{-1}$ to $g\in C_c(\R^n,\R^{n\times n})$ such that for any $r>0$,
   $$
   \omega_{r,\R^n} (g)\le C\Big(\omega_r\big((\g \vp)^{-1}\big)+r\,\|(\g \vp)^{-1}\|_{L^\infty(U)}\Big)\le C\Big( \omega_r(\g \vp)+r\Big).
   $$
   In particular, this implies that there exists a constant $C=C(\|\g\vp\|_{L^\infty(U)},\|\g\vp^{-1}\|_{L^\infty(W)}, U,K)>0$ such that for any $y\in B_\vae(\vp(x_1))$, it holds that
   \begin{align*}
       \big\|\g \vp^{-1}(y)-\g \vp^{-1}(\vp(x_0))\big\|&=\big\|\big(\g \vp(\vp^{-1}(y))\big)^{-1}-\big(\g \vp(x_0)\big)^{-1}\big\|\\[2mm]
       &\le \big(K+\|\g \vp^{-1}\|_{L^\infty(W)}+1\big)\,\omega_{\vae,\R^n} (g)\\[2mm]
       &\le C\big( \omega_\vae(\g \vp)+\vae\big).
    \end{align*}
     Hence we have 
     \begin{align} \begin{aligned}\label{dvpdiff}
         &\left|\fint_{B_{\vae} (\vp(x_1))}\g (f\circ\vp^{-1})-\g \vp^{-1}(\vp(x_0))\fint_{B_{\vae} (\vp(x_1))}  \g f\circ \vp^{-1} \right|\\[2mm]
         &\le C \big(\omega_\vae(\g \vp)+\vae\big )\,\|\g f\|_{L^\infty(U)}.
         \end{aligned}
    \end{align}
       Since $\dist(\vp(x_1),\p W)>\vae$, we have $\dist (x_1,\p U)>\vae\,\|\g \vp\|^{-1}_{L^\infty(U)} $. Now we can apply the proof of Riemann's Theorem \cite[Theorem 2]{Rei74} and Lemma \ref{bmo-estimate} to get \begin{align}\label{vmoaftercov}
           \bigg| \fint_{B_{\vae} (\vp(x_1))}  \g f\circ \vp^{-1}-\fint_{B_{\vae\,\|\g \vp\|^{-1}_{L^\infty}}(x_1)} \g f\bigg|\le C\,\beta_\vae (\g f).
       \end{align}
       Since $\p U$ is Lipschitz, by  Corollary \ref{covmoext} we can extend $\g f$ to a function $\tilde f\in VMO(\R^n,\R^n)$ with 
       $$
    \forall r>0,\qquad \beta_r(\tilde f)\le C\big (\beta_r(\g f)+r\,\|\g f\|_{L^\infty(U)}\big ).
       $$
       Hence it holds that
       \begin{align}\begin{aligned}\label{modiff2}
           \bigg|\fint_{B_{\vae\,\|\g \vp\|^{-1}_{L^\infty}}(x_1)} \g f-\fint_{B_\vae(x_0)} \g f \bigg|  \le C\,\beta_{\vae}(\tilde f)
            \le C\big( \beta_\vae(\g f)+\vae\,\|\g f\|_{L^\infty(U)}\big).
           \end{aligned}
        \end{align}
        We also have that 
        \begin{align}
        \begin{aligned} \label{modiff3}
               \left|\g f_\vae(x_0)-\fint_{B_\vae(x_0)}\g f \right|&= \left|\int_{B_\vae(x_0)}\eta_{\vae}(x_0-x)\,\left(\g f(x)-\fint_{B_\vae(x_0)}\g f\,dx\right) \right|\\[1mm]                &\le C\,\fint_{B_\vae(x_0)} \left|\g f-\fint_{B_\vae(x_0)}\g f \right|\\[1mm]
                & \le C\, \beta_\vae(\g f). 
        \end{aligned}
        \end{align}
        Combining \eqref{vmoaftercov}, \eqref{modiff2}, \eqref{modiff3} we obtain by the triangle inequality 
        \begin{align}
        \begin{aligned}\label{modiff4}
                   & \left|\g \vp^{-1}(\vp(x_0))\fint_{B_{\vae}(\vp(x_1))}\g f\circ \vp^{-1}- \g(f_\vae \circ\vp^{-1})(\vp(x_0))\right|\\[1mm]
                   &\le \|\g \vp^{-1}\|_{L^\infty}\left|\fint_{B_{\vae}(\vp(x_1))}\g f\circ \vp^{-1}-\g f_\vae(x_0)\right|\\[1mm] 
                   &\le C\big( \beta_\vae(\g f)+\vae\,\|\g f\|_{L^\infty(U)}\big).
        \end{aligned}
        \end{align}
        Finally, \eqref{vmocov} follows from \eqref{dvpdiff} and \eqref{modiff4}.
\end{proof}

We are now ready to establish an approximation result required for the proof of Theorem \ref{th-approx}. The underlying idea goes back to \cite{SU}, see also \cite[Remark~2.1]{Kuwertli12} for the $W^{2,2}$ case in dimension $2$. For completeness, we present a detailed proof.

\begin{thm}\label{th:Approx}
    Let $(N,\ti g)$ be an $n$-dimensional Riemannian manifold without boundary, $M$ be a manifold of dimension $\ell$ smoothly embedded in $\R^m$. Let $U\subset N$ be open and precompact. Assume $\p U$ is Lipschitz. Let $f\in 
    W_{\textup{imm}}^{1,\infty}(U,M)$ with $df\in \text{VMO}(U)$. Then there exists a sequence of $C^\infty$ immersions $\{f_k\}_{k=1}^\infty\subset C^\infty(\overline U,M)$ and a positive constant $\La$ such that $df_k\rightarrow df $ a.e. on $U$ and \begin{equation}\label{seqbound}
        \La^{-1}|v|_{\ti g}\le|d(f_k)(v)|_{\R^m}\le \La |v|_{\ti g},\qquad \text{for all } \,v\in TU\,\text{ and all }\, k\in\N.
    \end{equation} 
  Moreover, if $f\in W^{2,p}(U)$ for any $1\le p<\infty$, $f_k$ can be chosen to strongly converge to $f$ in $W^{2,p}(U)$.
    \end{thm}
\begin{proof}
    Since $\p U$ is Lipschitz and $f\in W^{1,\infty}(U)$, $f$ has a unique Lipschitz continuous representative defined on $\overline U$. In particular, $f(\overline{U})$ is compact. 
    Since $M$ is smooth, there exists a bounded tubular neighborhood $W$ of $M$ with a retraction map $\pi\in C^\infty( W, M)$ such that $\pi(y)=y$ for any $y\in M$. Since $f(U)$ is compact, there exists a precompact neighborhood $W'$ of $f(U)$ strictly contained in $W$. Since $\overline U$ is compact, there exist finitely many coordinate charts $\{(U_j,\phi_j)\}_{j=1}^s$ of $N$ such that $\overline U\subset \bigcup_{j=1}^s U_j$. Since $\p U$ is Lipschitz, we can further assume that for each $U_j\nsubset U$, we have $\phi_j(U_j)=B^n\subset \R^n$, $0\in \phi_j(\p U)$ and, for some Lipschitz map $\gamma_j\colon\R^{n-1}\rightarrow \R$,
    $$
    \phi_j(U\cap U_j)=\{(x_1,\dots,x_n)\in B^n:x_n>\gamma_j(x_1,\dots,x_{n-1})\}.
    $$ 
    Let $\{\xi_j\}_{j=1}^s$ be a $C^\infty$ partition of unity such that $\supp(\xi_j)\subset V_j\Subset U_j$ and $\sum_{j=1}^s \xi_j=1$ on $U$, where each $V_j$ is a precompact open subset of $U_j$. 
    Let $\ti f_j\coloneqq f\xi_j$, then $\ti f_j\in W^{1,\infty}(U,\R^k)$ and $d\ti f_j\in  VMO(U)$. Let 
    \begin{align}\label{eq:choice_ep}
    0<\vae< \min_{1\leq j\leq s}\left[ \frac{1}{K_j+1} \min\Big(\text{dist}\big(\phi_j(\supp(\xi_j)),\phi_j(\p V_j)\big),\text{dist}\big(\phi_j(V_j),\phi_j(U_j)^c\big)\Big) \right].
    \end{align}
    Define $\eta_\vae$ as in Lemma \ref{lmcovvmo} (ii). For $1\le j\le s$, if $U_j\subset U$, for $x\in V_j$ we define 
    $$ 
    F_{j,\vae}(x)\coloneqq \int_{B_\vae(\phi_j(x))}\ti f_j (\phi_j^{-1}(z)) \,\eta_\vae (\phi_j(x)-z)\,dz.
    $$ 
    Then it holds that $F_{j,\vae}\in C_c^\infty(V_j,\R^k)\subset C_c^\infty(U,\R^k)$. 
    If $U_j\cap \p U\neq \emptyset$, let $K_j\coloneqq\text{Lip}(\gamma_j)+2$. We denote $e_n\coloneqq(0,\dots,0,1)\in \R^n$. Then for any $z\in \phi_j(\overline V_j\cap \overline U)$, $\vae>0$, there holds 
    $$
    B_\vae(z+K_j\,\vae\, e_n)\cap \phi_j(U_j\cap\p U)=\emptyset.
    $$
    By choice of $\vae$ in \eqref{eq:choice_ep}, it holds that
    $$
    B_\vae(z+K_j\,\vae\, e_n) \subset \phi_j(U_j\cap U).
    $$
    We define for $x\in \phi_j(\overline V_j\cap \overline U)$:
    $$ 
    F_{j,\vae}(x)\coloneqq \int_{B_\vae(0)}\ti f_j\big (\phi_j^{-1}(\phi_j(x)+K_j\,\vae\, e_n-z)\big) \,\eta_\vae (z)\,dz.
    $$
    Thus, we have $F_{j,\vae}=0$ in a neighborhood of $\p V_j\cap U$. We extend the map $F_{j,\vae}$ to a function in $C^\infty(\overline U)$ by setting $F_{j,\vae}=0$ on $\overline U\setminus V_j$.\\
    
    We now define $F_\vae\coloneqq\sum_{j=1}^s F_{j,\vae}$. Then it holds that $F_\vae\rightarrow f$ in $C(\overline U)$ as $\vae\rightarrow 0$, and $F_\vae(\overline U)\subset W'$ for $\vae$ small enough. Let $f _{ \vae}\coloneqq\pi\circ F_\vae$. The family $\{f _{ \vae}\}_{\vae>0}$ is bounded in $W^{1,\infty}(U,M)$ for $\vae$ sufficiently small. Since $f=\pi\circ f$, the Sobolev
    convergence of $f_{ \vae}$ to $f$ follows from the same convergence of $F_\vae$ to $f$. Now it suffices to prove that there exists a constant $\La>0$ such that for all $\vae$ small enough, we have
         \be\label{seqbound1} 
         \forall v\in T\Sigma,\qquad \La^{-1}|v|_{\ti g}\le|d(f _{ \vae})(v)|_{\R^m}\le \La |v|_{\ti g}. 
         \ee
Let $1\le j,j'\le s$, since $\overline V_j\cap \overline V_{j'}\subset  U_j\cap U_{j'}$, we can find a pre-compact region $U_{j,j'}$ with smooth boundary (see \cite[Proposition 8.2.1]{Daners08}) such that 
    $$
    \overline {V_j}\cap \overline {V_{j'}} \subset U_{j,j'}\Subset  U_j\cap U_{j'}.
    $$ 
    If $U_j\cap U_{j'}=\emptyset$, we set 
    $U_{j,j'}=\emptyset$. Now we can apply \eqref{vmocov} to the function $\ti f_j\circ \phi_{j'}^{-1}$ and transition map $\vp=\phi_j\circ \phi_{j'}^{-1}$ defined on $\phi_{j'}(U_{j,j'})$. In the case where $U_{j'}\cap \p U\neq \emptyset$, we extend the following map to a VMO function defined on an open subset of $\phi_{j'}(U_{j'})$ strictly containing $\phi_{j'}(U_{j,j'})$, as in the proof of Lemma \ref{lmcovvmo}:
    \[
    \g (\ti f_j\circ \phi_{j'}^{-1})\colon\phi_{j'}(U_{j'}\cap U)\rightarrow \R^{k\times n}.
    \]
    For $z\in \phi_j(V_j\cap \overline U)$, we set $z^\vae\coloneqq z$ if $V_j\subset U$, and $z^\vae\coloneqq z+\vae \,K_j\,e_n$ otherwise. Then we obtain the following estimate where the right-hand side is independent of $z$:
    \begin{gather*}
        \left|\g(F_{j',\vae}\circ \phi_{j}^{-1})(z)-\fint_{B_\vae(z^\vae)}\g (\ti f_{j'}\circ \phi_{j}^{-1})\right|= \underset{\vae\to 0}{o}(1).
    \end{gather*}
  Consequently, it holds that
  \begin{align}\begin{aligned}
  \label{osdiff}
      &\left|\g (F_\vae\circ\phi_j^{-1})(z)-\fint_{B_\vae(z^\vae)}\g (f\circ \phi_{j}^{-1})\right|\\
      &\le \sum_{j'=1}^s \left|\g (F_{j',\vae}\circ \phi_{j}^{-1})(z)-\fint_{B_\vae(z^\vae)}\g (\ti f_{j'}\circ \phi_{j}^{-1})\right|\\
      &= \underset{\vae\to 0}{o}(1).
      \end{aligned}
    \end{align}
     Since $df\in \text{VMO}(U)$, we also have 
     \begin{align*}
          \fint_{B_\vae(z^\vae)}\left|\g (f\circ \phi_{j}^{-1})-\fint_{B_\vae(z^\vae)}\g (f\circ \phi_{j}^{-1})\right|=\underset{\vae\to 0}{o}(1).
      \end{align*} 
      Together with \eqref{osdiff}, then we get the following estimate where the right-hand side is independent of $z\in \phi_j(V_j\cap \overline U)$:
      \begin{align}\label{osdiff1}
          \fint_{B_\vae(z^\vae)}\left|\g (f\circ \phi_{j}^{-1})-\g (F_\vae\circ\phi_j^{-1})(z)\right|=\underset{\vae\to 0}{o}(1).
      \end{align} 
       Since $f\in W_{\textup{imm}}^{1,\infty}(U,M)$, there exists a constant $C_j>0$ such that for a.e. $z\in \phi_j(V_j)$ and for any $v\in\mathbb S^{n-1}$ it holds \begin{equation}\label{fimmcon}
          \big| d (f\circ \phi_j^{-1})_z(v)\big| \ge C_j.
          \end{equation}
    Since $\overline {W'}$ is compact and $\pi$ is smooth in an open set containing $\overline{W'}$, the map $\g \pi\colon\overline{W'}\rightarrow \R^{k\times k}$ is Lipschitz. 
    We define 
    \begin{align}\label{defal}
    \alpha\coloneqq\|\g \pi\|_{C^0(\overline{W'})}+1+\text{Lip}(\g \pi|_{\overline{W'}})\,\|\g (f\circ \phi_j^{-1})\|_{L^\infty(\phi_j(V_j\cap \overline U))}.
   \end{align}
    By \eqref{osdiff1}, there exists $\vae_0>0$ such that for any $\vae\in (0,\vae_0)$, we have on the one hand that $F_\vae\in C^\infty(\overline U,M)$ is well-defined with $F_\vae(\overline U)\subset W'$. On the other hand for any $z\in \phi_j(\supp(\xi_j)\cap\overline U)$, there holds $B_\vae(z^\vae)\subset \phi_j(V_j\cap \overline U)$, and there exists $z'\in B_\vae(z^\vae)$ satisfying $f\circ \phi_j^{-1}$ is differentiable at $z'$ with the following inequalities:
    \begin{itemize}
        \item oscillation of $\g(f\circ \phi_j^{-1})$:
        \begin{align}\label{oscdf}
            \big|\g (F_\vae\circ \phi_{j}^{-1})(z)-\g (f\circ\phi_j^{-1})(z')\big|\le (2\al)^{-1}C_j.
        \end{align}
        \item $C^0$ distance:
        \begin{align}
        \begin{aligned}\label{c0dis}
            & \big|(F_\vae\circ \phi_{j}^{-1})(z)-(f\circ\phi_j^{-1})(z')\big|\\[2mm]
            &\le \big|(F_\vae\circ \phi_{j}^{-1})(z)-(f\circ \phi_{j}^{-1})(z)\big|+\big|(f\circ \phi_{j}^{-1})(z)-(f\circ \phi_{j}^{-1})(z')\big|\\[2mm]
            &\le (2\al)^{-1} C_j.
            \end{aligned}
        \end{align}
    \end{itemize}
 The last inequality follows from the facts that $f\circ \phi_j^{-1}$ is Lipschitz continuous on $\phi_j(V_j\cap U)$ and $\|F_\vae-f\|_{C(\overline U)}\rightarrow 0$ as $\vae\rightarrow 0$. \\
 
 Given a point $z'$ where $f\circ \phi_j^{-1}$ is differentiable, it holds $d (f\circ \phi_{j}^{-1})_{z'}(v)\in T_{f\circ \phi_{j}^{-1}(z') }M$ for any $v\in \R^n$ since $f(U)\subset M$. In particular, we have $$\big(d\pi \circ d (f\circ \phi_{j}^{-1})\big)_{z'}(v)=  d (f\circ \phi_{j}^{-1})_{z'}(v).$$ 
 Hence, it holds for any $z\in \phi_j(\supp(\xi_j)\cap \overline U)$, $0<\vae<\vae_0$, $v\in \mathbb S^{n-1}$ that
 \begin{align}\begin{aligned}\label{dpi>}
     & \big|\big(d\pi\circ d (F_\vae\circ\phi_j^{-1})\big)_z(v)\big|\\[2mm]
        &\ge \  \big|d (f\circ \phi_{j}^{-1})_{z'}(v)\big|-\big|\big(d\pi\circ d (F_\vae\circ\phi_j^{-1})\big)_z(v)-\big(d \pi \circ d (f\circ \phi_{j}^{-1})\big)_{z'}(v)\big|.
        \end{aligned}
 \end{align}
 By the triangle inequality, we have for any $v\in\mathbb S^{n-1}$  \begin{align}\begin{aligned}\label{dpitri}
      &\big|\big(d\pi\circ d (F_\vae\circ\phi_j^{-1})\big)_z(v)-\big(d \pi \circ d (f\circ \phi_{j}^{-1})\big)_{z'}(v)\big|\\[3mm]
      &\le \big|\g \pi_{(F_\vae\circ \phi_{j}^{-1})(z)}\big(\g (F_\vae\circ \phi_{j}^{-1})_z-\g (f\circ\phi_j^{-1})_{z'}\big)(v)\big|\\[2.8mm]
      &\quad\quad+\big| \big(\g \pi_{(F_\vae\circ \phi_{j}^{-1})(z)}-\g \pi_{(f\circ\phi_j^{-1})(z')}\big) \g (f\circ \phi_{j}^{-1})_{z'}(v)\big|\\[3mm]
      &\le \big|\g \pi\big((F_\vae\circ \phi_{j}^{-1})(z)\big)\big| \ \big|\g (F_\vae\circ \phi_{j}^{-1})(z)-\g (f\circ\phi_j^{-1})(z')\big|\\[3.5mm]
        &\quad\quad+\text{Lip}(\g \pi|_{\overline{W'}})\  \big|(F_\vae\circ \phi_{j}^{-1})(z)-(f\circ\phi_j^{-1})(z')\big|\ \big|\g (f\circ\phi_j^{-1})(z')\big|.
 \end{aligned}
 \end{align}
 Combining \eqref{defal}--\eqref{c0dis} and \eqref{dpitri}, we obtain \begin{align}\begin{aligned}\label{<c_j}
     &\big|\big(d\pi\circ d (F_\vae\circ\phi_j^{-1})\big)_z(v)-\big(d \pi \circ d (f\circ \phi_{j}^{-1})\big)_{z'}(v)\big|\\[2.5mm]
 &\le (2\al)^{-1}C_j\Big(\|\g \pi\|_{C(\overline{W'})}+\text{Lip}(\g \pi|_{\overline{W'}})\,\|\g (f\circ \phi_j^{-1})\|_{L^\infty(\phi_j(V_j\cap \overline U))}\Big)\\[2.5mm]
        &\le 2^{-1}C_j.
        \end{aligned}
  \end{align}
   By \eqref{fimmcon}, \eqref{dpi>}, \eqref{<c_j} we finally conclude that for any $z\in \phi_j(\supp(\xi_j)\cap \overline U)$, $0<\vae<\vae_0$, $v\in \mathbb S^{n-1}$, it holds 
    \begin{align*}
    & \big|\big(d \pi\circ d (F_\vae\circ\phi_j^{-1})\big)_z(v)\big|
        \ge  2^{-1}C_j.
    \end{align*} 
    The estimate \eqref{seqbound1} is obtained since $\overline U\subset \bigcup_{j=1}^s \supp(\xi_j)$.
\end{proof}
We are now ready to prove Theorem \ref{th-approx}.
\begin{hproof4}
Let $\bP\in W^{2,2}_{\textup{imm}}(\Sigma,\R^m)$ be weakly conformal. Let $\ti g$ be a reference Riemannian metric on $\Sigma$. By Theorem \ref{th:Approx}, there exists a sequence of $C^\infty$ immersions $\bP_k$ satisfying $\bP_k\rightarrow \bP$ in $W^{2,2}(\Sigma,\R^m)$ and there exists a constant $\Lambda>0$ such that \be\label{dphi_k_bound}
\Lambda^{-1} |X|_{\ti g} \le | d(\vec{\Phi}_k)(X)|  \le \Lambda |X|_{\ti g},\qquad \text{for all }  X\in T\Sigma\,\text{ and all } \,k\in \N.\ee
Write $\bP_\infty\coloneqq\bP$, and we construct $\{\Psi_k\}$, $\{h_k\}$ as in Theorem \ref{thm-con-stru}. Since each $\bP_k$ is smooth, from the construction we see that each $\Psi_k$ is also smooth, and we define $\bP_k'\coloneqq\bP_k\circ \Psi_k$, $h_k'\coloneqq\Psi_k^*h_k$, $g_k'\coloneqq\bP_k'^*g_{\text{std}}$. Then we have $g_k'=\Psi_k^*(\bP_k^*g_{\text{std}})$ is conformal to $h_k'$ since $h_k$ is conformal to $\bP_k^*g_{\text{std}}$, and $h_k'\rightarrow h$ in $C^\infty(\Sigma,T^*\Sigma\ot T^*\Sigma)$. Moreover, by Theorem \ref{thm-con-stru} (iii) we have $\bP_k'\rightarrow \bP$ in $W^{2,2}(\Sigma,\R^m)$. Write $g'_k=e^{2\alpha_k}h_k'$. 

Now it remains to show that $\alpha_k\rightarrow \alpha$ in $C^0(\Sigma)$. Let $p\in \Sigma$. Since $h_k'\rightarrow h$ in $C^\infty(\Sigma)$, by a similar argument as in Theorem \ref{thmconchart} 
(replacing $\vec e_{i,k}$ by orthonormal frames on $T\Sigma$, and $d\vec e_{i,k}$ by covariant derivatives of these vector fields with respect to $h'_k$), there exists a neighborhood $U$ of $p$ and $C^\infty$ diffeomorphisms $\vp_k$ from $U$ to $\vp_k(U)\subset \R^2$ such that for any $k\in \N\cup\{\infty\}$, 
it holds that $\vp_k(p)=0$, $D^2\Subset\vp_k(U)$, and $\vp_k^*g_{\text{std}}$ is conformal to $h_k'$ (write $h_\infty'\coloneqq h$). Moreover, we have $\vp_k\rightarrow \vp_{\infty}$ in $C^\infty_{\textup{loc}}(U)$. 
Since $h_k'$ is conformal to $g_k'$, we also obtain that $\bP_k'\circ \vp_k^{-1}$ is conformal. 
We denote
$$ 
(\vp_k^{-1})^*h_k'=e^{2\tilde\la_k} (dx_1^2+dx_2^2),\quad (\vp_k^{-1})^*g_k'=(\bP_k'\circ \vp^{-1}_k)^*g_{\text{std}}=e^{2\la_k}(dx_1^2+dx_2^2).
$$
Then we have $\tilde\la_k\rightarrow\tilde\la_\infty$ in $C^\infty(\overline{D^2})$ since it holds $\vp_k^{-1}\rightarrow\vp_{\infty}^{-1}$ in $C^\infty(\overline{D^2})$ and $h_k'\rightarrow h_{\infty}'=h$ in $C^\infty(\Sigma)$. Since the sequence $\{\la_k\}$ is bounded in $L^\infty(D^2)$ by \eqref{dphi_k_bound} and Theorem \ref{thm-con-stru} (v), 
and $\bP_k'\circ \vp_k^{-1}\rightarrow \bP\circ \vp_{\infty}^{-1}$ in $W^{2,2}(D^2,\R^m)$, we have $\la_k\rightarrow \la_{\infty}$ in $W^{1,2}(D^2)$. Now by applying the argument in Corollary \ref{coconfactor} to $\{\bP_k'\circ \vp_k^{-1}\}$, we have $\la_k\rightarrow\la_\infty$ in $W^{2,1}_{\textup{loc}}(D^2)$. Since $\alpha_k=(\la_k-\tilde\la_k)\circ \vp_k$ on $U$, and $p$ is arbitrarily chosen on $\Sigma$, we conclude that $\alpha_k\rightarrow \alpha$ in $W^{2,1}(\Sigma)$, hence also in $C^0(\Sigma)$ (see \cite[Theorem 3.3.4 \& 3.3.10]{helein2002}).
\end{hproof4}

\section{Willmore surfaces}\label{sec-willsur}

Motivated by the generalization of Willmore surfaces in higher dimension, we provide in this section a new proof of the regularity of Willmore surfaces which does not involve the choice of conformal coordinates. In \Cref{sec:CL}, we write the Euler--Lagrange equation of Willmore surfaces in divergence form. In \Cref{wilreg}, we prove the regularity of Willmore surfaces.

\subsection{The Euler--Lagrange equation and conservation laws}\label{sec:CL}

In this section, we show that the Euler--Lagrange system of the Willmore functional can be written as a div-curl system. To do so, we need Noether theorem to rewrite it in divergence form. 
\begin{Th}[{Noether theorem \cite{Noe}}]
    Let $l\colon \R^m\times \mathbb R^{m\times n}\rightarrow \R$ for $l(z,p)$ being $C^1$ in $z$ and $C^2$ in $p$. Let $X$ be a tangent vector field on $\R^m$, $F(t,z)$ denote the flow of $X$ at time $t$ with $F(0,z)=z$. We say $X$ is an \textit{infinitesimal symmetry} of $l$ if 
\[l(u,\g u)=l(F(t,u), \g(F(t,u))),\qquad \text{for all }\, u\in C^1(B^n,\R^m).\]
    Let $u$ be a critical point of $L(u)\coloneq\int_{B^n}l(u,\g u) \,d\mathcal L^n$, i.e. for any $\omega\in C^\infty_c(B^n)$, it holds \[
        \lf. \frac{d}{dt}\, L(u+t\omega)\rg |_{t=0}=0.\]
    Then for any infinitesimal symmetry $X$ of $l$, we have \begin{align}\label{eq:noether}\mbox{div}\, \lf(\lf(\frac{\p l} {\p p}\cdot X\rg)\circ u\rg)=0.\end{align}
\end{Th}
In the same paper, Noether in fact considered higher-order Lagrangians of the form $$L=\int_{B^n} l(u,\g u,\dots,\g^k u)\,d\mathcal L^n \quad (k\ge 2),$$ where $l$ is a smooth function. While the classical conservation law \eqref{eq:noether} may not hold for $k\ge 2$, symmetries of the Lagrangian density $l$ still generate conservation laws in the form $\text{div }J=0$, where $J$ is a Noether current expressible in terms of $u$ and its derivatives.\\

Let $\Sigma$ be a $2$-dimensional closed smooth orientable manifold. The notions of Gauss map, second fundamental form and mean curvature associated to $\bP\in W^{2,2}_{\textup{imm}}(\Sigma, \R^3)$ have been defined in Notation \ref{not-met}.


\medskip


\begin{Dfi}
Let $\vec\Phi\in W^{2,2}_{\text{imm}}(\Sigma, \R^3)$. We define the \textit{Willmore energy}:
\[
W(\vec{\Phi})\coloneqq\int_{\Sigma} \big| \bH \big|^2 \,d\textup{vol}_{g}.
\]
The map $\vec\Phi$ is said to be a critical point for $W$ if for any $\vec w\in C^\infty(\Sigma,\R^3)$, there holds
\begin{equation*}
   \lf. \frac d{dt}\,W(\vec\Phi+t\,\vec w)\rg|_{t=0}=0.
\end{equation*}
Such a critical point is called a \textit{weak Willmore immersion}. Similarly, we can define weak Willmore immersions in $W^{2,2}_{\text{imm}}(D^2,\R^3)$.
\end{Dfi}
Throughout Sections \ref{sec-willsur} and \ref{sec:4d}, we use the Einstein summation convention, and we leave out the symbols $\otimes$ for sections of $\bigwedge \R^m\ot \bigwedge T^*\Sigma$  ($m\in \N^+$). For instance, we write $\p_i\bP \,dx^j=\p_i\bP\ot dx^j$.

\subsubsection*{The Noether current associated to translations.}
Now we compute the Euler-Lagrange equation satisfied by weak Willmore immersions. In fact, the divergence-form equation can be seen as a consequence of the pointwise invariance of $H^2\,\dvol_g$ by translations in the ambient space, as pointed out in~\cite{Ber} (see also~\cite{marque19}).
\begin{Th}[{\cite[Theorem 1.5]{Riv14}}]\label{thm:willequ}
A weak immersion $\vec \Phi\in W^{2,2}_{\text{imm}}(\Sigma, \R^3)$ is a weak Willmore immersion if and only if the following equation holds in $\mathcal D'\big(\Sigma,\R^3\ot \bwe^2T^*\Si\big)$:
\be\label{willmore eq} 
d *_g\lf(-2H\,d\vec n +d\vec H -H^2\,d\vec \Phi\rg)=0.
\ee
\end{Th}
\begin{proof}
Let $(x^1,x^2)$ be local coordinates for $\Sigma$ that maps onto $D^2$ and $\bP\in W^{2,2}_{\imm}(D^2,\R^3)$. Let $\vec w\in C^\infty(\ov{D^2},\R^3)$. We consider the variation 
\begin{align*}
    \bP_t=\bP+t\bw.
\end{align*}
Denote $g_t=g_{\bP_t}$, $\bn_t=\bn_{\bP_t}$, etc. Since $g^{ij}_t\,g_{jk,t}=\delta_{ik}$, taking derivative with respect to $t$, we have
\begin{gather}\begin{dcases}\label{dg^ij}
    \lf.\frac{d}{dt} \,g_{ij,t}\rg |_{t=0}=\lf.\frac{d}{dt}\, \p_{i} \bP_t\rg|_{t=0}\cdot\p_{j} \bP+\lf.\frac{d}{dt} \,\p_{j} \bP_t\rg|_{t=0}\cdot\p_{i}\bP=\p_i\vec w \cdot\p_{j}\bP+\p_{j}\vec w \cdot\p_{i}\bP,
    \\[2mm]
    \lf.\frac{d}{dt} \,g_t^{ij}\rg|_{t=0}=-g^{ik}g^{\ell j}\lf.\frac{d}{dt} \,g_{k\ell,t}\rg|_{t=0}=-g^{ik}g^{\ell j}(\p_{k}\vec w \cdot\p_{\ell}\bP+\p_{\ell}\vec w \cdot\p_{k}\bP).
    \end{dcases}
\end{gather}
By Jacobi's formula and writing $g=(g_{ij})$ for simplicity, it follows that 
\begin{align}\begin{aligned}
    \label{dvol}
    \lf.\frac{d}{dt} (\det g_t)^{\frac 12}\rg|_{t=0}&=\frac 12 (\det g)^{\frac 12} \,\text{Tr}\lf(g^{-1}\lf.\frac{d}{dt}\, g_t\rg|_{t=0}\rg)\\
    &=\frac12 (\det g)^{\frac 12} \,g^{ij}(\p_i\vec w \cdot\p_{j}\bP+\p_{j}\vec w \cdot\p_{i}\bP)\\[2mm]
    &=(\det g)^{\frac 12}\,g^{ij}\,\p_{j}\vec w \cdot\p_{i}\bP.
    \end{aligned}
\end{align}
Since $|\bn|=1$, it holds 
\begin{align*}
    \bn_t\cdot \frac{d}{dt} \,\bn_t  =\frac 12\, \frac{d}{dt}\, |\bn_t|^2=0.
\end{align*}
Thus the variation of the Gauss map is given by 
\begin{align*}
    \lf.\frac{d}{dt}\, \bn_t\rg|_{t=0}=g^{ij}\lf(\p_{i}\bP \cdot \lf.\frac{d}{dt}\,\bn_t\rg|_{t=0}\rg)\p_{j}\bP=-g^{ij}\lf(\bn \cdot \lf.\frac{d}{dt}\,\p_{i}\bP_t\rg|_{t=0}\rg)\,\p_{j}\bP=-g^{ij}\, (\bn \cdot \p_i\vec w)\,\p_{j}\bP.
\end{align*} 
We now compute the variation of the mean curvature:
\begin{align}\begin{aligned}
    \label{dH}
    2\lf.\frac{d}{dt} \,H_t\rg|_{t=0}&=g^{ij}\lf.\frac{d}{dt}\,  {\II}_{ij,t}\rg|_{t=0}+{\II}_{ij}\lf.\frac{d}{dt}\, g_t^{ij}\rg|_{t=0}\\[1mm]
    &=-g^{ij} \,\p_{j}\bP\cdot \lf.\frac{d}{dt}\,\p_{i}\bn_t\rg|_{t=0}-g^{ij}\, \p_{i}\bn\cdot \p_{j}\vec w+\II_{ij}\lf.\frac{d}{dt} \,g_t^{ij}\rg|_{t=0}.
    \end{aligned}
\end{align}
Thanks to the fact that $\Delta_g\bP\perp \p_l\bP$, we have
\begin{align}
\begin{aligned}\label{dpn}
    (\det g)^{\frac 12}\,g^{ij}\, \p_{j}\bP\cdot \lf.\frac{d}{dt}\,\p_{i}\bn_t\rg|_{t=0}&=-(\det g)^{\frac 12}\,g^{ij}\,\p_j\bP \cdot \p_i\big (g^{k\ell} (\bn \cdot \p_{k}\vec w)\,\p_{\ell}\bP\big)\\
    &=-\p_i\lf (\big((\det g)^{\frac 12}\,g^{ij}\,\p_j\bP\big) \cdot \big(g^{k\ell}(\bn \cdot \p_{k}\vec w)\, \p_{\ell}\bP\big)\rg)\\[1mm]
    &=-\p_i \lf ( (\det g)^{\frac 12}\,g^{ij}\, g_{j\ell}\,g^{k\ell} (\bn \cdot \p_{k}\vec w)\rg)\\[1mm]
    &=- \p_i \lf ((\det g)^{\frac 12} \,g^{ik}  (\bn \cdot \p_{k}\vec w)\rg)\\[1mm]
    &=- \p_i \lf ((\det g)^{\frac 12} \,g^{ij}  (\bn \cdot \p_{j}\vec w)\rg).
    \end{aligned}
\end{align}
Since the second fundamental form is symmetric and $\p_i\bn \cdot \bn=0$, we also have \begin{align}\begin{aligned}
    \label{dg}
 \II_{ij}\lf.\frac{d}{dt} \,g^{ij}_t\rg|_{t=0}&=   -2\,\II_{ij}\,g^{ik}g^{\ell j}\,\p_{k}\vec w \cdot\p_{\ell}\bP\\
    &=2\, \p_{k}\vec w\cdot \lf(g^{ik}(\p_i\bn \cdot \p_j\bP)\,g^{\ell j}\,\p_{\ell}\bP\rg)\\[2mm]
    &=2 \,\p_{k}\vec w \cdot\big( g^{ik}\, \p_i\bn\big)\\[2.5mm]
    &=2\,g^{ik}\,\p_{k}\vec w\cdot \p_i\bn\\[2.5mm]
    &=2\,g^{ij}\,\p_{j}\vec w\cdot \p_i\bn.
    \end{aligned}
\end{align} 
Combining \eqref{dvol}--\eqref{dg}, we obtain the pointwise a.e. variation
\begin{align}\label{ptvarwill}
\begin{aligned}
    &\lf.\frac d{dt}\big(H_t^2\mko(\det g_t)^{\frac 12}\big)\rg|_{t=0}\\
    &= 2H\lf.\frac{d}{dt}\, H_t\rg|_{t=0} (\det g)^{\frac 12}+H^2 \lf.\frac{d}{dt} (\det g_t)^{\frac 12}\rg|_{t=0}\\
    &=H\mkt \p_i \lf ((\det g)^{\frac 12} \mko g^{ij}  (\bn \cdot \p_{j}\vec w)\rg)+ H\mko (\det g)^{\frac 12}\,g^{ij}\mko \p_{i}\bn\cdot \p_{j}\vec w+H^2\mko(\det g)^{\frac 12}\mko g^{ij}\mko \p_{j}\vec w \cdot\p_{i}\bP.
    \end{aligned}
\end{align}
For any $1\le j\le 2$, $a\in L^\nf \cap W^{1,2}(D^2)$ and $f\in L^{2}(D^2)$, we have 
\begin{align*}
    \|a \mkt \p_{j} f  \|_{L^1+W^{-1,2}(D^2)}
    &\le \|\p_{j}(af)\|_{W^{-1,2}(D^2)}+\|f\mkt\p_{j} a \|_{L^1(D^2)}\\[0.3ex]
    &\le \|af\|_{L^{2}(D^2)}+\|\p_{j} a\|_{L^2(D^2)} \|f\|_{L^{2}(D^2)}\\[0.5ex]
    &\le \|a\|_{L^\nf \cap W^{1,2}(D^2)} \mko\|f\|_{L^{2}(D^2)}.
\end{align*}
Hence, there exists a universal constant $C>0$ such that for all $a\in L^\nf \cap W^{1,2}(D^2)$ and $T\in L^1+W^{-1,2}(D^2)$, it holds that
\begin{align}\label{prorL1W-12d}
    \|aT\|_{L^1+W^{-1,2}(D^2)}\le C\mko\|a\|_{L^\nf \cap W^{1,2}(D^2)} \mko \|T\|_{L^1+W^{-1,2}(D^2)}.
\end{align}
By~\eqref{ptvarwill} and~\eqref{prorL1W-12d}, we obtain in $W^{-1,2}+L^1(D^2)$ that 
\begin{align}\label{ptvarH^2det12}
\begin{aligned}
     &\lf.\frac d{dt}\big(H_t^2\mko(\det g_t)^{\frac 12}\big)\rg|_{t=0}\\
     &=-\p_i H\, (\det g)^{\frac 12} \,g^{ij}  (\bn \cdot \p_{j}\vec w)+ \p_i \lf (H\mko (\det g)^{\frac 12} \mko g^{ij}  (\bn \cdot \p_{j}\vec w)\rg)\\
     &\quad+\lf\langle H\,d\bn+H^2\,d\bP,d\vec w\rg\rangle_g (\det g)^{\frac 12}\\
     &=\lf\lan -\bn \,dH+H\,d\bn+H^2\,d\bP,d\vec w\rg\ran_g(\det g)^{\frac 12}+ \p_i \lf (H\mko (\det g)^{\frac 12} \mko g^{ij}  (\bn \cdot \p_{j}\vec w)\rg).
\end{aligned}
\end{align}
We define \begin{align}\label{defV2d}
    \vec V\coloneqq \bn \, dH-H\mko d\bn-H^2\mko d\bP=-2H\mko d\vec n +d\vec H -H^2\mko d\vec \Phi.
\end{align}
From~\eqref{ptvarH^2det12} it follows that in $W^{-2,2}+W^{-1,1}(D^2,\bwe^2T^* D^2)$,
\begin{align}\label{ptvarH^2dvol}
\begin{aligned}
    \lf.\frac d{dt}(H_t^2\,\dvol_{g_t})\rg|_{t=0}&= (*_g\,\vec V) \dwe d\vec w+\p_i \lf (H\mko (\det g)^{\frac 12} \mko g^{ij}  (\bn \cdot \p_{j}\vec w)\rg)dx^1\we dx^2\\
    &=\vec w\cdot d*_g\vec V- d*_g \big(\vec V\cdot \bw-\vec H\cdot d\bw\big).
\end{aligned}
\end{align}
Using integration by parts and denoting by $\big\lan\cdot,\cdot\big\ran$ the canonical pairing between $\mca D'(D^2,\R^3\ot \bwe^2 T^*D^2)$ and $C_c^\nf(D^2,\R^3)$, for $\vec w\in C_c^\nf(D^2,\R^3)$ we then obtain\footnote{We first prove~\eqref{varwillint} for immersions $\bP\in C^\nf(\ov{D^2},\R^3)$. The general case then follows from Theorem~\ref{th:Approx} together with~\eqref{ptvarwill}.}
 \begin{align}\label{varwillint}
    \lf.\frac{d}{dt}\, W(\bP_t)\rg|_{t=0}
    =\big\lan  d*_g\vec V,\vec w\big\ran.
\end{align}
Therefore, by a partition of unity argument, we have $\bP\in W^{2,2}_{\textup{imm}}(\Sigma,\R^3)$ is Willmore if and only if $d*_g \vec V=0$, i.e., the equation~\eqref{willmore eq} holds in $\mathcal D'\big(\Sigma,\R^3\ot \bwe^2T^*\Si\big)$. 
\end{proof}
To prove the regularity of weak Willmore immersions, by using local coordinates $(x^1,x^2)$, we can still assume $\bP\in W^{2,2}_{\textup{imm}}(D^2,\R^3)$. By Definition \ref{defweakimm}, there exists $\La\ge 1$ such that (see Notation~\ref{not-met})
\begin{align}\label{weakimmcon2d}
\Lambda^{-1}|v|_{\R^2}^2 \le | d\vec{\Phi}_x(v)|_{\R^3}^2  \le \Lambda | v|_{\R^2}^2,\qquad \text{for a.e. }x\in D^2  \text{ and all } v\in T_x D^2.
\end{align}
Defining $\vec V$ as in~\eqref{defV2d}, we have 
\begin{align*}
    *_g\,\vec V\in L^1+W^{-1,2}(D^2,\R^3\ot T^*D^2).
\end{align*}
By \eqref{willmore eq} and weak Poincar\'e lemma (see \cite[Chapter I, Theorem 2.24]{Dem}), there exists $\vec L\in \mathcal D'(D^2,\R^3)$ such that 
\begin{align}\label{defL}
    d\vec L= *_g\,\vec V.
\end{align}
In particular, it holds that $d\bL\in W^{-1,2}+L^1(D^2)$, and by Corollary \ref{corofbb}, we have $\bL \in L^2_{\text{loc}}(D^2)$. The quantity $\bL$ alone is not sufficient to deduce some additional regularity since $\bL$ does not possess additional derivatives a priori. Using the pointwise invariance of $H^2\,\dvol_g$ by dilations and rotations, we deduce some conservation laws related to $\bL$ and $\bP$.
\subsubsection*{The Noether currents associated to dilations and rotations.}
\begin{Th}[{\cite[Theorem 5.59]{Ri16}}]\label{thmcon1}
Let $\bP\in W^{2,2}_{\text{imm}}(D^2, \R^3)$ be a weak Willmore immersion. Then there exists $\bL\in L^2_{\text{loc}}(D^2,\R^3)$ satisfying \eqref{defL}\footnote{Up to additive constants, this $\bL$ is $-1/2$ of the one defined in \cite[Theorem 5.59]{Ri16}. There is a typo in \cite[Equation (5.214b)]{Ri16}: the $+$ sign should be changed to $-$.} and moreover, we have
\begin{gather}
\begin{gathered}\label{conser1}
d\left(\vec L \cdot d\bP \right)=0, \quad d\left(\bL\times d\bP+H\,d\bP\right)=0.
    \end{gathered}
\end{gather}
\end{Th}

\begin{proof}
We have already proved $\bL\in L^2_{\text{loc}}(D^2,\R^3)$, hence it remains to prove the relations \eqref{conser1}. We first consider the variation $\bP_t=(1+t)\bP$ with 
$$
    \bw=\lf.\frac d{dt}\mko \bP_t\rg|_{t=0}=\bP.
$$
Denote $H_t$ and $g_t$ as in the proof of Theorem~\ref{thm:willequ}. Since $H_t^2\,dvol_{g_t}=H^2\,\dvol_g$, combining~\eqref{ptvarH^2dvol} and \eqref{defL}, we obtain that
\begin{align*}
    0&=   \frac d{dt}(H_t^2\,\dvol_{g_t})\Big|_{t=0}\\
    &=- \mko d*_g \big(\vec V\cdot \bw-\vec H\cdot d\bw\big)\\[0.5ex]
    &=-\mko d(d\vec L\cdot \bw)\\[0.5ex]
    &=d\vec L\dwe d\bP\\[0.5ex]
    &=d(\vec L\cdot d\bP).
\end{align*}
To prove the second conservation law in~\eqref{conser1}, for $\vec a\in \R^3$ we consider the variation $\bP_t$ satisfying
\begin{align*}
  \begin{dcases} 
  \frac d{dt}\mko \bP_t=\vec a\times \bP_t, 
  \qquad t\in \R,\\[0.3ex]
  \bP_0=\bP.
  \end{dcases}
\end{align*} 
For this variation, the equations~\eqref{ptvarwill}, \eqref{ptvarH^2det12}, and~\eqref{ptvarH^2dvol} remain valid. Moreover, for each $t\in \R$, there exists $Q_t\in \text{SO}(3)$ depending only on $\vec a$ and $t$ such that $\bP_t=Q_t\circ \bP$ on $D^2$. Consequently, we have $H_t^2\,dvol_{g_t}=H^2\,\dvol_g$ for all $t\in \R$.  It follows that
\begin{align*}
    0&=   \frac d{dt}(H_t^2\,\dvol_{g_t})\Big|_{t=0}\\
    &=- \mkt d*_g \big(\vec V\cdot \bw-\vec H\cdot d\bw\big)\\[0.4ex]
    &=-\mkt d\big(d\bL\cdot (\vec a\times \bP)-\bH\cdot  (\vec a\times *_g \,d\bP)\mko \big)\\[0.4ex]
    &= \vec a\cdot d\big(d\vec L\times \bP-\bH \times*_g \,d\bP \big)\\[0.4ex]
    &=-\mkt \vec a \cdot d\big(\bL\times d\bP+\bH \times*_g \,d\bP\big).
\end{align*}
Since $\bn\times *_g\, d\bP=d\bP$, and $\vec a\in \R^3$ is arbitrary, we then obtain
\begin{align*}
    d\big(\bL\times d\bP+H\mko d\bP\big)=0.
\end{align*}
This completes the proof.
\end{proof}
By \eqref{conser1} and weak Poincar\'e lemma, there exist $ S\in W^{1,2}_{\textup{loc}}(D^2)$ and $\bR\in W^{1,2}_{\textup{loc}}(D^2,\R^3)$ such that 
\begin{align*}
    dS=\vec L \cdot d\bP, \qquad \text{ and }\qquad d\bR=\bL\times d\bP+H\mko d\bP.
\end{align*}

From these relations, we derive a system on $S,\bR$ which does not involve $\bL$ anymore.
    \begin{Th}[{\cite[Theorem 5.60]{Ri16}}]\label{thmcon2}
    Let $\bP\in W^{2,2}_{\text{imm}}(D^2, \R^3)$ be a weak Willmore immersion, and define $\bL$ as in Theorem \ref{thmcon1}. Then there exist $S\in W^{1,2}_{\text{loc}}(D^2)$ and $\bR\in W^{1,2}_{\text{loc}}(D^2,\R^3)$ satisfying 
    \begin{align}\label{con1}
    dS=\vec L \cdot d\bP, \qquad \text{ and }\qquad d\bR=\bL\times d\bP+H\,d\bP.
    \end{align}
    Moreover, we have 
    \begin{align}\label{con2}
        dS=-*_g d\bR\cdot\bn,\qquad \text{ and }\qquad d\bR=*_g\big(\bn \times  d\bR+ dS\   \bn\big).
    \end{align}
    \end{Th}
    
    \begin{proof}
First, we prove the second equation in \eqref{con2}. Since $d\bP\times \bn=*_g\,d\bP$, taking the cross product of $\bL\times d\bP$ with $\bn$, we obtain
\begin{align*}
    \bn \times (\bL\times d\bP)&=-\bL\times (d\bP\times \bn)-d\bP\times (\bn\times\bL)\\[1mm]
    &=-\bL\times *_g\, d\bP+(\bn\cdot d\bP)\, \bL- (\bL\cdot d\bP) \,\bn\\[1mm]
    &=-*_g(\bL\times  d\bP)-dS\ \bn.
\end{align*}
It follows that 
\begin{align*}
    \bn \times  d\bR+ dS\  \bn&= \bn \times(\bL\times d\bP+H\,d\bP) + dS\  \bn\\[1mm]
    &=-*_g(\bL\times  d\bP)+H\,\bn\times d\bP\\[1mm]
    &=-*_g(\bL\times  d\bP+H\,d\bP)\\
    &=-*_gd\bR.
\end{align*}
We then obtain the second equation in \eqref{con2}:
\begin{align*}
    d\bR=*_g\big(\bn \times  d\bR+ dS\  \bn).
\end{align*}
The first equation in \eqref{con2} follows immediately:
\begin{align*}
    -*_g d\bR\cdot\bn=(\bn \times  d\bR+ dS\  \bn)\cdot \bn=dS.
\end{align*}
\end{proof}

By differentiating the relations \eqref{con2}, we obtain a div-curl system in $\bR$ and $S$.

\begin{Th}[{\cite[Corollary 5.61]{Ri16}}]\label{thm-sys-srh}
     Let $\bP\in W^{2,2}_{\text{imm}}(D^2, \R^3)$ be a weak Willmore immersion, and define $\bL,S,\bR$ as before. Then we have (see Notation \ref{not-overset})
         \begin{subnumcases}{}
         \Delta_g S=*_g\,\big(d\,\bn \dwe d\bR\big), \label{lap1}\\[3mm]
         \Delta_g \bR=*_g\,\big(d\bR\times d\bn+dS\wedge d\bn\big),\label{lap2}\\[3mm]
          \Delta_g \bP=*_g\,\big(dS\wedge d\bP+d\bR   \overset{\times}\wedge d\bP \big).\label{lap3}
     \end{subnumcases}
\end{Th}

\begin{proof}
Equations \eqref{lap1} and \eqref{lap2} follow from the relations \eqref{con2}, by applying the operator $d^{*_g}$: 
\begin{gather*}
    \left\{
    \begin{aligned}
     & \Delta_g S=*_g \,d *_g dS=*_g \,d(d\bR\cdot \bn)=*_g \,(d\bn\dwe d\bR),\\[3mm]
     & \Delta_g \bR=*_g\, d *_g d\bR=-*_g d(\bn \times  d\bR+ dS\,  \bn)=*_g\,(d\bR\times d\bn+dS\wedge d\bn).
     \end{aligned}
     \right.
\end{gather*}
Thanks to \eqref{con1}, we have \begin{align*}
    d\bR   \overset{\times}\wedge d\bP&= (\bL\times d\bP+H\,d\bP)\overset{\times}\wedge d\bP\\
    &=d\bP\wedge (\bL \cdot d\bP)-(d\bP \dwe d\bP)\,\bL+H\,d\bP \overset{\times}\wedge d\bP\\[1mm]
    &=d\bP\wedge dS+2H (\p_1\bP \times \p_2\bP)\,dx^1\wedge dx^2\\[1mm]
    &=d\bP\wedge dS +2\bH *_g 1\\[1mm]
    &=-dS\wedge d\bP+ (\Delta_g \bP) *_g 1.
\end{align*}
This is \eqref{lap3}.
\end{proof}

\paragraph{The Noether current associated to inversions.}
Although $H^2\,\dvol_g$ is not pointwise invariant by inversions, this is the case of the Lagrangian $(H^2-K)\,\dvol_{g}$, see for instance~\cite{Chen}. In this section, we derive the associated conservation law \eqref{eq:inv}, which is equivalent to \eqref{lap3}.\\

Let $\bP\colon D^2\to \R^3$ be a Willmore immersion. As in \cite{marque19}, we consider the variation $\bP_t\coloneqq \vp_t\circ \bP$ where, for a given $\vec a\in \R^3$, the map $\vp_t$ is given by
\begin{align*}
    \vp_t(x)\coloneqq\frac{\frac{x}{|x|^2}+t\vec a}{\big|\frac{x}{|x|^2}+t\vec a\big|^2}.
\end{align*}
Then we have \begin{align*}
    \bw\coloneqq\lf.\frac d{dt} \,\bP_t\rg|_{t=0}=|\bP|^2 \vec a-2\mko(\bP\cdot \vec a)\mko\bP.
\end{align*}
Since $\bP$ is Willmore, by~\eqref{ptvarH^2dvol} and \eqref{defL}, we obtain that
\begin{align}\label{dH^2dvol}
    \lf.\frac d{dt}(H_t^2\,\dvol_{g_t})\rg|_{t=0}= \mkt d*_g \big(\vec H\cdot d\bw-\vec V\cdot \bw\big).
\end{align}
Denote by $\text{adj}(\II_t)$ the adjugate matrix of $(\II_{ij,t})$, i.e. the matrix satisfying the pointwise identity
\begin{align*}
    \text{adj}(\II_t)^{ij} \mko \II_{jk,t}=\det(\II_{ij,t})\mko\de^{i}_k.
\end{align*}
We have 
\begin{align}\label{dKdvol}
    \begin{aligned}
    \lf.\frac d{dt}\det(\II_{ij,t})\rg|_{t=0}&=\text{adj}(\II)^{ij} \lf.\frac d{dt}\II_{ij,t}\rg|_{t=0}\\
    &=\text{adj}(\II)^{ij}\Big(\p_{i}\bP\cdot \p_j\big(g^{k\ell}\, (\bn \cdot \p_k\vec w)\,\p_{\ell}\bP\big)- \p_{i}\bn\cdot \p_{j}\vec w\Big).
    \end{aligned}
\end{align}
Hence by \eqref{dvol},
\begin{align*}
        \lf.\frac d{dt}\big(K_t\det(g_t)^{\frac 12}\big) \rg|_{t=0}&=\lf.\frac d{dt}\Big(\det(\II_{ij,t})\,(\det g_t)^{-\frac 12}\Big) \rg|_{t=0} \\
        &=(\det g)^{-\frac 12}\,\text{adj}(\II)^{ij}\Big(\p_{i}\bP\cdot \p_j\big(g^{k\ell}\, (\bn \cdot \p_k\vec w)\,\p_{\ell}\bP\big)- \p_{i}\bn\cdot \p_{j}\vec w\Big)\\
        &\quad -\det(\II_{ij})\,(\det g)^{-\frac 12}g^{ij}\mko\p_{j}\vec w \cdot\p_{i}\bP.
\end{align*}
We compute 
\begin{align*}
       (\det g)^{-\frac 12}\,\text{adj}(\II)^{ij} \p_{i}\bn\cdot \p_{j}\vec w&=- (\det g)^{-\frac 12}\,\text{adj}(\II)^{ij}(g^{k\ell}\II_{ik}\p_\ell\bP)\cdot \p_{j}\vec w\\[1ex]
       &=- (\det g)^{-\frac 12}\det(\II_{ij})g^{ij}\, \p_i \bP\cdot \p_j\bw.
\end{align*}
Then \eqref{dKdvol} implies
\begin{align}\label{dKdet12}
      \lf.\frac d{dt}\big(K_t\det(g_t)^{\frac 12}\big)\rg|_{t=0}=(\det g)^{-\frac 12}\,\text{adj}(\II)^{ij}\p_{i}\bP\cdot \p_j\big(g^{k\ell}\, (\bn \cdot \p_k\vec w)\,\p_{\ell}\bP\big).
\end{align}
We will now show that 
\begin{align}\label{pjadjII0}
    \p_j \Big( (\det g)^{-\frac 12}\,\text{adj}(\II)^{ij}\p_{i}\bP\Big) \cdot \p_\ell \bP=0.
\end{align}
We define the map $A_g\colon T_x^*(D^2)\to  T_x^*(D^2)$ by 
\begin{align*}
    A_g(dx^i)\coloneqq (\det g)^{-1} \text{adj}(\II)^{ik} g_{kj}\, dx^j.
\end{align*}
Denoting by $g$, $g^{-1}$ the matrix $(g_{ij})$ and $(g^{ij})$ respectively, we have 
\begin{align*}
     (\det g)^{-1} \text{adj}(\II)^{ik} g_{kj}
     &= (\det g)^{-1} \big(\text{adj}(\II) g\big)^i_j\\
     &=\big(\text{adj}(g^{-1}\II)\big)^i_j\\
     &=\text{Tr}(g^{-1}\II)\de^i_j- \big(g^{-1}\II\big)^i_j\\
    & =2H\de_j^i -g^{ik}\II_{kj}.
\end{align*}
Then Lemma~\ref{lmuse} implies
\begin{align*}
    A_g(d\bP)&= (\det g)^{-1} \text{adj}(\II)^{ik} g_{kj}\mko \p_i\bP\, dx^j \\[0.3ex]
    &=2H \p_j\bP\,dx^j-g^{ik}\II_{kj}\mko \p_i\bP\, dx^j\\[0.3ex]
    &=2Hd\bP+d\bn\\[0.3ex]
    &=*_g\, d\bn \times \bn.
\end{align*}
Hence we have
\begin{align*}
    \p_j \Big( (\det g)^{-\frac 12}\,\text{adj}(\II)^{ij}\p_{i}\bP\Big)dx^1\we dx^2&=d*_g A_g(d\bP)\\
    &=d(\bn\times d\bn)\\[1ex]
    &=d\bn\times d\bn \\[1ex]
    &=K\bn \,\dvol_g.
\end{align*}
Thus, the identity \eqref{pjadjII0} is proved.
Since $\vp_t$ is conformal, we have by pointwise conformal invariance that $(H_t^2-K_t)\,\dvol_{g_t}=(H^2-K)\,\dvol_{g}$. Then by combining \eqref{dH^2dvol}--\eqref{pjadjII0}, we obtain that
\begin{align}\label{delH2*K}
\begin{aligned}
    0&=\lf.\frac d{dt}\Big((H_t^2-K_t)\,\dvol_{g_t}\Big)\rg|_{t=0}\\
    &=d*_g \Big(\bH \cdot d\bw -\bw\cdot \vec V\Big)-\p_j\Big((\det g)^{-\frac 12}\,\text{adj}(\II)^{ij} \bn\cdot \p_i\bw\Big)dx^1\we dx^2\\
    &=d*_g \Big(\bH \cdot d\bw + *_g \,d\bL\cdot \bw-A_g(\bn\cdot d\bw)\Big).
    \end{aligned}
\end{align}
By the expression $\bw=|\bP|^2 \vec a-2\mko(\bP\cdot \vec a)\mko\bP$, we have
\begin{align*}
    \bn\cdot d\bw&=2\mko \bn\cdot \Big((d\bP\cdot \bP)\mko \vec a-(d\bP\cdot \vec a)\mko \bP\Big)-2\mko (\bP\cdot \vec a) \mko \bn\cdot d\bP\\
    &=2\mko \vec a\cdot \Big((d\bP\cdot \bP) \bn-(\bn\cdot \bP)\mko d\bP)\Big)\\
    &=2\mko \vec a\cdot \Big((d\bP\times \bn)\times \bP\Big)\\
    &=2\mko \vec a\cdot (*_g\, d\bP\times \bP).
\end{align*}
We also have
\begin{align*}
    A_g(\bn\cdot d\bw)&=2\mko \vec a \cdot \Big(\big(A_g(d\bP)\times \bn\big)\times \bP \Big)\\
    &=2\mko \vec a \cdot \Big(\big(*_g d\bn\times \bn)\times \bn\big)\times \bP \Big)\\
    &=-2\mko \vec a\cdot (*_g\,d\bn \times \bP).
\end{align*}
The identity \eqref{delH2*K} then implies 
\begin{align*}
    \vec a \cdot d\Big(2\mko H\bP\times d\bP-2\mkt  d\bn \times \bP-d\bL|\bP|^2+2(d\bL\cdot \bP)\bP\Big)=0.
\end{align*}
Since $\vec a$ is arbitrarily chosen, we get 
\begin{align}\label{conlawinv}
    d\Big(2H\bP\times d\bP-2 \mkt d\bn \times \bP-d\bL|\bP|^2+2(d\bL\cdot \bP)\bP\Big)=0
\end{align}
Moreover, we have
\begin{align}\label{-dLbp2+2dL}
\begin{aligned}
    &-d\bL|\bP|^2+2(d\bL\cdot \bP)\bP\\[1ex]
    &=d\Big(2(\bL\cdot \bP)\bP-\bL|\bP|^2\Big)+2\bL\, (d\bP\cdot \bP)-2(\bL\cdot \bP)d\bP-2(\bL\cdot d\bP)\bP\\
    &=d\Big(2(\bL\cdot \bP)\bP-\bL|\bP|^2\Big)+2\bP\times (\vec L\times d\bP)-2dS\,\bP\\
    &=d\Big(2(\bL\cdot \bP)\bP-\bL|\bP|^2\Big)+2\bP\times d\bR-2H\bP\times d\bP-2dS\,\bP.
    \end{aligned}
\end{align}
Combining~\eqref{conlawinv}--\eqref{-dLbp2+2dL}, we then obtain
\begin{align*}
    2\,d\Big(-d\bn \times \bP+\bP\times d\vec R-dS\,\bP\Big)=0.
\end{align*}
It follows that
\begin{align}\label{eq:inv}
    d\Big(*_gd\bP-\bR\times d\bP-S\,d\bP \Big)=0.
\end{align}
This is the conservation law corresponding to inversions.
\subsection{Regularity of Willmore surfaces without conformal coordinates}\label{wilreg}

The goal of this section is to provide a new proof of the regularity of Willmore surfaces which does not involve the choice of conformal coordinates. The difficulty lies into proving the continuity of the Gauss map, see Theorem \ref{th:n_continuous}. The higher regularity then follows from standard bootstrap argument from the fact that $\bP$ can be represented locally as a graph, see Corollary \ref{cor:Willmore_graph}.
\begin{Dfi}[{Morrey spaces}]
    Let $1\le p<\infty$, $0\le \la\le n$, for a measurable function $f\colon U\rightarrow \R$, we say $f\in M^{p,\la}(U)$ if \begin{align*}
    \sup_{r>0,x_0\in U} r^{-\la}\int_{B_r(x_0)\cap U} |f(y)|^p\,dy<\infty,
\end{align*} 
and we define the Morrey norm \begin{align*}
    \|f\|_{M^{p,\la}(U)}\coloneqq\lf(\sup_{r>0,x_0\in U} r^{-\la}\int_{B_r(x_0)\cap U} |f(y)|^p\,dy\rg)^{\frac 1p}.
\end{align*}
\end{Dfi}

We start by proving some Morrey decay on the conserved quantities, which provides a Morrey decay on the mean curvature.

\begin{Th}\label{thmor1}
 Let $\bP\in W^{2,2}_{\text{imm}}(D^2, \R^3)$ be a weak Willmore immersion, and define $\bL,S,\bR$ as before. Let $\La$ be the constant in \eqref{weakimmcon2d} associated to $\bP$. Then there exists $\gamma=\gamma(\La) >0$ such that $\g S,\g \bR,\bH\in M^{2,\gamma}_{\textup{loc}}(D^2)$.
\end{Th}
\begin{proof}
By dilation and translation, it suffices to prove $\g S,\g \bR,\bH\in M^{2,\gamma}(D_{1/2}(0))$. By definition of a weak immersion, the coefficients $a^{ij}\coloneqq g^{ij}(\det g)^{1/2}$ satisfy the uniform ellipticity condition \eqref{elli}. Fix a small positive number $\vae_0$ which will be determined later. Since $ \bP\in W^{2,2}(D^2)$, we have 
$\g \bn\in L^2(D^2)$, hence there exists $r_0\in (0,\frac 14)$ depending on $\e_0$ and $\bP$ such that 
\begin{align}\label{smaint}
    \sup_{p\in D_{ 1/2}} \int _{D_{r_0}(p)} |\g \bn|^2 \,dx^1\we dx^2<\vae_0.
\end{align}
 Let $p\in D_{1/2}(0)$, $r\in (0,r_0)$. By Lemma \ref{lm-chali}, we can define $\Psi_S\in W^{1,2}_0(D_{r}(p))$ to be the unique solution of (see Section \ref{geono})
\begin{gather}\label{psiS}
 \begin{cases}  
 \p_i(a^{ij} \,\p_j\Psi_S)=\g ^{\perp}\bn \cdot \g \bR & \quad\mbox{in } D_{r}(p),\\[3mm]
 \Psi_S=0  &\quad \mbox{on } \p D_{r}(p).
\end{cases}\end{gather}
We also define $\vec{\Psi}_{\bR}\in W^{1,2}_0(D_{r}(p))$ to be the unique solution of
\begin{gather}\label{psiR}
 \begin{cases}  
 \p_i(a^{ij} \,\p_j\vec\Psi _{\bR})=\g ^{\perp}\bR\times \g \bn+ \g \bn\ \g ^{\perp}S & \quad \mbox{in } D_{r}(p),\\[3mm]
   \vec\Psi _{\bR}=0 &\quad \mbox{on } \p D_{r}(p).
    \end{cases}
    \end{gather}
Let $v_S\coloneqq S-\Psi_S$ and $\bv_{\bR}\coloneqq\bR-\vec\Psi _{\bR}$. We combine the relations \eqref{lap1} and \eqref{psiS} together with 
\begin{align*}
   (\det g)^{\frac 12}*_g \big(d\bn \dwe d\bR\big)=\g^{\perp}\bn \cdot \g \bR.
\end{align*} 
 We obtain
   \begin{align}\label{har1}
       \p_i(a^{ij}\, \p_jv_S)=0 \quad \mbox{  in } D_{r}(p).
   \end{align}
   As a result, it holds
   \begin{align}\begin{aligned}\label{intpar1}
       \int_{D_{r}(p)} |dS|_g^2\, d\textup{vol}_g&= \int_{D_{r}(p)} |dv_S|_g^2\, d\textup{vol}_g+ \int_{D_{r}(p)} |d\Psi_S|_g^2\, d\textup{vol}_g+2 \int_{D_{r}(p)} \langle dv_S,d\Psi_S\rangle_g\, d\textup{vol}_g\\
       &\ge \int_{D_{r}(p)} |dv_S|_g^2\, d\textup{vol}_g+2\int_{D_{r}(p)}  a^{ij}\,\p_jv_S\,\p_i \Psi_S\,dx^1\we dx^2\\
       &=\int_{D_{r}(p)} |dv_S|_g^2\, d\textup{vol}_g-2\int_{D_{r}(p)}  \p_i(a^{ij}\,\p_jv_S) \Psi_S\,dx^1\we dx^2\\
       &=\int_{D_{r}(p)} |dv_S|_g^2\, d\textup{vol}_g.
       \end{aligned}
   \end{align}
   Similarly, we have the system
   \begin{align}\label{har2}
        \p_i(a^{ij}\, \p_j\bv_{\bR})=0 \qquad \mbox{  in } D_{r}(p).
   \end{align}
   From this, we deduce the estimate
   \begin{gather}\label{intpar2}  
        \int_{D_{r}(p)} |d\bR|_g^2\, d\textup{vol}_g \ge \int_{D_{r}(p)} |d\bv_{\bR}|_g^2\, d\textup{vol}_g.
   \end{gather}
   Let $\rho\in (0,1)$ be determined later. It holds
   \begin{align*}
       &\int _{D_{\rho r}(p)} \lf(|dS|_g^2+|d\bR|_g^2\rg) d\textup{vol}_g \\
       &\le 2 \int _{D_{\rho r}(p)} \lf (|dv_S|_g^2+|d\bv_{\bR}|_g^2 \rg) d\textup{vol}_g+2 \int _{D_{ r}(p)} \lf(|d\Psi_S|_g^2+|d\vec{\Psi}_{\bR}|_g^2\rg) d\textup{vol}_g.
   \end{align*}
   By \eqref{immcon}, \eqref{psiS}, \eqref{psiR}, \eqref{har1}, \eqref{har2}, Lemmas \ref{derdecay} and \ref{lm-chali}, there exist positive constants $C_1,\,\alpha$ depending only on $\La$ such that
   \begin{align*}
       &\int _{D_{\rho r}(p)} \lf( |dS|_g^2+|d\bR|_g^2\rg) d\textup{vol}_g\\
       &\le C_1\lf( \rho^{2\alpha}\int_{D_{r}(p)} \lf(|dv_S|_g^2+|d\bv_{\bR}|_g^2\rg) d\textup{vol}_g+\|\g \bn\|^2_{L^2(D_{r}(p))}\int_{D_{ r}(p)} \lf(|dS|_g^2+|d\bR|_g^2\rg) d\textup{vol}_g\rg).
   \end{align*}
   From \eqref{smaint}, \eqref{intpar1} and \eqref{intpar2}, we deduce  \begin{align}\label{decayest}
   \begin{aligned}
       \int _{D_{\rho r}(p)} \lf(|dS|_g^2+|d\bR|_g^2\rg) d\textup{vol}_g \le C_1( \rho^{2\alpha}+\vae_0)\int_{D_{ r}(p)} \lf(|dS|_g^2+|d\bR|_g^2\rg)  d\textup{vol}_g.
       \end{aligned}
   \end{align}
   Now we choose 
   \begin{align}\label{detcon}
   \vae_0\coloneqq \frac{1}{4C_1}, \quad \rho\coloneqq\min\lf\{ \frac{1}{(4C_1)^{1/(2\alpha)} },  \frac{1}{2} \rg\}.
   \end{align} 
   For any $p\in D_{1/2}(0)$, $s\in (0,1)$, let $k\in \N_0$ such that $\rho^{k+1}\le s<\rho^{k}$. Thanks to \eqref{decayest} and \eqref{detcon} we have \begin{align*}
       \int_{D_{sr_0}(p)} \lf(|dS|_g^2+|d\bR|_g^2\rg) d\textup{vol}_g&\le  \int_{D_{\rho^kr_0}(p)}\lf( |dS|_g^2+|d\bR|_g^2\rg) d\textup{vol}_g\\
       &\le C(\La)\,2^{-k-1}\int_{D_{r_0+\frac 12}}\lf(|\g S|^2+|\g\bR|^2\rg)\\
       &\le C(\La)\,2^{-\log_\rho s}\int_{D_{\frac 34}}\lf(|\g S|^2+|\g\bR|^2\rg)\\
       &=C(\La)\, s^{-\log_\rho 2}\int_{D_{\frac 34}}\lf(|\g S|^2+|\g\bR|^2\rg).
   \end{align*}
Consequently, if we let $\gamma\coloneqq-\log_\rho 2\in (0,\infty)$ which depends on $\La$ only, then we have 
\begin{align*}
    \sup_{p\in D_{1/2}(0),r\le r_0}r^{-\gamma}\int_{D_r(p)}\lf( |\g S|^2+|\g\bR|^2\rg)<\infty.
\end{align*}
By \eqref{lap3}, we then have \begin{align*}
     \sup_{p\in D_{1/2}(0),r\le r_0}r^{-\gamma}\int_{D_r(p)} H^2<\infty.
\end{align*}
\end{proof}

We now prove that the Gauss map is continuous.

\begin{Th}\label{th:n_continuous}
 Let $\bP\in W^{2,2}_{\textup{imm}}(D^2, \R^3)$ be a weak Willmore immersion satisfying \eqref{weakimmcon2d}. Then there exists $\tau=\tau(\La) >0$ such that $\g \bn \in M^{2,\tau}_{\textup{loc}}(D^2)$. In particular, there exists $\al=\al(\La)\in (0,1)$ such that $\bn\in C^{0,\al}_{\textup{loc}}(D^2)$. 
\end{Th}
\begin{proof}
     Let $a^{ij}\coloneqq g^{ij}(\det g)^{1/2}$. By Lemma \ref{lmuse}, we have 
     \[
     -2H\,d\bP=d\bn +\bn \times *_g\, d\bn.
     \]
   We deduce that the following system holds in $\D'(D^2)$:
   \begin{align}\label{dnequ}
   \begin{aligned}
        \p_i(a^{ij}\,\p_j\bn)&=(\det g)^{\frac 12} *_gd*_gd\bn \\[1mm]
         & =(\det g)^{\frac 12} *_g d(\bn\times d\bn-2H*_gd\bP)\\[1mm]
        &=2\,\p_1\bn \times \p_2\bn-2\p_i(a^{ij}\,H\,\p_j\bP).
    \end{aligned}
    \end{align}
  Fix a small positive number $\vae_0$ which will be determined later. As in the proof of Theorem \ref{thmor1}, there exists $r_0\in (0,\frac 14)$ depending on $\e_0$ and $\bP$ such that 
  \begin{align*}
    \sup_{p\in D_{ 1/2}} \int _{D_{r_0}(p)} |\g \bn|^2 \,dx^1\we dx^2<\vae_0.
\end{align*}
Let $p\in D_{1/2}(0)$, $r\in (0,r_0)$. By Lemma \ref{lm-chali}, we can define $\bPS_1\in W^{1,2}_0(D_r(p))$ to be the unique solution of 
\begin{align}\label{eq:systemPsi1}
 \begin{cases}  \p_i(a^{ij}\, \p_j\bPS_1)=2\,\p_1\bn\times\p_2\bn \quad &\mbox{in } D_{r}(p),\\[3mm]
   \bPS_1=0 \quad &\mbox{on } \p D_{r}(p).
    \end{cases}
    \end{align}
    We have the following a priori estimate:
    \[
    \|\g \bPS_1\|^2_{L^2(D_r(p))}\le C(\La)\, \|\g \bn\|^4_{L^2(D_r(p))}\le C(\La)\, \e_0\, \|\g \bn\|^2_{L^2(D_r(p))}.
    \]
    Furthermore, by Riesz representation theorem for Hilbert spaces, there exists a unique $\bPS_2\in W^{1,2}_0(D_r(p))$ satisfying the estimate $\|\g \bPS_2\|_{L^2(D_r(p))}\le C(\La)\, \|H\|_{L^2(D_r(p))}$ and the following equation holds 
    \begin{align}\label{eq:systemPsi2}
        \p_i(a^{ij}\, \p_j\bPS_2)=-2\,\p_i(a^{ij}\,H\,\p_j\bP).
    \end{align}
    Let $\bPS\coloneqq\bPS_1+\bPS_2\in W^{1,2}_0(D_r(p))$ and $\bv\coloneqq\bn-\bPS$. By Theorem \ref{thmor1}, there exist two constants $C=C(\bP)$ and $\gamma=\gamma(\La)>0$ such that
    \begin{align*}
        \|\g\bPS\|^2_{L^2(D_r(p))}&\le C(\La)\big(\e_0\, \|\g\bn\|^2_{L^2(D_r(p))}+\|H\|_{L^2(D_r(p))}^2\big)\\[2mm]
    &\le C(\bP)\lf( \e_0\, \|\g\bn\|^2_{L^2(D_r(p))}+ r^\gamma\rg). 
    \end{align*}
    Thanks to \eqref{eq:systemPsi1} and \eqref{eq:systemPsi2}, we also have 
    \[
    \p_i(a^{ij}\,\p_j\bv)=0.
    \]
    It follows that, as in the proof of Theorem \ref{thmor1}, it holds
    \[ 
    \int_{D_{r}(p)} |d\bn|_g^2\, d\textup{vol}_g \ge \int_{D_{r}(p)} |d\bv|_g^2\, d\textup{vol}_g.
    \]
    Let $\rho\in (0,1)$ be determined later. There exist constants $\al=\al(\La)>0$ defined as in Lemma \ref{derdecay}, $C_2=C_2(\La)>0$, $C_3=C_3(\bP)>0$ such that
    \begin{align}\label{eqdecayest}
   \begin{aligned}
       \int _{D_{\rho r}(p)}|d\bn|_g^2\,d\textup{vol}_g&\le 2 \int _{D_{\rho r}(p)}|d\bv|_g^2\,d\textup{vol}_g+2 \int _{D_{ r}(p)}|d\vec{\Psi}|_g^2\,d\textup{vol}_g\\
       &\le C_2\lf( \rho^{2\alpha}\int_{D_{r}(p)} |d\bv|_g^2\,d\textup{vol}_g+\e_0\, \|\g\bn\|^2_{L^2(D_r(p))}+C(\bP)\, r^\gamma\rg)\\
       &\le C_2\,( \rho^{2\alpha}+\vae_0)\int_{D_{ r}(p)}|d\bn|_g^2\,d\textup{vol}_g+C_3\,r^\gamma.
       \end{aligned}
   \end{align}
   Now we choose 
   \begin{align*}
   \vae_0\coloneqq \frac{1}{4C_2}, \quad \rho\coloneqq\min\lf\{ \frac{1}{(4C_2)^{1/(2\alpha)} }, \frac{1}{2^{1/\gamma}}, \frac{1}{2} \rg\}.
   \end{align*} 
   For any $p\in D_{1/2}(0)$, $s\in (0,1)$, let $k\in \N_0$ such that $\rho^{k+1}\le s<\rho^{k}$. Then by \eqref{eqdecayest} it holds
   \begin{align*}
       \int_{D_{sr_0}(p)} |d\bn|_g^2\,d\textup{vol}_g&\le  \int_{D_{\rho^k r_0}(p)} |d\bn|_g^2\,d\textup{vol}_g\\
       &\le C_3\,r_0^\gamma\lf(\rho^{\gamma(k-1)}+\frac 12\,\rho^{\gamma(k-2)}+\cdots + \frac{1}{2^{k-1}} \rg) +C(\La)\,2^{-k-1}\int_{D_{r_0+\frac 12}}|\g \vec n|^2.
    \end{align*}
    Since $\rho\leq \frac{1}{2}$, we obtain
    \begin{align*}
       \int_{D_{sr_0}(p)} |d\bn|_g^2\,d\textup{vol}_g &\le C(\bP)\bigg (2^{-\log_\rho s}\int_{D_{\frac 34}}|\g \vec n|^2+ \frac{ r_0^\gamma k}{2^{k-1} }\bigg)\\
       &\le C(\bP) \bigg(s^{-\log_\rho 2}\int_{D_{\frac 34}}|\g \vec n|^2+r_0^\gamma \,\lf(\frac 23 \rg)^{k-1}\bigg)\\[1.5mm]
       &\le C(\bP)\, s^{\log_{\rho} \frac 23}.
   \end{align*}
Therefore, if we let $\tau\coloneqq\log_\rho \frac 23\in (0,\infty)$ which depends on $\La$ only, then we have \begin{align*}
    \sup_{p\in D_{1/2}(0),r\le r_0}r^{-\tau}\int_{D_r(p)} |\g \bn|^2<\infty.
\end{align*}
The H\"older continuity of $\bn$ then follows from standard knowledge of the Morrey--Campanato spaces (see for instance \cite[Theorem~3.5.2]{Morrey08} and \cite{Adams15}).
\end{proof}

If the Gauss map is continuous, then one can write the image of $\bP$ as a graph.

\begin{Lm}\label{lm:loc-graph}
    Let $\bP\in W^{2,2}_{\textup{imm}}(D^2, \R^m)$ ($m\ge 3$). Assume the Gauss map $\bn$ is continuous, then there exist an open neighborhood $U\subset D^2$ of $0$, a $W^{2,2}\cap W^{1,\infty}$ homeomorphism $\Psi\colon U\rightarrow\Psi(U)\subset\R^2$, and $f\in W^{2,2}(\Psi(U),\R^{m-2})$ such that upon rotating and relabeling the coordinate axes if necessary, we have  $$\bP\circ \Psi^{-1}(x_1,x_2)=(x_1,x_2,f(x_1,x_2)),\qquad \text{for all }(x_1,x_2)\in  \Psi(U).$$ 
    Thus $\bP$ is locally the graph of a $W^{2,2}$ map.
\end{Lm}
\begin{proof}
    Upon rotating and relabeling the coordinate axes, we assume $\bn(0)=e_3\wedge\dots\wedge e_m$, where $(e_1,\dots,e_m)$ denotes the canonical oriented basis of $\R^m$. By~\eqref{weakimmcon2d}, we have 
    \be\label{lowbou} 
    \La^{-2}\le\det(g_{ij})= |\p_1\bP\wedge \p_2\bP|^2\le \La^{2}. 
    \ee
   Since $\bn$ is continuous, there exists a neighborhood $U_1$ of $0$ such that $|\bn-\bn\mko(0)|<\frac 12$ on $U_1$. By definition of the Gauss map, we have the following inequality on $U_1$:
    \[
    \lf\lan\frac{\p_1\bP\wedge \p_2\bP}{|\p_1\bP\wedge \p_2\bP|}, e_1\wedge e_2\rg\ran =1+\lf\langle \star\,\bn-\star\,\bn\mko(0) ,e_1\wedge e_2\rg\rangle >\frac 12.
    \]
    Writing $\bP=(\Phi_1,\dots,\Phi_m)$, on the other hand we obtain
    \[ 
    \lf\lan\frac{\p_1\bP\wedge \p_2\bP}{|\p_1\bP\wedge \p_2\bP|}, e_1\wedge e_2\rg\ran=|\p_1\bP\wedge \p_2\bP|^{-1}\det (\p_i\Phi_j)_{1\le i,j\le 2}\le \La \det (\p_i\Phi_j)_{1\le i,j\le 2}.
    \]
    It follows that $\det(\p_i\Phi_j)_{1\le i,j\le 2}>(2\La)^{-1} $ on $U_1$. Then by Lemma \ref{thinv}, there exists an open neighborhood $U\subset U_1$ of $0$ such that $\Psi\coloneqq(\Phi_1,\Phi_2)$ is injective on $U$ with $\Psi^{-1}\in W^{1,\infty}\cap W^{2,2}(\Psi(U),\R^2)$. We obtain that for any $(x_1,x_2)\in \Psi(U)$, there holds $$\bP\circ \Psi^{-1}(x_1,x_2)=(x_1,x_2,\Phi_3\circ \Psi^{-1}(x_1,x_2),\dots,\Phi_m\circ \Psi^{-1}(x_1,x_2)).$$ 
\end{proof}

Consequently, a weak Willmore immersion is locally a graph, and we can prove smoothness under the graph coordinates.

\begin{Co}\label{cor:Willmore_graph}
 Let $\Sigma$ be a $2$-dimensional closed smooth manifold, $\bP\in W^{2,2}_{\textup{imm}}(\Sigma,\R^3)$. 
 Then $\bP$ is a graph near $\bP(p)$ for any $p\in\Sigma$, and $\bP\in C^\infty$ under the graph coordinates.
\end{Co}
\begin{proof}
   Let $p\in \Sigma$. Then by Theorem \ref{th:n_continuous} and Lemma \ref{lm:loc-graph}, upon affine transformation on $\R^m$ if necessary, we can without loss of generality assume $p=0\in \R^2$ and $\bP(x_1,x_2)=(x_1,x_2,f(x_1,x_2))$ on $D^2$ for some function $f\in W^{2,2}(D^2, \R^{m-2})$. Define $S,\vec R$ as before. Recall that by \eqref{dnequ} and Theorem \ref{thm-sys-srh}, we have the following system of elliptic equations 
   \begin{subnumcases}{}
         \Delta_g S=*_g\big(d\bn \dwe d\bR\big),\label{eq1} \\[2mm]
         \Delta_g \bR=*_g\big(d\bR\times d\bn+dS\wedge d\bn\big),\label{eq2}\\[2mm]
          2H=*_g\big(dS\wedge d\bP+d\bR   \overset{\times}\wedge d\bP \big),\label{eq3}\\[2mm]
          d*_gd\bn=-d\big(2H*_g d\bP\big)+d\bn\overset{\times}{\wedge} d\bn.  \label{eq4}
    \end{subnumcases} 
          Moreover, by Theorems \ref{thmor1} and \ref{th:n_continuous}, there exist $\tau\in (0,1)$ and $\alpha\in (0,1)$ such that 
    \begin{gather}\label{mor}
    H,\g\vec R, \g S,\g \vec n\in M^{2,\tau}_{\textup{loc}}(D^2),\quad \bn\in C_{\textup{loc}}^{0,\alpha}(D^2).
    \end{gather} 
        Using the expression of $\bP$ by $f$, we have\label{n=df}
              \be\bn=\frac{(-\p_1f,-\p_2 f,1)}{\sqrt{1+|\g f|^2}}.\ee 
        Hence by direct computation we see that $\g f, \g \bP\in C^{0,\alpha}_{\textup{loc}}(D^2)$. In particular, $g=g_{\bP}\in C_{\textup{loc}}^{0,\alpha}(D^2)$. 

         By \cite[Eq.~(1.5)]{Mingi11} and standard estimates on the Riesz potential \cite[Proposition~3.2(ii)]{Adams75}, we obtain $M^{1,\tau}(D^2)\hookrightarrow W^{-1,p}(D^2)$ for any $p\in(2,\frac{2-\tau}{1-\tau})$. Now since $*_g\big(d\bn \dwe d\bR\big) \in M^{1,\tau}_{\loc}(D^2)$, 
    we have $*_g\big(d\bn \dwe d\bR\big) \in W^{-1,p}_{\loc}(D^2)$ for some $p>2$. Applying \cite[Theorem~3.1]{Byun05} to the equation~\eqref{eq1} then yields $\g S\in L^p_{\textup{loc}}(D^2)$. Similarly, we have $\g \bR\in L^p_{\textup{loc}}(D^2)$, and hence by~\eqref{eq3}, $H\in L^p_{\textup{loc}}(D^2)$. Moreover, since $d(2H*_g d\bP)\in W^{-1,p}_{\loc}(D^2)$, combining~\eqref{eq4} and~\eqref{mor} with a similar argument as above implies $\g \bn\in L^p_{\textup{loc}}(D^2)$. It follows that $$d\bn \dwe d\bR\in \begin{cases}
         L^{p/2}_{\textup{loc}}(D^2)\subset W^{-1,\frac{2p}{4-p}}_{\textup{loc}}(D^2) \quad& \text{if } p<4,\\[3mm]   W^{-1,q}_{\textup{loc}}(D^2)\text{ for any $4<q<\infty$}\quad &\text{otherwise}.
          \end{cases}$$
          Hence by \eqref{eq1} and \cite[Theorem 3.1]{Byun05} again, we have \begin{align*}
       \g S\in \begin{cases}
              L_{\textup{loc}}^{(2/p-1/2)^{-1}}(D^2) \quad& \text{if }p<4,\\[3mm]
            L^q_{\textup{loc}}(D^2) 
            \quad& \text{otherwise.}
        \end{cases}
    \end{align*}
    The same also holds for $\g \bR,\,\g \bn,\, H$, and after finitely many iterations, we get $S,\bR,\bn\in W^{1,q}_{\textup{loc}}(D^2)$ for any $q<\infty$. By computing $\g \bn$ in the expression \eqref{n=df}, we get $f\in W^{2,q}_{\textup{loc}}(D^2)$ hence $\bP\in W^{2,q}_{\textup{loc}}(D^2)$. 
    
    Now return to the equations \eqref{eq1}--\eqref{eq4}. The right-hand side of each equation is in $L^q_{\textup{loc}}(D^2)$ for any $q<\infty$, and since $S,\bR,\bn,g\in W^{1,q}_\loc (D^2)$, by \cite[Theorem 4.1]{Chiarenza93} and the proof of \cite[Theorem 8.8]{Gilbarg01}, we obtain $S,\bR,\bn\in W^{2,q}_{\textup{loc}}(D^2)$, and then $\bP\in W^{3,q}_{\textup{loc}}(D^2)$... We finally get $\bP\in C^\infty$ under the graph coordinates.
\end{proof}

\section{Generalized Willmore functionals for 4-dimensional submanifolds}\label{sec:4d}

In this part we implement the approach introduced by the third author in \cite{Riv08} for the Willmore energy in order to deduce conserved quantities for critical points to scaling invariant Lagrangians in 4 dimension. The motivation being that, with the help of this conservation laws, one can hope to develop a strategy to devise a proof of the smoothness of weak critical points.
The interpretations by Bernard \cite{Ber} of the conservation laws found in \cite{Riv08} as being the Noether currents associated to the generators of the invariance group has enlightened the field of the analysis of conformally invariant Lagrangians of immersions. The computations of these conservation laws in dimension 4 have been first realised by Bernard in \cite{Bernard25}. Here, we present these computations for the simpler Lagrangien $\int |d H|_g^2\,\dvol_g$. The proof of the regularity of the critical points has been first given in a joint work of the three authors together with Bernard in \cite{bernard2025}.

\begin{Dfi}
Let $\vec\Phi\in W^{3,2}_{\text{imm}}(B^4, \R^5)$. We define the energy 
\[
E(\vec{\Phi})\coloneqq\frac{1}{2}\int_{B^4} | d H |_g^2 \,d\textup{vol}_{g}.
\]
The map $\vec\Phi$ is said to be a weak critical point of $E$ if for any $\vec w\in C^\infty_c(B^4,\R^5)$, there holds
\begin{equation*}
   \lf. \frac d{dt}\,E(\vec\Phi+t\,\vec w)\rg|_{t=0}=0.
\end{equation*}
\end{Dfi}

\subsubsection*{The Noether current associated to translations.}
We first compute the Euler--Lagrange equation satisfied by weak critical points of $E$. Similar to the 2-dimensional case, this divergence-form equation is also a consequence of the pointwise invariance of $|d H|_g^2\,\dvol_g$ by translations in the ambient space, see~\cite[Section~A.2]{Bernard25}.
\begin{Lm}
\label{lm-critic-W^{3,2}} An immersion $\vec{\Phi}\in W^{3,2}_{imm}(B^4, \R^5)$ is a weak critical point of $E$ if and only if the following equation holds in $\mathcal D'(B^4,\R^5\ot \bwe^4T^*B^4)$:
\be
\label{Euler-1}
d*_g\Big(2\,\langle d\vec{\Phi},dH\rangle_g \,dH -\frac 12\,d(\bn \, \Delta_g H)+ \Delta_g H\,d\bn-|dH|_g^2\,d\bP\Big)=0.
\ee
\end{Lm}
\begin{proof}
   Let $\bP\in W^{3,2}_{\imm}(B^4,\R^5)$ and $\vec w\in C^\infty(\ov{B^4},\R^5)$. We consider the variation 
\begin{align*}
    \bP_t=\bP+t\bw.
\end{align*}
    Similar to Section \ref{sec:CL}, we denote $g_t\coloneqq g_{\vec{\Phi}_t}$, $\det g_t\coloneqq\det (g_{ij,\bP_t})$, etc. The same computation as in~\eqref{dg^ij}--\eqref{dg} implies 
\begin{align}\label{ptvarH4d}\begin{dcases}
    \lf.\frac{d}{dt} \,g_t^{ij}\rg|_{t=0}=-g^{ik}g^{sj}(\p_{k}\vec w \cdot\p_{s}\bP+\p_{s}\vec w \cdot\p_{k}\bP),\\[1mm]
    \lf.\frac{d}{dt} (\det g_t)^{\frac 12}\rg|_{t=0}=(\det g)^{\frac 12}\,g^{ij}\,\p_{j}\vec w \cdot\p_{i}\bP,\\[1mm]
     4\lf.\frac{d}{dt} \,H_t\rg|_{t=0}= (\det g)^{-\frac 12}\, \p_i \lf ((\det g)^{\frac 12} \,g^{ij}  (\bn \cdot \p_{j}\vec w)\rg)+g^{ij}\, \p_{i}\bn\cdot \p_{j}\vec w.
    \end{dcases}
\end{align}
 We obtain the pointwise a.e. variation
\begin{align}\label{ptvar4d}
\begin{aligned}
    &\frac d{dt}\big(|d H_t|_{g_t}^2\,\dvol_{g_t}\big)\Big|_{t=0}\\[0.5ex]
    &= \lf(\frac{d}{dt}\, g_t^{ij}\Big|_{t=0}\p_i H\mkt\p_j H+ 2\,\bigg\langle dH,d\Big(\frac{d}{dt}\mkt H_t\Big|_{t=0}\Big) \bigg\rangle_{\!g} \rg)\dvol_g+ |dH|_g^2\mkt \frac{d}{dt}\,\dvol_{g_t}\Big|_{t=0}\\
    &=-\mko 2\mkt \langle d\vec{\Phi},dH\rangle_g\cdot
    \lf\langle dH,d\vec{w}\rg\rangle_g \dvol_g-2\mkt (*_g \,dH)\we d\Big(\frac{d}{dt}\mkt H_t\Big|_{t=0}\Big)\\[1.3ex]
    &\quad +|dH|_g^2\, g^{ij}\mko (\p_{j}\vec w \cdot\p_{i}\bP)\,\dvol_g.
    \end{aligned}
\end{align}
 Now for $a\in L^\nf\cap W^{2,2}(B^4)$ and $f\in L^{2}(B^4)$, we write 
\begin{align*}
         a \, \p_i\p_j f=\p_i\p_j(af)-\p_i(f\,\p_j\mko a  ) -\p_j(f\,\p_i\mko a )+f\,\p_i\p_j\mko a.
\end{align*}
We have the following estimates:
\begin{align*}\begin{dcases}
   \big \|\p_i\p_j (af)\big\|_{W^{-2,2}(B^4)}\le \|af \|_{L^{2}(B^4)}\le \|a\|_{L^\nf(B^4)}\mko\|f\|_{L^{2}(B^4)},\\[1ex]
    \big\|\p_i(f\,\p_j\mko a  )\big\|_{W^{-1,\frac 43}(B^4)}\le \|f\,\p_j\mko a  \|_{L^{\frac 43}(B^4)}\le  \|\g a\|_{L^4(B^4)}\mko \|f\|_{L^{2}(B^4)},\\[0.7ex]
    \|f\,\p_i\p_j\mko a\|_{L^1(B^4)}\le \|f\|_{L^2(B^4)}\|\g^2 a\|_{L^2(B^4)}.
    \end{dcases}
\end{align*}
Hence there exists a universal constant $C>0$ such that for all $a\in L^\nf\cap W^{2,2}(B^4)$ and $T\in W^{-2,2}+L^1(B^4)$, it holds that
\begin{align}\label{W-22L1estaT}
    \|aT\|_{W^{-2,2}+L^1(B^4)}\le C\mko\|a\|_{L^\nf\cap W^{2,2}(B^4)}\|T\|_{W^{-2,2}+L^1(B^4)}.
\end{align}
 Combining~\eqref{ptvarH4d}--\eqref{W-22L1estaT}, we obtain in $W^{-2,2}+L^1(B^4,\bwe^4T^*B^4)$ that
\begin{align}\label{ptvardH^2det12}
\begin{aligned}
     &\frac d{dt}\big(|d H_t|_{g_t}^2\,\dvol_{g_t}\big)\Big|_{t=0}\\
     &= -2\bigg(\langle d\vec{\Phi},dH\rangle_g\cdot
    \lf\langle dH,d\vec{w}\rg\rangle_g + \Delta_g H\mko \frac{d}{dt}\mko H_t\Big|_{t=0}\bigg)\,\dvol_g+ 2\mkt d\mko\bigg(\frac{d}{dt} H_t \Big|_{t=0} *_gdH\bigg)\\[0.3ex]
     &\quad +|dH|_g^2\, g^{ij}\mko (\p_{j}\vec w \cdot\p_{i}\bP)\,\dvol_g\\[0.3ex]
     &=\Big(-2\mkt \langle d\vec{\Phi},dH\rangle_g\cdot
    \langle dH,d\vec{w}\rangle_g + \frac 12\mko \p_i( \Delta_g H)\mko g^{ij}\mkt \bn\cdot \p_j \vec w-\frac 12\mko \Delta_g H\mkt  g^{ij}\mkt \p_i\bn\cdot\p_j\vec w\Big)\,\dvol_g\\[0.3ex]
    &\quad +d*_g \Big( -\frac 12\lap_g H \,  \bn \cdot d \bw+2\,\frac{d}{dt} H_t \Big|_{t=0} \,dH\Big)+|dH|_g^2\, g^{ij}\mko (\p_{j}\vec w \cdot\p_{i}\bP)\,\dvol_g\\[0.3ex]
    &=\Big\langle -2\,\langle d\vec{\Phi},dH\rangle_g \,dH+ \frac 12\mko\bn \,d(\Delta_g H)- \frac 12\mko \Delta_g H\mko d\bn+|dH|_g^2\,d\bP,d\vec w\Big\rangle_g\,\dvol_g\\
    &\quad +d*_g \Big( -\frac 12\lap_g H \,  \bn \cdot d \bw+2\,\frac{d}{dt} H_t \Big|_{t=0}\, dH\Big).
\end{aligned}
\end{align}
We define 
\begin{align}\label{defV4d}\begin{aligned}
    \vec V
    &\coloneqq 2\,\langle d\vec{\Phi},dH\rangle_g \,dH -\frac 12\,d(\bn \, \Delta_g H)+ \Delta_g H\,d\bn-|dH|_g^2\,d\bP\\
    &=2\,\langle d\vec{\Phi},dH\rangle_g \,dH- \frac 12\mko\bn \,d(\Delta_g H)+ \frac 12\mko \Delta_g H\mko d\bn-|dH|_g^2\,d\bP.
    \end{aligned}
\end{align}
Then we have $\vec V\in W^{-2,2}+L^1(B^4,\R^5\ot T^*B^4)$.
From~\eqref{ptvardH^2det12} it follows that in $W^{-3,2}+W^{-1,1}(B^4,\bwe^4T^*B^4)$,
\begin{align}\label{ptvardH^2dvol}
\begin{aligned}
   \frac d{dt}\big(|d H_t|_{g_t}^2\,\dvol_{g_t}\big)\Big|_{t=0}&= (*_g\,\vec V) \dwe d\vec w+d*_g \Big( -\frac 12\lap_g H \,  \bn \cdot d \bw+2\,\frac{d}{dt} H_t \Big|_{t=0}\, dH\Big)\\
    &=\vec w\cdot d*_g\vec V-d*_g \Big(\vec V\cdot \bw +\frac 12\mko \lap_g H \,  \bn \cdot d \bw-2\,\frac{d}{dt} H_t \Big|_{t=0}\, dH\Big).
\end{aligned}
\end{align}
As in~\eqref{varwillint}, for $\vec w\in C_c^\nf(B^4,\R^5)$ we obtain
 \begin{align*}
  &\frac d{dt}\, E(\bP_t)\Big|_{t=0}=\frac 12\mko \big\lan d*_g \vec V,\bw \big\ran.
\end{align*}
Therefore, we conclude that $\bP\in W^{3,2}_{\textup{imm}}(B^4,\R^5)$ is a weak critical point of $E$ if and only if $d*_g \vec V=0$.
\end{proof}

 Let $\bP\in W^{3,2}_{\textup{imm}}(B^4,\R^5)$ be a weak critical point of $E$, and define $\vec V$ as in~\eqref{defV4d}. By the embeddings $L^1(B^4)\hookrightarrow W^{-1,(4/3,\nf)}(B^4)\hookrightarrow W^{-2,(2,\nf)}(B^4)$ (see Lemma~\ref{lm:LpembW-1}), we obtain $*_g \,\vec V\in W^{-2,(2,\nf)}\big(B^4,\R^5\ot \bwe^3 T^*B^4\big)$. By weak Poincar\'e lemma (see for instance \cite[Cor.~3.4]{Costa10}) and Lemma~\ref{lm-critic-W^{3,2}}, there exists $\bL\in W^{-1,(2,\infty)}\big(B^4,\R^5\ot \bigwedge ^2 T^*B^4\big)$ such that \begin{align}\label{d*gL=*gV4d}
     d *_g \bL=*_g \,\vec V.
      \end{align}

\subsubsection*{The Noether currents associated to dilations and rotations.}
Since $|d H|_g^2\,\dvol_g$ is pointwise invariant under dilations and rotations, we can apply Noether theorem to find the corresponding conservation laws as in Theorem~\ref{thmcon1}. The operators $\,\resg$ and $\,\res$ are defined in \eqref{defresg}, and we define bilinear maps $\overset{\res}{\we}$, $\wres$ in the same way as in~\eqref{defdwe}: the upper operators act on the $\bwe \R^5$-factors, and the base operators $\we$, $\resg$, act on the $\bwe T^*_x B^4$-factors.
\begin{Prop}
    \label{prop:conlawdilrot}
Let $\bP\in W^{3,2}_{\text{imm}}(B^4, \R^5)$ be a weak critical point of $E$. Then there exists $\bL\in W^{-1,(2,\nf)}(B^4,\R^5\ot \bwe^2 T^*B^4)$ satisfying \eqref{d*gL=*gV4d} and moreover, we have
\begin{subnumcases}{}
    d*_g\big( \vec{L} \,\dres  d\bP+d(H^2)\mko\big)=0,\label{cons-law-dilation}\\[1ex]  d*_g\Big( -\vec L \wres d\vec\Phi -\frac 12\mko \lap_g H\mkt\vec n \we d\vec\Phi\Big)=0.\label{con-law-rotation-4d}
\end{subnumcases}
\end{Prop}
\begin{proof}
To prove~\eqref{cons-law-dilation}, we first consider the variation $\bP_t=(1+t)\bP$ with 
$$
    \bw=\frac d{dt}\mko \bP_t\Big|_{t=0}=\bP.
$$
Denote $H_t$ and $g_t$ as in the proof of Lemma~\ref{lm-critic-W^{3,2}}. Since $|d H_t|_{g_t}^2\,\dvol_{g_t}=|d H|_g^2\,\dvol_g$, combining~\eqref{ptvardH^2dvol}--\eqref{d*gL=*gV4d}, we obtain that
\begin{align}\label{dilconcomp1}
\begin{aligned}
    0&=   \frac d{dt}\big(|d H_t|_{g_t}^2\,\dvol_{g_t}\big)\Big|_{t=0}\\
    &=-\mkt d*_g \Big(\vec V\cdot \bw +\frac 12\mko \lap_g H \,  \bn \cdot d \bw-2\,\frac{d}{dt} H_t \Big|_{t=0}\, dH\Big)\\
    &=-\mkt d\Big(\bP\cdot d*_g\vec L-2\,\frac{d}{dt} H_t \Big|_{t=0}\, *_gdH\Big)\\
    &= d\Big((*_g\,\vec L)\dwe d\bP+2\,\frac{d}{dt} H_t \Big|_{t=0}\, *_gdH\Big).
    \end{aligned}
\end{align}
By~\eqref{ptvarH4d}, we have 
\begin{align*}
    \frac{d}{dt} H_t\Big|_{t=0}=\frac 14\mkt g^{ij}\mkt\p_i\bn\cdot \p_j\bP=-H.
\end{align*}
Then it follows from~\eqref{dilconcomp1} together with~\eqref{*_gcomres} that
\begin{align*}
    d*_g\big( \vec{L} \,\dres  d\bP+d(H^2)\mko\big)=-\mko d\big(\mko(*_g\,\bL)\dwe d\bP-2H*_gdH\big)=0.
\end{align*}
Next, we prove~\eqref{con-law-rotation-4d}.  For $\vec a\in \bwe^2\R^5$ we define $\bP_t$ by
\begin{align*}
  \begin{dcases} 
  \frac d{dt}\mko \bP_t=\vec a\,\res \bP_t, 
  \qquad t\in \R,\\[0.3ex]
  \bP_0=\bP.
  \end{dcases}
\end{align*} 
For this variation, the equation~\eqref{ptvardH^2dvol} remains valid. Since $ \frac d{dt}\mko \bP_t\cdot \bP_t=0$, we have $|\bP_t|^2=|\bP|^2$ for all $t\in \R$. 
Hence, there exists $Q_t\in \text{SO}(5)$ depending only on $\vec a$ and $t$ such that $\bP_t=Q_t\circ \bP$ on $B^4$. Consequently, we have $|d H_t|_{g_t}^2\,\dvol_{g_t}=|d H|_g^2\,\dvol_g$ for all $t\in \R$.  Using the identity $\vec a\,\res(\vec b\we \vec c)=\vec c\cdot (\vec a\,\res \vec b)$ for $\vec b,\vec c\in \R^5$, it follows that
\begin{align}\label{rotconcomp1}
\begin{aligned}
    0&=   \frac d{dt}\big(|d H_t|_{g_t}^2\,\dvol_{g_t}\big)\Big|_{t=0}\\
    &=-\mkt d*_g \Big(\vec V\cdot \bw +\frac 12\mko \lap_g H \,  \bn \cdot d \bw-2\,\frac{d}{dt} H_t \Big|_{t=0}\, dH\Big)\\
    &=-\mkt d\Big((\vec a\,\res \bP)\cdot d*_g\bL+\frac 12\mko \lap_g H \,  \bn \cdot (\vec a\,\res d\bP)-2\,\frac{d}{dt} H_t \Big|_{t=0}\, dH \Big)\\
    &= -\mkt \vec a\,\res d\Big(\bP\we d*_g\vec L+\frac 12\mko \lap_g H\,d\bP\we \bn\Big)+2\mko d\Big(\frac{d}{dt} H_t \Big|_{t=0}\, dH \Big)\\
    &=-\mkt \vec a \,\res d\Big((*_g\,\bL) \ovs{\ovwe}\we d\bP-\frac 12\mko \lap_g H\,\bn\we d\bP\Big)+2\mko d\Big(\frac{d}{dt} H_t \Big|_{t=0}\, dH \Big).
\end{aligned}
\end{align}
By~\eqref{ptvarH4d}, we compute
\begin{align*}
    4\, \frac{d}{dt} H_t \Big|_{t=0}&= *_g\, d\big(\bn\cdot *_g\,d\bw \big)+ d\bn\,\dres d\bw\\
    &=\vec a\,\res \Big(*_g d(*_g\,d\bP\we \bn)+d\bP\wres d\bn\Big)\\
    &=\vec a\,\res \Big(4\vec H\we \bn +*_g\big( d\bn\ovs{\ovwe}{\we}(*_g\,d\bP)\mko\big)+d\bP\wres d\bn \Big)\\
    &=\vec a\,\res \Big(d\bn\wres d\bP+ d\bP\wres d\bn \Big)\\[0.8ex]
    &=0.
\end{align*}
Since $\vec a\in \bwe^2\R^5$ is arbitrary, by~\eqref{rotconcomp1} and~\eqref{*_gcomres} we then obtain
\begin{align*}
    d*_g\Big( -\vec L \wres d\vec\Phi -\frac 12\mko \lap_g H\mkt\vec n \we d\vec\Phi\Big)=d\Big((*_g\,\bL) \ovs{\ovwe}\we d\bP-\frac 12\mko \lap_g H\,\bn\we d\bP\Big)=0.
\end{align*}
This completes the proof.

\end{proof}

By the embedding results in Lemma~\ref{lm:LpembW-1}, for any index $1\le j\le 4$ and any maps $a\in L^\nf \cap W^{1,4}(B^4)$ and $f\in L^{2,\nf}(B^4)$, we have 
\begin{align*}
    \|a \mkt \p_{j} f  \|_{W^{-1,(2,\nf)}(B^4)}
    &\le \|\p_{j}(af)\|_{W^{-1,(2,\nf)}(B^4)}+C \mko\|f\mkt\p_{j} a \|_{L^{\frac{4}{3},\nf}(B^4)}\\
    &\le \|af\|_{L^{2,\nf}(B^4)}+C\mko\|\p_{j} a\|_{L^4(B^4)} \|f\|_{L^{2,\nf}(B^4)}\\[0.7ex]
    &\le C\mko \|a\|_{L^\nf \cap W^{1,4}(B^4)} \mko\|f\|_{L^{2,\nf}(B^4)}.
\end{align*}
Hence, there exists a universal constant $C>0$ such that for all $a\in L^\nf \cap W^{1,4}(B^4)$ and $T\in W^{-1,(2,\nf)}(B^4)$, it holds that
\begin{align}\label{prornd}
    \|aT\|_{W^{-1,(2,\nf)}(B^4)}\le C \mko\|a\|_{L^\nf \cap W^{1,4}(B^4)} \mko \|T\|_{W^{-1,(2,\nf)}(B^4)}.
\end{align}  
Using the inequality~\eqref{prornd}, we obtain that
\begin{align*}
    \begin{dcases}
       *_g\big( \vec{L} \,\dres  d\bP+d(H^2)\mko\big)\in W^{-1,(2,\nf)}\Big(B^4,\bwe\nolimits^3 T^*B^4\Big),\\[0.5ex]
        *_g\Big( -\vec L \wres d\vec\Phi -\frac 12\mko \lap_g H\mkt\vec n \we d\vec\Phi\Big) \in W^{-1,(2,\nf)}\Big(B^4,\bwe\nolimits^2\R^5\ot \bwe\nolimits^3 T^*B^4\Big).
    \end{dcases}
\end{align*}
We define the codifferential $d^{*_g}$ as in~\eqref{defd*glap}. Then by Proposition~\ref{prop:conlawdilrot} and the weak Poincar\'e lemma \cite[Cor.~3.4]{Costa10}, there exist $S\in L^{2,\nf}(B^4,\bigwedge^2T^*B^4)$ and $\vec R\in L^{2,\nf}\big(B^4,\bigwedge^2 \R^5\ot \bigwedge^2T^*B^4\big)$ such that 
\begin{gather}\label{d*SRsys4d}
\begin{dcases}
    d^{*_g} S=-\vec L\,\dres d\bP-d(H^2),\\[1ex]
    d^{*_g} \vec R = -\vec L \wres d\vec\Phi -\frac 12\mko \lap_g H\mkt\vec n \we d\vec\Phi.
            \end{dcases}
\end{gather}

Following Remark~\ref{rm-conf-will}, it is natural to ask in 4 dimension this time the following question.

\begin{op}
\label{op-q-100}
Interpret variationally the following equations: Suppose there exists $\vec{L} $ in $ W^{-1,(2,\nf)}\big(B^4,\R^5\ot \bwe^2 T^*B^4\big)$ such that
\begin{align}\label{4d-will-const}
  \begin{dcases} 
  d*_g\big( \vec{L} \,\dres  d\bP+d(H^2)\mko\big)=0,\\[1ex]  
  \displaystyle      d*_g\Big( -\vec L \wres d\vec\Phi -\frac 12\mko \lap_g H\mkt\vec n \we d\vec\Phi\Big)=0.
\end{dcases}
    \end{align}
Does the system of equations (\ref{4d-will-const}) correspond to Euler--Lagrange equations related to the variations of the Dirichlet Energy of the mean curvature under some constraint?
Recall from  Remark~\ref{rm-conf-will} that in 2 dimension the corresponding system of equations is equivalent to the {\it Conformal Willmore Equations} obtained by taking variations of the Willmore equations with some constraint on the underlying conformal class induced by the metric.   
\end{op}

Now we deduce the counterpart to \eqref{lap3} for weak critical points of $E$.\footnote{Partly due to the fact that $|d H|_g^2\,\dvol_g$ is not conformally invariant, we do not obtain a divergence-form identity as in~\eqref{eq:inv}.}

\begin{Lm}
\label{lm-HSR}
With the above notations there holds
\be
\label{HSR}
2\mko \Delta_g H  \mkt \vec{n}-\langle d(H^2), d\vec{\Phi}\rangle_g= (d^{*_g}\vec{R})\,\overset{\res}{\res}_g d\vec{\Phi}+ d^{*_g}S\,\resg d\vec{\Phi}.
\ee
\end{Lm}
 \begin{proof}
Applying the identity~\eqref{prodrule} to $\,\res$ acting on $\bwe \R^5$, we have 
\begin{align}\begin{aligned}
    \label{II.7-a}
    \big(\vec L \wres d\vec\Phi \big)\,\ovs{\res}{\res}_g d\vec{\Phi}&=-\big(\vec{L}\,\dres d\vec{\Phi}\big)\,\resg d\vec{\Phi}-\vec L \,\resg \big(d\bP\dwe d\bP\big)\\[1ex]
    &=-\big(\vec{L}\,\dres d\vec{\Phi}\big)\,\resg d\vec{\Phi}.
\end{aligned}
\end{align}  
We also have 
\begin{align}\begin{aligned}
\label{II.7-b}
    (\vec n \we d\vec\Phi)\,\ovs{\res}{\res}_g d\vec{\Phi}
    &=d\bP \,\resg \big(\bn\cdot d\bP\big)-\bn\mko \big( d\bP\,\dres d\bP\big)\\[2mm]
    &=-\mko 4\mko \bn.
\end{aligned}\end{align}
Therefore, by~\eqref{d*SRsys4d} we obtain
\begin{align*}    
   (d^{*_g} \vec R) \,\ovs{\res}{\res}_g d\vec{\Phi}+d^{*_g}S\,\resg  d\vec{\Phi}&=-\frac 12\mko \Delta_gH (\vec n \we d\vec\Phi)\,\ovs{\res}{\res}_g d\vec{\Phi}-d(H^2)\,\resg d\vec{\Phi}\\[1ex]
    &=2\mko \Delta_g H  \mkt \vec{n}-\langle d(H^2), d\vec{\Phi}\rangle_g.
 \end{align*}
 This completes the proof.
 \end{proof}

We summarize our results of this section in the following theorem.
\begin{thm}
    Assume $\bP\in W_{imm}^{3,2}(B^4,\R^5)$ is a critical point of the functional 
$$
    E(\vec{\Phi})\coloneqq \frac 12\int_{B^4}\,|dH|_g^2 \, \dvol_g. 
$$
Then there exist $\bL\in W^{-1,(2,\nf)}\big(B^4,\R^5\ot \bwe^2 T^*B^4\big)$, $S\in L^{2,\nf}(B^4,\bigwedge^2T^*B^4)$, and $\vec R\in L^{2,\nf}\big(B^4,\bigwedge^2 \R^5\ot \bigwedge^2T^*B^4\big)$ such that the following system holds:
\begin{equation}\label{sys-LRS}
    \begin{dcases}
        d *_g \bL=*_g\Big(2\,\langle d\vec{\Phi},dH\rangle_g \,dH -\frac 12\,d(\bn \, \Delta_g H)+ \Delta_g H\,d\bn-|dH|_g^2\,d\bP\Big),\\[0.5ex]
        d^{*_g} S=-\vec L\,\dres d\bP-d(H^2),\\[0.5ex]
           d^{*_g} \vec R = -\vec L \wres d\vec\Phi -\frac 12\mko \lap_g H\mkt\vec n \we d\vec\Phi,\\[0.5ex]
        (d^{*_g}\vec{R})\,\overset{\res}{\res}_g d\vec{\Phi}+ d^{*_g}S\,\resg d\vec{\Phi}=2\mko \Delta_g H  \mkt \vec{n}-\langle d(H^2), d\vec{\Phi}\rangle_g.
\end{dcases}
\end{equation}
\end{thm}
As mentioned at the begining of this section, this set of identities and Noether currents obtained through the application of Noether theorem is the starting point to the proof of the regularity of weak immersions $\vec{\Phi}$ critical points to $E$ (see \cite{bernard2025}).


\begin{thebibliography}{99}

\bibitem{Adams75} Adams, David R. \textit{A note on Riesz potentials}. Duke Math. J. {\bf 42} (1975), no.~4, 765--778.
\bibitem{Adams15} Adams, David R. \textit{Morrey spaces}. Lecture Notes in Applied and Numerical Harmonic Analysis. Birkh\"auser/Springer, Cham, 2015.

\bibitem{Adams96} Adams, David R.; Hedberg, Lars Inge.
\textit{Function spaces and potential theory}. Grundlehren der mathematischen Wissenschaften, Vol. 314. Springer-Verlag, Berlin, 1996.

\bibitem{Ahlfors06}Ahlfors, Lars V. \textit{Lectures on quasiconformal mappings}, 2nd edition, with supplemental chapters by C. J. Earle, I. Kra, M. Shishikura and J. H. Hubbard. University Lecture Series, Vol. 38. American Mathematical Society, Providence, RI, 2006.

\bibitem{allard1975} Allard, William K. \textit{On the first variation of a varifold: boundary behavior}. Ann. of Math. (2) {\bf 101} (1975), 418--446.

\bibitem{allen2025} Allen, Ben F.; Gover, A. Rod. \textit{Higher Willmore energies from tractor coupled GJMS operators}. Preprint (2025), \url{https://arxiv.org/abs/2411.01835}.

\bibitem{anastasiou2025} Anastasiou, Giorgos; Araya, Ignacio J.; Bueno, Pablo; Moreno, Javier; Olea, Rodrigo; Vilar Lopez, Alejandro. \textit{Higher-dimensional Willmore energy as holographic entanglement entropy}. J. High Energy Phys. (2025), no.~1, Paper No.~81, 45 pp.

\bibitem{Aubinjp09} Aubin, Jean-Pierre; Frankowska, Hélène. \textit{Set-valued analysis}, reprint of the 1990 edition. Modern Birkhäuser Classics.
Birkhäuser Boston, Inc., Boston, MA, 2009.

\bibitem{aubin1998} Aubin, Thierry. \textit{Some nonlinear problems in {R}iemannian geometry}. Springer Monographs in Mathematics, Springer-Verlag, Berlin, 1998.

\bibitem{babich1993} Babich, M.; Bobenko, A. \textit{Willmore tori with umbilic lines and minimal surfaces in hyperbolic space}. Duke Math. J. {\bf 72} (1993), no.~1, 151--185.


\bibitem{baraket1996} Baraket, Sami. \textit{Estimations of the best constant involving the $L^\infty$ norm in Wente's inequality}. Ann. Fac. Sci. Toulouse Math. (6) {\bf 5} (1996), no.~3, 373--385.

\bibitem{BaKu} Bauer, Matthias; Kuwert, Ernst. \textit{Existence of minimizing Willmore surfaces of prescribed genus}.
Int. Math. Res. Not. 2003, no.~10, 553--576.

\bibitem{Bennett88} Bennett, Colin; Sharpley, Robert. \textit{Interpolation of operators}. Pure and Applied Mathematics, Vol. 129. Academic Press, Inc., Boston, MA, 1988.

\bibitem{Ber} Bernard, Yann. \textit{Noether's theorem and the Willmore functional}. Adv. Calc. Var. {\bf 9} (2016), no.~3, 217--234.

\bibitem{Bernard25}Bernard, Yann. \textit{Structural equations for critical points of conformally invariant curvature energies in $4$d}. Preprint (2025), \url{https://arxiv.org/abs/2509.01179}.

\bibitem{bernard2025} Bernard, Yann; Lan, Tian; Martino, Dorian; Rivière, Tristan. \textit{The Regularity of Critical Points to the Dirichlet Energy of the Mean Curvature in Dimension 4}. Preprint (2025), \url{https://arxiv.org/abs/2511.01765}.

\bibitem{bernard2023} Bernard, Yann; Laurain, Paul; Marque, Nicolas. \textit{Energy estimates for the tracefree curvature of Willmore surfaces and applications}. Arch. Ration. Mech. Anal. {\bf247} (2023), no.~1, Paper No.~8, 28 pp.

\bibitem{bern-rivi-conf} Bernard, Yann; Rivi\`ere, Tristan.
\textit{Local Palais-Smale sequences for the Willmore functional}.
Comm. Anal. Geom. {\bf19} (2011), no.~3, 563--599.

\bibitem{bernard2014} Bernard, Yann; Rivière, Tristan. \textit{Energy quantization for Willmore surfaces and applications}. Ann. of Math. (2) {\bf180} (2014), no.~1, 87–136.

\bibitem{bethuel1993} Bethuel, F.; Ghidaglia, J.-M. \textit{Improved regularity of solutions to elliptic equations involving Jacobians and applications}. J. Math. Pures Appl. (9) {\bf72} (1993), no.~5, 441–474.

\bibitem{Bla}Blaschke, Wilhelm. \textit{Vorlesungen \"uber Differentialgeometrie und geometrische Grundlagen von Einsteins Relativit\"atstheorie III: Differentialgeometrie der Kreise und Kugeln}. Grundlehren der mathematischen Wissenschaften, Vol. 29. Springer, Berlin, 1929.

\bibitem{blatt2009} Blatt, Simon. \textit{A singular example for the Willmore flow}. Analysis (Munich) {\bf 29} (2009), no.~4, 407--430.

\bibitem{blitz2023generalized} Blitz, Samuel; Gover, A. Rod; Waldron, Andrew. \textit{Generalized Willmore energies, $Q$-curvatures, extrinsic Paneitz operators, and extrinsic Laplacian powers}, Commun. Contemp. Math. {\bf 26} (2024), no.~5, Paper No.~2350014, 50 pp.

\bibitem{Bourgain03} Bourgain, Jean; Brezis, Ha\"im. 
\textit{On the equation $\mbox{div } Y = f$ and application to control of
phases}. J. Amer. Math. Soc. {\bf16} (2003), no.~2, 393--426.

\bibitem{brazda2024} Brazda, Katharina; Kružík, Martin; Stefanelli, Ulisse. \textit{Generalized minimizing movements for the varifold Canham-Helfrich flow}. Adv. Calc. Var. {\bf17} (2024), no.~3, 727--751.

\bibitem{brendle2021}Brendle, Simon. \textit{The isoperimetric inequality for a minimal submanifold in Euclidean space}. J. Amer. Math. Soc. {\bf 34} (2021), no.~2, 595--603.

\bibitem{Brezis11} Brezis, Ha\"im. \textit{Functional analysis, Sobolev spaces and partial differential equations}. Universitext. Springer, New York, 2011.

\bibitem{bryant1984} Bryant, Robert L. \textit{A duality theorem for Willmore surfaces}. J. Differential Geom. {\bf 20} (1984), no.~1, 23--53.

\bibitem{bryant1987} Bryant, Robert L. \textit{Surfaces in conformal geometry}. In: The mathematical heritage of Hermann Weyl (Durham, NC, 1987), 227--240. Proc. Sympos. Pure Math., Vol. 48. American Mathematical Society, Providence, RI, 1988.

\bibitem{Bry}  Bryant, Robert L. {\it On the conformal volume of $2$-tori}. Preprint (2025), \url{https://arxiv.org/abs/1507.01485}.

\bibitem{Byun05} Byun, Sun-Sig. \textit{Elliptic equations with BMO coefficients in Lipschitz domains}. Trans. Amer. Math. Soc. {\bf357} (2005), no.~3, 1025--1046.

\bibitem{CL63} Chanillo, Sagun; Li, Yan Yan. \textit{Continuity of solutions of uniformly elliptic equations in $R^2$}. Manuscripta Math. {\bf77} (1992), no.~4, 415--433.

\bibitem{chavel1984} Chavel, Isaac. \textit{Eigenvalues in Riemannian geometry}, including a chapter by Burton Randol, with an appendix by Jozef Dodziuk. Pure and Applied Mathematics, Vol. 115. Academic Press, Inc., Orlando, FL, 1984.

\bibitem{Chen} Chen, Bang-Yen. \textit{Some conformal invariants of submanifolds and their applications}. Boll. Un. Mat. Ital. (4) {\bf10} (1974), 380--385.

\bibitem{chen2013} Chen, Jingyi; Lamm, Tobias. \textit{A Bernstein type theorem for entire Willmore graphs}, J. Geom. Anal. {\bf 23} (2013), no.~1, 456--469.

\bibitem{chen2017} Chen, Jingyi; Li, Yuxiang. \textit{Radially symmetric solutions to the graphic Willmore surface equation}, J. Geom. Anal. {\bf 27} (2017), no.~1, 671--688.

\bibitem{ChLa} Chern, Shiing-shen; Lashof, Richard K.
\textit{On the total curvature of immersed manifolds}.
Amer. J. Math. {\bf79} (1957), 306--318.

\bibitem{Chiarenza93}Chiarenza, Filippo; Frasca, Michele; Longo, Placido. \textit{$W^{2,p}$-solvability of the Dirichlet problem for nondivergence elliptic equations with VMO coefficients}. Trans. Amer. Math. Soc. {\bf336} (1993), no.~2, 841--853.

\bibitem{chodosh2016} Chodosh, Otis; Maximo, Davi. \textit{On the topology and index of minimal surfaces}. J. Differential Geom. {\bf 104} (2016), no.~3, 399--418.

\bibitem{chodosh2023} Chodosh, Otis; Maximo, Davi. \textit{On the topology and index of minimal surfaces II}. J. Differential Geom. {\bf 123} (2023), no.~3, 431--459.
\bibitem{Cianchi98}Cianchi, Andrea; Pick, Luboš.
\textit{Sobolev embeddings into BMO, VMO, and $L_{\infty}$}.
Ark. Mat. 36 (1998), no.~2, 317--340.

\bibitem{CFMOZ}Clop, A.; Faraco, D.; Mateu, J.; Orobitg, J.; Zhong, X. \textit{Beltrami equations with coefficient in the
Sobolev space $W^{1,p}$}. Publ. Mat. {\bf53} (2009), no.~1, 197--230.

\bibitem{CLMS}Coifman, R.; Lions, P.-L.; Meyer, Y.; Semmes, S. \textit{Compensated
compactness and Hardy spaces}. J. Math. Pures. Appl. (9) {\bf72} (1993), no.~3, 247--286.

\bibitem{Costa10}Costabel, Martin; McIntosh, Alan. \textit{On Bogovskiĭ and regularized Poincaré integral operators for de Rham complexes on Lipschitz domains}. Math. Z. {\bf265} (2010), no.~2, 297--320.

\bibitem{curry2018} Curry, Sean N.; Gover, A. Rod. \textit{An introduction to conformal geometry and tractor calculus, with a view to applications in general relativity}. In: Asymptotic analysis in general relativity, 86--170. London Mathematical Society Lecture Note Series, Vol. 443. Cambridge University Press, Cambridge, 2018.

\bibitem{dalio2025}  Da Lio, Francesca; Gianocca, Matilde; Rivière, Tristan. \textit{Morse index stability for critical points to conformally invariant Lagrangians}. J. Eur. Math. Soc. (2025), published online first, DOI 10.4171/JEMS/1718.

\bibitem{dalio2020} Da Lio, Francesca; Palmurella, Francesco; Rivière, Tristan. \textit{A resolution of the Poisson problem for elastic plates}. Arch. Ration. Mech. Anal. {\bf236} (2020), no.~3, 1593--1676.

\bibitem{dalio2021}Da Lio, Francesca; Rivière, Tristan. \textit{Critical chirality in elliptic systems}. Ann. Inst. H. Poincar\'e C Anal. Non Lin\'eaire {\bf38} (2021), no.~5, 1373--1405.

\bibitem{dalio2025b} Da Lio, Francesca; Rivière, Tristan. \textit{Conservation Laws for $p$-Harmonic Systems with Antisymmetric Potentials and Applications}. Arch. Ration. Mech. Anal. {\bf249} (2025), no.~2, Paper No.~21, 51 pp.

\bibitem{dall2024} Dall'Acqua, Anna; Müller, Marius; Schätzle, Reiner M.; Spener, Adrian. \textit{The Willmore flow of tori of revolution}. Anal. PDE {\bf 17} (2024), no.~9, 3079--3124.

\bibitem{dellacqua2024} Dall'Acqua, Anna; Schätzle, Reiner M. \textit{Rotational symmetric Willmore surfaces with umbilic lines}. Preprint (2024), \url{https://www.math.uni-tuebingen.de/user/schaetz/publi/dallacqua-schaetzle-24.pdf}.

\bibitem{Daners08}Daners, Daniel. \textit{Domain perturbation for linear and semi-linear boundary value problems}. Handbook of differential equations: stationary partial differential equations, Vol. 6, 1--81. Elsevier/North-Holland, Amsterdam, 2008.

\bibitem{Dem}Demailly, Jean-Pierre. \textit{Complex analytic and differential geometry}, Universit\'e de Grenoble I, Grenoble, 2012.

\bibitem{dorfmeister2019} Dorfmeister, Josef F.; Wang, Peng. \textit{Willmore surfaces in spheres: the DPW approach via the conformal Gauss map}. Abh. Math. Semin. Univ. Hambg. {\bf89} (2019), no.~1, 77–103.

\bibitem{Durrett19}Durrett, Rick. \textit{Probability: theory and examples},
5th edition.
Cambridge Series in Statistical and Probabilistic Mathematics, Vol. 49. Cambridge University Press, Cambridge, 2019.

\bibitem{ejiri1988} Ejiri, Norio. \textit{Willmore surfaces with a duality in $S^N(1)$}. Proc. London Math. Soc. (3) {\bf 57} (1988), no.~2, 383--416.

\bibitem{escher1998} Escher, Joachim; Mayer, Uwe F.; Simonett, Gieri. \textit{The surface diffusion flow for immersed hypersurfaces}. SIAM J. Math. Anal. {\bf29} (1998), no.~6, 1419--1433.

\bibitem{evans} Evans, Lawrence C. \textit{Partial differential equations}, 2nd edition. Graduate Studies in Mathematics, Vol. 19. American Mathematical Society, Providence, RI, 2010.

\bibitem{Federer96} Federer, Herbert. \textit{Geometric measure theory}, reprint of the 1969 edition. Classics in Mathematics, Vol. 153.
Springer-Verlag, Berlin, 1996.

\bibitem{ferus1990} Ferus, Dirk; Pedit, Franz. \textit{$S^1$-equivariant minimal tori in $S^4$ and $S^1$-equivariant Willmore tori in $S^3$}. Math. Z. {\bf 204} (1990), no.~2, 269--282.

\bibitem{fischer1985} Fischer-Colbrie, D. \textit{On complete minimal surfaces with finite Morse index in three-manifolds}. Invent. Math. {\bf 82} (1985), no.~1, 121--132.

\bibitem{gallot1988} Gallot, Sylvestre. \textit{Isoperimetric inequalities based on integral norms of Ricci curvature}. Colloque Paul Lévy sur les Processus Stochastiques (Palaiseau, 1987). Ast\'erisque No. 157-158 (1988), 191--216.

\bibitem{ge1998} Ge, Yuxin. \textit{Estimations of the best constant involving the $L^2$ norm in Wente's inequality and compact $H$-surfaces in Euclidean space}.
ESAIM Control Optim. Calc. Var. 3 (1998), 263–300.

\bibitem{Germain} Germain, Sophie. \textit{Recherches sur la théorie des surfaces élastiques}. Mme Ve Courcier, Libraire pour les sciences, Paris, 1821.

\bibitem{gianocca2025} Gianocca, Matilde. \textit{Optimal weighted Wente's inequality}. Comm. Partial Differential Equations {\bf 50} (2025), no.~1-2, 245--257.

\bibitem{Gilbarg01} Gilbarg, David; Trudinger, Neil S. \textit{Elliptic partial differential equations of second order}, reprint of the 1998 edition. Classics in Mathematics, Vol. 224. Springer-Verlag, Berlin, 2001.

\bibitem{glaros2019} Glaros, Michael; Gover, A. Rod; Halbasch, Matthew; Waldron, Andrew. \textit{Variational calculus for hypersurface functionals: singular Yamabe problem Willmore energies}. J. Geom. Phys. {\bf138} (2019), 168–193.

\bibitem{gover2017} Gover, A. Rod; Waldron, Andrew. \textit{Renormalized volume}. Comm. Math. Phys. {\bf 354} (2017), no.~3, 1205--1244.

\bibitem{gover2020} Gover, A. Rod; Waldron, Andrew. \textit{A calculus for conformal hypersurfaces and new higher Willmore energy functionals}. Adv. Geom. {\bf 20} (2020), no.~1, 29--60.

\bibitem{gover2021} Gover, A. Rod; Waldron, Andrew. \textit{Conformal hypersurface geometry via a boundary Loewner-Nirenberg-Yamabe problem}. Comm. Anal. Geom. {\bf 29} (2021), no.~4, 779--836.

\bibitem{grafakos2014-2} Grafakos, Loukas. \textit{Modern Fourier analysis}, 3rd edition. Graduate Texts in Mathematics, Vol. 250. Springer, New York, 2014.

\bibitem{graham2017} Graham, C. Robin. \textit{Volume renormalization for singular Yamabe metrics}. Proc. Amer. Math. Soc. {\bf 145} (2017), 1781--1792.

\bibitem{GJMS1992} Graham, C. Robin; Jenne, Ralph; Mason, Lionel J.; Sparling, George A. J. (1992), "Conformally invariant powers of the Laplacian. I. Existence", Journal of the London Mathematical Society, Second Series, 46 (3): 557–565.

\bibitem{graham2014} Graham, C. Robin; Karch, Andreas. \textit{Minimal area submanifolds in AdS $\times$ compact}. J. High Energy Phys. (2014), no.~4, Paper No.~168, 22 pp.

\bibitem{graham2020} Graham, C. Robin; Reichert, Nicholas. \textit{Higher-dimensional Willmore energies via minimal submanifold asymptotics}. Asian J. Math. {\bf 24} (2020), no.~4, 571--610.

\bibitem{graham1999} Graham, C. Robin; Witten, Edward. \textit{Conformal anomaly of submanifold observables in AdS/CFT correspondence}. Nuclear Phys. B {\bf 546} (1999), no.~1-2, 52--64.

\bibitem{guven} Guven, Jemal. \textit{Conformally invariant bending energy for hypersurfaces}. J. Phys. A {\bf 38} (2005), no.~37, 7943--7955.

\bibitem{hanlinpde} Han, Qing; Lin, Fanghua. \textit{Elliptic partial differential equations}, 2nd edition. Courant Lecture Notes, Vol. 1. Courant Institute of Mathematical Sciences, New York; American Mathematical Society, Providence, RI, 2011.

\bibitem{Heinonen00}Heinonen, Juha; Kilpeläinen, Tero. \textit{BLD-mappings in $W^{2,2}$ are locally invertible}. Math. Ann. {\bf318} (2000), no.~2, 391--396.
\bibitem{helein1998} Hélein, Frédéric. \textit{Willmore immersions and loop groups}. J. Differential Geom. {\bf 50} (1998), no.~2, 331--385.
\bibitem{helein2002} Hélein, Frédéric. \textit{Harmonic maps, conservation laws and moving frames}, translated from the 1996 French original, with a foreword by James Eells, 2nd edition. Cambridge Tracts in Mathematics, Vol. 150. Cambridge University Press, Cambridge, 2002. 

\bibitem{HelNd1} Heller, Lynn; Ndiaye, Cheikh Birahim. \textit{First explicit constrained Willmore minimizers of non-rectangular conformal class}. Adv. Math. {\bf386} (2021), Paper No.~107804, 47 pp.

\bibitem{Hencl14}Hencl, Stanislav; Koskela, Pekka. \textit{Lectures on mappings of finite distortion}. Lecture Notes in Mathematics, Vol. 2096. Springer, Cham, 2014.
\bibitem{henningson1998} Henningson, M.; Skenderis, K. \textit{The holographic Weyl anomaly}. J. High Energy Phys. (1998), no.~7, Paper No.~23, 12 pp.
\bibitem{hirsch2024} Hirsch, Jonas; Kusner, Rob; Mäder-Baumdicker, Elena. \textit{Geometry of complete minimal surfaces at infinity and the Willmore index of their inversions}. Calc. Var. Partial Differential Equations {\bf 63} (2024), no.~8, Paper No.~190, 30 pp.
\bibitem{hirsch2023} Hirsch, Jonas; Mäder-Baumdicker, Elena. \textit{On the index of Willmore spheres}. J. Differential Geom. {\bf 124} (2023), no.~1, 37--79.
\bibitem{hsu1992} Hsu, Lucas; Kusner, Rob; Sullivan, John. \textit{Minimizing the squared mean curvature integral for surfaces in space forms}. Experiment. Math. {\bf 1} (1992), no.~3, 191--207. 
\bibitem{hughes1985} Hughes, John F. \textit{Another Proof That Every Eversion of the Sphere Has a Quadruple Point}. Amer. J. Math. {\bf107} (1985), no.~2, 501--505.
\bibitem{Hum} Hummel, Christoph.
\textit{Gromov's compactness theorem for pseudo-holomorphic curves}.
Progress in Mathematics, Vol. 151. Birkh\"auser Verlag, Basel, 1997.
\bibitem{Hunt}Hunt, Richard A.
\textit{On $L(p,q)$ spaces}.
Enseign. Math. (2) {\bf12} (1966), 249--276.
\bibitem{Imayoshi92} Imayoshi, Y.; Taniguchi, M. \textit{An introduction to Teichm\"uller spaces}, translated and revised from the Japanese by the authors. Springer-Verlag, Tokyo, 1992.
\bibitem{Iwaniec79} Iwaniec, Tadeusz; Kopiecki, Ryszard.
\textit{Stability in the differential equations for quasiregular mappings}. In: Analytic functions (Proceedings of the Seventh Conference held at Kozubnik, 1979), 203--214.
Lecture Notes in Mathematics, Vol. 798.
Springer, Berlin, 1980.

\bibitem{Jo80} Jones, Peter W. \textit{Extension Theorems for BMO}. Indiana Univ. Math. J. {\bf29} (1980), no.~1, 41--66.

\bibitem{J06} Jost, J\"urgen. \textit{Compact Riemann surfaces: an introduction to contemporary mathematics}, 3rd edition. Universitext. Springer-Verlag, Berlin, 2006.
\bibitem{Jost17} Jost, J\"urgen. \textit{Riemannian geometry and geometric analysis}, 7th edition. Universitext. Springer, Cham, 2017.

\bibitem{J98}Jost, J\"urgen; Li-Jost, Xianqing. \textit{Calculus of variations}. Cambridge Studies in Advanced Mathematics, Vol. 64. Cambridge University Press, Cambridge, 1998.

\bibitem{keller2014} Keller, Laura Gioia Andrea; Mondino, Andrea; Rivière, Tristan.
\textit{Embedded surfaces of arbitrary genus minimizing the Willmore energy under isoperimetric constraint}. Arch. Ration. Mech. Anal. {\bf212} (2014), no.~2, 645–682.

\bibitem{klingenberg1978} Klingenberg, Wilhelm. {\it Lectures on closed geodesics}. Grundlehren der mathematischen Wissenschaften, Vol. 230. Springer-Verlag, Berlin-New York, 1978.

\bibitem{kohsaka2006} Kohsaka, Yoshihito; Nagasawa, Takeyuki. \textit{On the existence of solutions of the Helfrich flow and its center manifold near spheres}. Differential Integral Equations {\bf19} (2006), no.~2, 121–142.

\bibitem{kusner1989} Kusner, Rob. \textit{Comparison surfaces for the Willmore problem}. Pacific. J. Math. {\bf138} (1989), no.~2, 317--345.

\bibitem{KLW2024} Kusner, Rob; Lü, Ying; Wang, Peng. \textit{The Willmore problem for surfaces with symmetry}. \url{https://arxiv.org/abs/2410.12582}.

\bibitem{kusner} Kusner, Rob; Mondino, Andrea; Schulze, Felix. \textit{Willmore Bending Energy on the Space of Surfaces}. MSRI Emissary, Spring, 2016. \url{https://www.msri.org/system/cms/files/204/files/original/Emissary-2016-Spring-Web.pdf}.

\bibitem{Kuwertli12}Kuwert, Ernst; Li, Yuxiang. \textit{$W^{2,2}$-conformal immersions of a closed Riemann surface into $\mathbb R^n$}. Comm. Anal. Geom. {\bf20} (2012), no.~2, 313--340.

\bibitem{kuwert2001} Kuwert, Ernst; Schätzle, Reiner. \textit{The Willmore flow with small initial energy}. J. Differential Geom. {\bf57} (2001), no.~3, 409--441.

\bibitem{KS13} Kuwert, Ernst; Schätzle, Reiner.
\textit{Minimizers of the Willmore functional under fixed conformal class}.
J. Differential Geom. 93 (2013), no.~3, 471–530.

\bibitem{lamm2015} Lamm, Tobias; Nguyen, Huy The. \textit{Branched Willmore spheres}. J. Reine Angew. Math. {\bf 701} (2015), 169--194.

\bibitem{Lan} Lan, Tian. \textit{Analysis of surface energies depending on the first
and second fundamental form}. Master thesis, ETH Zürich, March 2023.

\bibitem{laurain2025} Laurain, Paul; Martino, Dorian. \textit{Huber theorem revisited in dimensions 2 and 4}. Preprint (2025), \url{https://arxiv.org/abs/2502.05541}.

\bibitem{laurain2018} Laurain, Paul; Rivière, Tristan. \textit{Energy quantization of Willmore surfaces at the boundary of the moduli space}. Duke Math. J. {\bf167} (2018), no.~11, 2073–2124.

\bibitem{LaRi2} Laurain, Paul; Rivi\`ere, Tristan. \textit{Optimal estimate for the gradient of Green's function on degenerating surfaces and applications}. Comm. Anal. Geom. {\bf26} (2018), no.~4, 887--913.

\bibitem{jlsmo} Lee, John M. \textit{Introduction to smooth manifolds}, 2nd edition. Graduate Texts in Mathematics, Vol. 218. Springer, New York, 2013.

\bibitem{LY} Li, Peter; Yau, Shing Tung. \textit{A new conformal invariant and its applications to the Willmore conjecture and the first eigenvalue on compact surfaces}. Invent. Math. {\bf69} (1982), no. 2, 269–291.

\bibitem{li2016} Li, Yuxiang. \textit{Some remarks on Willmore surfaces embedded in $\R^3$}. J. Geom. Anal. {\bf 26} (2016), no.~3, 2411--2424.

\bibitem{li2013} Li, Yuxiang; Luo, Yong; Tang, Hongyan. \textit{On the moving frame of a conformal map from 2-disk into {$\R^n$}}. Calc. Var. Partial Differential Equations {\bf 46} (2013), no.~1-2, 31--37.

\bibitem{li2024} Li, Yuxiang; Yin, Hao; Zhou, Ye. \textit{3-circle Theorem for Willmore surfaces II--degeneration of the complex structure}. Preprint (2024), \url{https://arxiv.org/abs/2411.06453}.

\bibitem{loewner1974} Loewner, Charles; Nirenberg, Louis. \textit{Partial differential equations invariant under conformal or projective transformations}. In: Contributions to analysis (a collection of papers dedicated to Lipman Bers), 245–272. Academic Press [Harcourt Brace Jovanovich, Publishers], New York-London, 1974.

\bibitem{luo2014} Luo, Yong; Sun, Jun. \textit{Remarks on a Bernstein type theorem for entire Willmore graphs in $R^3$}. J. Geom. Anal. {\bf 24} (2014), no.~3, 1613--1618.

\bibitem{MPW2025} Ma, Xiang; Pedit, Franz; Wang, Peng. \textit{Classification of Willmore $2$-spheres in $S^n$}.\url{https://arxiv.org/abs/2512.00540}.

\bibitem{ma2017} 
Ma, Xiang; Wang, Changping; Wang, Peng. \textit{Classification of Willmore two-spheres in the 5-dimensional sphere}. J. Differential Geom. {\bf 106} (2017), no.~2, 245--281.

\bibitem{ma2006} Ma, Xiang. \textit{Adjoint transform of Willmore surfaces in $\mathbb S^n$}. Manuscripta Math. {\bf120} (2006), no.~2, 163–179.

\bibitem{ma} Ma, Xiang; Wang, Peng. \textit{Willmore 2-spheres in $S^n$: a survey}. In: Geometry and topology of manifolds (Shanghai, 2014), 211--233. Springer Proceedings in Mathematics \& Statistics, Vol. 154. Springer, Tokyo, 2016.

\bibitem{mader2025} Mäder-Baumdicker, Elena; Seidel, Jona. \textit{The Willmore energy landscape of spheres and avoidable singularities of the Willmore flow}. Preprint (2025), \url{https://arxiv.org/abs/2506.23359}.

\bibitem{maldacena1998} Maldacena, Juan. \textit{The large $N$ limit of superconformal field theories and supergravity}. Adv. Theor. Math. Phys. {\bf 2} (1998), no.~2, 231--252.

\bibitem{marque19} Marque, Nicolas. \textit{Moduli spaces of Willmore immersions}. Doctoral thesis, Université Paris Cité, December 2019.

\bibitem{marque2021} Marque, Nicolas. \textit{Conformal Gauss map geometry and application to Willmore surfaces in model spaces}. Potential Anal. {\bf 54} (2021), no.~2, 227--271.

\bibitem{marque20212} Marque, Nicolas. \textit{Minimal bubbling for Willmore surfaces}. Int. Math. Res. Not. IMRN 2021, no.~23, 17708--17765.

\bibitem{marque2025} Marque, Nicolas; Martino, Dorian. \textit{Some global properties of umbilic points of Willmore immersions in the 3-sphere}. Preprint (2025), \url{https://arxiv.org/abs/2506.10720}.

\bibitem{Marques14}Marques, Fernando C.; Neves, Andr\'e. \textit{Min-max theory and the Willmore conjecture}. Ann. of Math. {\bf179} (2014), no.~2, 683--782.

\bibitem{martino2024duality} Martino, Dorian. \textit{A duality theorem for a four dimensional Willmore energy}. Preprint (2024), \url{https://arxiv.org/abs/2308.11433}.

\bibitem{martino2024} Martino, Dorian. \textit{Classification of branched Willmore spheres}. Preprint (2024), Proc. Amer. Math. Soc. (Accepted), \url{https://arxiv.org/abs/2306.07965}.

\bibitem{martino2025} Martino, Dorian. \textit{Energy quantization for Willmore surfaces with bounded index}. J. Eur. Math. Soc. (2025), published online first, DOI 10.4171/JEMS/1629.

\bibitem{MarRiv1} Martino, Dorian; Rivière, Tristan. \textit{Construction of harmonic coordinates for weak immersions}. Preprint (2025), \url{https://arxiv.org/abs/2510.10601}.

\bibitem{MarRiv2} Martino, Dorian; Rivière, Tristan. \textit{Weak immersions with second fundamental form in a critical Sobolev space}. Preprint (2025), \url{https://arxiv.org/abs/2510.10594}.

\bibitem{martio88} Martio, O.; Väisälä, J. \textit{Elliptic equations and maps of bounded length distortion}. Math. Ann. {\bf282} (1988), no.~3, 423--443.

\bibitem{max1981} Max, Nelson; Banchoff, Tom. \textit{Every sphere eversion has a quadruple point}. In: Contributions to Analysis and Geometry (Baltimore, MD, 1980), 191--209. Johns Hopkins University Press, Baltimore, MD, 1981.
\bibitem{mayer2003} Mayer, Uwe F.; Simonett, Gieri. \textit{Self-intersections for Willmore flow}. In: Evolution equations: applications to physics, industry, life sciences and economics (Levico Terme, 2000), 341--348. Progress in Nonlinear Differential Equations and Their Applications, Vol. 55. Birkhäuser Verlag, Basel, 2003.
\bibitem{Mazya2011} Maz'ya, Vladimir. \textit{Sobolev spaces: with applications to elliptic partial differential equations}, 2nd, revised and augmented edition. Grundlehren der
mathematischen Wissenschaften, Vol. 342. Springer, Heidelberg, 2011.

\bibitem{Me63}Meyers, Norman G. \textit{An $L^p$-estimate for the gradient of solutions of second order elliptic divergence equations}. Ann. Scuola Norm. Sup. Pisa Cl. Sci. (3) {\bf17} (1963), 189--206.

\bibitem{michael1973} Michael, J. H.; Simon, L. M. \textit{Sobolev and mean-value inequalities on generalized submanifolds of $R^n$}. Comm. Pure Appl. Math. {\bf26} (1973), 361–379.

\bibitem{michelat2018} Michelat, Alexis. \textit{Morse index estimates of min-max Willmore surfaces}. Preprint (2018), \url{https://arxiv.org/abs/1808.07700}.

\bibitem{michelat-these} Michelat, Alexis. \textit{Morse theoretical aspects of the Willmore energy}. Doctoral thesis, Diss. ETH No.~26223, 2019.

\bibitem{michelat2020} Michelat, Alexis. \textit{On the Morse index of Willmore spheres in $S^3$}. Comm. Anal. Geom. {\bf28} (2020), no.~6, 1337--1406.

\bibitem{michelat2025} Michelat, Alexis. \textit{Morse index stability of branched Willmore immersions}. Preprint (2025), \url{https://arxiv.org/abs/1808.07700}.

\bibitem{michelat2016} Michelat, Alexis; Rivière, Tristan. \textit{A viscosity method for the min-max construction of closed geodesics}. ESAIM Control Optim. Calc. Var. {\bf 22} (2016), no.~4, 1282--1324.


\bibitem{michelat2022} Michelat, Alexis; Rivière, Tristan. \textit{The classification of branched Willmore spheres in the 3-sphere and the 4-sphere}. Ann. Sci. \'Ec. Norm. Sup\'er. (4) {\bf 55} (2022), no.~5, 1199--1288.

\bibitem{MiRi2023} Michelat, A., Rivière, T. Pointwise Expansion of Degenerating Immersions of Finite Total Curvature. J Geom Anal 33, 24 (2023). 

\bibitem{michelat2023} Michelat, Alexis; Rivière, Tristan. \textit{Morse {i}ndex {s}tability of {W}illmore {i}mmersions}, Memoirs of the European Mathematical Society, \emph{to appear}, memoir combining the two articles \url{https://arxiv.org/abs/2306.04608} and \url{https://arxiv.org/abs/2306.04609}, 2025+.

\bibitem{Mingi11}Mingione, Giuseppe. \textit{Gradient potential estimates}. J. Eur. Math. Soc. (JEMS) {\bf13} (2011), no.~2, 459--486.

\bibitem{mondino2018} Mondino, Andrea; Nguyen, Huy T. \textit{Global conformal invariants of submanifolds}. Ann. Inst. Fourier (Grenoble) {\bf68} (2018), no.~6, 2663--2695.

\bibitem{MoRo} Montiel, Sebastián; Ros, Antonio. \textit{Minimal immersions of surfaces by the first eigenfunctions and conformal area}. Invent. Math. {\bf83} (1986), no. 1, 153–166.


\bibitem{montiel2000} Montiel, Sebastián. \textit{Willmore two-spheres in the four-sphere}. Trans. Amer. Math. Soc. {\bf 352} (2000), no.~10, 4469--4486.
\bibitem{moore2017} Moore, John Douglas. {\it Introduction to global analysis: minimal surfaces in Riemannian manifolds}. Graduate Studies in Mathematics, Vol. 187. American Mathematical Society, Providence, RI, 2017.

\bibitem{Morrey08}Morrey, Charles B., Jr. \textit{Multiple integrals in the calculus of variations},
reprint of the 1966 edition. Classics in Mathematics. Springer-Verlag, Berlin, 2008.

\bibitem{MS1995} Müller, S.; Šverák, V. \textit{On surfaces of finite total curvature}. J. Differential Geom. {\bf42} (1995), no.~2, 229–258.

\bibitem{Nag88}Nag, Subhashis. \textit{The complex analytic theory of Teichm\"uller spaces}. Canadian Mathematical Society Series of Monographs and Advanced Texts. Wiley-Interscience Publication, John Wiley \& Sons, Inc., New York, 1988.

\bibitem{ndiaye2015} Ndiaye, C.B.; Schätzle, R.M. \textit{New examples of conformally constrained Willmore minimizers of explicit type}. Adv. Calc. Var. {\bf8} (2015), no. 4, 291–319.

\bibitem{Noe} Noether, E. \textit{Invariante Variationsprobleme}. Nachr. König. Gesellsch. Wiss. Göttingen Math-Phys. Kl. (1918), 235--257.

\bibitem{palmer1991} Palmer, Bennett. \textit{The conformal Gauss map and the stability of Willmore surfaces}. Ann. Global Anal. Geom. {\bf 9} (1991), no.~3, 305--317.
\bibitem{palmurella2022} Palmurella, Francesco; Rivière, Tristan. \textit{The parametric approach to the Willmore flow}. Adv. Math. {\bf 400} (2022), Paper No.~108257, 48 pp.
\bibitem{palmurella2024}  Palmurella, Francesco; Rivière, Tristan. \textit{The parametric Willmore flow}. J. Reine Angew. Math. {\bf 811} (2024), 1--91.

\bibitem{pinkall1985} Pinkall, U. \textit{Hopf tori in $S^3$}. Invent. Math. {\bf 81} (1985), no.~2, 379--386.
\bibitem{Plot}Plotnikov, P. I. \textit{Isothermal Coordinates of $W^{2,2}$ Immersions: A Counterexample}. Tr. Mat. Inst. Steklova {\bf327} (2024), Matematicheskie Aspekty Mekhaniki, 265--282.

\bibitem{poisson} Poisson, Sim\'eon Denis. \textit{M\'emoire sur les surfaces élastiques}. Mem. Cl. Sei. Math. Phys., Inst. de France (1814), 167--225.

\bibitem{pozetta2021} Pozzetta, Marco. \textit{On the Plateau--Douglas problem for the Willmore energy of surfaces with planar boundary curves}. ESAIM Control Optim. Calc. Var. {\bf27} (2021), Paper No.~S2, 35 pp.
\bibitem{Rei74} Reimann, H. M. \textit{Functions of bounded mean oscillation and quasiconformal mappings}. Comment. Math. Helv. {\bf49} (1974), 260--276.
\bibitem{Reshetnyak89} Reshetnyak, Yu. G. \textit{Space mappings with bounded distortion}, translated from the Russian by H. H. McFaden. Translations of Mathematical Monographs, Vol. 73. American Mathematical Society, Providence, RI, 1989.
\bibitem{Riv08} Rivi\`ere, Tristan. \textit{Analysis aspects of Willmore surfaces}. Invent. Math. {\bf174} (2008), no.~1, 1--45.

 \bibitem{Riv13}  Rivi\`ere, Tristan. \textit{Sequences of smooth global isothermic immersions}. Comm. Partial Differential Equations {\bf38} (2013), no.~2, 276--303.

\bibitem{Riv14} Rivi\`ere, Tristan. \textit{Variational principles for immersed surfaces with $L^2$-bounded second fundamental form}. J. Reine Angew. Math. {\bf695} (2014), 41--98.
\bibitem{Riv15}  Rivi\`ere, Tristan. \textit{Critical weak immersed surfaces within sub-manifolds of the Teichm\"uller space}. Adv. Math. {\bf283} (2015), 232--274.
\bibitem{Ri16} Rivi\`ere, Tristan. \textit{Weak immersions of surfaces with $L^2$-bounded second fundamental form}. In: Geometric analysis, 303--384.
IAS/Park City Mathematics Series, Vol. 22. American Mathematical Society, Providence, RI, 2016. 
\bibitem{riviere2017}Rivi\`ere, Tristan. \textit{A viscosity method in the min-max theory of minimal surfaces}. Publ. Math. Inst. Hautes Études Sci. {\bf126} (2017), 177--246.
\bibitem{riviere2021} Rivi\`ere, Tristan. \textit{Willmore minmax surfaces and the cost of the sphere eversion}. J. Eur. Math. Soc. (JEMS) {\bf23} (2021), no.~2, 349--423.

\bibitem{rupp2023} Rupp, Fabian. \textit{The volume-preserving Willmore flow}. Nonlinear Anal. {\bf230} (2023), Paper No.~113220, 30 pp.

\bibitem{rupp20240} Rupp, Fabian. \textit{The Willmore flow with prescribed isoperimetric ratio}. Comm. Partial Differential Equations {\bf49} (2024), no.~1-2, 148–184.

\bibitem{rupp2024} Rupp, Fabian; Scharrer, Christian; Schlierf, Manuel. \textit{Gradient flow dynamics for cell membranes in the Canham-Helfrich model}. Preprint (2024), \url{https://arxiv.org/abs/2408.07493}.

\bibitem{scharrer2022} Scharrer, Christian. \textit{Embedded Delaunay tori and their Willmore energy}. Nonlinear Anal. {\bf223} (2022), Paper No.~113010, 23 pp.

\bibitem{scharrer20252} Scharrer, Christian; West, Alexander. \textit{Energy quantization for constrained Willmore surfaces}. Preprint (2025), \url{https://arxiv.org/abs/2505.19898}.

\bibitem{scharrer2025} Scharrer, Christian; West, Alexander. \textit{On the minimization of the Willmore energy under a constraint on total mean curvature and area}. Arch. Ration. Mech. Anal. {\bf249} (2025), no.~2, Paper No.~17, 62 pp.

\bibitem{schatzle2017} Schätzle, Reiner M. \textit{The umbilic set of Willmore surfaces}. Preprint (2017), \url{https://arxiv.org/abs/1710.06127}.

\bibitem{SU} Schoen, Richard; Uhlenbeck, Karen K. \textit{Boundary regularity and the Dirichlet problem for harmonic maps}. J. Differential Geom. {\bf18} (1983), no.~2, 253--268.

\bibitem{schygulla2012} Schygulla, Johannes. \textit{Willmore minimizers with prescribed isoperimetric ratio}. Arch. Ration. Mech. Anal. {\bf203} (2012), no.~3, 901–941.


 
\bibitem{Sim} Simon, Leon. \textit{
Existence of surfaces minimizing the Willmore functional}.
Comm. Anal. Geom. {\bf1} (1993), no.~2, 281--326.
\bibitem{simonett2001} Simonett, Gieri. \textit{The Willmore flow near spheres}. Differential Integral Equations {\bf14} (2001), no.~8, 1005--1014.

\bibitem{steinhar} Stein, Elias M. \textit{Harmonic analysis: real-variable methods, orthogonality, and oscillatory integrals}, with the assistance of Timothy S. Murphy. Princeton Mathematical Series, Vol. 43. Princeton University Press, Princeton, NJ, 1993.

\bibitem{thom} Thomsen, G. \textit{Über konforme Geometrie I: Grundlagen der konformen Flächentheorie}. Abh. Math. Sem. Univ. Hamburg {\bf3} (1924), no.~1, 31--56.


\bibitem{topping1997} Topping, Peter. \textit{The optimal constant in Wente's $L^\nf$ estimate}. Comment. Math. Helv. {\bf72} (1997), no.~2, 316–328.

\bibitem{Tromba92} Tromba, Anthony J. \textit{Teichm\"uller theory in Riemannian geometry}, lecture notes prepared by Jochen Denzler. Lectures in Mathematics ETH Zürich. Birkhäuser Verlag, Basel, 1992.

\bibitem{Uh82}Uhlenbeck, Karen K. \textit{Connections with $L^p$ bounds on curvature}. Comm. Math. Phys. {\bf83} (1982), no.~1, 31--42.

\bibitem{Weiner78}Weiner, Joel L. \textit{On a problem of Chen, Willmore, et al}. Indiana Univ. Math. J. {\bf27} (1978), no.~1, 19–35.

\bibitem{Wen} Wente, Henry C. \textit{An existence theorem for surfaces of constant mean curvature}. J. Math. Anal. Appl. {\bf26} (1969), 318--344.

\bibitem{Wil}Willmore, T. J. \textit{Note on embedded surfaces}. An. Şti. Univ. ``Al. I. Cuza'' Iaşi Secţ. I a Mat. (N.S.) {\bf11B} (1965), 493--496.

\bibitem{wu2025} Wu, Nan; Yan, Zetian. \textit{Energy reduction for Fourth order Willmore energy}. Preprint (2025), \url{https://arxiv.org/abs/2504.08105}.

\bibitem{wu2023} Wu, Yunqing. \textit{Classification of Willmore surfaces with vanishing Gaussian curvature}. J. Geom. Anal. {\bf 33} (2023), no.~7, Paper No.~209, 22 pp.

\bibitem{xia2004} Xia, Qiao Ling; Shen, Yi Bing. \textit{Weierstrass type representation of Willmore surfaces in $S^n$}. Acta Math. Sin. (Engl. Ser.) {\bf20} (2004), no.~6, 1029–1046.

\bibitem{zhang} Zhang, Yongbing. \textit{Graham-Witten's conformal invariant for closed four-dimensional submanifolds}. J. Math. Study {\bf 54} (2021), no.~2, 200--226.

\bibitem{Zie} Ziemer, William P. \textit{Weakly differentiable functions: Sobolev spaces and functions of bounded variation}. Graduate Texts in Mathematics, Vol. 120. Springer-Verlag, New York, 1989. 

\bibitem{zwiebach2009} Zwiebach, Barton. {\it A first course in string theory}, with a foreword by David Gross, 2nd edition. Cambridge University Press, Cambridge, 2009.
\end{thebibliography}
\end{document}